\documentclass[10.9pt,a4paper]{amsart}
\usepackage{a4wide}
\usepackage[utf8]{inputenc}
\usepackage[english]{babel}

\usepackage{amsfonts,amssymb,amsmath}

\usepackage{thmtools}
\usepackage{comment}

\usepackage{graphicx}

\usepackage[dvipsnames]{xcolor}
\usepackage[colorlinks,
    linkcolor={blue!50!black},
    citecolor={blue!50!black},
    urlcolor={black!80!black}]{hyperref}
 \usepackage{enumerate}

\usepackage[capitalize]{cleveref}

\usepackage{tikz}
\usepackage{tikz-cd}
\usepackage[all]{xy}
\usepackage{enumitem}
\makeatletter
\newcommand{\mylabel}[2]{#2\def\@currentlabel{#2}\label{#1}}
\usetikzlibrary{calc, matrix, arrows, cd}

\newcommand{\bsm}{\left(\begin{smallmatrix}}
\newcommand{\esm}{\end{smallmatrix}\right)}
\newtheorem{innercustomthm}{Theorem}
\newenvironment{customthm}[1]
  {\renewcommand\theinnercustomthm{#1}\innercustomthm}
  {\endinnercustomthm}
  \newtheorem{innercustomprop}{Proposition}
\newenvironment{customprop}[1]
  {\renewcommand\theinnercustomprop{#1}\innercustomprop}
  {\endinnercustomprop}

\numberwithin{equation}{section}

\newtheorem{theorem}[equation]{Theorem}

\newtheorem{lemma}[equation]{Lemma}
\newtheorem{proposition}[equation]{Proposition}

\theoremstyle{definition}
\newtheorem{definition}[equation]{Definition}

\newtheorem{example}[equation]{Example}

\newtheorem{remark}[equation]{Remark}
\newtheorem{plan}[equation]{Plan}
\newtheorem{convention}[equation]{Convention}
\newtheorem{construction}[equation]{Construction}
\newtheorem{notation}[equation]{Notation}
\newtheorem*{claim}{Claim}
\newtheorem*{claim*}{Claim}

\newcommand{\Z}{\mathbb{Z}}
\newcommand{\Q}{\mathbb{Q}}

\DeclareMathOperator{\exc}{exc}
\newcommand{\Hom}{\operatorname{Hom}}
\newcommand{\Kun}{\operatorname{Kun}}
\newcommand{\Homotopy}{\operatorname{Homotopy}}
\newcommand{\proj}{\operatorname{proj}}
\newcommand{\pt}{\operatorname{pt}}

\newcommand{\incl}{\operatorname{incl}}
\newcommand{\cd}{\operatorname{cd}}
\newcommand{\pr}{\operatorname{pr}}
\newcommand{\PD}{\operatorname{PD}}
\newcommand{\hAut}{\operatorname{hAut}}
\newcommand{\coker}{\operatorname{coker}}
\newcommand{\Ext}{\operatorname{Ext}}

\newcommand{\Herm}{\operatorname{Herm}}
\newcommand{\id}{\operatorname{id}}

\newcommand{\ks}{\operatorname{ks}}
\newcommand{\CW}{\operatorname{CW}}
\newcommand{\ev}{\operatorname{ev}}

\newcommand{\im}{\operatorname{im}}
\newcommand{\wt}{\widetilde}
\newcommand{\wh}{\widehat}
\DeclareMathOperator{\Res}{Res}

\DeclareMathOperator{\Ind}{Ind}
\DeclareMathOperator{\Sesq}{Sesq}

 \newcommand{\mult}{\operatorname{mult}}
  \newcommand{\btimes}{\boxtimes}
  
  \newcommand{\unaryminus}{\scalebox{0.75}[1.0]{\( - \)}}
 
 \definecolor{bettergreen}{rgb}{0.0, 0.5 0.0}

\begin{document}
\title{4-manifolds with a given boundary}

\author[A.~Conway]{Anthony Conway}
\address{The University of Texas at Austin, Austin TX}
\email{anthony.conway@austin.utexas.edu}
\author[D.~Kasprowski]{Daniel Kasprowski}
\address{University of Southampton, United Kingdom}
\email{d.kasprowski@soton.ac.uk }

\maketitle

\begin{abstract}
This paper studies the homotopy and homeomorphism classifications of~$4$-manifolds with boundary.
Given $4$-manifolds $X_0$ and $X_1$ with fundamental group $\pi$, we consider the problem of extending
 a homotopy equivalence~$h \colon \partial X_0 \to \partial X_1$ 
 to a homotopy equivalence~$X_0 \to~X_1$.
 We solve this problem in broad settings for a class of groups that includes free groups,  finite cyclic groups,  finite dihedral groups,  solvable Baumslag-Solitar groups,  and many~$3$-manifold groups.
When the fundamental group is additionally assumed to be good, we use surgery theory to list situations when a homeomorphism  $h \colon \partial X_0 \to \partial X_1$ extends to a homeomorphism~$X_0 \to X_1$.
The outcome recovers results of Boyer in the simply-connected case and work of the first author and Powell when~$\pi \cong \Z$ and the~$\partial X_i$ have torsion Alexander module.
\end{abstract}

\medbreak

\section{Introduction}
\label{sec:Introduction}

Known homeomorphism classifications of closed orientable topological~$4$-manifolds include the case where the fundamental group is trivial~\cite{Freedman}, finite cyclic~\cite{HambletonKreck}, infinite cyclic~\cite{FreedmanQuinn}, and a solvable Baumslag-Solitar group~$BS(1,n)$~\cite{HambletonKreckTeichner}.
When the boundary is nonempty, a complete classification is known for simply-connected 4-manifolds~\cite{BoyerUniqueness,BoyerRealization} and for~$4$-manifolds with~$\pi_1 \cong \Z$ and certain restrictions on the boundary~\cite{ConwayPowell,ConwayPiccirilloPowell}.

A little more is known concerning the problem of deciding when two closed $4$-manifolds are homotopy equivalent since one need not assume the fundamental group be good.
The earliest results go back to Whitehead~\cite{Whitehead4manifold} and Wall~\cite{WallPoincare} and now include several classes of finite groups~\cite{HambletonKreck,Bauer,KasprowskiPowellRuppik,KasprowskiNicholsonRuppik} as well as~$\PD_3$-groups~\cite{HKPR}.
For manifolds with boundary, there are no results besides those for which the homeomorphism classification is known.

This article achieves the homotopy classification of~$4$-manifolds with a given boundary for large classes of groups under the assumption that the inclusion induced map~$\pi_1(\partial X) \to \pi_1(X)$ is surjective (Theorems~\ref{thm:IntroTorsion} and~\ref{thm:1d3dIntro}).
Surgery theory is then used to derive the homeomorphism classification in several cases (Theorems~\ref{thm:Homeopi1=ZIntro} and~\ref{thm:Homeopi1=BSIntro}).
Our results hold more generally for~$4$-dimensional Poincar\'e pairs, which are pairs~$(X,Y)$ that are homotopy equivalent to finite CW pairs and for which there are classes~$[X] \in H_4(X,Y)$ and~$[Y] \in H_3(Y)$ such that the cap products with~$[X]$ and~$[Y]$ induce Poincar\'e duality isomorphisms (see Definition~\ref{def:PoincarePair}).
During this introduction however, we restrict ourselves to manifolds for simplicity.
We work in the topological category.
Manifolds are assumed to be compact,  path-connected and~oriented.

\subsection{The homotopy classification of~$4$-manifolds  with boundary.}

Given~$4$-manifolds~$X_0$ and~$X_1$ with isomorphic fundamental groups, connected boundaries and~$\iota_j \colon \pi_1(\partial X_j) \to \pi_1(X_j)$ surjective for~$j=0,1$,  we aim to find necessary and sufficient conditions for a degree one homotopy equivalence~$h \colon \partial X_0 \to \partial X_1$ to extend to a homotopy equivalence~$X_0 \to X_1$.

For such an extension to exist, it is necessary that there be an isomorphism~$u \colon \pi_1(X_0) \to \pi_1(X_1)$ with $u \circ \iota_0=\iota_1 \circ h_*$ and,  if such an isomorphism exists, then it is unique.
Henceforth when such an isomorphism exists, we use it to identify~$\pi_1(X_0)$ with $\pi:=\pi_1(X_1)$.
Another necessary condition can be formulated in terms of the \emph{relative equivariant intersection form}
$$\lambda_{X_i}^\partial \colon H_2(X_i;\Z[\pi]) \times  H_2(X_i,\partial X_i ;\Z[\pi]) \to \Z[\pi].$$
Namely we say that $\Z[\pi]$-isomorphisms
\begin{align*}
&F \colon H_2(X_0;\Z[\pi]) \to H_2(X_1;\Z[\pi]) \\
&G \colon H_2(X_0,\partial X_0;\Z[\pi]) \to H_2(X_1,\partial X_1;\Z[\pi])
\end{align*}
are \emph{compatible} with the homotopy equivalence $h$ if they fit into the commutative diagram
$$
\xymatrix@C0.4cm{
H_2(\partial X_0;\Z[\pi])
\ar[r]\ar[d]_{h_*}&
H_2(X_0;\Z[\pi])
\ar[r]\ar[d]_{F}&
H_2(X_0,\partial X_0;\Z[\pi])
\ar[r]\ar[d]_{G}&
H_1(\partial X_0;\Z[\pi])
\ar[d]_{h_*}\\
H_2(\partial X_1;\Z[\pi])
\ar[r]&
H_2(X_1;\Z[\pi])
\ar[r]&
H_2(X_1,\partial X_1;\Z[\pi])
\ar[r]&
H_1(\partial X_1;\Z[\pi])
}
$$
and intertwine $\lambda_{X_0}^\partial $ and $\lambda_{X_1}^\partial$, meaning that for all~$x \in H_2(X_0;\Z[\pi])$ and all~$y \in H_2(X_0,\partial X_0;\Z[\pi])$,  
$$\lambda_{X_1}^\partial(F(x),G(y))=\lambda_{X_0}^\partial(x,y).$$
We call $(F,G,h)$ a \emph{compatible triple.}
A rapid verification shows that the existence of a compatible triple is a  necessary condition for $h$ to extend to a homotopy equivalence.

\medskip

We describe our solution to the problem of extending $h \colon \partial X_0 \to \partial X_1$ in increasing order of complexity, starting with the case where~$H_1(\partial X_i;\Z[\pi])^*:=\Hom_{\Z[\pi]}(H_1(\partial X_i;\Z[\pi]),\Z[\pi])$ vanishes.
When~$\pi$ is free abelian, this is equivalent to 
$H_1(\partial X_i;\Z[\pi])$ being $\Z[\pi]$-torsion.

Our first result holds for some classes of groups~$\pi$ that are either finite or of dimension at~most~$3$.

\begin{theorem}
\label{thm:IntroTorsion}
Let $\pi$ be a group that is either finite abelian with at most $2$ generators, or finite dihedral,  or a surface group, or a solvable Baumslag--Solitar group, or a torsion-free $3$-manifold group,
let~$X_0$ and~$X_1$ be~$4$-manifolds with~$\pi_1(X_j) \cong \pi$ and~$\iota_j \colon \pi_1(\partial X_j) \to \pi_1(X_j)$ surjective for~$j=0,1$. 
Fix a degree one homotopy equivalence~$h \colon \partial X_0 \to \partial X_1$.

When~$H_1(\partial X_1;\Z[\pi])^*=0$,  the following assertions are equivalent:
\begin{itemize}
\item The homotopy equivalence~$h$ extends to a homotopy equivalence~$X_0 \to X_1$.
\item There exists an isomorphism~$u \colon \pi_1(X_0) \to \pi_1(X_1)$ with $u \circ \iota_0=\iota_1 \circ h_*$ and a compatible triple~$(F,G,h)$ with 
$$F_*(k_{X_0,\partial X_0})=h^*(k_{X_1,\partial X_1}) \in H^3(B\pi,\partial X_0;\pi_2(X_1)).$$
When the group is infinite,
the condition involving the relative $k$-invariants can be omitted.
\end{itemize}
Additionally,  given a $k$-invariant preserving compatible triple $(F,G,h)$ as above,  the homotopy equivalence $X_0 \to X_1$ extending $h$ can be chosen to induce $F$ and $G$ on $\Z[\pi]$-homology.
\end{theorem}

Here, a \emph{$3$-manifold group} refers to the fundamental group of a closed $3$-manifold, and~$k_{X_i,\partial X_i}$ denotes the \emph{relative~$k$-invariant} of~$(X_i,\partial X_i)$ for $i=0,1$.
We refer to~\cite{ConwayKasprowskiKinvariant} for further details on the definition of $k_{X_i,\partial X_i}$: in this article,  only its properties are needed, and these are collected in Section~\ref{sec:Relativek}.
We nevertheless mention that,  essentially,~$k_{X_i,\partial X_i}$ is the
obstruction to extending the map~$\partial X_i \to X_i \to P_2(X_i)$ to a section~$B\pi \to P_2(X_i)$ of the fibration~$P_2(X_i) \to B\pi$.
Here,~$P_2(X_i)$ denotes the Postnikov $2$-type of $X_i$.

\begin{remark}
\label{rem:CompatiblePairTripleIntro}
If $h$ is a homeomorphism, the case~$\pi=1$ of Theorem~\ref{thm:IntroTorsion} can be recovered from work of Boyer~\cite[Theorem~0.7 and Proposition~0.8]{BoyerUniqueness}, whereas the case~$\pi=\Z$ can be deduced from~\cite[Theorem~1.10]{ConwayPowell}.
Theorem~\ref{thm:IntroTorsion} is new in all other cases.
For $\pi \in \{1,\Z\}$, 
the results of~\cite{BoyerUniqueness,ConwayPowell} are formulated in terms of compatible pairs $(F,h)$, i.e. pairs $(F,h)$
with
$$
\xymatrix@C0.4cm@R0.4cm{
H_2(\partial X_0;\Z[\pi])
\ar[r]\ar[d]_{h_*}&
H_2(X_0;\Z[\pi])
\ar[r]\ar[d]_{F}&
{\overbrace{H_2(X_0,\partial X_0;\Z[\pi])}^{\cong H_2(X_0;\Z[\pi])^*}}
\ar[r]\ar[d]_{(F^*)^{-1}}&
H_1(\partial X_0;\Z[\pi])
\ar[d]_{h_*}\\
H_2(\partial X_1;\Z[\pi])
\ar[r]&
H_2(X_1;\Z[\pi])
\ar[r]&
{\underbrace{H_2(X_1,\partial X_1;\Z[\pi])}_{\cong H_2(X_1;\Z[\pi])^*}}
\ar[r]&
H_1(\partial X_1;\Z[\pi])
}
$$
instead of compatible triples $(F,G,h)$,  but for $\pi \in \{ 1,\Z\}$ the former determine the latter; see Proposition~\ref{prop:CompatiblePair}.
Note that for a compatible pair $(F,h)$,  the condition~$\lambda_{X_1}^\partial(F(x),(F^*)^{-1}(y))=\lambda_{X_0}^\partial(x,y)$ for all $x\in H_2(X_0;\Z[\pi]), y \in H_2(X_0,\partial X_0;\Z[\pi])$ is automatically satisfied.
We refer to Section~\ref{sub:CompatibleTriple} for further details on compatible pairs and triples.
We also wish to emphasise that even under the assumptions of this remark, the proof of Theorem~\ref{thm:IntroTorsion} is entirely different from the arguments in~\cite{BoyerUniqueness} and~\cite{ConwayPowell}.
\end{remark}

The next theorem states our results when the hypothesis on the $\Z[\pi]$-homology of the boundary is relaxed.
These results hold whenever~$\pi$ is a torsion-free $3$-manifold group.
This includes the cases where $\pi$ is trivial or free.

\begin{theorem}
\label{thm:1d3dIntro}
Let $\pi$ be a torsion-free $3$-manifold group,  and let~$X_0$ and~$X_1$ be~$4$-manifolds with~$\pi_1(X_j) \cong \pi$ and~$\iota_j \colon  \pi_1(\partial X_j) \to \pi_1(X_j)$ surjective  for $j=0,1$.
Fix a degree one homotopy equivalence $h \colon \partial X_0 \to \partial X_1$.

Assume~$H_1(\partial X_i;\Z[\pi]) \cong L \oplus T$ with $L$ a~$\Z[\pi]$-free module and $T^*=0$.
\begin{itemize}
\item When the universal covers of $X_0$ and $X_1$ are not spin,~$h$ extends to a homotopy equivalence~$X_0 \to X_1$ if and only if there is a group isomorphism~$u \colon \pi_1(X_0) \to \pi_1(X_1)$ with~$u \circ \iota_0=\iota_1 \circ h_*$ and a $k$-invariant preserving compatible triple $(F,G,h)$.
\item  When the universal covers of $X_0$ and $X_1$ are spin,  $h$ extends to a homotopy equivalence~$X_0 \to X_1$ if and only if there is a group isomorphism~$u \colon \pi_1(X_0) \to \pi_1(X_1)$ with~$u \circ \iota_0=\iota_1 \circ h_*$, a $k$-invariant preserving compatible triple~$(F,G,h)$,  and the universal cover of $X_0 \cup_h X_1$ is spin. 
\end{itemize}
When the universal covers are spin, given a~$k$-invariant preserving compatible triple~$(F,G,h)$,  the homotopy equivalence $X_0 \to X_1$ extending $h$ can be chosen to induce~$F$ on~$\Z[\pi]$-homology.
\end{theorem}

In the second item, $X_0 \cup_h X_1$ refers to the pushout of $X_0$ and $X_1$ along the homotopy equivalence~$h$.
Propositions~\ref{prop:HomotopyPushout} and~\ref{prop:pi1HomotopyPushout} show that~$X_0 \cup_h X_1$ is a~$4$-dimensional Poincar\'e pair with fundamental group $\pi$, allowing us to make sense of its equivariant intersection form and whether or not its universal cover is spin.

We note that even when~$\pi=1$,  Theorem~\ref{thm:1d3dIntro} does not seem to have appeared in the literature (when~$h$ is a homeomorphism and~$\pi=1$, the result follows from~\cite[Theorem~0.7 and Proposition~0.8]{BoyerUniqueness}).
In this case,~$H_1(\partial X_i)$ necessarily decomposes as the direct sum of a free module and torsion module (because~$\Z$ is a PID),  the relative~$k$-invariants vanish and compatible triples can be replaced by compatible pairs.

\begin{remark}
We reframe the conditions on the universal cover in terms of the equivariant intersection form of the~$X_i$.
Recall that a hermitian form~$\lambda \colon  H \times H \to \Z[\pi]$ is \emph{weakly even} if for every~$x \in H$ we have~$\lambda(x,x)=a+\overline{a}$ for some~$a \in \Z[\pi]$,  \emph{odd} if it is not weakly even, and \emph{even} if there exists a sesquilinear form~$q \colon H  \times H\to \Z[\pi]$  with~$\lambda(x,y)=q(x,y)+\overline{q(y,x)}$ for every~$x,y \in H$.
Even forms are weakly even, and if~$H\cong L \oplus T$ with~$L$ free and~$T^*=0$ then weakly even forms are even (Lemma~\ref{lem:WeaklyEvenEven}).

Given a~$4$-manifold~$X$, Lemma~\ref{lem:WeaklyEvenUnivCoverSpin} shows that~$\widetilde{X}$ is spin if and only if~$\lambda_{X}$ is weakly even.
The first and second scenarios of Theorem~\ref{thm:1d3dIntro}  can therefore be thought of as the ``weakly even" and ``not weakly even" cases respectively.
\end{remark}

\subsection{Homeomorphism classifications}
\label{sub:HomeoIntro}

When the fundamental group is good,  in favourable circumstances, surgery theoretic arguments make it possible to upgrade homotopy classifications to homeomorphism classifications.
We list several sample results focusing on situations where the~$2$-type need not be mentioned.

We begin by stating the surgery theoretic input.

\begin{proposition}
\label{prop:SurgerySequence}
Assume that $\pi:=\pi_1(X_1)$ is trivial,  finite cyclic, or a solvable Baumslag--Solitar group.
Let~$X_0$ and $X_1$ be $4$-manifolds with fundamental group $\pi$ and $\pi_1(\partial X_i) \to \pi_1(X_i)=:\pi$ surjective for~$i=0,1$.
If $\ks(X_0)=\ks(X_1)$ and a homeomorphism $\partial X_0 \to \partial X_1$ extends to a  homotopy equivalence~$\phi \colon X_0 \to X_1$,  then it extends to a homeomorphism $\psi \colon X_0 \to X_1$ that satisfies~$\psi_*=\phi_*$ on $\pi_1(X_0),H_2(\widetilde{X}_0)$ and $H_2(\widetilde{X}_0,\partial \widetilde{X}_0)$.

\end{proposition}
\begin{proof}
In the case where $\pi$ is trivial or a Baumslag--Solitar group,
the argument for closed manifolds is by now well known (see e.g.~\cite{KirbyTaylor,KasprowskiLand} and~\cite{HKPR}).
The tools involve the surgery exact sequence coupled with a Novikov pinching argument both of which go through without change for manifolds with boundary.
For the finite cyclic case, we refer to~\cite[Proposition~4.6]{LeeWilczy}.
The equality $\psi_* =\phi_*$ on $\pi_1(X_0)$ and $\pi_2(X_0) \cong H_2(\widetilde{X}_0)$ is explained in~\cite[Proof of Proposition~4.6]{LeeWilczy} and the proof also yields the equality on $H_2(\widetilde{X}_0,\partial \widetilde{X}_0)$ and extends to the other fundamental groups we have listed.
\end{proof}


When $X_0$ and $X_1$ are simply-connected,  Theorem~\ref{thm:1d3dIntro} together with Proposition~\ref{prop:SurgerySequence} recover (with a  different proof) the result due to Boyer~\cite[Theorem 0.7 and Proposition~0.8]{BoyerUniqueness} that an orientation-preserving homeomorphism~$h \colon \partial X_0 \to \partial X_1$ extends to a homeomorphism~$X_0 \to X_1$ if
\begin{itemize}
\item either $H_1(\partial X_i)$ is torsion and there is a compatible pair $(F,h)$  (when the $X_i$ are not spin,  add the requirement that $\ks(X_0)=\ks(X_1)$), 
\item or the $X_i$ are not spin,   there is a compatible pair $(F,h)$,  and $\ks(X_0)=\ks(X_1)$,
\item or the $X_i$ are spin,  there is a compatible pair $(F,h)$, and the union~$X_0 \cup_h X_1$ is spin.
\end{itemize}
We therefore focus on case in which $\pi$ is nontrivial, starting with the infinite cyclic case.

\begin{theorem}
\label{thm:Homeopi1=ZIntro}
Let $X_0$ and $X_1$ be $4$-manifolds with $\pi_1(X_j) \cong \Z$ and $\iota_j \colon \pi_1(\partial X_j) \to \pi_1(X_j)$ surjective for~$j=0,1$.
Fix an orientation-preserving homeomorphism $h \colon \partial X_0 \to \partial X_1$.
\begin{itemize}
\item When $H_1(\partial X_i;\Z[\Z])$ is torsion,  $h$ extends to a homeomorphism~$X_0 \to X_1$ if and only if there is an isomorphism $u \colon \pi_1(X_0) \cong \pi_1(X_1)$ with $u \circ \iota_0=\iota_1 \circ h_*$,  and there is a compatible pair~$(F,h)$; when the $X_i$ are not spin,  add the condition that~$\ks(X_0)=\ks(X_1)$.
\item When the $X_i$ are not spin and~$H_1(\partial X_i;\Z[\Z]) \cong L \oplus T$ with $L$ free and $T$ torsion,  $h$ extends to a homeomorphism~$X_0 \to X_1$ if and only if there is an isomorphism $u \colon \pi_1(X_0) \cong \pi_1(X_1)$ with $u \circ \iota_0=\iota_1 \circ h_*$,  there is a compatible pair $(F,h)$, and $\ks(X_0)=\ks(X_1)$.
\item When the $X_i$ are spin and~$H_1(\partial X_i;\Z[\Z]) \cong L \oplus T$ with $L$ free and $T$ torsion,  $h$ extends to a homeomorphism~$X_0 \to X_1$ if and only if there is an isomorphism $u \colon \pi_1(X_0) \cong \pi_1(X_1)$ with $u \circ \iota_0=\iota_1 \circ h_*$,   there is a compatible pair~$(F,h)$,  and the union~$X_0 \cup_h X_1$ is spin.
\end{itemize}
Additionally,  when $H_1(\partial X_i;\Z[\Z])$ is torsion or when the $X_i$ are spin, given a $k$-invariant preserving compatible pair $(F,h)$ as above,  the homeomorphism $X_0 \to X_1$ extending $h$ can be chosen to induce~$F$ on $\Z[\Z]$-homology.
\end{theorem}
\begin{proof}
Since~$\pi_1(X_i) \cong \Z$ for~$i=0,1$,  Proposition~\ref{prop:CompatiblePair} and the existence of a compatible pair~$(F,h)$ ensure the existence of a compatible triple $(F,G,h)$.
Theorems~\ref{thm:IntroTorsion} and~\ref{thm:1d3dIntro} imply that the homeomorphism $h \colon \partial X_0 \to \partial X_1$ extends to a homotopy equivalence $X_0 \to X_1$.
If~$X:=X_0 \cup_h X_1$ is spin, then the addivity of the Kirby-Siebenmann invariant (see e.g.~\cite[Theorem 9.2]{FriedlNagelOrsonPowell}) and of the signature give
$$ \ks(X_0)+\ks(X_1)=\ks(X)=\frac{\sigma(X)}{8}=\frac{\sigma(X_0)-\sigma(X_1)}{8}=0 \mod 2.$$
Here, we used that the signature of $X_i$ is determined by the equivariant intersection form of~$X_i$.
When $H_1(X_i;\Z[\Z])$ is torsion and the $X_i$ are spin,~$X=X_0 \cup_h X_1$ is automatically spin~\cite[Theorem 3.12]{ConwayPowell}; this also follows from the combination of Theorem~\ref{thm:obstruction0ImpliesEvenIntro} and Lemma~\ref{lem:WeaklyEvenUnivCoverSpin}.
Proposition~\ref{prop:SurgerySequence} then implies that~$h$ in fact extends to a homeomorphism.
The fact that a given isometry~$F$ can be realised follows from the corresponding realisation statement in Theorems~\ref{thm:IntroTorsion} and~\ref{thm:1d3dIntro} and from the equality $\psi_*=\phi_*$ in Proposition~\ref{prop:SurgerySequence}.
\end{proof}

The first item of Theorem~\ref{thm:Homeopi1=ZIntro} is~\cite[Theorem 1.10]{ConwayPowell} (but the proof is entirely different), whereas a result similar in nature to the third item was announced in 2020 during an online talk by Kreck and Teichner.

\begin{remark}
The third item of \cref{thm:Homeopi1=ZIntro} implies the following related statement. 
	Let~$X_0$ and~$X_1$ be 4-manifolds with~$\pi_1(X_i) \cong \Z$, a preferred spin structure~$\sigma_i$,  and~$\iota_i\colon \pi_1(\partial X_i) \to \pi_1(X_i)$ surjective for~$i = 0, 1$. 
	Fix a homeomorphism~$h\colon \partial X_0 \to \partial X_1$ that preserves the induced spin structures~$\partial \sigma_i$. 
	When~$H_1(\partial X_i;\Z[\Z])\cong L\oplus T$ with~$L$ free and~$T$ torsion,  the homeomorphism~$h$ extends to a spin structure preserving homeomorphism~$X_0 \to X_1$ if and only if there is an isomorphism~$u\colon \pi_1(X_0) \cong \pi_1(X_1)$ with~$u\circ \iota_0 = \iota_1 \circ h_*$, and there is a compatible pair~$(F, h)$. 
	
	To see this,  note that since~$h$ preserves the induced spin structures, the union~$X_0\cup_hX_1$ is spin, so~\cref{thm:Homeopi1=ZIntro} ensures that~$h$ extends to a homeomorphism~$H \colon X_0 \to X_1$. 
	Let~$\sigma_0'$ be the pullback of~$\sigma_1$ along~$H$. 
	Then~$h$ pulls back~$\partial\sigma_1$ to~$\partial \sigma_0'$ and thus~$\partial \sigma_0=\partial \sigma_0'$. 
This implies  that $\sigma_0'=\sigma_0$: indeed, since
$H_1(\partial  X_0;\Z_2) \to H_1(X_0;\Z_2)$ is surjective,~$\operatorname{Spin}(X_0) \to \operatorname{Spin}(\partial X_0)$ is injective.
	It follows that~$H$ preserves the spin structures~$\sigma_i$ as claimed.
\end{remark}

For finite cyclic groups and solvable Baumslag--Solitar groups, we obtain the following result.
\begin{theorem}
\label{thm:Homeopi1=BSIntro}
Let $\pi$ be either a solvable Baumslag--Solitar group or a finite cyclic group.
Let $X_0$ and $X_1$ be $4$-manifolds with~$\pi_1(X_j) \cong \pi$ and~$\iota_j \colon \pi_1(\partial X_j) \to \pi_1(X_j)$ surjective.
Fix an orientation-preserving homeomorphism $h \colon \partial X_0 \to \partial X_1$.

When~$H_1(\partial X_j;\Z[\pi])^*=0$ for $j=0,1$, the following assertions are equivalent.
\begin{itemize}
\item The homeomorphism $h$ extends to a homeomorphism~$X_0 \to X_1$. \item There is an isomorphism $u \colon \pi_1(X_0) \cong \pi_1(X_1)$ with $u \circ \iota_0=\iota_1 \circ h_*$, there is a compatible triple~$(F,G,h)$, and~$\ks(X_0)=\ks(X_1)$.
\end{itemize}
Additionally,  given a $k$-invariant preserving compatible triple $(F,G,h)$ as above,  the homeomorphism $X_0 \to X_1$ extending $h$ can be chosen to induce $F$ and $G$ on $\Z[\pi]$-homology.
\end{theorem}
\begin{proof}
The argument is the same as in the proof of Theorem~\ref{thm:Homeopi1=ZIntro}: Theorem~\ref{thm:IntroTorsion} ensures that $h$ extends to a homotopy equivalence $X_0 \to X_1$ and
Proposition~\ref{prop:SurgerySequence} then implies that $h$ extends to a homeomorphism.
\end{proof}

\begin{remark}
Under the assumptions of Theorem~\ref{thm:Homeopi1=BSIntro}, it is conceivable that, in some cases,  the existence of a compatible triple implies that $\ks(X_0)=\ks(X_1)$ as is the case for $BS(1,0) \cong~\Z$ (recall Theorem~\ref{thm:Homeopi1=ZIntro}).
For example,  when $\pi\cong BS(1,n)$ and the universal covers of~$X_0$ and $X_1$ are spin,  the existence of a compatible triple ensures that the universal cover of~$X=X_0 \cup_h X_1$ is 
spin (combine Theorem~\ref{thm:obstruction0ImpliesEvenIntro} with Lemma~\ref{lem:WeaklyEvenUnivCoverSpin}) which in turn, when $n$ is even,  implies that~$X$ is spin~\cite[Theorem B]{HambletonKreckTeichner} and thus~$\ks(X_0)=\ks(X_1)$ (by the same reasonning as in the proof of Theorem~\ref{thm:Homeopi1=ZIntro} together with~\cite[Remark 4.2]{HambletonKreckTeichner}).
\end{remark}

\subsection*{Organisation}

Section~\ref{sec:PostnikovIntro} describes our results in greater detail and proves Theorems~\ref{thm:IntroTorsion} and~\ref{thm:1d3dIntro}.
The remainder of the paper decomposes into four parts.
Sections~\ref{sec:Prep}-\ref{sec:DefineAction} set up our primary and secondary obstructions to the extension problem.
Sections~\ref{sec:Relativek} and~\ref{sec:NoPostnikov} are concerned with relative~$k$-invariants and with reformulating the primary obstruction in terms of compatible triples.
Sections~\ref{sec:KunnethYS1}-\ref{sec:Odd} are concerned with the analysis of the secondary obstruction.
Finally,  Section~\ref{sec:cd3} focuses on describing situations in which our theorems apply.
The appendix collects some facts about twisted homology and obstruction theory.

\subsection*{Acknowledgments}
We are grateful to 
Mark Powell, Peter Teichner, and Simona Veselá for helpful discussions.
AC was partially supported by the NSF grant DMS~2303674.

\subsection*{Conventions}
Spaces are assumed to be compact and connected unless mentioned otherwise.
CW complexes are assumed to be finite.

Given a ring~$R$ with involution, and a left~$R$-module~$H$, we write~$\overline{H}$ for the right~$R$-module whose underlying group agrees with that of~$H$ but with the~$R$-module structure~$x\cdot r:=\overline{r} x$ for~$x \in H$ and~$r \in R$.
Similarly, we write~$H^*:=\overline{\Hom_{\operatorname{left-}R}(H,R)}$ to indicate that we are using the involution on~$R$ to turn the right~$R$-module $\Hom_{\operatorname{left-}R}(H,R)$ into a left~$R$ module.

Finally, given a hermitian form $b \colon H \times H \to R$ and a map $f  \colon V \to H$ we write $f^*b \colon V \times V \to R$ for the hermitian form $(x,y) \mapsto b(f(x),f(y)).$
If $f=g^*$ is an induced map on cohomology,  then we write $g^*b$ instead of $(f^*)^*(b).$

\section{The homotopy classification in terms of the Postnikov~$2$-type.}
\label{sec:PostnikovIntro}

We describe our results in greater detail and then explain how they recover the statements from the introduction.
The organisation of this section matches the organisation of the paper but the proofs of the results stated here are deferred to later sections.

Since our results hold more generally than for manifolds, we begin with some terminology.
\begin{definition}
\label{def:PoincarePair}
An $n$-dimensional \emph{Poincar\'e complex} $Y$ consists of a space that is homotopy equivalent to 
a CW complex, together with a class $[Y] \in H_n(Y)$ such that for every~$k \geq 0$, the cap product with $[Y]$ induces an isomorphism
$$ -\cap [Y] \colon H^k(Y;\Z[\pi_1(Y)]) \xrightarrow{\cong} H_{n-k}(Y;\Z[\pi_1(Y)]).$$
An $n$-dimensional \emph{Poincar\'e pair}
consists of a pair~$(X,Y)$ together with a class $[X] \in H_n(X,Y)$ such that
\begin{itemize}
\item the space~$Y$ is an $(n-1)$-dimensional Poincar\'e complex,
\item the inclusion~$Y \subset X$ is a cofibration;
\item there is a homotopy equivalence of pairs $(X,Y) \simeq (X',Y')$ to a CW pair;
\item  the connecting homomorphism satisfies~$\partial([X])=[Y]$ and,  for every $k \geq 0$, the cap product with $[X]$ induces an isomorphism
$$ -\cap [X] \colon H^k(X;\Z[\pi_1(X)]) \xrightarrow{\cong} H_{n-k}(X,Y;\Z[\pi_1(X)]).$$
\end{itemize} 
\end{definition}

\begin{remark}
\label{rem:PoincarePair}
We briefly comment on Definition~\ref{def:PoincarePair}.
\begin{itemize}
\item Definition~\ref{def:PoincarePair} differs slightly from Wall's definition~\cite{WallPoincare}.
Indeed,  in addition to the condition on cap products,  Wall requires Poincar\'e complexes to be finitely dominated~CW complexes and, similarly, Poincar\'e pairs are assumed to be CW pairs (thus ensuring that the inclusion~$Y \subset X$ is automatically a cofibration).
We relax this condition so that, as proved in Lemma~\ref{lem:manifolds-pd-pairs} below,  if $M$ is an $n$-manifold, then $(M,\partial M)$ is an $n$-dimensionsal Poincar\'e pair, even if $M$ or $\partial M$ happen not to admit a~CW structure. 
In particular,  all the results obtained in this section, which are stated in terms of~$4$-dimensional Poincar\'e pairs, also apply to~$4$-manifolds.
\item Using his definition, Wall proves that if $(X,Y)$ is a Poincar\'e pair, then the cap product with~$[X]$ in fact induces isomorphisms on (co)homology with all coefficients~\cite[Lemma~1.2]{WallPoincare}.
His proof also applies with our definition.
\end{itemize}
\end{remark}

\subsection{Primary and secondary obstructions}
\label{sub:With2TypeIntro}

A \emph{Postnikov $2$-type} of a space~$X$ consists of a pair~$(P_2(X),c)$, where $P_2(X)$ is a space with~$\pi_i(P_2(X))=0$ for~$i \geq 3$ and~$c \colon X \to P_2(X)$ is a $3$-connected map.
The pair $(P_2(X),c)$ is unique up to (an essentially unique) homotopy equivalence and thus, for brevity, we will often refer to the space~$P_2(X)$ as the Postnikov $2$-type of $X$.

\begin{notation}
\label{not:Intro}
Fix~$4$-dimensional Poincar\'e pairs~$(X_0,Y_0)$ and~$(X_1,Y_1)$ with~$\pi_1(X_i) \cong \pi$,  the same~$2$-type, and~$\iota_j \colon \pi_1(Y_j) \to \pi_1(X_j)$ surjective for $j=0,1$.
Fix a degree one  homotopy equivalence~$h \colon Y_0 \to Y_1$. 
Write~$B:=P_2(X_1)$ and assume without loss of generality that~$Y_1 \subset B$: if~$c \colon X_1 \to B'$ is a~$3$-connected map into a model for the~$2$-type,  then the mapping cylinder~$B$ of~$c$ is homotopy equivalent to~$B'$ and contains~$Y_1$ as a subspace.
This leads to a~$3$-connected map~$c_1 \colon (X_1,Y_1) \to (B,Y_1)$.
Finally fix a $3$-connected map~$c_0 \colon X_0 \to B$ with~$c_0|_{Y_0}=c_1|_{Y_1} \circ h$ so that it induces a map of pairs~$c_0 \colon (X_0,Y_0) \to (B,Y_1).$
This also implies that the group isomorphism~$u:=(c_1)_*^{-1} \circ (c_0)_* \colon \pi_1(X_0) \to \pi_1(X_1)$ satisfies~$u \circ \iota_0=\iota_1 \circ h_*$.
\end{notation}

The maps $c_0$ and $c_1$ lead to relative pairings on $(B,Y_1)$ which are more convenient to describe in cohomology.
For this reason,  we work with the cohomological version of the relative intersection form $\lambda_{X_i}^\partial$,  i.e. with the pairing
\begin{align*}
b_{X_i}^\partial \colon H^2(X_i,Y_i;\Z[\pi]) \times  H^2(X_i;\Z[\pi]) &\to \Z[\pi] \\
(\alpha,\beta) &\mapsto \langle \beta,\PD_{X_i}(\alpha)\rangle.
\end{align*}
If $Y_i$ is empty, we write $b_{X_i}$ instead of $b_{X_i}^\partial $.
Poincar\'e duality induces an isometry~$b_{X_i}^\partial \cong \lambda_{X_i}^\partial$; see Remark~\ref{rem:PairingCohomHom} for further details.
For $i=0,1$ we then consider the pulled back pairing 
\begin{align*}
c_i^*b_{X_i}^\partial \colon H^2(B,Y_1;\Z[\pi]) \times  H^2(B;\Z[\pi]) &\to \Z[\pi] \\
(x,y) &\mapsto b_{X_i}^\partial(c_i^*(x),c_i^*(y)).
\end{align*}

The first main definition of this section is the following.

\begin{definition}
\label{def:PrimaryObstructionIntroduction}
The \emph{primary obstruction} to~$(X_0,Y_0)$ and~$(X_1,Y_1)$ being homotopy equivalent rel.~$h$ is the pairing
$$c_0^*b_{X_0}^\partial-c_1^*b_{X_1}^\partial  \colon H^2(B,Y_1;\Z[\pi]) \times  H^2(B;\Z[\pi]) \to \Z[\pi].$$
\end{definition}

\begin{remark}
In this definition, the homotopy equivalence $h$ is permitted to have degree $\pm 1$.
\end{remark}

In order to describe the secondary obstruction,  we write~$U:=X_0 \cup_h X_1$ for the pushout of~$(X_0,Y_0)$ and~$(X_1,Y_1)$ along~$h$ and set~$c:=c_0 \cup c_1 \colon U \to B$.
We refer to Propositions~\ref{prop:HomotopyPushout} and~\ref{prop:pi1HomotopyPushout} for a proof of the fact that $U$ is a~$4$-dimensional Poincar\'e complex with fundamental group~$\pi$ (so that one can make sense of its equivariant intersection form).


The next proposition follows from the combination of Propositions~\ref{prop:FundamentalClassUnion} and~\ref{prop:PullbackProperties}.
In essence, it shows that when the primary obstruction vanishes, a secondary  obstruction can be defined.

\begin{proposition}
\label{prop:SecondaryIntro}
Assume the maps~$c_0 \colon X_0 \to B$ and $c_1  \colon X_1 \to B$ satisfy~$c_0^*b_{X_0}^\partial =c_1^*b_{X_1}^\partial$.
For~$a,b \in H^2(Y_1;\Z[\pi])$, 
and~$a',b' \in H^2(B;\Z[\pi])$ with~$c_1|_\partial^*(a')=a,c_1|_\partial^*(b')=b$, consider
\begin{align*}
b(c_0,c_1) \colon H^2(Y_1;\Z[\pi]) \times H^2(Y_1;\Z[\pi]) &\to \Z[\pi] \\
(a,b)& \mapsto b_U(c^*(a'),c^*(b')).
\end{align*}
This assignment does not depend on the choice of $a',b'$,  and the pairing $b(c_0,c_1)$ is hermitian.
\end{proposition}

In order to make precise the idea that $b(c_0,c_1)$ serves as a secondary obstruction, we need to arrange for the vanishing of $b(c_0,c_1)$ not to depend on the choice of $c_0$ and $c_1$ within their homotopy classes over $B$.
Write $\hAut_{Y_1}(B)$ for the group of rel. $Y_1$ homotopy classes of homotopy equivalences~$B \to B$ that restrict to the identity on $Y_1$ and, given a~$3$-connected map~$c_1 \colon X_1 \to B$,  consider the group
$$ \mathcal{G}=  \{ \varphi \in \hAut_{Y_1}(B) \mid  (\varphi \circ c_1)^*b_{X_1}^\partial = c_1^*b_{X_1}^\partial \}.$$
Observe that $\mathcal{G}$ does not depend on $c_1$ within its homotopy class over $B$.
Proposition~\ref{prop:Action} proves that~$\mathcal{G}$ acts on the set~$\Herm(H^2(Y_1;\Z[\pi]))$ of hermitian forms on $H^2(Y_1;\Z[\pi])$,  and Proposition~\ref{prop:Indepc0} shows that, once considered in~$\Herm(H^2(Y_1;\Z[\pi]))/\mathcal{G}$,  the vanishing of the pairing~$b(c_0,c_1)$ does not depend on the choice of~$c_0$ (within its homotopy class over $B$) satisfying~$c_1|_{Y_1} \circ h = c_0|_{Y_0}.$

\begin{definition}
\label{def:SecondaryIntro}
Assume the maps~$c_0 \colon X_0 \to B$ and $c_1  \colon X_1 \to B$ satisfy~$c_0^*b_{X_0}^\partial =c_1^*b_{X_1}^\partial$,
The \emph{secondary obstruction} to~$(X_0,Y_0)$ and~$(X_1,Y_1)$ being homotopy equivalent rel. $h$ refers to the hermitian form
$$b(c_0,c_1) \in \Herm(H^2(Y_1;\Z[\pi]))/\mathcal{G}.$$
\end{definition}

We describe informally how we think of the hermitian form $b(c_0,c_1)$.
Given $a \in H^2(Y_1;\Z[\pi])$,  since the universal covers $\widetilde{X}_0$ and $\widetilde{X}_1$,  a loop $\gamma$ representing $\PD_{Y_1}(a) \in H_1(Y_1;\Z[\pi])$ bounds immersed discs in $\widetilde{X}_0$ and $\widetilde{X}_1$ whose union forms an immersed sphere $S(a)$ representing the Poincaré dual~$\PD_{X_0 \cup_h X_1}(c^*(a')) \in H_2(X_0 \cup_h X_1;\Z[\pi])$.
We then think of $b(c_0,c_1)(x,y)$ as $\lambda_{X_0 \cup_h X_1}(S(x), S(y))$.

\subsection{If the obstructions vanish, then the homotopy equivalence extends}

One of our main technical results shows that for large classes of fundamental groups, the primary and secondary obstructions are the only obstructions to extending~$h$ to a homotopy equivalence.
When~$\pi$ is finite, the statement involves the Whitehead $\Gamma$-group $\Gamma(\pi_2(X_1))$; see Section~\ref{sub:WhiteheadGamma} for a brief recollection of this construction.

\medbreak 

Recall that a group $\pi$ has \emph{cohomological dimension at most~$n$}, denoted~$\cd(\pi)\leq n$,  if the~$\Z[\pi]$-module~$\Z$,  endowed with the trivial $\pi$-action,  has a projective resolution of length~$n$.
The proof of the following theorem can be found in Section~\ref{sec:DefineAction}.
\begin{theorem}
\label{thm:MainTheoremBoundaryIntro}
Let~$(X_0,Y_0)$ and~$(X_1,Y_1)$ be~$4$-dimensional Poincar\'e pairs with fundamental group~$\pi_1(X_i) \cong \pi$,  Postnikov $2$-type $B:=P_2(X_1)$ and~$\pi_1(Y_i) \to \pi_1(X_i)$ surjective for $i=0,1$,  and let~$h \colon Y_0 \to Y_1$ be a degree one homotopy equivalence.
Assume that
\begin{itemize}
\item either~$\cd(\pi)\leq 3$ and the map $\ev^* \colon H^2(X_i;\Z[\pi])^{**} \to H_2(X_i;\Z[\pi])^*$ is an isomorphism,
\item or~$\pi_1(X_i)$ is finite and~$\Gamma(\pi_2(X_1)) \otimes_{\Z[\pi]} \Z$ is torsion free.
\end{itemize}
The homotopy equivalence $h$ extends to a homotopy equivalence $X_0 \to X_1$ if and only if
there exist~$3$-connected maps~$c_0 \colon X_0 \to B$ and~$c_1 \colon X_1 \to B$ such that~$c_1|_{Y_1} \circ h = c_0|_{Y_0}$ and
\begin{enumerate}
\item the pulled back pairings~$c_0^*b_{X_0}^\partial$ and~$c_1^*b_{X_1}^\partial$ on~$H^2(B;\Z[\pi])$ are equal,
\item  the hermitian pairing~$b(c_0,c_1)$ vanishes in~$\Herm(H^2(Y_1;\Z[\pi]))/\mathcal{G}$.
\end{enumerate}
Additionally,  when $b(c_0,c_1)=0 \in \Herm(H^2(Y_1;\Z[\pi]))$, the homotopy equivalence $f \colon X_0 \to X_1$ extending $h$ can be chosen to satisfy $c_1 \circ f \simeq c_0$.
\end{theorem}

As we describe in Section~\ref{sec:cd3}, the conditions in Theorem~\ref{thm:MainTheoremBoundaryIntro} hold for the trivial group, the infinite cyclic group, surface groups, Baumslag--Solitar groups, for torsion-free $3$-manifold groups,  as well as for classes of finite groups,  including finite cyclic groups.

For large classes of groups,
Theorem~\ref{thm:MainTheoremBoundaryIntro} provides a solution to the problem of deciding whether a homotopy equivalence~$h \colon Y_0 \to Y_1$ extends to a homotopy equivalence~$X_0 \to X_1$.
Since practical applications typically require statements that do not involve the $2$-type,  the next sections describe settings where Theorem~\ref{thm:MainTheoremBoundaryIntro} can be formulated without appealing to $c_0,c_1$ and $B$.

\subsection{The primary obstruction and compatible triples}
\label{sub:Without2TypeIntro}

Theorem~\ref{thm:RealiseAlgebraic3Type} shows that if there is an isomorphism $u \colon \pi_1(X_0) \cong \pi_1(X_1)$ with $u \circ \iota_0=\iota_1 \circ h_*$ and an isomorphism~$F \colon \pi_2(X_0) \cong \pi_2(X_1)$  with~$h^* (k_{X_1,Y_1})=F_*(k_{X_0,Y_0})$, then there are~$3$-connected maps~$c_0 \colon X_0 \to B$ and $c_1 \colon X_1 \to B$ that satisfy~$c_0|_{Y_0} =  c_1|_{Y_1} \circ h$.
The next result, in which the homotopy equivalence $h$ is permitted to have degree $\pm 1$,
shows that the vanishing of the primary obstruction is equivalent to the existence of a compatible triple.

\begin{theorem}
\label{thm:PrimaryObstructionRecastIntro}
Assume that~$H^2(\pi;\Z[\pi]) \to H^2(Y_1;\Z[\pi])$ is injective and $H_1(Y_1;\Z[\pi]) \cong L \oplus T$, with $L$ a free $\Z[\pi]$-module and $T^*=0$.
The following assertions are equivalent:
\begin{enumerate}
\item there is an isomorphism $u \colon \pi_1(X_0) \cong \pi_1(X_1)$ with $u \circ \iota_0=\iota_1 \circ h_*$ and a compatible triple~$(F,G,h)$ with $h^*(k_{X_1,Y_1})=F_*(k_{X_0,Y_0})$;
\item there are $3$-connected maps~$c_0 \colon X_0 \to B,c_1 \colon X_1 \to B$ such that~$c_0|_{Y_0} =  c_1|_{Y_1} \circ h$ with
$$c_0^*b_{X_0}^\partial =c_1^*b_{X_1}^\partial \quad \text{ and } \quad (c_1)_*^{-1}(c_0)_*(k_{X_0,Y_0})=h^*(k_{X_1,Y_1}).$$
\end{enumerate}
Given $(F,G,h)$ as in~$(1)$,  the maps~$c_0$ and~$c_1$ can be chosen to satisfy~$(c_1)_*^{-1}(c_0)_*=F$ as well as~$(c_1)_*^{-1}(c_0)_*=G$.
Conversely, given $c_0$ and $c_1$ as in~$(2)$,  the maps~$F$ and~$G$ are defined as~$F:=(c_1)_*^{-1}(c_0)_*$ and~$G:=(c_1)_*^{-1}(c_0)_*$.
\end{theorem}
\begin{proof}
The~$(1) \Rightarrow (2)$ direction is Proposition~\ref{prop:Simplify}.
The~$(2) \Rightarrow (1)$ direction is Proposition~\ref{prop:3ConnImpliesCompatible}.
\end{proof}

\begin{remark}
Proposition~\ref{prop:TorsionImpliesH2mapInj} shows~$H^2(\pi;\Z[\pi]) \to H^2(Y_1;\Z[\pi])$ is injective if~$H_1(Y_1;\Z[\pi])^*=~0.$
\end{remark}

We record settings in which compatible triples can be replaced by compatible pairs.
The proof of the following result, in which one need not assume that $h$ has degree one,  can be found in Section~\ref{sub:CompatibleTriple}.

\begin{proposition}
\label{prop:CompatiblePair}
If $H^k(\pi;\Z[\pi])=0$ for $k=2,3$, then the existence of a compatible pair ensures the existence of a compatible triple.
\end{proposition}

We also record settings in which the condition on the relative $k$-invariants can be omitted.

\begin{proposition}
\label{prop:NokinvariantIntro}
Let $F \colon \pi_2(X_0) \cong \pi_2(X_1)$ be an isomorphism.
Assume that
\begin{itemize}
\item either $\cd(\pi)\leq 2$ and there is a compatible triple~$(F,G,h)$,
\item or
$H_4(\pi)=0$ and~$H_1(Y_0;\Z[\pi])^*=0$, 
\end{itemize}
then~$F_*(k_{X_0,Y_0})=h^*(k_{X_1,Y_1})$.
\end{proposition}
\begin{proof}
The first item is Proposition~\ref{prop:NokinvariantIntrocd2}, the second is Proposition~\ref{prop:NokinvariantIntrocd3}.
\end{proof}

\subsection{The vanishing of the secondary obstruction I: odd forms. }

We now turn to the vanishing of the secondary obstruction.
We briefly recall the notation from Notation~\ref{not:Intro} and additionally fix $3$-connected maps for which the primary obstruction vanishes.

\begin{notation}
\label{not:IntroAgainAgain}
Let~$(X_0,Y_0)$ and~$(X_1,Y_1)$ be~$4$-dimensional Poincar\'e pairs with~$\pi_1(X_j) \cong \pi$,
$2$-type $B:=P_2(X_1)$ and~$\iota_j \colon \pi_1(Y_j) \to \pi_1(X_j)$ surjective for~$j=0,1$.
Fix a degree one homotopy equivalence~$h \colon Y_0 \to Y_1$
and assume there are $3$-connected maps~$c_0 \colon X_0 \to B,c_1 \colon X_1 \to B$ such that~$c_0|_{Y_0} =  c_1|_{Y_1} \circ h$ with~$c_0^*b_{X_0}^\partial =c_1^*b_{X_1}^\partial$.
Note that this implies~$\lambda_{X_0} \cong \lambda_{X_1}$.
\end{notation}

We also fix terminology concerning hypotheses that will often be assumed.
\begin{definition}
\label{def:HypothesesName}
We say that~\emph{$\ev^*$ is an isomorphism} if the following is an isomorphism for~$i=0,1$:
$$\ev^* \colon H^2(X_i;\Z[\pi])^{**} \to H_2(X_i;\Z[\pi])^*.$$
\end{definition}

Proposition~\ref{prop:ev*Iso} shows that this condition holds if $\pi$ is a free product of a free group and~$\PD_3$ groups (e.g. if $\pi$ is trivial,  free,  or, more generally, a torsion-free $3$-manifold group),  or if~$\pi=\pi_1(\Sigma)$ is the fundamental group of an orientable surface and~$\pi_1(\Sigma \times S^1) \to \pi$ is the projection onto the first coordinate, or if $\pi$ is a solvable Baumslag--Solitar group.

When $\lambda_{X_0} \cong \lambda_{X_1}$ is odd, the following result (which is proved in Section~\ref{sec:Odd}) shows that the secondary obstruction vanishes under some additional hypotheses.

\begin{theorem}
\label{thm:VanishOddIntro}
Assume that $\cd(\pi) \leq 3$.
If $\ev^*$ is an isomorphism,~$H^2(\pi;\Z[\pi]) \to H^2(Y_1;\Z[\pi])$ is injective,
$\lambda_{X_1}$ is odd,  and~$H_1(Y_1;\Z[\pi]) \cong L\oplus T$ with $L$ free and~$T^*=0$, then
$$\Herm(H^2(Y_1;\Z[\pi]))/\mathcal{G}=0.$$
In particular,  the secondary obstruction vanishes:
$$b(c_0,c_1)=0 \in \Herm(H^2(Y_1;\Z[\pi]))/\mathcal{G}.$$
\end{theorem}

\begin{remark}
The prevalence of the condition $\cd(\pi) \leq 3$ in our results is explained by Proposition~\ref{prop:StablyFreePD<4}: $\cd(\pi_1(X_i)) \leq 3$ is equivalent to requiring that $\pi_2(X_i)$ be $\Z[\pi_1(X_i)]$-projective.
\end{remark}

\subsection{The vanishing of the secondary obstruction II: (weakly) even forms. }

Continuing with Notation~\ref{not:IntroAgainAgain},  the next result (which is proved in Section~\ref{sec:WeaklyEven}) shows that contrarily to the odd case,  when~$\lambda_{X_0} \cong \lambda_{X_1}$ is even,  an additional sufficient condition is needed for the secondary obstruction to vanish.

\begin{theorem}
\label{thm:VanishSpinIntro}
Assume that $\cd(\pi) \leq 3$.
If $\ev^*$ is an isomorphism,~$\lambda_{X_0 \cup_h X_1}$ is weakly even,  and~$H_1(Y;\Z[\pi])\cong L \oplus T$ with~$L$ free and $T^*=0$,
then
$$b(c_0,c_1)=0 \in \Herm(H^2(Y_1;\Z[\pi]))/\mathcal{G}.$$
In fact, there is a $c_0' \simeq c_0$ with $c_1|_{Y_1} \circ h=c_0'|_{Y_0}$ such that $b(c_0',c_1)=0  \in \Herm(H^2(Y_1;\Z[\pi]))$.
\end{theorem}

Note that if~$\lambda_{X_0 \cup_h X_1}$ is (weakly) even, then so is $\lambda_{X_i}$ for $i=0,1$.
Our next result (also proved in Section~\ref{sec:WeaklyEven}) shows that modulo the difference between even  and weakly even forms,  the condition on the union from Theorem~\ref{thm:VanishSpinIntro} is also necessary.
\begin{theorem}
\label{thm:obstruction0ImpliesEvenIntro}
Assume that the hermitian form~$\lambda_{X_0} \cong \lambda_{X_1}$ is weakly even.
If the secondary obstruction~$b(c_0,c_1)$ vanishes in~$\Herm(H^2(Y_1;\Z[\pi]))/\mathcal{G}$, then~$\lambda_{X_0 \cup_h X_1} \cong b_{X_0 \cup_h X_1}$ is weakly even.
\end{theorem}

At this stage, it is useful to record three lemmas that we will constantly refer to when working with even hermitian forms.

\begin{lemma}
\label{lem:WeaklyEvenEven}
Let $H\cong L \oplus T$ be a $\Z[\pi]$-module where $L$ is free and $T^*=0$.
If a hermitian form~$\lambda \colon H \times H \to \Z[\pi]$ is weakly even, then it is even.
\end{lemma}
\begin{proof}
A weakly even hermitian form~$b$ on a free module $L$ is even: choose a basis~$e_1,\ldots,e_n$ for~$L$ with~$b(e_i,e_j)=x_i+x_i^*$, define~$q(e_i)=x_i$ and note that~$\lambda=q+q^*$.
It follows that the pairing $\lambda$ is even on $L \subset H$.
But since $T^*=0$,  it follows that~$\lambda|_{T \times T}=0$
and so $\lambda$ itself is even.
\end{proof}

\begin{lemma}
\label{lem:WeaklyEvenAlgebra}
Let $\pi$ be a group with no elements of order~$2$ and let $H$ be a $\Z[\pi]$-module.
A hermitian form $(H,b)$ over  is weakly even if and only if $b(x,x)_1$ is even for every $x \in H$.
\end{lemma}
\begin{proof}
In general, without any condition $\pi$, a hermitian form $b$ over $\Z[\pi]$ is weakly even if and only if its coefficient $b(x,x)_g$ at $g \in \pi$ is even for every $g \in \pi$ of order $2$.
\end{proof}

\begin{lemma}
\label{lem:WeaklyEvenUnivCoverSpin}
Let $X$ be either a $4$-dimensional Poincar\'e complex or be part of a Poincar\'e pair~$(X,Y)$.
The form~$\lambda_X$ is weakly even if and only if~$\widetilde{X}$ is spin.
\end{lemma}
\begin{proof}
Set~$\pi:=\pi_1(X)$.
If~$\lambda_X$ is weakly even, then for every $x \in \pi_2(X)$, the coefficient~$(\lambda_X(x,x))_g$ at $g \in \pi$ is even for every $g \in \pi$ of order $2$.
In particular $Q_{\widetilde{X}}(x,x)=(\lambda_X(x,x))_e$ is even for every~$x \in X$.
Thus the intersection form~$Q_{\widetilde{X}}$ is even and therefore~$\widetilde{X}$ is spin.

We prove the converse, first when $X$ is a manifold, then when $X$ is a Poincar\'e complex and finally when $(X,Y)$ is a Poincar\'e pair.
We first assume~$X$ is a manifold.
If~$\widetilde{X}$ is spin, then~$\lambda_X$ admits a quadratic refinement~$\mu_X$, say~$\lambda_X=\mu_X+\mu_X^*$; this is explained in detail in~\cite[Remark~1.7]{KasprowskiPowellTeichner}.
Thus~$(\lambda_X(x,x))_g$ is even for every~$g \in \pi$ of order~$2$ and we conclude that~$\lambda_X$ is weakly even.
We now assume that~$X$ is a Poincar\'e complex.
Since~$X$ is reducible~\cite{Land},  there is a degree one map~$f \colon M \to X$ and using surgery we can assume it is an isomorphism on~$\pi_1$. 
By naturality, since~$\widetilde{X}$ is spin, so is~$\widetilde{M}$, implying that~$\lambda_M$ is weakly even.
Since~$\lambda_M=\lambda_X\oplus \lambda_K$, where~$\lambda_K$ is the form restricted to the surgery kernel, it follows that~$\lambda_X$ is weakly even as well.

We conclude with the general case where~$(X,Y)$ is a Poincar\'e pair.
In what follows, we write~$\widetilde{Y} \subset \widetilde{X}$ for the preimage of~$Y$ in the universal cover~$\widetilde{X}$.
Since~$\widetilde{X}$ is spin,~$(\lambda_X(x,x))_e$ is even for every~$x \in H_2(X;\Z[\pi]),$ and it remains to prove that~$(\lambda_X(x,x))_g$ is even for every~$x \in H_2(X;\Z[\pi])$ and every~$g \in \pi$ of order~$2$, i.e.  that~$a \cup ga=0 \in H^4_{cs}(\widetilde{X},\widetilde{Y};\Z_2) \cong H^4_{cs}(\widetilde{X} \cup_{\widetilde{Y}} \widetilde{X},\widetilde{X};\Z_2)$ for every element~$a \in H^2_{cs}(\widetilde{X},\widetilde{Y};\Z_2) \cong H^2_{cs}(\widetilde{X} \cup_{\widetilde{Y}} \widetilde{X},\widetilde{X};\Z_2)$.
This follows from the naturality of the cup product under the map~$j \colon (\widetilde{X} \cup_{\widetilde{Y}} \widetilde{X},\emptyset) \to (\widetilde{X} \cup_{\widetilde{Y}} \widetilde{X},\widetilde{X})$ because~$j^*$ induces an isomorphism on~$H^4(-;\Z_2)$ and~$j^*(ga \cup a)=gj^*(a) \cup j^*(a)=0$.
In the last equality, we used the statement for Poincar\'e complexes: the double~$\widetilde{X} \cup_{\widetilde{Y}} \widetilde{X}$ is a Poincar\'e complex by Proposition~\ref{prop:HomotopyPushout}.
\end{proof}

\subsection{The vanishing of the secondary obstruction III : the main technical result.}
\label{sub:MainTechnicalIntro}

Both Theorem~\ref{thm:VanishOddIntro} and Theorem~\ref{thm:VanishSpinIntro} rely on being able to change $b(c_0,c_1)$ by an even form.
This is the content of the next theorem which is one of the more technical results of this paper.
More precisely we show that for certain groups with $\cd(\pi) \leq 3$,  for any sesquilinear form~$q$ on~$H^2(Y_1;\Z[\pi])$, it is possible to homotope $c_1|_{Y_1}$ back to itself so that the secondary invariant changes from $b(c_0,c_1)$ to~$b(c_0,c_1)+q+q^*$, where $q^*(x,y):=\overline{q(y,x)}$.
Continuing with the set-up from Notation~\ref{not:IntroAgainAgain}, the resulting statement reads as follows.

\begin{theorem}
\label{thm:VanishEvenIntro} 
Assume that $\cd(\pi) \leq 3$.
If $\ev^*$ is an isomorphism
and~$b(c_0,c_1)$ is even, then
$$b(c_0,c_1)=0 \in \Herm(H^2(Y_1;\Z[\pi]))/\mathcal{G}.$$
In fact, there is a $c_0' \simeq c_0$ with $c_1|_{Y_1} \circ h=c_0'|_{Y_0}$ such that $b(c_0',c_1)=0  \in \Herm(H^2(Y_1;\Z[\pi]))$.
\end{theorem}

The proof of Theorem~\ref{thm:VanishEvenIntro} is fairly arduous and will be spread out between Sections~\ref{sec:KunnethYS1} and~\ref{sec:ProofGoal}.
The intuition behind (the technical result underlying) Theorem~\ref{thm:VanishEvenIntro} is explained in Remark~\ref{rem:Intuition}.

\begin{remark}
\label{rem:BoyerComparison}
Given simply-connected $4$-manifolds $X_0,X_1$ an orientation-preserving homeomorphism $h \colon \partial X_0 \to \partial X_1$ and an isometry $F\colon \lambda_{X_0}\cong \lambda_{X_1}$ that is compatible with $h$, Boyer defines an obstruction 
$$\theta(F,h)\in \im(H^1(Y_1) \to H^1(Y;\Z_2)) $$
which is shown to vanish if and only if $h$ extends to a homeomorphism inducing $F$~\cite[Theorem~0.7]{BoyerUniqueness}.
The definition involves the spin structures of~$Y_1$ and satisfies the following properties~\cite[Proposition 0.8]{BoyerUniqueness}:
\begin{enumerate}
\item if $H_1(Y_1;\Q)=0$, then~$\theta(F,h)=0$,
\item when $\lambda_{X_i}$ is even,  $\theta(F,h)=0$ if and only if $X_0 \cup_h X_1$ is spin,
\item when $\lambda_{X_i}$ is odd,  there is an isometry $F'$ that is compatible with $h$ such that $\theta(F',h)=0$.
\end{enumerate}
Informally, we think of our secondary invariant $b(c_0,c_1)$ as an analogue of Boyer's $\theta(F,h)$.
Indeed despite their definitions being very different (our invariant does not involve spin structures and Boyer's does not reference $2$-types),  they share similar properties as illustrated by Theorems~\ref{thm:VanishOddIntro} and~\ref{thm:VanishEvenIntro} as well as the fact that~$b(c_0,c_1)$ vanishes if~$H_1(Y_1;\Z[\pi])^*=0$.
\end{remark}

We further compare the target $\Herm(H^2(Y_1;\Z[\pi]))/\mathcal{G}$ of our secondary obstruction with Boyer's target, namely $I^1(Y):=\im(H^1(Y_1) \to H^1(Y;\Z_2))$.
For this, given a $3$-dimensional Poincar\'e complex $Y$ and 
an epimorphism $\pi_1(Y) \twoheadrightarrow \pi$,  we generalise Boyer's $I^1(Y)$ by setting
$$I^1(Y;\Z[\pi]):=\im(H^1(Y;\Z[\pi])
 \xrightarrow{\ev}
   \Hom(H_1(Y;\Z[\pi]),\Z[\pi])
\to
  \Hom(H_1(Y;\Z[\pi]),\Z_2)),$$
where the second map is induced by augmentation modulo $2$.
When $\pi=1$,  evaluation induces an isomorphism $I^1(Y) \to I^1(Y;\Z)$, whence our notation (which is chosen to mimick Boyer's).

Given $x \in \Z[\pi]$, write $(x)_1 \in \Z$ for the coefficient of the neutral element and consider the group homomorphism
\begin{align*}
\widehat{} \ \colon \Herm(H^2(Y;\Z[\pi])) &\to I^1(Y;\Z[\pi]) \\
b &\mapsto (x \mapsto b(\PD_{Y}^{-1}(x),\PD_{Y}^{-1}(x))_1 \mod 2).
\end{align*}
The following proposition relates $\Herm(H^2(Y_1;\Z[\pi]))$ to $I^1(Y;\Z[\pi])$ and leads to an upper bound on the cardinality of $\Herm(H^2(Y_1;\Z[\pi]))/\mathcal{G}$; here we emphasise that the definition of $\mathcal{G}$ involves the Postnikov $2$-type $B$ and the relative hermitian form $b_{X_1}^\partial$; recall Notation~\ref{not:IntroAgainAgain}.
\begin{proposition}
\label{prop:Herm/G}
The map $\ \widehat{} \ $ satisfies the following properties.
\begin{enumerate}
\item The kernel of $\ \widehat{} \ $ consists of hermitian forms $b$ such that $b(x,x)_1$ is even for every $x$.
In particular, if $\pi$ has no nontrivial elements of order $2$, 
the kernel consists of weakly even forms.
\item If $H^2(Y;\Z[\pi]) \cong L \oplus T$~with $L$ free and~$T^*=0$,  then the map $\ \widehat{} \ $ is surjective and therefore,  if $\pi$ has no nontrivial elements of order $2$,  induces an isomorphism
$$\left(\Z[\pi]/\langle x+\overline{x} \mid x \in \Z[\pi] \rangle\right)^{\oplus b_1(Y;\Z[\pi])} \cong \frac{\Herm(H^2(Y;\Z[\pi]))}{\lbrace \text{(weakly) even forms}\rbrace}  \xrightarrow{\widehat{ } \ , \ \cong} I^1(Y;\Z[\pi]),$$
where a (weakly) even form acts on a hermitian form by $b \cdot b':=b+b'$.
\item Use $\ \widehat{} \ $ to transport the action $\mathcal{G} \curvearrowright \Herm(H^2(Y_1;\Z[\pi]))$ (which we denote $\varphi \cdot b:=b(\varphi) +b$) to an action $\mathcal{G} \curvearrowright I^1(Y_1;\Z[\pi])$ (given by $\varphi \cdot x:=\widehat{b}(\varphi) +x$) so that $\ \widehat{} \ $ induces a morphism
$$\widehat{} \  \colon \Herm(H^2(Y_1;\Z[\pi]))/\mathcal{G} \to I^1(Y_1;\Z[\pi])/\mathcal{G}.$$
that is an isomorphism if~$\cd(\pi) \leq 3,\ev^*$ is an isomorphism, and~$H^2(Y_1;\Z[\pi]) \cong L \oplus T$ with~$L$ free and~$T^*=0$.
\end{enumerate}
\end{proposition}
\begin{proof}
The first item follows from the definition of~$\ \widehat{} \ $ and Lemma~\ref{lem:WeaklyEvenAlgebra}.
The surjectivity follows from the definitions
and, since $H^2(Y;\Z[\pi]) \cong L \oplus T$ with $L$ free and~$T^*=0$, Lemma~\ref{lem:WeaklyEvenEven} ensures that weakly even forms are even.
We prove the third item.
The fact that~$\ \widehat{} \ $ descends to a homomorphism on the orbit sets follows from the equation $b(\psi \circ \varphi)=b(\psi)+b(\varphi)$ (this equality is proved in Proposition~\ref{prop:Action}), whereas the surjectivity follows from the first two items of this proposition.
It remains to prove injectivity.
If $\widehat{b}=\widehat{b(\varphi)}$ for some $\varphi \in \mathcal{G}$, then $b$ and $b(\varphi)$ differ by a weakly even form and thus by an even form because $H^2(Y;\Z[\pi]) \cong L \oplus T$.
Apply Theorem~\ref{thm:GoalActionpdf} to realise this form as $b(\varphi')$ for some $\varphi' \in \mathcal{G}$ so that $b=b(\varphi)+b(\varphi')=b(\varphi \circ \varphi')$, as required.
\end{proof}

We emphasize the take-away from the second and third items of this proposition: since~$\mathcal{G}$ 
can act by odd forms,  we are not able to entirely determine~$\Herm(H^2(Y_1;\Z[\pi]))/\mathcal{G}$, even when~$\cd(\pi) \leq 3$ and $\ev^*$ is an isomorphism.
Under these assumptions, the third item of Proposition~\ref{prop:Herm/G} nevertheless yields
$$ | \Herm(H^2(Y_1;\Z[\pi]))/\mathcal{G}| 
\leq |I^1(Y_1;\Z[\pi])/\mathcal{G}| 
\leq  |I^1(Y_1;\Z[\pi])|
\leq  |\Hom(H_1(Y),\Z_2))|
=|H^1(Y;\Z_2)|
$$
We conclude with instances in which~$ \Herm(H^2(Y_1;\Z[\pi]))/\mathcal{G}$ can be be determined.

\begin{proposition}
Assume~$\cd(\pi) \leq 3$.
Assume additionally that~$H_1(Y_1;\Z[\pi]) \cong L\oplus T$ with $L$ free and~$T^*=0$, and that~$\ev^*$ is an isomorphism.
\begin{itemize}
\item If~$\lambda_{X_0} \cong \lambda_{X_1}$ is odd and~$H^2(\pi;\Z[\pi]) \to H^2(Y_1;\Z[\pi])$ is injective,  then 
$$\Herm(H^2(Y_1;\Z[\pi]))/\mathcal{G}=0.$$
\item If~$\lambda_{X_i}$ is weakly even, then $\ \widehat{} \ $ induces an isomorphism
$$\Herm(H^2(Y_1;\Z[\pi]))/\mathcal{G} \cong I^1(Y_1;\Z[\pi]).$$
\end{itemize}
\end{proposition}
\begin{proof}
The first item is Theorem~\ref{thm:VanishOddIntro}.
The second follows from the third item of Proposition~\ref{prop:Herm/G} and from the assertion that, when $\lambda_{X_i}$ is weakly even, $I^1(Y_1;\Z[\pi])/\mathcal{G}=I^1(Y_1;\Z[\pi])$.
This assertion amounts to showing that the action~$\varphi \cdot x:=\widehat{b}(\varphi) +x$ is trivial, i.e. that $\widehat{b}(\varphi)$ is weakly even.
But this is the first claim in the proof of Theorem~\ref{thm:obstruction0ImpliesEvenIntro} at the end of Section~\ref{sec:WeaklyEven}. 
\end{proof}

In particular, when $\pi$ is trivial, our target is isomorphic to Boyer's in the even case.
In the odd case,  our target is trivial whereas Boyer's need not be so.

\subsection{Proof of the statements from the introduction}

This section shows how the previous results lead to the statements from the introduction.
We recall the relevant statements both for the reader's convenience and to record that they hold more generally than stated above: namely they hold for $4$-dimensional Poincar\'e pairs.
In order to avoid repeating the set-up,  we fix the following notation for the remainder of this section.
\begin{notation}
Fix~$4$-dimensional Poincar\'e pairs~$(X_0,Y_0)$ and~$(X_1,Y_1)$ with~$\pi_1(X_i) \cong \pi$,  the same~$2$-type, and~$\iota_j \colon \pi_1(Y_j) \to \pi_1(X_j)$ surjective for $j=0,1$.
Additionally, fix a degree one homotopy equivalence~$h \colon Y_0 \to Y_1$.
\end{notation}

The next result generalises Theorem~\ref{thm:IntroTorsion} from the introduction.

\begin{theorem}
\label{thm:IntroTorsionPoincarePair}
Let $\pi$ be a group and assume 
\begin{itemize}
\item either that $\pi$ is finite and~$\Gamma(\pi_2(X_1)) \otimes_{\Z[\pi]} \Z$ is torsion free,
\item or that~$\cd(\pi) \leq 3$ and that~$\ev^*$ is an isomorphism.
\end{itemize}
When~$H_1(Y_1;\Z[\pi])^*=0$,  the following assertions are equivalent:
\begin{enumerate}
\item the homotopy equivalence~$h$ extends to a homotopy equivalence~$X_0 \to X_1$;
\item there exists an isomorphism $u \colon \pi_1(X_0) \cong \pi_1(X_1)$ with $u \circ \iota_0=\iota_1 \circ h_*$ and a compatible triple $(F,G,h)$  with 
$$F_*(k_{X_0,Y_0})=h^*(k_{X_1,Y_1}).$$
When $\cd(\pi) \leq 3$, the condition involving the relative $k$-invariants can be omitted.
\end{enumerate}
Additionally,  given a $k$-invariant preserving compatible triple $(F,G,h)$ as above,  the homotopy equivalence $X_0 \to X_1$ extending $h$ can be chosen to induce $F$ and $G$ on $\Z[\pi]$-homology.
\end{theorem}
\begin{proof}
The~$(1) \Rightarrow (2)$ direction is not overly challenging and so we focus on the $(2) \Rightarrow (1)$ direction.
By Theorem~\ref{thm:PrimaryObstructionRecastIntro},  the existence of a $k$-invariant preserving compatible triple ensures the existence of $c_0,c_1$ with vanishing primary obstruction.
Since~$H_1(Y_i;\Z[\pi])^*=0$, there is no secondary obstruction.
Theorem~\ref{thm:MainTheoremBoundaryIntro} now ensures that the homotopy equivalence extends.
When~$\cd(\pi) \leq 3$, Proposition~\ref{prop:NokinvariantIntro} ensures that the $k$-invariant condition can be ignored. 
The fact that a given compatible triple can be realised follows from the corresponding clauses in Theorems~\ref{thm:PrimaryObstructionRecastIntro} and~\ref{thm:MainTheoremBoundary}.
\end{proof}

\begin{remark}
We explain why Theorem~\ref{thm:IntroTorsion} follows from Theorem~\ref{thm:IntroTorsionPoincarePair}.
If $\pi$ is finite abelian with two generators or is finite dihedral, then~$\Gamma(\pi_2(X_1)) \otimes_{\Z[\pi]} \Z$ is torsion free~\cite{HambletonKreck,KasprowskiPowellRuppik,KasprowskiNicholsonRuppik}, whereas if $\pi$ is either a surface group,  or a solvable Baumslag--Solitar group,  or the free product of a free group with~$\PD_3$ groups,  then $\cd(\pi)\leq 3$, and Proposition~\ref{prop:ev*Iso} then ensures that $\ev^*$ is an isomorphism.
Recall that a group $\pi$ is a \emph{$\PD_n$ group} if $B\pi$ is a $n$-dimensional Poincar\'e complex.
\end{remark}

The next result generalises Theorem~\ref{thm:1d3dIntro} from the introduction.

\begin{theorem}
\label{thm:IntroSpinPoincarePair}
Let $\pi$ be a free product of a free group and~$\PD_3$ groups
and assume $H_1(Y_j;\Z[\pi]) \cong L \oplus T$ with $L$ a~$\Z[\pi]$-free module and $T^*=0$.
\begin{itemize}
\item When the universal covers of $X_0$ and $X_1$ are not spin,~$h$ extends to a homotopy equivalence~$X_0 \to X_1$ if and only if there is a group isomorphism $u \colon \pi_1(X_0) \cong \pi_1(X_1)$ with~$u \circ \iota_0=\iota_1 \circ h_*$, and there is a $k$-invariant preserving compatible triple $(F,G,h)$.
\item  When the universal covers of $X_0$ and $X_1$ are spin,  $h$ extends to a homotopy equivalence~$X_0 \to X_1$ if and only if there is a group isomorphism $u \colon \pi_1(X_0) \cong \pi_1(X_1)$ with~$u \circ \iota_0=\iota_1 \circ h_*$,  there is a $k$-invariant preserving compatible triple~$(F,G,h)$,  and  the universal cover of $X_0 \cup_h X_1$ is spin.
\end{itemize}
When the universal covers are spin, given a~$k$-invariant preserving compatible triple~$(F,G,h)$,  the homotopy equivalence $X_0 \to X_1$ extending $h$ can be chosen to induce~$F$ on~$\Z[\pi]$-homology.
\end{theorem}
\begin{proof}
Lemma~\ref{lem:WeaklyEvenUnivCoverSpin} ensures that the universal covers of~$X_0$ and~$X_1$ are spin if and only if~$\lambda_{X_0} \cong \lambda_{X_1}$ is weakly even.
Also,  the condition~$H_1(Y_1;\Z[\pi])\cong L \oplus T$ ensures that weakly even hermitian forms on $H^2(Y_1;\Z[\pi])$ are even hermitian (recall Lemma~\ref{lem:WeaklyEvenEven}).

The existence of a homotopy equivalence extending~$h$ is readily seen to imply that there is a~$k$-invariant preserving compatible triple and that there are~$3$-connected maps~$c_0,c_1$ into the $2$-type of~$X_1$ with~$b(c_0,c_1)=0$ (see Theorem~\ref{thm:MainTheoremBoundary} for this latter point).
In the even case, Theorem~\ref{thm:obstruction0ImpliesEvenIntro} then ensures that~$X_0 \cup_h X_1$ has weakly even equivariant intersection form.
Lemma~\ref{lem:WeaklyEvenUnivCoverSpin} ensures that the universal cover of $X_0 \cup_h X_1$ is spin.

We focus on the converse.
First, we note that Propositions~\ref{prop:H2piZpi=0} and~\ref{prop:ev*Iso} show that our assumption on the group~$\pi$ ensures that the hypotheses of Theorems~\ref{thm:MainTheoremBoundaryIntro},~\ref{thm:PrimaryObstructionRecastIntro} and~\ref{thm:VanishSpinIntro} are satisfied.
Next,
by Theorem~\ref{thm:PrimaryObstructionRecastIntro},  the existence of a compatible triple ensures the existence of~$3$-connected maps~$c_0 \colon X_0 \to B,c_1 \colon X_1 \to B$ such that~$c_0|_{Y_0} =  c_1|_{Y_1} \circ h$
and with vanishing primary obstruction.
In the nonspin case,  Theorem~\ref{thm:VanishOddIntro} guarantees that $b(c_0,c_1)=0 \in \Herm(H^2(Y_1;\Z[\pi]))/\mathcal{G}.$
In the spin case, 
Lemma~\ref{lem:WeaklyEvenUnivCoverSpin} ensures that~$\lambda_{X_0 \cup_h X_1}$ is weakly even and Theorem~\ref{thm:VanishSpinIntro} then implies that~$b(c_0,c_1)=0 \in \Herm(H^2(Y_1;\Z[\pi]))/\mathcal{G}.$
Theorem~\ref{thm:MainTheoremBoundaryIntro} now ensures that the homotopy equivalence extends.
For the final sentence, Theorem~\ref{thm:VanishSpinIntro} ensures that there is a~$c_0' \simeq c_0$ with~$c_1|_{Y_1} \circ h=c_0'|_{Y_0}$ and~$b(c_0',c_1)=0  \in \Herm(H^2(Y_1;\Z[\pi]))$, so the fact that a given compatible triple can be realised follows from the corresponding clauses in Theorems~\ref{thm:PrimaryObstructionRecastIntro} and~\ref{thm:MainTheoremBoundary}.
\end{proof}

\section{Preparatory results}
\label{sec:Prep}

This section collects some preliminary results needed to define the primary and secondary obstructions.
Section~\ref{sub:PoincarePairs} records some facts about Poincar\'e pairs.
Given a~$4$-dimensional Poincar\'e pair~$(X,Y)$ with~$\pi_1(Y) \to \pi_1(X)=:\pi$ surjective,  Section~\ref{sub:pi2} focuses on the~$\Z[\pi]$-module~$\pi_2(X)$.
Next, given a space~$B$ and a~$\sigma \in H_4(B)$,  Section~\ref{sub:Cap} considers the pairing
$$
b_\sigma  \colon H^2(B;\Z[\pi]) \times H^2(B;\Z[\pi])  \to  \Z[\pi],
(\alpha,\beta) \mapsto \langle \beta,\alpha \cap \sigma \rangle.
$$
Section~\ref{sub:Pushout} is concerned with the 
pushout of two Poincar\'e pairs along a homotopy equivalence.
Finally the short Section~\ref{sub:ExtensionLemma}  states and proves an extension lemma.

\subsection{Poincar\'e pairs}
\label{sub:PoincarePairs}

This short section collects two lemmas concerning Poincar\'e pairs.
\medbreak

We begin by proving a fact that was mentioned in Remark~\ref{rem:PoincarePair}.

\begin{lemma}
	\label{lem:manifolds-pd-pairs}
	If $M$ is an $n$-manifold, then $(M,\partial M)$ is a $n$-dimensional Poincar\'e pair.
\end{lemma}
\begin{proof}
	Since~$\partial M$ is a manifold, it satisfies Poincar\'e duality and is homotopy equivalent to a CW complex, see e.g.~\cite[Theorem~3.16]{FriedlNagelOrsonPowell}.
	Hence~$\partial M$ is a Poincar\'e complex. 
	Since~$M$ satisfies Poincar\'e duality and~$\partial M\hookrightarrow M$ is a cofibration (thanks to the existence of a collar neighborhood of~$\partial M$),  it only remains to show that~$(M,\partial M)$ is homotopy equivalent to a CW pair.

The manifolds $M$ and $\partial M$ are homotopy equivalent to CW complexes $X$ and $Y$, respectively, see e.g.~\cite[Theorem~3.16]{FriedlNagelOrsonPowell}.
Choose a homotopy equivalence $h\colon \partial M\to Y$,  a homotopy inverse $\overline{h} \colon Y \to \partial M$, and consider the map
	\[f\colon Y\xrightarrow{\overline{h}} \partial M\hookrightarrow M\simeq X.\]
By cellular approximation,  after possibly changing $f$ by a homotopy, we can assume that it is cellular.
	Since~$Y$ is a CW complex and~$f$ is cellular, it follows that the mapping cylinder~$M(f)$ is a CW complex containing $Y$ as a subcomplex.
	We will show that $(M,\partial M)\simeq (M(f),Y)$.
	
	The homotopy equivalence~$g\colon M\simeq X\hookrightarrow M(f)$ restricts to a map~$g|_{\partial M}\colon \partial M\to M(f)$ such that the composition $Y\xrightarrow{\overline{h}} \partial M\xrightarrow{g|_{\partial M}}M(f)$ is homotopic to the composition $Y\xrightarrow{f}X\hookrightarrow M(f)$. 
	The latter is homotopic to the inclusion $Y\hookrightarrow M(f)$. 
It follows that~$g|_{\partial M}$ is homotopic to the composition $\partial M\xrightarrow{h}Y\hookrightarrow M(f)$. 
Applying homotopy extension, $g$ is homotopic to a homotopy equivalence $g'\colon M\to M(f)$ that restricts to the homotopy equivalence~$h\colon \partial M\to Y$. 
Since the maps~$\partial M\to M$ and $Y\to M(f)$ are cofibrations, the map $(g',h)\colon (M,\partial M)\to (M(f),Y)$ is a homotopy equivalence of pairs by e.g.~\cite[7.4.2]{BrownTopologyGroupoids}.
\end{proof}

\begin{lemma}
	\label{lem:PoincarePairIs4Complex}
	For $n\geq 4$, every $n$-dimensional Poincar\'e pair $(X,Y)$ is homotopy equivalent to a~CW pair $(X',Y')$ such that $Y'$ is $(n-1)$-dimensional and $X'$ is $n$-dimensional.
\end{lemma}
\begin{proof}
	By definition of a Poincar\'e pair, $Y$ is an $(n-1)$-dimensional Poincar\'e complex. 
Next, since~$H^i(Y;\Z[\pi_1(Y)]) \cong H_{n-i-1}(Y;\Z[\pi_1(Y)])=0$ for $i > (n-1)$, a result of Wall~\cite[Corollary 5.1]{WallFinitenessII}
ensures that~$Y$ is homotopy equivalent to an~$(n-1)$-dimensional CW complex $Y'$.
	
	If $Y$ is empty, then the previous paragraph can be applied to $X$, establishing the lemma in this case.
Thus,  from now on,  we assume that $Y$ is nonempty.
 Set~$\pi:=\pi_1(X)$ and note that since~$H^i(X;\Z[\pi]) \cong H_{n-i}(X,Y;\Z[\pi])=0$ for $i > n$,  a second application of~\cite[Corollary~5.1]{WallFinitenessII},  ensures that~$X$ is homotopy equivalent to an $n$-complex $X''$.
 Choose a homotopy equivalence~$h\colon Y\to Y'$,  a homotopy inverse $\overline{h} \colon Y' \to \partial Y$, and consider the map
$$f \colon Y'\xrightarrow{\overline{h}} Y \hookrightarrow X \simeq X''.$$
By cellular approximation,  after possibly changing $f$ by a homotopy, we can assume that it is cellular.
	Since~$Y'$ is an $(n-1)$-dimensional CW complex,~the mapping cylinder $M(f)$ is an~$n$-dimensional CW complex.
The exact same argument as the one in the last paragraph of the proof of Lemma~\ref{lem:manifolds-pd-pairs} then produces a homotopy equivalences of pairs $ (X,Y)\to (M(f),Y')$.
\end{proof}

\subsection{The second homotopy group}
\label{sub:pi2}

We study $\pi_2(X)$ as a $\Z[\pi_1(X)]$-module.
The main result,  stated in Proposition~\ref{prop:StablyFreePD<4},  describes a necessary and sufficient condition for it to be projective.

The next result serves as a first step towards deciding when $\pi_2(X)$ is projective.

\begin{lemma}
\label{lem:pi2Projective-new}
Let~$(X,Y)$ be a 4-dimensional Poincar\'e pair such that~$\pi_1(Y)\to \pi_1(X)=:\pi$ is surjective. 
    There exist finitely generated free~$\Z[\pi]$-modules~$C_0$,~$C_1$ and~$C_2$ and an exact sequence of~$\Z[\pi]$-modules
    \[0\to \pi_2(X)\to C_2\to C_1\to C_0\to \Z\to 0.\]
   In particular,  $\pi_2(X)$ is free as an abelian group.
\end{lemma}
\begin{proof}
By Lemma~\ref{lem:PoincarePairIs4Complex}, we can assume that~$X$ is obtained from~$Y$ by attaching cells of dimension at most~$4$.
Since~$\pi_1(Y)\to \pi_1(X)=:\pi$ is surjective, we can furthermore assume that~$X$ is obtained from~$Y$ by attaching cells of dimension~$2, 3$ and~$4$ (to see this,  adapt the proof that~$1$-handles can be traded for~$3$-handles, as e.g. in~\cite[page 258]{DET},  with the birth and death of cancelling~$1,2$-handles and~$2,3$-handles replaced by elementary expanses and collapses of~$1,2$-cells and~$2,3$-cells).
It follows that~$C^1(X,Y;\Z[\pi])=0$ and therefore
\[\pi_2(X)\cong H_2(X;\Z[\pi])\cong H^2(X,Y;\Z[\pi])=\ker \left(d^3\colon C^2(X,Y;\Z[\pi])\to C^3(X,Y;\Z[\pi])\right).\] 
Since~$H^3(X,Y;\Z[\pi])\cong H_1(X;\Z[\pi])=0$ and~$H^4(X,Y;\Z[\pi])\cong H_0(X;\Z[\pi])\cong \Z$,  considering the free $\Z[\pi]$-modules~$C_i:=C^{4-i}(X,Y,\Z[\pi])$ provides the required exact sequence.
\end{proof}

The next proposition characterises when $\pi_2(X)$ is projective.

\begin{proposition}
\label{prop:StablyFreePD<4}
Let~$(X,Y)$ be a 4-dimensional Poincar\'e pair such that~$\pi_1(Y)\to \pi_1(X)=:\pi$ is surjective. 
The following assertions are equivalent:
\begin{itemize}
\item the fundamental group~$\pi$ of~$X$ satisfies~$\cd(\pi)\leq 3$,
\item the~$\Z[\pi]$-module~$\pi_2(X)$ is projective.
\end{itemize}
\end{proposition}
\begin{proof}
If~$\pi_2(X)$ is projective, then the exact sequence from \cref{lem:pi2Projective-new} gives a $\Z[\pi]$-projective resolution of~$\Z$ of length~$3$. 
    Hence~$\cd(\pi)\leq 3$ as needed.
Conversely, if~$\cd(\pi)\leq 3$, then there is a projective resolution
$0\to P_3\to P_2\to P_1\to P_0\to \Z\to 0$, and Schanuel's lemma applied to the exact sequence from \cref{lem:pi2Projective-new}
implies that $\pi_2(X) \oplus P_2\oplus C_1\oplus P_0\cong P_3\oplus C_2\oplus P_1\oplus C_0.$
    It follows that~$\pi_2(X)$ is projective,  as required. 
\end{proof}

\subsection{The cap product pairing}
\label{sub:Cap}

Let~$(X,Y)$ be a~$4$-dimensional Poincar\'e pair such that the inclusion induced map~$\pi_1(Y) \to \pi_1(X)=:\pi$ is surjective, and let~$c\colon X \to B$ be a~$3$-equivalence to its Postnikov~$2$-type.
Given~$\sigma \in H_4(B)$,  we consider the pairing (which is hermitian by Lemma~\ref{lem:cup})
\begin{align*}
b_\sigma  \colon H^2(B;\Z[\pi]) \times H^2(B;\Z[\pi]) & \to  \Z[\pi] \\
(\alpha,\beta) &\mapsto \langle \beta,\alpha \cap \sigma \rangle.
\end{align*}
The goal of this section is to study the assignment~$\sigma \mapsto b_\sigma$ and to find conditions ensuring it is injective.
The first step is to relate~$H_4(B)$ to~$\Z \otimes_{\Z[\pi]} H_4(\widetilde{B})$.
Indeed, Whitehead's exact sequence (from \cite{WhiteheadCertainExact}) implies that the latter group is isomorphic~$\Z \otimes_{\Z[\pi]} \Gamma(\pi_2(B))$ and is therefore related to hermitian forms on~$H^2(B;\Z[\pi])$.

\medbreak

\begin{proposition}
\label{prop:H4(B)}
Let~$(X,Y)$ be a~$4$-dimensional Poincar\'e pair such that~$\pi_1(Y) \to \pi_1(X)=:\pi$ is surjective, and let~$c\colon X \to B$ be a~$3$-equivalence to its Postnikov~$2$-type.
If~$\cd(\pi)\leq 3$ or~$\pi$ is finite,  then the projection induces an isomorphism
$$ \Z \otimes_{\Z[\pi]} H_4(\widetilde{B}) \xrightarrow{\cong} H_4(B).$$
\end{proposition}
\begin{proof}
Consider the Leray--Serre spectral sequence associated to the fibration $\wt B\to B\to B\pi$:
\[H_p(\pi;H_q(\wt B))\Rightarrow H_{p+q}(B).\]
Since $\wt B$ is simply-connected, $H_1(\wt{B})=0$ and the Hurewicz theorem gives a surjection 
$0=\pi_3(B)=\pi_3(\widetilde{B}) \twoheadrightarrow H_3(\widetilde{B})$~\cite[page 390]{Hatcher} so $H_3(\widetilde{B})=0$.
	
We first assume that~$\cd(\pi)\leq 3$.
In this case, $H_p(\pi;H_0(\wt B))=0$ for $p\geq 4$. 
We also have~$H_p(\pi;H_2(\wt B))=0$ for $p\geq 1$ because \cref{prop:StablyFreePD<4} ensures that~$H_2(\wt B) \cong \pi_2(B) \cong \pi_2(X)$ is projective.
This shows that~$\Z \otimes_{\Z[\pi]} H_4(\wt B)$ is the only nonzero term on the $p+q=4$-line of the spectral sequence.
This term survives to the infinity page since all potential differentials into it have trivial source. 
The isomorphism $\Z \otimes_{\Z[\pi]} H_4(\widetilde{B}) \cong H_4(B)$ follows.

Next, we assume that~$\pi$ is finite. 
We first consider the Leray--Serre spectral sequence for the fibration~$\wt X\to X\to B\pi$.
	Since~$\pi$ is finite,~$H^1(X;\Z[\pi])\cong H^1(\pi;\Z[\pi])=0$. 
As~$Y$ is connected and~$\pi_1(Y)\to \pi_1(X)$ is surjective,  we deduce that~$H_3(\wt X)\cong H_3(X;\Z[\pi])\cong H^1(X,Y;\Z[\pi])= 0.$
Also,  note that~$H_i(\wt X)\cong H_i(X;\Z[\pi])\cong H^{4-i}(X,Y;\Z[\pi])= 0$ for every $i\geq 4$,  where for $i=4$ we used that $Y \to X$ is $\pi_1$-surjective.
	Hence the only two non-trivial terms on the $p+q=n$-line are~$H_{n-2}(\pi;H_2(\wt X))$ and~$H_{n}(\pi;H_0(\wt X))\cong H_n(\pi;\Z)$.
\begin{claim}
	For $n\geq 3$, the following third page differentials are isomorphisms:
	$$d_3 \colon H_n(\pi;\Z)\to H_{n-3}(\pi;H_2(\wt X)).$$
\end{claim} 
\begin{proof}
We first note that the~$d_3$\ differentials~$H_n(\pi;\Z)\to H_{n-3}(\pi;H_2(\wt X))$ are injective for~$n \geq 3$: otherwise~$\ker(d_3)$ would contribute to~$H_n(X)$
 via the infinity page. This gives a contradiction since~$H_i(X)\cong H^{4-i}(X,Y)=0$ for~$i \geq 3$ because $H_1(Y)\to H_1(X)$ is surjective.
By \cref{lem:pi2Projective-new}, there is an exact sequence
 \[0\to H_2(\widetilde{X})\to C_2\xrightarrow{d_2} C_1\xrightarrow{d_1} C_0\to \Z\to 0.\]
It follows by dimension shifting (i.e.  by using $H_i(\pi;\Z[\pi])=0$ for~$i>0$ and the long exact sequence associated to a short exact sequence of coefficients, see for example \cite[III.(7.1)]{Brown}) that
 \[H_n(\pi;H_2(\widetilde{X}))\cong H_{n+1}(\pi;\im d_2)\cong H_{n+2}(\pi;\im d_1) \cong H_{n+3}(\pi;\Z).\]
Since~$\pi$ is finite, the modules~$H_{n+3}(\pi;\Z)$ are finite for every~$n\geq -2$. Hence the~$d_3$ differentials~$H_n(\pi;\Z)\to H_{n-3}(\pi;\pi_2(X))$ are isomorphisms since they are injective maps between finite modules of the same cardinality.
This concludes the proof of the claim.
\end{proof}
Since~$X\to B$ is 3-connected,  the~$d_3$-differentials are also isomorphisms for~$B$. It follows that on the $p+q=4$-line of the infinity page, only the term $\Z\otimes_{\Z[\pi]}H_4(\wt B)$ survives.  This gives the required result.
\end{proof}

We are now in position to relate
$H_4(B)$ to hermitian forms on~$H^2(B;\Z[\pi])$.
As explained in Proposition~\ref{prop:H4(B)},  for~$\cd(\pi)\leq 3$ or~$\pi$ finite,  we have~$H_4(B) \cong \Z \otimes_{\Z[\pi]} H_4(\widetilde{B})$ and Whitehead's exact sequence implies that this latter group is isomorphic to~$\Z \otimes_{\Z[\pi]} \Gamma(\pi_2(B))$~\cite{WhiteheadCertainExact}.
We refer to Section~\ref{sub:WhiteheadGamma} for a brief recollection of $\Gamma$-groups, but for the moment the definition can be blackboxed: we only require the fact that $\Gamma(\pi_2(B))$ can be identified with the subgroup of symmetric elements of~$\pi_2(B) \otimes \pi_2(B)$ since $\pi_2(B)$ is free as an abelian group by Lemma~\ref{lem:pi2Projective-new}.

In order to relate $\Z \otimes_{\Z[\pi]} \Gamma(\pi_2(B))$ to hermitian forms on~$H^2(B;\Z[\pi])$,  consider
\begin{align*}
\mathcal{B} \colon \Z \otimes_{\Z[\pi]} \Gamma(\pi_2(B))  &\to  \operatorname{Herm}(\pi_2(B)^*)\\
1 \otimes \left( a \otimes b \right)  & \mapsto  \left( (f,g) \mapsto \overline{f(b)}g(a) \right).
\end{align*}
Indeed the target of~$\mathcal{B}$ can be related to~$\operatorname{Herm}(H^2(B;\Z[\pi]))$ using the evaluation-induced map 
$$\Herm(\ev) \colon \operatorname{Herm}(\pi_2(B)^*) \cong \operatorname{Herm}(H_2(B;\Z[\pi])^*) \to \operatorname{Herm}(H^2(B;\Z[\pi]))$$
and the map $\ev$ can be analysed using the universal coefficient spectral sequence.
\begin{proposition}
\label{prop:HillmanCite}
Let~$(X,Y)$ be a~$4$-dimensional Poincar\'e pair such that~$\pi_1(Y) \to \pi_1(X)=:\pi$ is surjective, and let~$c\colon X \to B$ be a~$3$-equivalence to its Postnikov~$2$-type.
Assume 
\begin{itemize}
\item either that~$\cd(\pi)\leq 3$ and~$\Herm(\ev) \colon \operatorname{Herm}(\pi_2(B)^*)\to \operatorname{Herm}(H^2(B;\Z[\pi]))$ is injective,
\item or that~$\pi$ is finite and that~$\Z\otimes_{\Z[\pi]}\Gamma(\pi_2(B))$ is torsion free, 
\end{itemize}
then the following composition is injective:
 \begin{align*}
b \colon H_4(B) &\xleftarrow{\cong}  \Z \otimes_{\Z[\pi_1]} H_4(\widetilde{B}) 
\xrightarrow{\cong } \Z \otimes_{\Z[\pi_1]} \Gamma(\pi_2(B)) 
\xrightarrow{\mathcal{B}}  \operatorname{Herm}(\pi_2(B)^*) 
\xrightarrow{}  \operatorname{Herm}(H^2(B;\Z[\pi])).
\end{align*}
Additionally it satisfies
\begin{equation}
\label{eq:bsigmaIsWhatYouThink}
b(\sigma)(x,y)=\langle y, x \cap \sigma\rangle.
\end{equation}
\end{proposition}
\begin{proof}
Proposition~\ref{prop:H4(B)} shows that the first map is an isomorphism.
The second map is known to be an isomorphism thanks to Whitehead's exact sequence~\cite{WhiteheadCertainExact}.
If~$\cd(\pi)\leq 3$, then~$\pi_2(B)$ is projective by \cref{prop:StablyFreePD<4}. In this case, the map~$\mathcal{B}$ is injective by~\cite[Theorem 7 with addendum]{HillmanStrongly}. If~$\pi$ is finite, the map~$\mathcal{B}$ is injective if and only if~$\Z\otimes_{\Z[\pi]}\Gamma(\pi_2(B))$ is torsion free by~\cite[Proposition~6.15]{HKPR}.
It remains to show that the last map is injective if~$\pi$ is finite. If~$\pi$ is finite, then~$H^3(\pi;\Z[\pi])=0$ and hence~$H^2(B;\Z[\pi]) \to H_2(B;\Z[\pi])^*$ is surjective. It follows that~$\operatorname{Herm}(\pi_2(B)^*) \to \operatorname{Herm}(H^2(B;\Z[\pi]))$ is injective as needed.
Finally,  the equality in~\eqref{eq:bsigmaIsWhatYouThink} follows from the commutativity of the diagram in~\cite[(2.7)]{HKPR}.
\end{proof}

\subsection{Pushouts of Poincar\'e pairs}
\label{sub:Pushout}

This section argues that the pushout of $n$-dimensional Poincar\'e pairs along a homotopy equivalence is an $n$-dimensional Poincar\'e complex.

\medbreak

Wall proved that if $(X_0,Y)$ and $(X_1,Y)$ are $n$-dimensional Poincar\'e pairs that are CW pairs,
then the union~$X_0 \cup_{Y} X_1$ (which is necessarily a CW complex) is an $n$-dimensional Poincar\'e complex~\cite[Theorem~2.1~i)]{WallPoincare}.
The next proposition builds on this result to include more general unions.

\begin{proposition}
	\label{prop:HomotopyPushout}
	Let $(X_0,Y_0)$ and $(X_1,Y_1)$ be $n$-dimensional Poincar\'e pairs.
	If $h \colon Y_0 \to Y_1$ is a homotopy equivalence, then the pushout $X_0 \cup_h X_1$ is an $n$-dimensional Poincar\'e complex.
\end{proposition}
\begin{proof}
	By \cref{lem:PoincarePairIs4Complex}, there are homotopy equivalences $(X_i,Y_i)\xrightarrow{(f_i,g_i)}(X_i',Y_i')$ with $(X_i',Y_i')$ a~CW pair for $i=0,1$. 
	Let $\overline{g}_0$ be a homotopy inverse of $g_0$, and let $h'\colon Y_0' \to Y_1'$ be a cellular map that is homotopic to $g_1\circ h\circ \overline{g}_0$. The mapping cylinder $M(h')$ of $h'$ is a CW complex that contains $Y_0'$ as a subcomplex. 
	Since the projection $(\pr,h')\colon (M(h'),Y_0')\to (X_1',Y_1')$ is a homotopy equivalence, $(M(h'),Y_0')$ is a Poincar\'e pair. 
	By \cite[Theorem~2.1~i)]{WallPoincare}, $X_0'\cup_{Y_0'}M(h')$ is a Poincar\'e complex. 
	It remains to show that $X_0\cup_h X_1$ and $X_0'\cup_{Y_0'}M(h')$ are homotopy equivalent.

By definition of $h'$, there is a homotopy, $h'\circ g_0 \simeq f_1| \circ h \colon Y_0 \to X_1$.
	By homotopy extension, there exists a map $f_1'\colon X_1\to X_1'$ that is homotopic to $f_1$ and such that $h'\circ g_0=f_1' \circ h$. 
Therefore the following diagram commutes:
	\[\begin{tikzcd}
		X_0'\ar[d,"="']&Y_0'\ar[r,hook]\ar[l,hook']\ar[d,"="']&M(h')\ar[d,"\pr","\simeq"']\\
		X_0'&Y_0'\ar[l,hook']\ar[r,"h'"]&X_1'\\
		X_0\ar[u,"f_0"',"\simeq"]&Y_0\ar[u,"g_0"',"\simeq"]\ar[l]\ar[r,"h"]&X_1. \ar[u,"f_1'"',"\simeq"]
	\end{tikzcd}\]
Since $Y_0 \to X_0$ and $Y_0' \to X_0'$ are cofibrations, the gluing lemma (see e.g.~\cite[Lemma 2.1.3]{MayPonto}) 
 then ensures that there are homotopy equivalences
	\[X_0\cup_hX_1\simeq X_0'\cup_{h'}X_1'\simeq X_0'\cup_{Y_0'}M(h').\]
	Thus $X_0\cup_hX_1$ is a Poincar\'e complex,  as required.
\end{proof}

We conclude with an observation concerning the fundamental group of such pushouts.

\begin{proposition}
\label{prop:pi1HomotopyPushout}
Let $(X_0,Y_0)$ and $(X_1,Y_1)$ be two $n$-dimensional Poincar\'e pairs with isomorphic fundamental groups, and let $h \colon Y_0 \to Y_1$ be a homotopy equivalence.
If $\iota_j \colon \pi_1(Y_i) \to \pi_1(X_i)$ is surjective for $j=0,1$ and if there is an isomorphism~$u \colon \pi_1(X_0) \to \pi_1(X_1)$ with $u \circ \iota_0=\iota_1 \circ h_*$,  then the inclusion 
 induces an isomorphism
$$\pi_1(X_1) \xrightarrow{\cong} \pi_1(X_0 \cup_h X_1).$$
\end{proposition}
\begin{proof}
Since $X_0 \cup_h X_1$ is a pushout,  so is~$\pi_1(X_0 \cup_h X_1)$.
Using that $\iota_0$ and $\iota_1$ are surjective, a direct verification shows that $\pi_1(X_1)$ also satisfies the universal property of the pushout.
The proposition then follows readily.
\end{proof}

\begin{remark}
\label{rem:X0cup-X1}
We fix a choice of fundamental class for $X_0 \cup_h \unaryminus X_1$.
Since $h$ is orientation-preserving,  under the composition
$$ H_4(X_0 \cup_h X_1) \to H_4(X_0\cup_h X_1,Y_1) \xleftarrow{\cong} H_4(X_0,Y_0) \oplus H_4(X_1,Y_1),$$
a fundamental class of $X_0 \cup_h \unaryminus X_1$ maps either to $(\unaryminus [X_0],[X_1])$ or to~$([X_0],\unaryminus [X_1])$.
In what follows,  we pick the latter choice and occasionally write $X_0 \cup_h \unaryminus X_1$ for emphasis.
\end{remark}
\subsection{An extension lemma}
\label{sub:ExtensionLemma}

This short section records a technical lemma which is a small generalisation of the following statement ``given a space~$X_0$, if a map~$c_1 \colon X_1 \to~B$ is~$n$-connected, then the induced map~$(c_1)_* \colon [X_0^{(n)},X_1] \to [X_0,B]$ is surjective for any CW complex $X_0$"

\begin{lemma}
\label{lem:SurjectivePushForward}
Let~$(X_0,D_0)$ be an~$n$-dimensional relative CW-space, i.e~$X_0$ is obtained from~$D_0$ by attaching cells of dimension at most~$n$.
Let~$c_1\colon X_1 \to B$ be an~$n$-connected map. Let~$D_1\subseteq X_1$ and let~$h\colon D_0\to D_1$ be given. For every map~$c_0\colon X_0\to B$ such that~$c_0|_{D_0}\simeq c_1|_{D_1}\circ h$ there exists a map~$f\colon X_0\to X_1$ extending~$h$ such that~$c_0\simeq c_1\circ f$.
\end{lemma}
\begin{proof}
We argue that without loss of generality we can assume that~$c_1$ is an inclusion and~$c_0|_{D_0}=h$.
By taking the mapping cylinder of~$c_1$, we can assume that~$c_1$ is an inclusion: indeed writing~$\proj \colon M(c_1) \xrightarrow{\simeq} B$, if $h$ extends to~$f \colon X_0 \to X_1$ with~$\incl_B \circ c_0\simeq \operatorname{incl}_{X_1} \circ f \colon X_0 \to M(c_1)$, then $c_1 \circ f=\proj \circ \incl_{X_1} \circ f=\proj \circ \incl_B \circ c_0 \simeq c_0.$
We can further assume that~$c_0|_{D_0}=h$: 
homotopy extension 
ensures that~$h \colon D_0 \to Y_1 \subset B$ extends to a map $c_0' \colon X_0 \to B$ that is homotopic to $c_0$,  so if $h$ extends to~$f \colon X_0 \to X_1$ with $c_1 \circ h \simeq c_0'$, then $c_1 \circ h \simeq c_0$.
\color{black}
	
Since we are now assuming that $c_1$ is the inclusion and $c_0|_{D_0}=h$,  the map $c_0$ determines a map  of pairs~$c_0\colon (X_0,D_0)\to (B,X_1)$.
Since $c_1$ is $n$-connected~\cite[Lemma~4.6]{Hatcher} implies that the map~$c_0\colon (X_0,D_0)\to (B,X_1)$ is homotopic relative~$D_0$ to a map~$f\colon X_0\to X_1$. 
Since the homotopy is relative~$D_0$,  it follows that~$f$ extends~$h$ as required.
\end{proof}

\section{The primary and secondary obstructions}
\label{sec:PrimarySecondary}

The next two sections aim to prove Theorem~\ref{thm:MainTheoremBoundaryIntro}.
Recall that given~$4$-dimensional Poincar\'e pairs~$(X_0,Y_0)$ and~$(X_1,Y_1)$ and a homotopy equivalence~$h \colon Y_0 \to Y_1$,  this theorem describes two successive obstructions to $h$ extending to a homotopy equivalence $X_0 \to X_1$.
Before describing our strategy, we fix some notation.

\begin{notation}
\label{notation:SectionPrimarySecondary}
Let~$(X_0,Y_0)$ and~$(X_1,Y_1)$ be~$4$-dimensional Poincar\'e pairs with~$\pi_1(X_j) \cong \pi$, Postnikov $2$-type $B:=P_2(X_1)$ and~$\iota_j \colon \pi_1(Y_j) \to \pi_1(X_j)$ surjective for~$j=0,1$.
Fix a degree one homotopy equivalence~$h \colon Y_0 \to Y_1$.
Given $3$-connected maps~$c_j \colon X_j \to B$ for~$j=0,1$, when we write~$(B, X_1)$ or~$(B,Y_1)$, it is understood that we have replaced~$B$ by the mapping cylinder of~$X_1 \to B$.
We additionally assume that~$c_1|_{Y_1} \circ h = c_0|_{Y_0}$.
Recall that that this ensures that~$u:=(c_1)_*^{-1} \circ (c_0)^* \colon \pi_1(X_0) \to \pi_1(X_1)$ satisfies $u \circ \iota_0=\iota_1 \circ h_*$.
\end{notation}

Section~\ref{sub:FundamentalClass} reduces the extension problem to one concerning an equality of fundamental classes.
Section~\ref{sub:FundamentalClassToIntersection} then recasts this later problem in terms of relative intersection forms.

\subsection{Reduction to fundamental classes}
\label{sub:FundamentalClass}

We reduce the problem of extending the homotopy equivalence $h \colon Y_0 \to Y_1$ to a question about the fundamental classes of~$X_0$ and~$X_1$.
The argument is inspired by work of Hambleton-Kreck~\cite{HambletonKreck}. 

\medbreak
We first describe a sufficient condition for a map extending $h$ to be a homotopy equivalence.

\begin{lemma}
\label{lem:DegreeOneCommute}
If a map~$f \colon  X_0 \to X_1$ extends~$h \colon Y_0 \to Y_1$ and satisfies~$c_1 \circ f \simeq c_0$, then it is a homotopy equivalence.
\end{lemma}
\begin{proof}
First,~$f$ induces an isomorphism on $\pi_1$ and $\pi_2$ because~$c_0$ and~$c_1$ are~$3$-connected.
Next, since~$(X_i,Y_i)$ is $1$-connected,  we obtain~$H_k(X_i;\Z[\pi]) \cong H^{4-k}(X_i,Y_i;\Z[\pi])=0$ for $k=3,4$.
Whitehead's theorem implies that $f$ is a homotopy equivalence.
\end{proof}

The next proposition shows that an equality involving fundamental classes ensures that the sufficient condition form Lemma~\ref{lem:DegreeOneCommute} is satisfied

\begin{proposition}
\label{prop:FundamentalClass}
If the maps~$c_0$ and~$c_1$ satisfy
$$(c_0)_*([X_0])=(c_1)_*([X_1]) \in H_4(B,Y_1),$$
then there exists a map~$f \colon  X_0 \to X_1$ extending~$h$ such that~$c_1 \circ f \simeq c_0$.
\end{proposition}
\begin{proof}
Since~$c_1$ is~$3$-connected, Lemma~\ref{lem:SurjectivePushForward} gives the existence of a map~$f \colon X_0^{(3)} \to X_1^{(3)} \subset X_1$ that restricts to~$h$ on~$Y_0$ and such that~$c_1 \circ f \simeq c_0.$
The homotopy extension property (applied to $c_1 \circ f \simeq c_0$ on $X_0^{(3)}$) ensures the existence of a~$3$-connected map $c_0'\simeq c_0 \colon X_0 \to B$ such that~$c_0'|_{X_0^{(3)}}=c_1|_{X_1^{(3)}} \circ f$.
Without loss of generality we can therefore assume that $c_1 \circ f = c_0$ on~$X_0^{(3)}$.
Consider the obstruction theoretic problem of extending~$f$ over~$X_0$ rel~$h$.

Using Lemma~\ref{lem:PoincarePairIs4Complex} we can assume that $X_0$ is obtained from $Y_0$ by attaching cells of dimension at most~$4$.
Indeed, if there is a homotopy equivalence of pairs~$(X_0,Y_0) \simeq (X_0',Y_0')$, we write $g \colon Y_0 \simeq Y_0'$ for the induced homotopy equivalence, choose a homotopy inverse $\overline{g}$,  and note that if $h \circ  \overline{g} \circ g \colon Y_0' \to Y_1$ extends, then by homotopy extension applied to $h \circ  \overline{g} \circ  g \simeq h$,  so does $h$.

Since~$f$ is defined on the whole of~$X_0$ apart from on the~$4$-cells,  the primary obstruction to our extension problem is a class in~$H^4(X_0,\partial X_0;\pi_3(X_1))$ whose vanishing ensures that~$f|_{X_0^{(2)} \cup \partial X_0 }$ extends to a map~$f \colon X_0 \to X_1$ rel~$h$.
This extension would then automatically satisfy~$c_1  \circ f \simeq c_0$:
since~$\pi_3(B)$ and~$\pi_4(B)$ vanish,  any homotopy~$c_1 \circ f|_{X_0^{(2)}} \simeq c_0|_{X_0^{(2)}}$ can be extended rel. ~$h$ to a homotopy~$c_1 \circ f \simeq c_0$.

We now analyse the obstruction in~$H^4(X_0,\partial X_0;\pi_3(X_1))$.
The composition of the Hurewicz map and the map~$\Z \otimes_{\Z[\pi]} H_4(B,X_1;\Z[\pi]) \to  H_4(B,X_1;\Z \otimes_{\Z[\pi]} \Z[\pi])$ induces an isomorphism
$$\Z \otimes_{\Z[\pi]} \pi_4(B,X_1)
\cong \Z \otimes_{\Z[\pi]} H_4(B,X_1;\Z[\pi])
 \cong H_4(B,X_1).$$
Here, the relative Hurewicz theorem applies because~$c$ is~$3$-connected while the second isomorphism follows from the UCSS because~$H_i(B,X_1;\Z[\pi])=0$ for~$i \leq 3$.

Since $B$ is $3$-coconnected, we obtain the following composition of isomorphisms:
\begin{align*}
H^4(X_0,Y_0;\pi_3(X_1)) 
&\cong H^4(X_0,Y_0;\pi_4(B,X_1))  
\cong H_0(X_0;\pi_4(B,X_1))  
\cong \Z \otimes_{\Z[\pi]} \pi_4(B,X_1) \\
& \cong H_4(B,X_1).
 \end{align*}
 We determine the image of the obstruction class in~$H_4(B,X_1)$ through these isomorphisms.
 \begin{claim}
 The image of the obstruction class in~$H_4(B,X_1)$ is given by 
$$(c_0)_*([X_0]) \in \im(H_4(B,Y_1) \to H_4(B,X_1)).$$
 \end{claim}
 \begin{proof}
 Write~$[X_0]=\sum_e n_e\otimes \wt e\in C_4^{\CW}(X_0,Y_0;\Z)$, where~$n_e\in \Z$ and the sum is taken over all~$4$-cells~$e$ of~$X_0\setminus Y_0$ with~$\widetilde{e}$ its lift to the universal cover (with respect to the fixed basepoint).
 Let~$\alpha_e \colon S^3 \to X_0^{(3)}\cup Y_0$ be the attaching map of~$e$.
The obstruction cocycle is the homomorphism~$C_4(\widetilde{X}_0,\widetilde{Y}_0) \to \pi_3(X_1)$ that maps~$\widetilde{e}$ to the homotopy class~$[S^3 \xrightarrow{\wt \alpha_e} \wt X_0^{(3)}\cup \wt Y_0\xrightarrow{\wt f} \wt X_1] \in \pi_3(X_1).$ 
We use~$\theta(f)$ to denote the cohomology class of this cocycle.

We describe a cocycle representative for the image of~$[\theta(f)]$ in~$H^4(X_0,Y_0;\pi_4(B,X_1))$.
Consider the commutative diagram
%
\[\xymatrix{
S^3\ar[r]^-{\wt \alpha_e}\ar[d]^\subset&\wt X_0^{(3)}\cup \wt Y_0\ar[r]^-{\wt f}&\wt X_1 \ar[d]^{\wt c_1}\\
D^4\ar[rr]^{\wt c_0|_{\wt e}}&& \wt B
}\]
to see that~$\wt e\mapsto \wt c_0|_{\wt e}$ determines the required cocycle and use~$\varphi$ to denotes its cohomology class.

We describe the image of this cohomology class in~$H_0(X_0,Y_0;\pi_4(B,X_1))\cong \Z \otimes_{\Z[\pi]} \pi_4(B,X_1)$.
By definition this class is~$\PD_{X_0}([\varphi])=\varphi \cap [X_0]$. 
We deduce that
\[\PD_{X_0}([\varphi]) = \varphi\cap \left[\sum_e n_e\otimes \wt e\right]=\left[\sum_e n_e \otimes \varphi(\wt e)\right]=\left[\sum_e n_e \otimes \wt c_0|_{\wt e}\right].\]
Under the Hurewicz map,~$\wt c_0|_{\wt e}\in \pi_4(B,X_1)$ is sent to~$(c_0)_*(\wt e)\in H_4(B,X_1)$, 
so~$\PD_{X_0}([\varphi])$ is sent to 
\[\sum_e n_e \otimes (c_0)_*(\wt e)=(c_0)_*\left(\sum_e n_e\otimes \wt e\right)=(c_0)_*([X_0]).\]
 This concludes the proof of the claim.
 \end{proof}
 We now assume that~$(c_0)_*([X_0])=(c_1)_*([X_1]) \in H_4(B,Y_1)$ and conclude the proof of the proposition.
The long exact sequence of the triple~$(B,X_1,Y_1)$ ensures that
$$(c_1)_*([X_1])=0 \in \im( H_4(B,Y_1) \to H_4(B,X_1)).$$
Combining these two equalities implies that 
$(c_0)_*([X_0])=0 \in \im( H_4(B,Y_1) \to H_4(B,X_1)).$
The claim then implies that  the required obstruction vanishes and the preceding explanation guarantees that~$h$ extends to a map~$f \colon  X_0 \to X_1$ with~$c_1 \circ f \simeq c_0$.
 \color{black}
\end{proof}

These results yield a tractable sufficient condition for~$h$ to extend to a homotopy equivalence.

\begin{proposition}
\label{prop:BauesBleileWeak}
If the maps~$c_0$ and~$c_1$ satisfy
$$(c_0)_*([X_0])=(c_1)_*([X_1]) \in H_4(B,Y_1),$$
 then~$h$ extends to a homotopy equivalence~$f \colon  X_0 \to X_1$ with~$c_1 \circ f \simeq c_0$.
\end{proposition}
\begin{proof}
Proposition~\ref{prop:FundamentalClass} proves the existence of a map~$f \colon  X_0 \to X_1$ extending~$h$ such that~$c_1 \circ f \simeq c_0$.
 Lemma~\ref{lem:DegreeOneCommute} then implies that~$f$ is a homotopy equivalence.
\end{proof}

The next proposition recasts this sufficient condition in terms of the vanishing of a class in~$H_4(B)$, thus placing us in a position to apply Proposition~\ref{prop:HillmanCite}.

\begin{proposition}
\label{prop:FundamentalClassUnion}
 There is a unique homology class~$\xi_0 \in H_4(B)$,  satisfying
$$ j_*(\xi_0)=(c_0)_*([X_0])-(c_1)_*([X_1]) \in H_4(B,Y_1).$$
 In particular,  we have 
$$ \xi_0=0 \ \ \  \Leftrightarrow  \ \  \ (c_0)_*([X_0])=(c_1)_*([X_1]) \in H_4(B,Y_1).$$
In fact,~$\xi_0=c_*([U])$, where~$(U,c):=(X_0 \cup_h \unaryminus X_1,c:=c_0 \cup c_1)$.
\end{proposition} 
 \begin{proof}
 Consider the connecting map~$\partial_B \colon H_4(B,Y_1) \to H_3(Y_1)$ as well as the corresponding connecting maps $\partial_0$ and $\partial_1$ for the pairs $(X_0,Y_0)$ and $(X_1,Y_1)$.
Since~$h$ has degree one,
a direct calculation shows that
\begin{align*}
\partial_B \circ (c_0)_*([X_0])-\partial_B \circ (c_1)_*([X_1])
&=h_* \circ \partial_0([X_0])-\partial_1([X_1]) 
=h_*([Y_0])-[Y_1] 
=0.
\end{align*}
The long exact sequence of the pair~$(B,Y_1)$ now ensures that the required~$\xi_0\in H_4(B)$ exists.
The uniqueness of~$\xi_0$ as well as the equivalence involving the vanishing of $\xi_0$ follow because the inclusion induced map~$H_4(B) \to H_4(B,Y_1)$ is injective.

Finally, we argue that~$j_* \circ c_*([U])=(c_0)_*([X_0])-(c_1)_*([X_1]) \in H_4(B,Y_1)$ as the uniqueness of~$\xi_0$ will then imply that~$\xi_0=c_*([U])$.
Consider the following commutative diagram:
$$
\xymatrix@R0.5cm{
H_4(U)\ar[d]^-{c_*}\ar[r]&H_4(U,Y_0)\ar[d]^{c_*}&H_4(X_0,Y_0) \oplus  H_4(X_1,Y_1)  \ar[l]_-\cong\ar[ld]^-{((c_0)_*,(c_1)_*)}  \\
H_4(B)\ar[r]^-{j_*}&H_4(B,Y_1).
}
$$
The top composition sends~$[U]$ to~$[X_0]-[X_1]$.
The conclusion now follows readily.
\end{proof}

We record the following observation on the hermitian form~$b_{\xi_0}(x,y) = \langle x , y \cap \xi_0 \rangle$ on~$H^2(B;\Z[\pi])$.
\begin{proposition}
\label{prop:bxiUnion}
The hermitian form $b_{\xi_0}$ satisfies~$ b_{\xi_0}=c^*b_U.$
\end{proposition}
\begin{proof}
Proposition~\ref{prop:FundamentalClassUnion} ensures that~$\xi_0=c_*([U])$.
Now, given~$x,y \in H^2(B;\Z[\pi])$, the result follows from a calculation:~$b_{\xi_0}(x,y)
=\langle x,y \cap \xi_0\rangle
=\langle x,y \cap c_*([U])\rangle
=\langle c^*(x),c^*(y) \cap [U] \rangle
=c^*b_X(x,y).
$
\end{proof}

\subsection{From fundamental classes to intersection forms}
\label{sub:FundamentalClassToIntersection}

We continue with Notation~\ref{notation:SectionPrimarySecondary}.
Additionally, as in Proposition~\ref{prop:FundamentalClassUnion},  we consider the unique homology class~$\xi_0 \in H_4(B)$ satisfying
$$ j_*(\xi_0)=(c_0)_*([X_0])-(c_1)_*([X_1]) \in H_4(B,Y_1).$$
Proposition~\ref{prop:FundamentalClassUnion} shows that $\xi_0 \in H_4(B)$ is an obstruction to $h$ extending to a homotopy equivalence.
Proposition~\ref{prop:HillmanCite} states that the problem of deciding whether~$\xi_0$ vanishes can be expressed terms of the pairing~$b_{\xi_0}(x,y) = \langle x , y \cap \xi_0 \rangle$ on~$H^2(B;\Z[\pi])$.
This section recasts the pairing $b_{\xi_0}$ in terms of the relative intersection forms of $X_0$ and $X_1$.
The outcome, namely Proposition~\ref{prop:FundamentalClassIntersectionForm},  is a preliminary version of Theorem~\ref{thm:MainTheoremBoundaryIntro} from the introduction.

\medbreak

The strategy outlined above relies on relative intersection forms.
\begin{definition}
\label{def:EquivariantIntersectionForm}
Let~$(X,Y)$ be a~$4$-dimensional Poincar\'e pair with $\pi:=\pi_1(X)$.
\begin{itemize}
\item The \emph{cohomological relative equivariant intersection form} is
\begin{align*}
b_X^\partial  \colon H^2(X,Y;\Z[\pi]) \times H^2(X;\Z[\pi]) & \to  \Z[\pi] \\
(\alpha,\beta) &\mapsto \langle \beta,\PD_X(\alpha) \rangle
=\langle \beta,\alpha \cap [X] \rangle.
\end{align*}
\item The \emph{relative equivariant intersection form} is
\begin{align*}
\lambda_X^\partial  \colon H_2(X;\Z[\pi]) \times H_2(X,Y;\Z[\pi]) & \to  \Z[\pi] \\
(x,y) &\mapsto \langle \PD_X^{-1}(y),x\rangle.
\end{align*}
\end{itemize}
\end{definition}

The appendix discusses these pairings in more detail but, for now,   we return to the pairing $b_{\xi_0}$ and collect some of its properties.

\begin{proposition}
\label{prop:PullbackProperties}
The pairing~$b_{\xi_0} \colon H^2(B;\Z[\pi]) \times H^2(B;\Z[\pi])  \to \Z[\pi],(x,y) \mapsto \langle x , y \cap \xi_0 \rangle$ satisfies the following properties.
\begin{itemize}
\item Consider the map
$$ j^* \colon H^2(B,Y_1;\Z[\pi]) \to H^2(B;\Z[\pi]).$$
For every~$x \in H^2(B,Y_1;\Z[\pi])$ and~$y \in H^2(B;\Z[\pi])$, 
$$ b_{\xi_0}(j^*(x),y)
=b_{X_0}^\partial(c_0^*(x),c_0^*(y)) -b_{X_1}^\partial(c_1^*(x),c_1^*(y)).$$
\item
If~$b_{\xi_0}(j^*(-),-) \equiv 0$, then~$b_{\xi_0}$ induces a hermitian form 
\begin{align*}
b(c_0,c_1) \colon H^2(Y_1;\Z[\pi]) \times H^2(Y_1;\Z[\pi]) 
&\xleftarrow{c_1|_{Y_1}^* \times c_1|_{Y_1}^*,\cong} 
 \frac{H^2(B;\Z[\pi])}{\im(j^*)} \times \frac{H^2(B;\Z[\pi])}{\im(j^*)} \\
 & \xrightarrow{b_{\xi_0}} \Z[\pi].
  \end{align*}
\end{itemize}
\end{proposition}
\begin{proof}
The first point follows from a direct calculation involving the definition of~$b_{\xi_0}$,  the naturality of the cap product, the definition of~$\xi_0$, and the naturality of the evaluation map:
 \begin{align*}
  b_{\xi_0}(j^*(x),y)
&=\langle y,j^*(x) \cap \xi_0 \rangle 
=\langle y,x\cap j_*(\xi_0) \rangle 
=\langle y,x \cap ((c_0)_*([X_0])-(c_1)_*([X_1])) \rangle \\
&=\langle c_0^*(y),c_0^*(x) \cap [X_0] \rangle-\langle c_1^*(y),c_1^*(x) \cap [X_1] \rangle \\
&=\langle c_0^*(y),\PD_{X_0}(c_0^*(x)) \rangle-\langle c_1^*(y),\PD_{X_1}(c_1^*(x)) \rangle 
=b_{X_0}^\partial(c_0^*(x),c_0^*(y)) -b_{X_1}^\partial(c_1^*(x),c_1^*(y)).
\end{align*}
We prove the second assertion.
If~$b_{\xi_0}(j^*(-),-) \equiv 0$, then~$b_{\xi_0}$ induces a pairing 
~$$ \frac{H^2(B;\Z[\pi])}{\im(j^*)}   \times H^2(B;\Z[\pi]) \to \Z[\pi].$$
A direct calculation using the fact that~$b_{\xi_0}$ is hermitian shows that $b_{\xi_0}(x+j^*(a),y+j^*(b))=b_{\xi_0}(x,y)$.
It follows that $b_{\xi_0}$ in fact descends to a pairing on~$H^2(B;\Z[\pi])/\im(j^*) \times H^2(B;\Z[\pi])/\im(j^*)$.

Next, we claim that~$c_1^* \colon H^2(B;\Z[\pi]) \to H^2(Y_1;\Z[\pi])$ is surjective.
Consider the diagram
	\[\begin{tikzcd}
		H^2(B;\Z[\pi])\ar[r,"{c_1^*}"]\ar[d,"c_1^*","\cong"']&H^2(Y_1;\Z[\pi])\ar[r,"\delta"]\ar[d,"="]&H^3(B,Y_1;\Z[\pi])\ar[d]\\
		H^2(X_1;\Z[\pi])\ar[r]&H^2(Y_1;\Z[\pi])\ar[r,"\delta"]&H^3(X_1,Y_1;\Z[\pi])
	\end{tikzcd}\]
and note that~$H^3(X_1,Y_1;\Z[\pi])\cong H_1(X_1;\Z[\pi])=0$ by Poincar\'e duality. It follows that~$H^2(X_1;\Z[\pi])\to H^2(Y_1;\Z[\pi])$ is surjective and hence so is~$c_1^* \colon H^2(B;\Z[\pi]) \to H^2(Y_1;\Z[\pi])$, as claimed.

The claim implies that there is is an exact sequence
$$H^2(B,Y_1;\Z[\pi]) \xrightarrow{j^*} H^2(B;\Z[\pi]) \xrightarrow{c_1^*} H^2(Y_1;\Z[\pi]) \to 0$$
and therefore~$c_1^*$ induces an isomorphism~$H^2(B;\Z[\pi])/\im(j^*) \cong H^2(Y_1;\Z[\pi]),$ as required.
This concludes the proof of the proposition.
\end{proof}

Recall from Definitions~\ref{def:PrimaryObstructionIntroduction} and~\ref{def:SecondaryIntro} that we occasionally refer to $c_0^*b_{X_0}^\partial-c_1^*b_{X_1}^\partial$ as the \emph{primary obstruction} and that, when this pairing vanishes, the hermitian form~$b(c_0,c_1)$ leads to the \emph{secondary obstruction}.
The terminology is motivated by the following proposition which shows that the vanishing of these pairings 
leads to our first sufficient condition for $h \colon Y_0 \to Y_1$ to extend to a homotopy equivalence $X_0 \to X_1$.

\begin{proposition}
\label{prop:FundamentalClassIntersectionForm}
Assume 
\begin{itemize}
\item  either that~$\cd(\pi)\leq 3$ and~$\operatorname{Herm}(H_2(B;\Z[\pi])^*)\to \operatorname{Herm}(H^2(B;\Z[\pi]))$ is injective,
\item or~$\pi_1(X_i)$ is finite and~$\Gamma(\pi_2(X)) \otimes_{\Z[\pi]} \Z$ is torsion free.
\end{itemize}
If there exists a~$3$-connected map~$c_i \colon X_i \to B$ for~$i=0,1$ such that~$c_1|_{Y_1} \circ h = c_0|_{Y_0}$ and
\begin{enumerate}
\item the pulled back pairings~$c_0^*b_{X_0}^\partial$ and~$c_1^*b_{X_1}^\partial$ on~$H^2(B;\Z[\pi])$ are equal,
\item the hermitian form~$b(c_0,c_1)$ on~$H^2(Y_1;\Z[\pi])$ is the zero pairing,
\end{enumerate} 
then~$ \xi_0=0 \in H_4(B)$ and thus~$h$ extends to a homotopy equivalence~$f \colon  X_0 \to X_1$ with~$c_1 \circ f \simeq c_0$.
\end{proposition}
\begin{proof}
Since~$c_0^*b_{X_0}^\partial=c_1^*b_{X_1}^\partial$ on~$H^2(B;\Z[\pi])$, Proposition~\ref{prop:PullbackProperties} implies that~$b(c_0,c_1)$ is defined.
If~$b(c_0,c_1)$ vanishes, then Proposition~\ref{prop:PullbackProperties} implies that~$b_{\xi_0}$ vanishes as well.
The assumptions of the present proposition together with Proposition~\ref{prop:HillmanCite} then ensure that~$\xi_0=0 \in H_4(B).$
Proposition~\ref{prop:FundamentalClass} implies that~$(c_0)_*([X_0])=(c_1)_*([X_1]) \in H_4(B,Y_1)$.
Proposition~\ref{prop:BauesBleileWeak} now implies that~$X_0$ and~$X_1$ are homotopy equivalent over~$B$ rel~$h$.
\end{proof}

\section{The action on hermitian forms and the proof of Theorem~\ref{thm:MainTheoremBoundaryIntro}}
\label{sec:DefineAction}

This section, which continues with Notation~\ref{notation:SectionPrimarySecondary},  aims to prove Theorem~\ref{thm:MainTheoremBoundaryIntro}.
Proposition~\ref{prop:FundamentalClassIntersectionForm} shows that~$c_0^*b_{X_0}^\partial-c_1^*b_{X_1}^\partial$ constitutes a first obstruction to extending $h \colon Y_0 \to Y_1$ to a homotopy equivalence,  and suggests the hermitian form~$b(c_0,c_1) \in \Herm(H^2(Y_1;\Z[\pi]))$ constitutes a secondary obstruction. 
This section constructs a group~$\mathcal{G}$ and an 
action~$\mathcal{G} \curvearrowright \Herm(H^2(Y_1;\Z[\pi]))$ so that for a given $c_1$,  the vanishing of~$b(c_0,c_1)$ in~$\Herm(H^2(Y_1;\Z[\pi]))/\mathcal{G}$ does not depend on the choice of
$$c_0 \in \mathcal{S}_0(c_1):=\{  c_0 \colon X_0 \to B \mid c_0 \text{ is }3\text{-connected}, c_1|_{Y_1} \circ h = c_0|_{Y_0} \text{ and } c_0^*b_{X_0}^\partial = c_1^*b_{X_1}^\partial  \} \Big/\simeq.$$
Theorem~\ref{thm:MainTheoremBoundaryIntro} then follows readily from this construction and Proposition~\ref{prop:FundamentalClassIntersectionForm}.

\medbreak

In what follows, given a pair~$(X,Y)$, we write 
\begin{align*}
& \hAut(X)=\{ \varphi \colon X \to X \mid \varphi \text{ is a homotopy equivalence} \} /\simeq  \\
& \hAut_Y(X)=\{ \varphi \colon X \to X \mid \varphi \text{ is a homotopy equivalence rel.  }Y\} /\simeq \text{ rel.  Y}.
\end{align*}

Any homotopy equivalence~$\varphi \colon B \to B$ rel.~$Y_1$ is readily seen to induce the identity on~$\pi_1(B).$
In particular,  any element of~$\hAut_{Y_1}(B)$ induces isomorphisms on the~$\Z[\pi]$-homology of~$B$.
%

\begin{construction}[The group~$\mathcal{G}$]
\label{cons:GroupG}
Consider the set
$$ \mathcal{G}=  \{ \varphi \in \hAut_{Y_1}(B) \mid  (\varphi \circ c_1)^*b_{X_1}^\partial = c_1^*b_{X_1}^\partial \}.$$
We leave it to the reader to verify that~$\mathcal{G}$ is a group under composition.
\end{construction}

The following two lemmas serve as preparatory results for the verification that~$\mathcal{G}$ acts on $\mathcal{S}_0(c_1)$.

 \begin{lemma}
 \label{lem:HomotopyRelBoundaryprep}
Let $(X,Y)$ be a pair of spaces,  and let~$B$ be a $3$-coconnected space.
 If~$f,f’ \colon X \to B$ are~$3$-connected maps with~$c:=f|_Y=f'|_Y$ both an inclusion and a cofibration,
then there is a homotopy equivalence~$\phi \colon B \to B$ that is the identity on~$c(Y)$ and such that~$f' \simeq \phi \circ f$ relative~$Y$.
 \end{lemma}
 \begin{proof}
Let~$\phi’ \colon B \to B$ be a homotopy equivalence such that~$f’ \simeq \phi’ \circ f$ and let~$\mathbf{H} \colon f’ \simeq \phi’\circ f$ be a homotopy.
Restricting to~$Y$ produces a homotopy~$\mathbf{H}|_Y \colon c\simeq \phi’\circ c$. 
Since $c$ is an inclusion, we can instead view this as a homotopy from $\id_{c(Y)}$ to $\phi'|_{c(Y)}$.
Use the homotopy extension property to extend $\mathbf{H}|_{Y}$ to a homotopy~$\mathbf{G}=(G_t)_{t \in [0,1]} \colon B \times I \to B$ between~$G_1=\phi'$ and
a map~$\phi:= G_0$ that satisfies $\phi|_{c(Y)}=\id_{c(Y)}$.
Note also that $\phi$ is a homotopy equivalence since it is homotopic to a homotopy equivalence.

The concatenation $\mathbf{H}'\colon X\times I\to B$ of the homotopies~$\mathbf{H} \colon f’ \simeq \phi’\circ f$ and~$\overline{\mathbf{G}} \circ f:=(G_{1-t} \circ f)_{t \in [0,1]} \colon \phi\circ f \simeq \phi’\circ f$ produces a homotopy~$f’ \simeq \phi\circ f.$
This is not yet the homotopy we require since it is not rel.~$Y$.
Indeed on the boundary it is the map~$Y \times I \to B$ given by the concatenation of~$\mathbf{H}|_{Y}$ and its ``reverse" $\overline{\mathbf{H}}_{Y}$.
This concatenation restricts to the identity on $Y\times \{0,1\}$
and is homotopic relative~$Y\times \{0,1\}$ to the homotopy~$\mathbf{I} \colon Y \times I \to B$ given by $I_t=c$ for every $t \in [0,1]$.
This latter homotopy is denoted $\mathbf{F} \colon Y \times I \times I \to B$.

Finally, we construct the rel. $Y$ homotopy~$f' \simeq \phi \circ f$.
Let $\mathbf{F}'\colon X\times \{0,1\}\times I\to B$ be given by~$\mathbf{F}'(x,0,t)=f'(x)$ and $\mathbf{F}'(x,1,t)=\phi\circ f'(x)$ and consider the map
$$\mathbf{H}' \cup \mathbf{F} \cup \mathbf{F}' \colon (X \times I \times \{0\}) \cup \overbrace{(Y \times I \times I) \cup (X \times \{0,1\} \times  I)}^{=\left((Y \times I) \cup (X \times \{0,1\})\right) \times I=:A \times I} \to B.$$
The homotopy extension property of~$(X\times I,A)$ ensures that it is possible to extend~$\mathbf{F} \cup \mathbf{F}' \colon A \times I \to B$ to a map~$(X \times I) \times I \to B$ that agrees with~$\mathbf{H}'$ on~$X \times I \times \{0\}$.
Restricting this map to~$X \times I  \times \{1\}$ gives the required homotopy~$f’ \simeq \phi \circ f$ rel.~$Y$. 
This homotopy is indeed relative $Y$ since on $Y\times I$ it agrees with $\mathbf{F}_1=\mathbf{I}$ by construction.
 \end{proof}
 
We require the following strengthening of \cref{lem:HomotopyRelBoundaryprep} that weakens the requirement that the map~$c\colon Y \to B$ be an inclusion.
 \begin{lemma}
 		\label{lem:HomotopyRelBoundary}
 		Let $(X,Y)$ be a pair of spaces,  let~$B$ be a $3$-coconnected space, and let $(B,Y')$ be a pair with $Y' \subset B$ a cofibration.
 		If~$f,f’ \colon X \to B$ are~$3$-connected maps with
 		$$c:=f|_Y=f'|_Y\colon Y\to Y'\subseteq B$$
 		 a homotopy equivalence, then there is a homotopy equivalence~$\phi \colon B \to B$ that is the identity on~$Y'$ and such that~$f' \simeq \phi \circ f$ relative~$Y$.
 \end{lemma}
\begin{proof}
	Choose a homotopy inverse~$c^{-1}\colon Y'\to Y$ of~$c$ and a homotopy~$\mathbf{G} \colon Y' \times I \to Y'$ from~$c\circ c^{-1}$ to~$\id_{Y'}$. 
	Consider the inclusion~$(X,Y)\to (M(i \circ c^{-1}),Y')$, where~$M(i \circ c^{-1})$ is the mapping cylinder of~$Y'\xrightarrow{c^{-1}} Y\xrightarrow{i} X$. 
	Taking~$\mathbf{G}$ on~$Y'\times I$, the maps~$f$ and~$f'$ extend to maps
	$$\wh f:=f \cup \mathbf{G},\wh f':=f' \cup \mathbf{G} \colon (M(i \circ c^{-1}),Y')\to (B,Y')$$
	  that restrict to the identity on~$Y'$. 
Here we are using the convention that the mapping cylinder of a map $g \colon C\to D$ is $M(g):=(D \cup (C \times [0,1]))/\sim$, where $(c,0) \sim f(c)$.

	  By Lemma~\ref{lem:HomotopyRelBoundaryprep}, there is a homotopy equivalence~$\phi\colon B\to B$ that is the identity on~$Y'$ and such that~$\wh f'\simeq \phi\circ \wh f$ rel.~$Y'$, say via~$\mathbf{H}=(H_t)_{t \in [0,1]}\colon M(ic^{-1})\times I\to B$. Restricting~$\mathbf{H}$ to~$X$ gives a homotopy~$\mathbf{H}|_X$ from~$f'$ to~$\phi\circ f$. This homotopy is not yet rel.~$Y$. 

The remainder of the proof is devoted to showing that~$\mathbf{H}|_X$ is homotopic rel.~$X\times \{0,1\}$ to a homotopy~$f' \simeq \phi\circ f$ that is rel.~$Y$.
Write~$M(c^{-1})$ for the mapping cylinder of~$c^{-1}\colon Y'\to Y$ and consider the map 
\begin{align*}
\wh c\colon Y\times I&\to M(c^{-1}) \\
(y,t)& \mapsto 
\begin{cases}
(c(y),t) & \quad \text{ if } t>0, \\
(c^{-1} \circ c(y),0) & \quad \text{ if } t=0.
\end{cases}
\end{align*}
We also consider the homotopy
$$\mathbf{G}':=\wh f\circ \wh c\colon Y\times I\to B$$
from $\mathbf{G}'_0=\wh f \circ (c^{-1} \circ c,0)=c\circ (c^{-1}\circ c)$ to~$\mathbf{G}'_1=\wh f \circ (c,1)=c$, 
and the composition
	 \begin{equation}\label{eq:homotopy}
	 	(Y \times I) \times I \xrightarrow{\wh c \times \id_I} M(c^{-1})\times I \xrightarrow{\mathbf{H}|} B.
	 \end{equation}
	We have $\widehat{f}' \circ \widehat{c}=\widehat{f} \circ \widehat{c}$ since $\wh f$ and $\wh f'$ are given by $\mathbf{G}$ on $Y'\times I$, and since $\mathbf{G}$ takes values in $Y'$ and~$\phi|_{Y'}=\id_{Y'}$, we also have $\phi\circ\widehat{f}\circ \widehat{c}=\widehat{f} \circ \widehat{c}$.
	Hence \eqref{eq:homotopy} is a homotopy from $H_0\circ \wh c=\widehat{f}' \circ \widehat{c}=\widehat{f} \circ \widehat{c}=\mathbf{G}'$ to~$H_1\circ \wh c=\phi\circ\widehat{f}\circ \widehat{c}=\widehat{f} \circ \widehat{c}=\mathbf{G}'$.
	Additionally, again by definition of $\widehat{c}$, on~$Y\times \{0\} \times I$, it is~$\mathbf{H}|_Y\circ c^{-1} \circ c$, whereas on~$Y\times \{1\} \times I$, it is $\mathbf{H}|_{Y'}\circ c=c\times\id_I=:\mathbf{cte}_c$ since $\mathbf{H}$ is relative $Y'$.
Hence on the boundary this homotopy restricts to
$$( \mathbf{G}' \sqcup \mathbf{G}') \cup  \mathbf{cte}_c   \cup  (c^{-1} \circ c \circ \mathbf{H}|_Y)  \colon   (Y \times I \times \{0,1\}) \cup (Y \times \{1\} \times I) \cup (Y \times \{0\} \times I) \to B.$$
It then becomes apparent that, reparametrising this homotopy,  the concatenation of~$\overline{\mathbf{G}}'$ and~$\mathbf{G}'$ is homotopic rel.~$Y\times\{0,1\}$ to~$c^{-1}\circ c\circ\mathbf{H}|_Y$.
Since the concatenation of~$\overline{\mathbf{G}}'$ and~$\mathbf{G}'$ is homotopic rel.~$Y\times\{0,1\}$ to the constant homotopy at $G'_0=c\circ c^{-1}\circ c$, 
we conclude that~$c^{-1}\circ c\circ \mathbf{H}|_Y \colon Y \times I\to B$ is homotopic rel.~$Y\times\{0,1\}$ to the constant homotopy~$Y\times I\to B$ at $c\circ c^{-1}\circ c$.

	Using a homotopy from~$c^{-1}\circ c$ to~$\id_Y$, then shows that~$\mathbf{H}|_Y$ is homotopic rel.~$Y\times\{0,1\}$ to the constant homotopy~$Y\times I\to Y$ given at~$c$.

Using homotopy extension as in the last paragraph of the proof of \cref{lem:HomotopyRelBoundaryprep}, it follows that~$\mathbf{H}|_X$ is homotopic rel. $X\times \{0,1\}$ to a homotopy that is relative~$Y$. 
This homotopy ensures that~$f' \simeq \phi \circ f$ relative~$Y$ as needed.
\end{proof}

 \begin{proposition}
 \label{prop:Transitive}
 The group $\mathcal{G}$ acts transitively on $\mathcal{S}_0(c_1)$ by composition.
 \end{proposition}
 \begin{proof}
We verify that for $\varphi \in \mathcal{G}$ and $c_0 \in \mathcal{S}_0(c_1)$, the composition $\varphi \circ c_0$ again lies in~$\mathcal{S}_0(c_1)$.
Since~$\varphi$ is rel. boundary, we have $c_1|_{Y_1} \circ h=c_0|_{Y_0}=\varphi \circ c_0|_{Y_0}.$
Next we must show that $(\varphi \circ c_0)^* b_{X_0}^\partial =c_1^*b_{X_1}^\partial$.
To do so,  apply~$c_0^*(c_1^*)^{-1}$ to~$(\varphi \circ c_1)^*b_{X_1}^\partial = c_1^*b_{X_1}^\partial$ and then use~$c_0^*b_{X_0}^\partial = c_1^*b_{X_1}^\partial$:
$$  (\varphi \circ c_0)^*b_{X_1}^\partial 
=
c_0^*(c_1^*)^{-1} (\varphi \circ c_1)^*b_{X_1}^\partial 
=c_0^*(c_1^*)^{-1} (c_1^*b_{X_1}^\partial)
= c_0^*b_{X_1}^\partial=c_1^*b_{X_1}^\partial.$$
We prove that the action is transitive.
Given~$c_0,c_0' \in \mathcal{S}_0(c_1)$, Lemma~\ref{lem:HomotopyRelBoundary} ensures that there is a homotopy equivalence~$\varphi \colon B \to B$ rel. $Y_1$ with~$\varphi \circ c_0' \simeq c_0$.
We verify that~$\varphi \in \mathcal{G}$ by using the definition of~$\mathcal{S}_0(c_1)$
to get~$(c_0')^*b_{X_0}^\partial=c_1^*b_{X_1}^\partial=c_0^*b_{X_0}^\partial=(\varphi \circ c_0')^* b_{X_0}^\partial$ and then applying~$c_1^* \circ ((c_0')^*)^{-1}$ to obtain 
$c_1^*b_{X_1}^\partial=(\varphi \circ c_1)^*  b_{X_1}^\partial$.
This concludes the proof of the proposition.
 \end{proof}

We now move towards the construction of an action of~$\mathcal{G}$ on~$ \Herm(H^2(Y_1;\Z[\pi]))$.

\begin{construction}[{The map~$\mathcal{G} \to \Herm(H^2(Y_1;\Z[\pi]))$}]
\label{cons:bphi}
The idea is to run the definition of the secondary obstruction~$b(c_0,c_1)$ but with~$(X_1,\varphi \circ c_1)$ in place of~$(X_0,c_0)$.
Here are the details.
Repeating the proof of Proposition~\ref{prop:FundamentalClassUnion} shows that the condition~$\varphi\circ c_1|_{Y_1} = c_1|_{Y_1}$ ensures the existence of a unique class~$\xi  \in H_4(B)$ with
\begin{equation}
\label{eq:ForDefinitionOfG}
j_*(\xi)=(\varphi\circ c_1)_*([X_1])-(c_1)_*([X_1]).
\end{equation}
Consider the pairing~$b_\xi(a,b)=\langle b,a \cap \xi \rangle$
on~$H^2(B;\Z[\pi])$.
Arguing as in Proposition~\ref{prop:PullbackProperties}, the condition~$(\varphi \circ c_1)^*b_{X_1}^\partial=c_1^*b_{X_1}^\partial$ ensures that this pairing induces a pairing on~$H^2(B;\Z[\pi])/\im(j^*)$ that can then be pulled back using~$c_1|_{Y_1}^*$ to a pairing
$$b(\varphi) \colon H^2(Y_1;\Z[\pi]) \times H^2(Y_1;\Z[\pi])\to \Z[\pi].$$
Explicitly,  for classes~$a,b \in H^2(Y_1;\Z[\pi])$, pick~$a',b' \in H^2(B;\Z[\pi])$ with~$c_1^*(a')=a,c_1^*(b')=b$, the hermitian pairing is defined as
$$b(\varphi)(a,b):=\langle b',a'\cap \xi \rangle.$$
\end{construction}

\begin{proposition}
\label{prop:Action}
The group~$\mathcal{G}$ acts on~$\Herm(H^2(Y_1;\Z[\pi]))$ by
$$\varphi \cdot \lambda:=b(\varphi) + \lambda.$$
\end{proposition}
\begin{proof}
This proposition will follow readily from the following claim.
\begin{claim}
The following equality holds for every~$\varphi,\psi \in \mathcal{G}$:
$$b(\psi \circ \varphi)=b(\psi) + b(\varphi).$$
\end{claim}
\begin{proof}
We set~$x=(c_1)_*([X_1])$ for brevity.
As in the definition of~$b(\varphi)$ and~$b(\psi)$ (recall~\eqref{eq:ForDefinitionOfG}),  we pick~$\xi_\varphi,\xi_\psi \in H_4(B)$ so that~$j_*(\xi_\varphi)=\varphi_*(x)-x$ and~$j_*(\xi_\psi)=\psi_*(x)-x$.
It follows that
$$i_*(\psi_*(\xi_\varphi)+\xi_\psi)=\psi_*(\varphi_*(x))-\psi_*(x)+\psi_*(x)-x=\psi_*(\varphi_*(x))-x.$$
The definition of~$b(\psi \circ \varphi)$ now implies that for~$a,b \in H^2(Y_1;\Z[\pi])$ and~$a',b' \in H^2(B;\Z[\pi])$ such that~$c_1^*(a')=a$ and~$c_1^*(b')=b$, we have
\begin{align*}
b(\psi \circ \varphi)(a,b)
&= \langle b',a' \cap( \psi_*(\xi_\varphi)+\xi_\psi) \rangle 
=\langle \psi^*(b'),\psi^*(a') \cap \xi_\varphi\rangle +   \langle b',a' \cap \xi_\psi \rangle \\
&=\langle \psi^*(b'),\psi^*(a') \cap \xi_\varphi\rangle +   b(\psi)(a,b).
\end{align*}
Finally, note that since~$\psi \circ c_1 =c_1$ we have~$(c_1)^*(\psi^*(a'))=a$ and~$(c_1)^*(\psi^*(b'))=b$, and so we can use~$\psi^*(a')$ and~$\psi^*(b')$ to calculate~$b(\varphi)$.
Namely, we have~$b(\varphi)(a,b)=\langle \psi^*(b'),\psi^*(a') \cap \xi_\varphi\rangle$ and the equality~$b(\psi \circ \varphi)=b(\psi) + b(\varphi)$ follows, thus proving the claim.
\end{proof}
For~$\varphi,\psi \in \mathcal{G}$ and~$\lambda \in  \Herm(H^2(Y_1;\Z[\pi]))$, the claim implies that~$(\psi \circ \varphi) \cdot \lambda=b(\psi) +b(\varphi) +\lambda$ and,  clearly, this equals~$\psi \cdot (\varphi \cdot \lambda)$.
Finally, we note that~$\id\in \mathcal{G}$ acts as the identity: in this case,~$i_*(0)=\id(x)-x$, so~$b(\id) = 0$, whence~$b(\id)+\lambda=\lambda.$
\color{black}
\end{proof}

The combination of this proposition with the next leads to the secondary obstruction from Definition~\ref{def:SecondaryIntro}.
 
 \begin{proposition}
 \label{prop:Indepc0}
If $c_0 \in \mathcal{S}_0(c_1)$ and $\varphi \in \mathcal{G}$, then
$$ b(\varphi  \circ c_0,c_1)=b(c_0,c_1)-b(\varphi) .$$
In particular, the hermitian form~$b(c_0,c_1)$,  once considered in~$\Herm(H^2(Y_1;\Z[\pi]))/\mathcal{G}$,  does not depend on the choice of~$c_0 \in \mathcal{S}_0(c_1)$. 
 \end{proposition}
 \begin{proof}
We argue that the second assertion follows from the first.
We consider~$c_0,c_0' \in \mathcal{S}_0(c_1)$ and aim to prove that~$b(c_0,c_1)=b(c_0',c_1) \in  \Herm(H^2(Y_1;\Z[\pi]))/\mathcal{G}.$
Proposition~\ref{prop:Transitive} ensures that there is a homotopy equivalence~$\varphi \in \mathcal{G}$ with~$\varphi \circ c_0 = c_0'$.
Since $\varphi \in \mathcal{G}$ and $c_0'=\varphi \circ c_0$, the first assertion yields~$b(\varphi) + b(c_0',c_1)=b(c_0,c_1)$.
By definition of the action, this is~$\varphi \cdot b(c_0',c_1)=b(c_0,c_1)$, as required.
We now prove the first assertion.

For brevity, write~$x_0:=c_0([X_0])$ as well as~$x_0':=(\varphi \circ c_0)_*([X_0])$ and~$x_1:=(c_1)_*([X_1])$.
Note that~$x_0'=\varphi_*(x_0).$
Let~$\xi_0,\xi_0',\xi_\varphi \in H_4(B)$ be the unique classes that respectively satisfy~$j_*(\xi_0)=x_1-x_0, j_*(\xi_0')=x_1-x_0'$ and~$j_*(\xi_\varphi)=\varphi_*(x_1)-x_1$, where $i \colon H_4(B) \to H_4(B,Y_1)$ denotes the inclusion induced map.
A direct calculation shows that 
\begin{align*}
j_*(\xi_\varphi +\xi_0')
&=\varphi_*(x_1)-x_1+x_1-x_0'
=\varphi_*(x_1)-\varphi_*(x_0)
=\varphi_*(x_1-x_0)
=\varphi_*(j_*(\xi_0)) \\
&=i_*(\varphi_*(\xi_0)).
\end{align*}
Since $i$ is injective, this implies that $\xi_\varphi +\xi_0'=\varphi_*(\xi_0)$.
It follows that for~$a,b \in H^2(Y_1;\Z[\pi])$ and~$a',b' \in H^2(B;\Z[\pi])$ such that~$c_1^*(a')=a$ and~$c_1^*(b')=b$, we have
\begin{align*}
\left(b(\varphi) +b(\varphi \circ c_0,c_1)\right)(a,b)
&=\langle b',a' \cap (\xi_\varphi +\xi_0') \rangle 
=\langle b',a' \cap \varphi_*(\xi_0) \rangle 
=\langle \varphi^*(b'),\varphi^*(a') \cap \xi_0 \rangle \\
&=b(c_0,c_1)(a,b).
\end{align*}
The last equality holds because,  since~$\varphi \circ c_1 = c_1$,  we have~$(c_1)^*(\varphi^*(a'))=a$ and~$(c_1)^*(\varphi^*(b'))=b$ and so we can use~$\varphi^*(a')$ and~$\varphi^*(b')$ to calculate~$b(c_0,c_1)(a,b)$.
This calculation establishes the equality~$b(\varphi) + b(\varphi \circ c_0,c_1)=b(c_0,c_1)$ required by the proposition.
\color{black}
 \end{proof}

 We can now prove Theorem~\ref{thm:MainTheoremBoundaryIntro}.
\begin{customthm}{\ref{thm:MainTheoremBoundaryIntro}}
\label{thm:MainTheoremBoundary}
Let~$(X_0,Y_0)$ and~$(X_1,Y_1)$ be~$4$-dimensional Poincar\'e pairs with fundamental group~$\pi_1(X_i) \cong \pi$,  Postnikov $2$-type $B:=P_2(X_1)$ and~$\pi_1(Y_i) \to \pi_1(X_i)$ surjective for $i=0,1$,  and let~$h \colon Y_0 \to Y_1$ be a degree one homotopy equivalence.
Assume that
\begin{itemize}
\item either~$\cd(\pi)\leq 3$ and the map~$\ev^* \colon H^2(X_i;\Z[\pi])^{**} \to H_2(X_i;\Z[\pi])^*$ is an isomorphism,
\item or~$\pi_1(X_i)$ is finite and~$\Gamma(\pi_2(X_1)) \otimes_{\Z[\pi]} \Z$ is torsion free.
\end{itemize}
The homotopy equivalence $h$ extends to a homotopy equivalence $X_0 \to X_1$ if and only if
there exist~$3$-connected maps~$c_0 \colon X_0 \to B$ and~$c_1 \colon X_1 \to B$ such that~$c_1|_{Y_1} \circ h = c_0|_{Y_0}$ and
\begin{enumerate}
\item the pulled back pairings~$c_0^*b_{X_0}^\partial$ and~$c_1^*b_{X_1}^\partial$ on~$H^2(B;\Z[\pi])$ are equal,
\item the hermitian pairing~$b(c_0,c_1)$ vanishes in~$\Herm(H^2(Y_1;\Z[\pi]))/\mathcal{G}$.
\end{enumerate}
Additionally,  when $b(c_0,c_1)=0 \in \Herm(H^2(Y_1;\Z[\pi]))$, the homotopy equivalence $f \colon X_0 \to X_1$ extending $h$ can be chosen to satisfy $c_1 \circ f \simeq c_0$.
\end{customthm}
\begin{proof}
If $h$ extends to a homotopy equivalence $f \colon X_0 \to X_1$ (which will necessarily have degree one), then we verify that picking an arbitrary $3$-connected map $c_1 \colon X_1 \to B$ and setting $c_0:=c_1 \circ f$ yields the result.
Firstly,  a rapid verification shows that~$c_0^*b_{X_0}^\partial=f^*\circ c_1^*b_{X_0}^\partial=c_1^*b_{X_1}^\partial$.
Secondly,  we see that~$(c_0)_*([X_0])=(c_1)_* \circ f_*([X_0])=(c_1)_*([X_1])$, so Proposition~\ref{prop:FundamentalClassUnion} ensures that $\xi_0=0$ and therefore $b_{\xi_0}=0.$
This implies that $b(c_0,c_1)=0 \in \Herm(H^2(Y_1;\Z[\pi])).$

We prove the converse.
Since~$c_1|_{Y_1} \circ h = c_0|_{Y_0}$ and~$c_0^*b_{X_0}^\partial=c_1^*b_{X_1}^\partial$, the pairing $b(c_0,c_1)$ is defined.
Since $b(c_0,c_1)=0 \in \Herm(H^2(Y_1;\Z[\pi]))/\mathcal{G}$,  there is a homotopy equivalence~$\varphi \colon B \to B$ rel.~$Y_1$ such that~$b(c_0,c_1)-b(\varphi)=0$.
Consider $c_0':=\varphi \circ c_0$.
Proposition~\ref{prop:Indepc0} yields
$$b(c_0',c_1)=b(c_0,c_1)-b(\varphi)=0.$$
Thus,  $c_0'$ satisfies all of the hypotheses from Proposition~\ref{prop:FundamentalClassIntersectionForm} which ensures that the homotopy equivalence $h$ extends to a homotopy equivalence $X_0 \to X_1$ with $c_1 \circ f \simeq c_0'$.
Here,  in the case that~$\cd(\pi) \leq 3$, Proposition~\ref{prop:FundamentalClassIntersectionForm} requires that~$\operatorname{Herm}(\pi_2(B)^*)\to \operatorname{Herm}(H^2(B;\Z[\pi]))$ is injective, but by Proposition~\ref{prop:Herm-iso} below, this condition holds if $\ev^* \colon H^2(X_i;\Z[\pi])^{**} \to H_2(X_i;\Z[\pi])^*$ is an isomorphism.
\end{proof}

\section{Relative $k$-invariants}
\label{sec:Relativek}

At this stage,  we have proved Theorem~\ref{thm:MainTheoremBoundaryIntro} which shows that the vanishing of the primary and secondary obstructions suffices to extend a homotopy equivalence.
The objective of the next four sections is to rework this theorem so that the outcome does not refer to the~$2$-type.
This requires some background on relative $k$-invariants which we now recall.

Given a pair $(X,Y)$ that is homotopy equivalent to a CW pair and a map $\nu \colon X \to B\pi_1(X)$, the relative $k$-invariant is a class $k_{X,Y}^\nu \in H^3(M(\nu|_Y),Y;\pi_2(X))$, where $M(\nu|_Y) \simeq B\pi_1(X)$ denotes the mapping cylinder of $\nu|_Y$.
We note that up to isomorphism,  $k_{X,Y}^\nu$ does not depend on $\nu$ and that its precise definition is not needed in the present article (we refer to~\cite{ConwayKasprowskiKinvariant} for the details) as we only require some of its properties.

\begin{convention}
\label{conv:IdentificationForKInvariantSection}
Fix pairs of spaces~$(X_0,Y_0)$ and $(X_1,Y_1)$ that are homotopy equivalent to~CW complexes.
Additionally, fix a map $h \colon Y_0 \to Y_1$ and an isomorphism $u \colon \pi_1(X_0) \to \pi_1(X_1)$ satisfying~$u \circ \iota_0=\iota_1 \circ h_*$  that we us to identify~$\pi_1(X_0)$ with $\pi:=\pi_1(X_1)$.
Given maps~$\nu_0 \colon X_0 \to B\pi$ and~$\nu_1 \colon X_1 \to B\pi$, we write
$$ (\id,h)^* \colon H^3(M(\nu_1|_{Y_1}),Y_0;\pi_2(X_1)) \to H^3(M(\nu_0|_{Y_0}),Y_0;\pi_2(X_1))$$
for the map induced by $h$ on the cohomology of the relevant mapping cylinders.
We refer to~\cite[Lemma 5.1]{ConwayKasprowskiKinvariant} for further details concerning the definition of this map.
\end{convention}

We begin with the main property of the relative $k$-invariant.

\begin{theorem}
	\label{thm:RealiseAlgebraic3Type}
	Let~$(X_0,Y_0)$ and~$(X_1,Y_1)$ be pairs of spaces that are homotopy equivalent to~CW pairs, and let $c_1\colon X_1\to P_2(X_1)$ be the $2$-type of $X_1$.
For a homotopy equivalence~$h \colon Y_0 \to Y_1$, an isomorphism~$u\colon \pi_1(X_0)\to \pi_1(X_1)$ with $u\circ(\iota_0)_*= (\iota_1)_* \circ h$, and an $u$-invariant isomorphism~$F \colon \pi_2(X_0) \to~\pi_2(X_1)$, the following assertions are equivalent:
	\begin{itemize}
		\item there is a $3$-connected map~$c_0 \colon X_0\to P_2(X_1)$ such that 
		\[(c_0)_*=\id,\quad c_1 \circ h\simeq c_0|_{Y_0},  \quad \text{and} \quad (c_1)_*^{-1}(c_0)_*=F;\]
		\item the relative $k$-invariants satisfy
		$$(\id,h)^*(k_{X_1,Y_1}^{\nu_1})=F_*(k_{X_0,Y_0}^{\nu_0}) \in H^3(M(\nu_0|_{Y_0}),Y_0;\pi_2(X_1))$$
for every $\nu_0 \colon X_0 \to B\pi$ and $\nu_1 \colon X_1 \to B\pi$.
	\end{itemize}
	\end{theorem}
	\begin{proof}
	This is a direct consequence of~\cite[Theorem 1.1]{ConwayKasprowskiKinvariant}.
	In that paper,  $F$ was not assumed to be an isomorphism, so $c_0$ was not claimed to be $3$-connected,.
	\end{proof}

We record two additional properties of relative $k$-invariants.
	
	\begin{lemma}
	\label{lem:kXXtokXY}
		Let $(X,Y)$ be a pair of spaces that is homotopy equivalent to a CW pair, and write~$i \colon Y\to X$ for the inclusion. 
Then,  for every $\nu \colon X \to B\pi_1(X)$, the following equality holds:
$$(\id,i)^*k_{X,X}^\nu=k_{X,Y}^{\nu}\in H^3(M(\nu|_Y),Y;\pi_2(X)).$$
	\end{lemma}	
	\begin{proof}
	The proof can be found in~\cite[Corollary 5.8]{ConwayKasprowskiKinvariant}.
	\end{proof}

\begin{lemma}
	\label{lem:Xi}
	Let $X$ be a space,  let $\nu \colon X \to B\pi_1(X)$ be  a map, set $\pi:=\pi_1(X)$, and given a~$\Z[\pi]$-module $A$,  consider the 
group homomorphism
\begin{align*}
\Phi\colon \Hom_{\Z[\pi]}(H_2(\widetilde{X}),A)&\to H^3(M(\nu),X;A) \\
	\varphi &\mapsto \varphi_*(k_{X,X}^\nu).
	\end{align*}
Then, writing $\delta$ for the connecting homomorphism in the exact sequence of the pair $(M(\nu),X)$, we have
$$\Phi\circ \ev=\delta\colon H^2(X;A)\to H^3(M(\nu),X;A).$$
\end{lemma}
	\begin{proof}
	The proof can be found in~\cite[Lemma 6.3]{ConwayKasprowskiKinvariant}.
	\end{proof}
	
\begin{convention}
\label{conv:kShorterNotation}	
	For better readability,  during the next section (as well as in the introduction and in Section~\ref{sec:PostnikovIntro}), we write
	 $$h^*(k_{X_1,Y_1})=F_*(k_{X_0,Y_0}) \in H^3(B\pi,Y_0;\pi_2(X_1))$$
if the following equality holds for every $\nu_0 \colon X_0 \to B\pi$ and $\nu_1 \colon X_1 \to B\pi$; 
			$$(\id,h)^*(k_{X_1,Y_1}^{\nu_1})=F_*(k_{X_0,Y_0}^{\nu_0}) \in H^3(M(\nu_0|_{Y_0}),Y_0;\pi_2(X_1)).$$
\end{convention}

\section{The primary obstruction and compatible triples}
\label{sec:NoPostnikov}

The goal of this section is to reformulate the vanishing of the primary obstruction in terms of compatible triples, thus proving Theorem~\ref{thm:PrimaryObstructionRecastIntro}.
Section~\ref{sub:RelativehermitianForms} briefly discusses relative hermitian forms, Section~\ref{sub:CompatibleTriple} recalls the definition of compatible triples from the introduction, Section~\ref{sub:PrimaryCompatible} proves Theorem~\ref{thm:PrimaryObstructionRecastIntro},  and Section~\ref{sub:NokInvariant} focuses on situations where the $k$-invariant condition can be done away with.

\subsection{Relative hermitian forms}
\label{sub:RelativehermitianForms}

This short section concerns relative hermitian forms.

\begin{definition}
\label{def:Relativehermitian}
Let $R$ be a ring and let~$P,Q$ be left $R$-modules. 
\begin{itemize}
\item A \emph{relative hermitian form} consists of a pair of $R$-linear maps~$(\widehat{\lambda}^\partial,j)$ with $j\colon P\to Q$ and $\widehat{\lambda}^\partial \colon Q \to \overline{\Hom_R(P,R)}$ such that $\widehat{\lambda}:=\widehat{\lambda}^\partial \circ j$ is a hermitian form.
We refer to $\widehat{\lambda}$ as the \emph{hermitian form associated} to~$(\widehat{\lambda}^\partial,j).$
The pairings corresponding to $\widehat{\lambda}^\partial$ and $\widehat{\lambda}$ are respectively denoted~$\lambda^\partial \colon P \times Q \to R$ and~$\lambda^\partial \colon P \times P \to R$.
\item Two relative hermitian forms~$(\widehat{\lambda}_0^\partial,j_0)$ and~$(\widehat{\lambda}^\partial_1,j_1)$ are \emph{isometric} if there are isomorphisms $F \colon P_0 \to P_1$ and $G \colon Q_0 \to Q_1$ with~$G \circ j_0=j_1 \circ F$ and $\lambda_1^\partial(F(x),G(y))=\lambda_0^\partial(x,y)$ for every~$x \in P_0$ and every $y \in Q_0$, i.e. $F$ and $G$ fit in the following commutative diagram
$$
\xymatrix{
P_0\ar[r]^{j_0}\ar[d]^{F}&Q_0 \ar[r]^{\widehat{\lambda}_0^\partial}\ar[d]^G& P_0^*  \\
P_1\ar[r]^{j_1}&Q_1 \ar[r]^{\widehat{\lambda}_1^\partial}& P_1^*.\ar[u]^{F^*} 
}
$$
In particular,  note that $F \colon \lambda_0 \cong \lambda_1$ is automatically an isometry.
\end{itemize}
\end{definition}

The next example describes the main motivation for this algebraic notion.

\begin{example}
Given a $4$-dimensional Poincar\'e complex $(X,Y)$, the pair
$$\left(\lambda_M^\partial,j_*\colon H_2(X;\Z[\pi_1(X)]) \to H_2(X,Y;\Z[\pi_1(X)])\right)$$
is a relative hermitian form.
Indeed, the associated form
$$ H_2(X;\Z[\pi_1(X)]) \xrightarrow{j_*} H_2(X,Y;\Z[\pi_1(X)]) \xrightarrow{\widehat{\lambda}_X^\partial} H_2(X;\Z[\pi_1(X)])^*$$
is the equivariant intersection form of $(X,Y)$ and is therefore hermitian by Lemma~\ref{lem:cup}.
\end{example}

The following lemma describes a situation in which an isometry of relative hermitian forms is determined by an isometry of the associated hermitian forms.

\begin{lemma}
\label{lem:DeterminedByhermitian}
If~$(\widehat{\lambda}_0^\partial,j_0)$ and~$(\widehat{\lambda}^\partial_1,j_1)$ are two relative hermitian forms such that $\widehat{\lambda}_0^\partial$ and $\widehat{\lambda}_1^\partial$ are isomorphisms, then an isometry $F \colon \lambda_0 \cong \lambda_1$ of the associated hermitian forms determines an isometry 
$$(F,F_!) \colon (\widehat{\lambda}_0^\partial,j_0) \cong (\widehat{\lambda}^\partial_1,j_1).$$ 
\end{lemma}
\begin{proof}
Given an isometry~$F \colon \lambda_0 \cong \lambda_1$,  we obtain the commutative diagram
\begin{equation}
\label{eq:WTSIsometryRelative}
\xymatrix{
P_0\ar[r]^{j_0}\ar[d]^{F}
&Q_0 \ar[r]^{\widehat{\lambda}_0^\partial,\cong}\ar@{-->}[d]^{F_!}
& P_0^*  \\
P_1\ar[r]^{j_1}&Q_1 \ar[r]^{\widehat{\lambda}_1^\partial,\cong}& P_1^*,\ar[u]^{F^*} 
}
\end{equation}
where the central dashed isomorphism is defined by
$$F_!:=\left(\widehat{\lambda}_1^{\partial}\right)^{-1} \circ (F^*)^{-1}\circ \widehat{\lambda}_0^{\partial}\colon Q_0 \to Q_1.$$
The right square of~\eqref{eq:WTSIsometryRelative} clearly commutes.
The left square also commutes because the outer rectangle commutes.
This shows that $(F,F_!)$ is an isometry of relative hermitian forms. 
\end{proof}

\subsection{Compatible triples and pairs.}
\label{sub:CompatibleTriple}

This section concerns compatible triples and pairs.

\begin{definition}
\label{def:CompatiblePair}
Let~$(X_0,Y_0)$ and~$(X_1,Y_1)$ be~$4$-dimensional Poincar\'e pairs with~$\pi_1(X_i) \cong \pi$ and for which~$\iota_j \colon \pi_1(Y_j) \to \pi_1(X_j)$ is surjective for $j=0,1$.
Let~$h \colon Y_0 \to Y_1$ be a homotopy equivalence and assume that there is an isomorphism~$u \colon \pi_1(X_0) \to \pi_1(X_1)$ with $u \circ \iota_0=\iota_1 \circ h_*$.
An isometry~$(F,G)$ between the relative hermitian forms~$(\lambda_{X_0}^\partial,(j_0)_*)$ and~$(\lambda_{X_1}^\partial,(j_1)_*)$ is \emph{compatible} with~$h$ if the following diagram commutes:
$$
\xymatrix@C0.5cm{
0
\ar[r]&
H_2(Y_0;\Z[\pi])
\ar[r]\ar[d]_{h_*}&
H_2(X_0;\Z[\pi])
\ar[r]^{(j_0)_*}\ar[d]_{F}&
H_2(X_0,Y_0;\Z[\pi])
\ar[r]\ar[d]_{G}&
H_1(Y_0;\Z[\pi])
\ar[r]\ar[d]_{h_*}&
0\\
0
\ar[r]&
H_2(Y_1;\Z[\pi])
\ar[r]&
H_2(X_1;\Z[\pi])
\ar[r]^{(j_1)_*}&
H_2(X_1,Y_1;\Z[\pi])
\ar[r]&
H_1(Y_1;\Z[\pi])
\ar[r]&
0.
}
$$
We refer to $(F,G,h)$ as a \emph{compatible triple}.
\end{definition}

\begin{convention}
For the remainder of this section, we fix 
Poincar\'e pairs $(X_0,Y_0), (X_1,Y_1)$ with Postnikov $2$-type $B$ and a homotopy equivalence $h \colon Y_0 \to Y_1$ as in Notation~\ref{not:Intro}.
\end{convention}

The next result (mentioned in Remark~\ref{rem:CompatiblePairTripleIntro} from the introduction) shows that for some fundamental groups (such as the trivial group, $\Z$,  and finite groups),  the isomorphism $G$ is superfluous. 
\begin{customprop}{\ref{prop:CompatiblePair}}
If $H^k(\pi;\Z[\pi])=0$ for $k=2,3$, then the existence of an isometry~$F \colon \lambda_{X_0} \cong \lambda_{X_1}$ with
$$
\xymatrix@C0.5cm{
0
\ar[r]&
H_2(Y_0;\Z[\pi])
\ar[r]\ar[d]_{h_*}&
H_2(X_0;\Z[\pi])
\ar[r]\ar[d]_{F}&
H_2(X_0,Y_0;\Z[\pi])
\ar[r]\ar[d]_{F_!}&
H_1(Y_0;\Z[\pi])
\ar[r]\ar[d]_{h_*}&
0\\
0
\ar[r]&
H_2(Y_1;\Z[\pi])
\ar[r]&
H_2(X_1;\Z[\pi])
\ar[r]&
H_2(X_1,Y_1;\Z[\pi])
\ar[r]&
H_1(Y_1;\Z[\pi])
\ar[r]&
0,
}
$$
where $F_!=\PD_{X_1} \circ \ev_1^{-1} \circ F^* \ev_0 \PD_{X_0}^{-1}$,  ensures the existence of a compatible triple.
\end{customprop}
\begin{proof}
The fact that~$H^k(\pi;\Z[\pi])=0$ for $k=2,3$ implies that~$\ev \colon H^2(X_i;\Z[\pi]) \to H_2(X_i;\Z[\pi])^*$ is an isomorphism for $i=0,1$.
It follows that $\widehat{\lambda}_{X_0}^\partial$ and~$\widehat{\lambda}_{X_1}^\partial$ are isomorphisms.
Lemma~\ref{lem:DeterminedByhermitian} then ensures that $(F,F_!,h)$ is a compatible triple.
\end{proof}

A pair $(F,h)$ as in Proposition~\ref{prop:CompatiblePair} is referred to as a \emph{compatible pair}.
These have appeared previously in~\cite{BoyerUniqueness} (for $\pi=1$) and in~\cite{ConwayPowell} (for $\pi \cong \Z$).

\subsection{The primary obstruction and compatible triples}
\label{sub:PrimaryCompatible}

Next, we prove the main statement of this section which makes it possible to reformulate the vanishing of the primary obstruction in terms of compatible triples (as stated in Theorem~\ref{thm:PrimaryObstructionRecastIntro}).
The combination of Proposition~\ref{prop:Simplify} and Proposition~\ref{prop:3ConnImpliesCompatible} implies Theorem~\ref{thm:PrimaryObstructionRecastIntro}.

\begin{proposition}
\label{prop:Simplify}
Assume that there is an isomorphism~$u \colon \pi_1(X_0) \to \pi_1(X_1)$ with $u \circ \iota_0=\iota_1 \circ h_*$ and that there is a compatible triple
$$
\xymatrix@C0.5cm{
H_2(Y_0;\Z[\pi])
\ar[r]\ar[d]_{h_*}&
H_2(X_0;\Z[\pi])
\ar[r]\ar[d]_{F}&
H_2(X_0,Y_0;\Z[\pi])
\ar[r]\ar[d]_{G}&
H_1(Y_0;\Z[\pi])
\ar[r]\ar[d]_{h_*}&
0\\
H_2(Y_1;\Z[\pi])
\ar[r]&
H_2(X_1;\Z[\pi])
\ar[r]&
H_2(X_1,Y_1;\Z[\pi])
\ar[r]&
H_1(Y_1;\Z[\pi])
\ar[r]&
0.
}
$$
Assume that~$F_*(k_{X_0,Y_0})=h^*(k_{X_1,Y_1})$,  and
\begin{itemize}
\item either that the map~$H^2(\pi;\Z[\pi]) \to H^2(Y_1;\Z[\pi])$ is injective, and that~$H_1(Y_0;\Z[\pi]) \cong L \oplus T$, with~$L$ a free~$\Z[\pi]$-module and~$T^*=0$.
\item or that $\pi$ is finite and that~$H_1(Y_0;\Z[\pi]) \cong L \oplus T \oplus \Z^k$, with~$L$ a free~$\Z[\pi]$-module and~$T^*=0$.
Here, $\Z$ is endowed with the $\Z[\pi]$-module structure induced by the trivial $\pi$-action. 
\end{itemize}
Given a $3$-connected map~$c_1 \colon X_1 \to B$,  there exists a $3$-connected map~$c_0 \colon X_0 \to B$ such that 
$$c_0|_{Y_0} =  c_1|_{Y_1} \circ h, \quad (c_1)_*^{-1}(c_0)_*=F, \quad \text{and} \quad (c_1)_*^{-1}(c_0)_*=G.$$
In particular, $c_0$ and $c_1$ satisfy
$$c_0^*b_{X_0}^\partial =c_1^*b_{X_1}^\partial.$$
%
\end{proposition}
\begin{proof}
By Theorem~\ref{thm:RealiseAlgebraic3Type}, the condition on the $k$-invariants ensures that there is a~$c'\colon X_0 \to B$ with~$c'|_{Y_0} \simeq  c_1|_{Y_1} \circ h$,  such that $c'$ is a $\pi_1$-isomorphism
and with $(c_1)_*^{-1}c'_*=F$.
Apply the homotopy extension property to~$c'|_{Y_0} \simeq c_1|_{Y_1} \circ h$ so that $c' \simeq c \colon X_0 \to B$ with~$c|_{Y_0}=c_1|_{Y_1} \circ h$.
In particular, $c$ induces a map of pairs~$c \colon (X_0,Y_0) \to (B,Y_1)$.
\begin{claim}
There exists a map 
$$\psi \colon H_2(X_0,Y_0;\Z[\pi]) \to H_2(X_1;\Z[\pi])$$
that makes the following diagram commute:
\begin{equation}
\label{eq:DiagonalDiagramWithFG}
\xymatrix{
H_2(X_0;\Z[\pi])
\ar[r]^-{(j_0)_*}\ar[d]_{F-(c_1)_*^{-1}c_*}&
H_2(X_0,Y_0;\Z[\pi])
\ar[d]^{G-(c_1)_*^{-1}c_*}\ar[ld]^\psi&
\\
H_2(X_1;\Z[\pi])
\ar[r]^-{(j_1)_*}&
H_2(X_1,Y_1;\Z[\pi]).
&
}
\end{equation}
\end{claim}
\begin{proof}
Since we already have $F=(c_1)_*^{-1}c_*$, we deduce that the following diagram commutes:
\begin{equation}
\label{eq:DiagramWithZeros}
\xymatrix@C0.6cm{
H_2(Y_0;\Z[\pi])
\ar[r]^{(i_0)_*} \ar[d]^0&
H_2(X_0;\Z[\pi])
\ar[r]^-{(j_0)_*}\ar[d]^{0}&
H_2(X_0,Y_0;\Z[\pi])
\ar[d]^{G-(c_1)_*^{-1}c_*}\ar[r]^-{\partial_0}&
H_1(Y_0;\Z[\pi])
\ar[r]\ar[d]^0&
0
\\
H_2(Y_1;\Z[\pi])
\ar[r]^-{(i_1)_*}&
H_2(X_1;\Z[\pi])
\ar[r]^-{(j_1)_*}&
H_2(X_1,Y_1;\Z[\pi])\ar[r]^-{\partial_1}&
H_1(Y_1;\Z[\pi])
\ar[r]&
0.
}
\end{equation}
In order to define $\psi$, it suffices to define a map
$$  \psi' \colon  H_1(Y_0;\Z[\pi]) \to H_2(X_1;\Z[\pi])$$
with $(j_1)_*\circ \psi'\circ \partial_0=G-(c_1)_*^{-1}c_*$.
Indeed $\psi$ is then obtained as $\psi:=\psi' \circ \partial_0$: a short verification ensures that $\psi$ makes~\eqref{eq:DiagonalDiagramWithFG} commute.

Since the diagram in~\eqref{eq:DiagramWithZeros} commutes and the bottom sequence is exact, the map $G-(c_1)_*^{-1}c_*$ has image in $\ker (\partial_1)$. 
Thus for $\partial_0(x) \in H_1(Y_0;\Z[\pi])$,  with $x\in H_2(X_0,Y_0;\Z[\pi])$ exactness ensures that~$G-(c_1)_*^{-1}c_*(x)=(j_1)_*(z)$ for some $z \in H_2(X_1;\Z[\pi])$.
The element $z$ is only well defined in~$\coker((i_1)_*)$ leading to a map
$$  \psi'' \colon  H_1(Y_0;\Z[\pi]) \to \coker((i_1)_*).$$
that satisfies
$(i_1)_* \circ \psi'' \circ \partial_0(x)=(i_1)_*(z)=(G-(c_1)_*^{-1}c_*)(x).$
We will show that the projection $\proj \colon H_2(X_1;\Z[\pi])\to \coker((i_1)_*)$ induces a surjection
\begin{equation}
\label{eq:WantSurj}
\proj_* \colon \Hom(H_1(Y_0;\Z[\pi]),H_2(X_1;\Z[\pi]))
\to\Hom(H_1(Y_0;\Z[\pi]),\coker((i_1)_*)).
\end{equation}
Any lift $\psi' \colon H_1(Y_0;\Z[\pi]) \to H_2(X_1;\Z[\pi])$ of $\psi''$ will satisfy $(j_1)_*\circ \psi'\circ \partial_0=(j_1)_*\circ \psi''\circ \partial_0=G-(c_1)_*^{-1}c_*$ and thus give the map~$\psi \colon H_2(X_0,Y_0;\Z[\pi]) \to H_2(X_1;\Z[\pi])$ required by the claim.

Towards showing that $\proj_*$ is surjective,  we assert that~$\ev \circ \PD_{X_1}^{-1} \colon H_2(X_1,Y_1;\Z[\pi]) \to H_2(X_1;\Z[\pi])^*$ restricts to an injection $\ker(\partial_1) \to H_2(X_1;\Z[\pi])^*.$
Consider the following diagram:
$$
\xymatrix@R0.5cm@C0.5cm{
0\ar[r]&H^2(\pi;\Z[\pi])\ar[r]\ar@{^{(}->}[dd]&H^2(X_1;\Z[\pi])\ar[r]^-{\ev}\ar[rd]_{\PD_{X_1}}^{\cong}\ar[ldd]^{i_1^*}&H_2(X_1;\Z[\pi])^* \\
&&&H_2(X_1,Y_1;\Z[\pi]) \ar[ld]^{\partial_1} \\
&H^2(Y_1;\Z[\pi])\ar[r]^-{\PD_{Y_1}}_{\cong}&H_1(Y_1;\Z[\pi]).
}
$$
The vertical map is assumed to be injective in our first hypothesis; whereas in the second set of hypotheses,  $\pi$ is finite,  and the injectivity is automatic because in that case~$H^2(\pi;\Z[\pi])=0$.
The triangle commutes by naturality,  and the square commutes by standard properties of the Poincar\'e duality isomorphism.
A short diagram chase using the exactness of the top row shows that if~$\ev \circ \PD_{X_1}^{-1}(x)=0$ for~$x \in \ker(\partial_1)$, then~$x=0.$
This establishes the assertion.

The assertion produces an injection $\Hom_{\Z[\pi]}(T,\coker((i_1)_*))
\hookrightarrow  \Hom_{\Z[\pi]}(T,H_2(X_1;\Z[\pi])^*)$.
This implies that~$\Hom(T,\coker((i_1)_*))=0$ because
\begin{align*}
\Hom_{\Z[\pi]}(T,H_2(X_1;\Z[\pi])^*)
&\cong \Hom_{\Z[\pi]}(T \otimes_{\Z[\pi]} H_2(X_1;\Z[\pi]),\Z[\pi]) \\
&\cong  \Hom_{\Z[\pi]}(H_2(X_1;\Z[\pi]),T^*)=0.
\end{align*}
We deduce that the map $\proj_*$ from~\eqref{eq:WantSurj} takes the form
$$
\xymatrix{ 
\proj_* \colon {\overbrace{\Hom(H_1(Y_0;\Z[\pi]),H_2(X_1;\Z[\pi]))}^{=\Hom(T,H_2)\oplus\Hom(L,H_2)\oplus \Hom(\Z^k,H_2)}}
\ar[r]&
{\overbrace{\Hom(H_1(Y_0;\Z[\pi]),\coker((i_1)_*)).}^{=\Hom(L,\coker)\oplus \Hom(\Z^k,\coker)}}
}
$$
where the summands involving $\Z$ are understood to occur only under the second set of hypotheses.

To show that this is surjective, we show that~$\Hom(L,H_2(X_1;\Z[\pi])) \to \Hom(L,\coker((i_1)_*))$ and~$\Hom(\Z^k,H_2(X_1;\Z[\pi])) \to \Hom(\Z^k,\coker((i_1)_*))$ are both surjective.

In the first case, this follows because $L$ is free and $\proj \colon H_2(X_1;\Z[\pi]) \to \coker((i_1)_*)$ is surjective.
We thus focus on the case in which the group~$\pi$ is finite.
Since $\pi$ is finite, we first note that~$H_3(X_1,Y_1;\Z[\pi]) \cong H^1(X_1;\Z[\pi])=0.$
Then,  by dualising the resulting short exact sequence~$0 \to H_2(Y_1;\Z[\pi]) \to H_2(X_1;\Z[\pi]) \to \coker((i_1)_*)\to 0$, 
we obtain
$$\ldots \to \Hom(\Z^k,H_2(X_1;\Z[\pi])) \xrightarrow{\proj_*} \Hom(\Z^k,\coker((i_1)_*)) \to \overbrace{\Ext^1_{\Z[\pi]}(\Z^k,H_{2}(Y_1;\Z[\pi]))}^{=H^1(\pi;H_{2}(Y_1;\Z[\pi]))} \to \ldots$$
Thus it suffices to prove that the rightmost term vanishes.
Since the group~$\pi$ is finite,~$H^1(\pi)=0$ and~$H^k(\pi;\Z[\pi])=\Ext^k(\pi,\Z[\pi])=0$ for $k>0$ 
whence the desired outcome:
$$H^1(\pi;H_{2}(Y_1;\Z[\pi]))
\cong H^1(\pi;H^{1}(Y_1;\Z[\pi]))
\cong H^1(\pi;H_1(Y_1;\Z[\pi])^*)
= H^1(\pi;L \oplus \Z^k)
=0.
$$
Thus the map $\proj_*$ from~\eqref{eq:WantSurj} is surjective.
As explained above this gives the existence of the map~$\psi$ required by the claim.
\end{proof}
For $i=0,1$, note that since~$H_i(X_0,Y_0;\Z[\pi])=0$, we have~$\operatorname{Ext}^{3-i}_{\Z[\pi]}(H_i(X_0,Y_0;\Z[\pi]),\pi_2(B))=0$.
The UCSS then implies that the following evaluation map is surjective:
$$H^2(X_0,Y_0;\pi_2(B)) \to \Hom(H_2(X_0,Y_0;\Z[\pi]),\pi_2(B)) \ni (c_1)_* \circ \psi$$
Here,  $\psi$ is the map from the claim.
Choose $[d] \in H^2(X_0,Y_0;\pi_2(B))$ with $\ev([d])=(c_1)_* \circ \psi$.
Proposition~\ref{prop:ObstructionTheory} applied with $(X,Y)=(X_0,Y_0)$ and $(B,B')=(B,Y_1)$ and $f_0=c$ shows that there exists a map~$c_0 \colon X_0 \to B$ with $c_0|_{Y_0}=c|_{Y_0}$ and $(c_0)_*=c_* \colon \pi_1(X_0) \to \pi_1(B)$ that makes the following diagram commute:
$$
\xymatrix{
H_2(X_0;\Z[\pi])
\ar[r]^-{(j_0)_*}\ar[d]_{(c_0)_*-c_*}&
H_2(X_0,Y_0;\Z[\pi])
\ar[d]^{(c_0)_*-c_*}\ar[ld]_-{(c_1)_* \circ \psi}&
\\
H_2(B;\Z[\pi])
\ar[r]^-{j_B}&
H_2(B,Y_1;\Z[\pi]).
&
}
$$
Applying $(c_1)_*^{-1}$ to this diagram gives the following commutative diagram:
\begin{equation}
\label{eq:DiagonalDiagramWithcc0}
\xymatrix{
H_2(X_0;\Z[\pi])
\ar[r]^-{(j_0)_*}\ar[d]_{(c_1)_*^{-1}((c_0)_*-c_*)}&
H_2(X_0,Y_0;\Z[\pi])
\ar[d]^{(c_1)_*^{-1}((c_0)_*-c_*)}\ar[ld]_-\psi&
\\
H_2(X_1;\Z[\pi])
\ar[r]^-{(j_1)_*}&
H_2(X_1,Y_1;\Z[\pi]).
&
}
\end{equation}
Taking the difference of the diagrams in~\eqref{eq:DiagonalDiagramWithFG} and~\eqref{eq:DiagonalDiagramWithcc0} gives 
$$
\xymatrix{
H_2(X_0;\Z[\pi])
\ar[r]^-{(j_0)_*}\ar[d]_{F-(c_1)_*^{-1}(c_0)_*}&
H_2(X_0,Y_0;\Z[\pi])
\ar[d]^{G-(c_1)_*^{-1}(c_0)_*}\ar[ld]_-0&
\\
H_2(X_1;\Z[\pi])
\ar[r]^-{(j_1)_*}&
H_2(X_1,Y_1;\Z[\pi]).
&
}
$$
This shows that~$(c_1)_*^{-1}(c_0)_*=F$ and~$(c_1)_*^{-1}(c_0)_*=G$.
In particular since $c_1$ is~$3$-connected, this equality ensures that~$c_0$ is a~$\pi_2$-isomorphism.
Since~$\pi_3(B)=0$ and Proposition~\ref{prop:ObstructionTheory} ensures that~$c_0$ is a $\pi_1$-isomorphism, we deduce that~$c_0$ is $3$-connected. 

Since $(F,G)$ is an isometry of relative forms, the conditions $(c_1)_*^{-1}(c_0)_*=F$ and~$(c_1)_*^{-1}(c_0)_*=G$ are seen to be equivalent to $\lambda_{X_0}^\partial((c_0)_*^{-1}(x),(c_0)_*^{-1}(y))=\lambda_{X_1}^\partial((c_1)_*^{-1}(x),(c_1)_*^{-1}(y))$. Proposition~\ref{prop:RelativeFormsCohomologyAreTheSame} therefore ensures that~$c_0^*b_{X_0}^\partial =c_1^*b_{X_1}^\partial.$
\end{proof}

We now prove the converse of Proposition~\ref{prop:Simplify}.

\begin{proposition}
\label{prop:3ConnImpliesCompatible}
Let $c_0 \colon X_0 \to B$ and $c_1 \colon X_1 \to B$ be $3$-connected maps.
If~$c_0|_{Y_0} = c_1|_{Y_1} \circ~h$ and~$c_0^*b_{X_0}^\partial=c_1^*b_{X_1}^\partial$, then~$u:=(c_1)_*^{-1} \circ (c_0)_* \colon \pi_1(X_0) \to \pi_1(X_1)$ satisfies~$u \circ \iota_0=\iota_1 \circ h_*$ and the triple~$((c_1)_*^{-1}(c_0)_*,(c_1)_*^{-1}(c_0)_*,h)$ is compatible.
\end{proposition}
\begin{proof}
The assertion concerning $u$ is immediate and we therefore focus on the second claim.
Since~$c_0|_{Y_0} = h$ and $c_1|_{Y_1}$ is the inclusion,  the following diagram commutes:
$$
\xymatrix{
0 \ar[r]& H_2(Y_0;\Z[\pi]) \ar[d]^{h_*}\ar[r]& H_2(X_0;\Z[\pi]) \ar[d]^{(c_1)_*^{-1}(c_0)_*}\ar[r]& H_2(X_0,Y_0;\Z[\pi])\ar[r]\ar[d]^{(c_1)_*^{-1}(c_0)_*} & H_1(X_0;\Z[\pi]) \ar[d]^{h_*}\ar[r] & 0 \\
0 \ar[r]& H_2(Y_1;\Z[\pi]) \ar[r]& H_2(X_1;\Z[\pi]) \ar[r]& H_2(X_1,Y_1;\Z[\pi]) \ar[r]& H_1(X_1;\Z[\pi]) \ar[r]& 0.
}
$$
Since~$c_0^*b_{X_0}^\partial=c_1^*b_{X_1}^\partial$, Proposition~\ref{prop:RelativeFormsCohomologyAreTheSame} implies that~$(c_0)_*^{-1}\lambda_{X_0}^\partial=(c_1)_*^{-1}\lambda_{X_1}^\partial$, i.e. that~$(c_1)_*^{-1}(c_0)_*$ gives rise to an isomorphism of relative hermitian forms.
\end{proof}

\subsection{Working without the $k$-invariant}
\label{sub:NokInvariant}

We show that if either~$\cd(\pi)\leq 2$ or if~$H_4(\pi)=0$
and~$H_1(Y_i;\Z[\pi])^*=0$, then the~$k$-invariant condition in Theorem~\ref{thm:PrimaryObstructionRecastIntro} can be done away with.
These results were stated in Proposition~\ref{prop:NokinvariantIntro}.

\medbreak

We begin with the case where~$\cd(\pi)\leq 2$.

\begin{proposition}
\label{prop:NokinvariantIntrocd2}
Let~$(X_0,Y_0)$ and~$(X_1,Y_1)$ be~$4$-dimensional Poincar\'e pairs with fundamental group~$\pi_1(X_i) \cong \pi$ and~$\iota_j \colon \pi_1(Y_i) \to \pi_1(X_i)$ surjective  for $i=0,1$.
Fix a degree one homotopy equivalence $h \colon Y_0 \to Y_1$ and assume there is a group isomorphism $u \colon \pi_1(X_0) \to \pi_1(X_1)$ that satisfies~$u \circ \iota_0=\iota_1 \circ h_*$.
\begin{enumerate}
	\item If $\cd(\pi)\leq 2$ and if there is a compatible triple~$(F,G,h)$, then 
$$F_*(k_{X_0,Y_0})=h^*(k_{X_1,Y_1}) \in H^3(B\pi,Y_0;\pi_2(X_1)).$$
\item More generally, if $H^3(\pi;H_2(X_1;\Z[\pi]))$ is $d$-torsion for some $d\in\mathbb{N}$ and if there is a compatible triple~$(F,G,h)$, then 
$$F_*(dk_{X_0,Y_0})=h^*(dk_{X_1,Y_1}) \in H^3(B\pi,Y_0;\pi_2(X_1)).$$
\end{enumerate}
\end{proposition}
\begin{proof}
If $\cd(\pi)\leq 2$, then $H^3(\pi;H_2(X_1;\Z[\pi]))=0$ and the first item indeed follows from the second item for $d=1$.
	
According to Convention~\ref{conv:kShorterNotation}, given~$\nu_0 \colon X_0 \to B\pi$ and~$\nu_1 \colon X_1 \to B\pi$, the proposition requires we prove~$F_*(dk_{X_0,Y_0}^{\nu_0})=(\id,h)^*(dk_{X_1,Y_1}^{\nu_1})$.
The idea of the proof is to run a diagram chase on the diagram in~\eqref{eq:DiagramForNok} below,  but doing so requires some preliminary observations.

We write~$G^! \colon H^2(X_0;\Z[\pi]) \to H^2(X_1;\Z[\pi])$ for the map~$\PD_{X_1}^{-1} \circ G \circ \PD_{X_0}$.
By definition of a compatible triple and the commutativity of  the following diagram we have~$F^* \circ \ev \circ G^!=\ev$ and~$i_0^*=h^* \circ i_1^* \circ G^!$:
$$
\xymatrix{
H^2(Y_0;\Z[\pi])\ar[rd]^{\PD_{Y_0}}_\cong&&H^2(X_0;\Z[\pi])\ar[ll]_-{i_0^*}\ar[r]^{\ev}\ar[d]^{\PD_{X_0}}_\cong\ar@/^4pc/[ddd]^{G^!}&H_2(X_0;\Z[\pi])^* \\
&H_1(Y_0;\Z[\pi])\ar[d]^{h_*}&H_2(X_0,Y_0;\Z[\pi])\ar[d]^{G}_\cong \ar[l]_{\partial_0}& \\
&H_1(Y_1;\Z[\pi])&H_2(X_1,Y_1;\Z[\pi])\ar[l]_{\partial_1}& \\
H^2(Y_1;\Z[\pi])\ar[ru]^{\PD_{Y_1}}_\cong \ar[uuu]_{h^*}&&H^2(X_1;\Z[\pi])\ar[ll]_-{i_1^*}\ar[r]^{\ev}\ar[u]_{\PD_{X_1}}^\cong&H_2(X_1;\Z[\pi])^*\ar[uuu]_{F^*}^\cong.
}
$$
Let $\Phi_i\colon \Hom_{\Z[\pi]}(H_2(X_i;\Z[\pi]),H_2(X_1;\Z[\pi]))\to H^3(M(\nu_i),X_i;H_2(X_1;\Z[\pi]))$ be the homomorphism from \cref{lem:Xi} and consider the following diagram in which $\Z[\pi]$-coefficients are understood:
\begin{equation}
\label{eq:DiagramForNok}
\begin{tikzcd}[column sep=normal]   \Hom(H_2(X_0),H_2(X_0))\ar[d,"\Phi_0"]\ar[r,"F_*{,}\cong"]&\Hom(H_2(X_0),H_2(X_1))\ar[d,"\Phi_0"]&\Hom(H_2(X_1),H_2(X_1))\ar[d,"\Phi_1"]\ar[l,"F^*{,}\cong"']\\       
H^3(M(\nu_0),X_0;H_2(X_0))\ar[d,"{(\id,i_0)^*}"]\ar[r,"F_*{,}\cong"]&H^3(M(\nu_0),X_0;H_2(X_1))\ar[d,"{(\id,i_0)^*}"]&H^3(M(\nu_1),X_1;H_2(X_1))\ar[d,"{(\id,i_1)^*}"]\\     
   H^3(M(\nu_0|_{Y_0}),Y_0;H_2(X_0))\ar[r,"F_*{,}\cong"]&H^3(M(\nu_0|_{Y_0}),Y_0,;H_2(X_1))&H^3(M(\nu_1|_{Y_1}),Y_1,;H_2(X_1)).\ar[l,"{(\id,h)^*{,}\cong}"']
    \end{tikzcd}
    \end{equation}
Here and in the remainder of this proof, we write $(\id,i_0)^*,(\id,i_1)^*$ and $(\id,h^*)$ to emphasise the presence of the mapping cylinders; see e.g.~\cite[Lemma 5.1]{ConwayKasprowskiKinvariant}.
The left-hand squares in the above diagram commute: each horizontal map is obtained by applying $F$ to the coefficients.
\begin{claim} 
The right hand square in the diagram~\eqref{eq:DiagramForNok} commutes for elements divisible by $d$:
$$d(\id,i_0)^*\circ \Phi_0\circ F^*=d(\id,h)^*\circ (\id,i_1)^*\circ \Phi_1.$$
\end{claim}
\begin{proof}
    Since \[\Ext^3_{\Z[\pi]}(H_0(X_1;\Z[\pi]),H_2(X_1;\Z[\pi]))=\Ext^3_{\Z[\pi]}(\Z,H_2(X_1;\Z[\pi]))=H^3(\pi;H_2(X_1;\Z[\pi]))\]
was assumed to be $d$-torsion,
    every element of divisibility $d$ is contained in the image of \[\ev\colon H^2(X_1;H_2(X_1;\Z[\pi]))\to \Hom_{\Z[\pi]}(H_2(X_1;\Z[\pi]),H_2(X_1;\Z[\pi])).\] 
    Hence it suffices to show that 
    \[(\id,i_0)^*\circ \Phi_0\circ F^*\circ \ev= (\id,h)^*\circ (\id,i_1)^*\circ \Phi_1\circ \ev.\]
    Using that $F^* \circ \ev \circ G^!=\ev$, the equality $\Phi_i\circ \ev=\delta_i$ (from~\cref{lem:Xi}), and the naturality of the connecting homomorphism in the long exact sequence of a pair, we have
    \begin{align*}
        (\id,i_0)^*\circ \Phi_0\circ F^*\circ \ev&=(\id,i_0)^*\circ \Phi_0\circ \ev\circ (G^!)^{-1}\\
        &=(\id,i_0)^*\circ \delta_0\circ (G^!)^{-1}\\
        &=\delta_0\circ i_0^*\circ (G^!)^{-1}.
    \end{align*}
Rewrite~$i_0^*=h^* \circ i_1^* \circ G^!$ as~$i_0^*\circ (G^!)^{-1}=h^*\circ i_1^*\colon H^2(X_1;H_2(X_1;\Z[\pi]))\to H^2(Y_0;H_2(X_1;\Z[\pi]))$. 
    Using this,  together with the equality~$\Phi_i\circ \ev=\delta_i$ (again \cref{lem:Xi}), and the naturality of the connecting homomorphism in the long exact sequence of a pair, we have
    \begin{align*}
        \delta_0\circ i_0^*\circ (G^!)^{-1}&=\delta_0\circ h^*\circ i_1^*\\
        &=(\id,h)^*\circ \delta_1\circ i_1^*\\
        &=(\id,h)^*\circ (\id,i_1)^*\circ \delta_1\\
        &=(\id,h)^*\circ (\id,i_1)^*\circ \Phi_1\circ \ev.
    \end{align*}
    This completes the proof of the claim.
\end{proof}
We can now conclude the proof of the proposition.
Consider~$d\id_{H_2(X_0;\Z[\pi])}$ as an element of the top left module of the diagram in~\eqref{eq:DiagramForNok}: it maps to~$d\id_{H_2(X_1;\Z[\pi])}$ on the top right and then down to~$dk_{X_1,Y_1}^{\nu_1}$ by the combination of Lemma~\ref{lem:Xi} and Lemma~\ref{lem:kXXtokXY}.
Since Lemmas~\ref{lem:Xi} and~\ref{lem:kXXtokXY} again ensure that~$d\id_{H_2(X_0;\Z[\pi])}$ maps to~$dk_{X_0,Y_0}^{\nu_0}$ in the bottom left, the commutativity of the entire diagram on elements of divisibility $d$ implies that
$F_*(dk_{X_0,Y_0}^{\nu_0})=(\id,h)^*(dk_{X_1,Y_1}^{\nu_1})$,
as required.
\end{proof}

Next, we turn to the case where~$H_4(\pi)=0$.

\begin{proposition}
\label{prop:NokinvariantIntrocd3}
Let~$(X_0,Y_0)$ and~$(X_1,Y_1)$ be~$4$-dimensional Poincar\'e pairs with fundamental group~$\pi_1(X_i) \cong \pi$ and~$\iota_j \colon \pi_1(Y_i) \to \pi_1(X_i)$ surjective  for $i=0,1$.
Fix a homotopy equivalence~$h \colon Y_0 \to Y_1$ and assume there is a group isomorphism~$u \colon \pi_1(X_0) \to \pi_1(X_1)$ that satisfies~$u \circ \iota_0=\iota_1 \circ h_*$.
Finally, fix a $\Z[\pi]$-isomorphism~$F \colon \pi_2(X_0) \cong \pi_2(X_1)$.

If 
$H_4(\pi)=0$ and~$H_1(Y_0;\Z[\pi])^*=0$, then~$F_*(k_{X_0,Y_0})=h^*(k_{X_1,Y_1}).$
\end{proposition}
\begin{proof}
According to Convention~\ref{conv:kShorterNotation}, given~$\nu_0 \colon X_0 \to B\pi$ and~$\nu_1 \colon X_1 \to B\pi$, the proposition requires we prove~$F_*(k_{X_0,Y_0}^{\nu_0})=(\id,h)^*(k_{X_1,Y_1}^{\nu_1})$.
Using $c_1 \colon X_1 \to B \supset Y_1$ for the~$2$-type,  we do so by constructing a~$c \colon  X_0\to B$ with $c=c_1 \circ h$ and~$F_*(k_{X_0,Y_0}^{\nu_0})=(c_1)_*^{-1}c_*(k_{X_0,Y_0}^{\nu_0})=(\id,h)^*(k_{X_1,Y_1}^{\nu_1})$.

We begin by constructing the map $c$.
The dimension shifting argument used during the proof of Proposition~\ref{prop:H4(B)} yields $H_n(\pi;\pi_2(B)) \cong H_{n+3}(\pi)$.
Poincar\'e duality and the UCSS then give
$$H^3(X_0,Y_0;\pi_2(B))\cong H_1(X_0;\pi_2(B))=\operatorname{Tor}_1^{\Z[\pi]}(\Z,\pi_2(B))
=H_1(\pi;\pi_2(B)) \cong H_4(\pi)=0.$$
Obstruction theory then ensures that there is a~$c'\colon X_0 \to B$ with~$c' \simeq c_1 \circ h$ and such that~$c'$ is a~$\pi_1$-isomorphism.
Apply the homotopy extension property to~$c' \simeq  c_1 \circ h$ in order to obtain a map~$c \simeq c' \colon X_0 \to B$ with~$c = c_1 \circ h$.
Applying Theorem~\ref{thm:RealiseAlgebraic3Type} with $F=(c_1)_*^{-1}c_*$ gives
$$(c_1)_*^{-1}c_*(k_{X_0,Y_0}^{\nu_0})=(\id,h)^*(k_{X_1,Y_1}^{\nu_1}).$$
It remains to prove that $F_*(k_{X_0,Y_0}^{\nu_0})=(c_1)_*^{-1}c_*(k_{X_0,Y_0})$.
For $i=0,1$, Poincar\'e duality and the UCSS give isomorphisms
$$H_3(X_i,Y_i;\Z[\pi])\cong H^1(X_i;\Z[\pi])\cong H^1(\pi;\Z[\pi]).$$
The same reasoning also yields~$H_2(Y_i;\Z[\pi]) \cong H^1(Y_i;\Z[\pi]) \cong H^1(\pi;\Z[\pi])$, where the second isomorphism uses~$H_1(Y_i;\Z[\pi])^*=0$.
The naturality of the UCSS implies that the connecting homomorphism~$H_3(X_i,Y_i;\Z[\pi]) \to H_2(Y_i;\Z[\pi])$ is an isomorphism.
We therefore obtain the short exact sequence
$$0\to H_2(X_i;\Z[\pi])\xrightarrow{(j_i)_*} H_2(X_i,Y_i;\Z[\pi])\to H_1(Y_i;\Z[\pi])\to 0.$$
Consider the long exact sequences of coefficients in cohomology for~$M(\nu_0|_{Y_0})$.
 Since both~$F$ and~$(c_1)_*^{-1}c_*$ induce~$h_*$ on~$H_1(Y_0;\Z[\pi])$, we obtain the following commutative diagram:
$$
\xymatrix{
H^2(M(\nu_0|_{Y_0}),Y_0;H_1(Y_0;\Z[\pi])) \ar[r]\ar[d]^-{0}
&H^3(M(\nu_0|_{Y_0}),Y_0;H_2(X_0;\Z[\pi])) \ar[d]^{F_*-(c_1)_*^{-1}c_*} \ar[r]^-{(j_0)_*}
& \ldots\\
H^2(M(\nu_0|_{Y_0}),Y_0;H_1(Y_1;\Z[\pi])) \ar[r]
&H^3(M(\nu_0|_{Y_0}),Y_0;H_2(X_1;\Z[\pi])).
}
$$
Since~$(j_0)_*(k_{X_0,Y_0}^{\nu_0})=0$ (see~\cite[Lemma 6.1]{ConwayKasprowskiKinvariant}),~$k_{X_0,Y_0}^{\nu_0}$ pulls back to~$H^2(M(\nu_0|_{Y_0}),Y_0;H_1(Y_0;\Z[\pi]))$. 
A diagram chase now yields~$F_*(k_{X_0,Y_0}^{\nu_0})=(c_1)_*^{-1}c_*(k_{X_0,Y_0}^{\nu_0})=(\id,h)^*(k_{X_1,Y_1}^{\nu_1}),$ as needed.
\end{proof}

\section{Twisted Künneth calculations in~$Y \times S^1$}
\label{sec:KunnethYS1}

We begin our analysis of the secondary obstruction.
Since the subsequent sections involve heavy use of twisted cup and cross products,  Section~\ref{sub:Twisted} introduces some background on the topic, whereas Section~\ref{sub:YxS1} then focuses on calculations in~$Y \times S^1$.

\subsection{Twisted homology and cohomology}
\label{sub:Twisted}

We recall the definition of twisted (co)homology as well as twisted cup, cap and cross products, mostly following the notation from~\cite{ConwayNagel}.
Further details can be found in the appendix.
Let $S$ be a ring with involution.
In what follows, given a left~$S$-module~$M$, we write~$\overline{V}$ for the right~$S$-module whose underlying group agrees with that of~$V$ but with the~$S$-module structure~$v\cdot s=\overline{s} v$ for~$v \in V$ and~$s \in S$.

\begin{definition}
\label{def:TwistedHomology}
Given a pair~$(X,Y)$, a ring~$R$ and a~$(R,\Z[\pi_1(X)])$-bimodule~$M$, \emph{the twisted chain complex} and \emph{twisted cochain complex} of~$(X,Y)$ with coefficients in~$M$ are respectively defined as the following chain complexes of left~$R$-modules: 
\begin{align*}
C_*(X,Y;M)&:=M \otimes_{\Z[\pi_1(X)]} C_*(\widetilde{X},\widetilde{Y}), \\
C^*(X,Y;M)&:=\Hom_{\operatorname{right-}\Z[\pi_1(X)]}(\overline{C_*(\widetilde{X},\widetilde{Y}}),M).
\end{align*}
The \emph{twisted (co)homology of $(X,Y)$} with coefficients in~$M$ is the homology of these chain complexes:
\begin{align*}
H_*(X,Y;M)&:=H_*(C_*(X,Y;M)), \\
H^*(X,Y;M)&:=H_*(C^*(X,Y;M)). 
\end{align*}
These are left~$R$-modules 
with actions given by~$r \cdot (m \otimes \sigma)=rm \otimes \sigma$ and~$(r \cdot f)(\sigma)=rf(\sigma)$ for~$f \in C^*(X,Y;M),\sigma \in C_*(X,Y;M)$ and~$m \in M$.
Note also that~$f(\gamma \sigma)=f(\sigma \cdot \gamma^{-1})=f(\sigma)\gamma^{-1}$ for every~$\gamma \in \pi_1(X)$.
\end{definition}

\begin{construction}
\label{cons:Box}
Consider a 
ring~$R$ with involution and two~$(R, \Z[\pi])$--modules~$M$ and~$N$. 
We equip the abelian group~$M \otimes_\Z N$ with the structure of an~$(R, \Z[\pi] \times R)$--bimodule, denoted~$M\boxtimes N$, as follows.
The left and right actions of~$R$ are~via 
\[ r \cdot (m \otimes n) \cdot s = ( r \cdot m ) \otimes ( \overline{s} \cdot n )\]
for~$r ,s \in R$ and~$m \in M$ and~$n \in N$.
The right~$\pi$--action is the diagonal action
\[ (m\otimes n) \cdot g = (m \cdot g) \otimes (n \cdot g) \]
for~$g \in \pi$ and~$m \in M$ and~$n \in N$.
As in~\cite{ConwayNagel}, it is to stress the diagonal action that we use the symbol~$\boxtimes$ instead of the more common~$\otimes$.
Homology and cohomology with coefficients in $M \btimes N$ inherits the structure of an $(R,R)$-bimodule.
\end{construction}

This construction is relevant to twisted cup, cap and cross products.
These will be reviewed in detail in the appendix, but since they will be used heavily in the next section, we record some relevant facts.
For a space $X$ and~$(R,\Z[\pi_1(X)])$-bimodules~$M,N,$ the \emph{twisted cup product} and \emph{twisted cap product} are~$(R,R)$-linear pairings
\begin{align*}
\cup \colon H^i(X;M) \otimes_\Z \overline{H^j(X;N)} \to H^{i+j}(X;M \boxtimes N), \\
\cap \colon H^i(X;M) \otimes_\Z \overline{H_k(X;N)} \to H_{k-i}(X;M \boxtimes N).
\end{align*}
Similarly for spaces~$X_1,X_2$, rings $R_1,R_2,R$ and a~$(R_1,\Z[\pi_1(X_1)] \times R)$-bimodule $M$ as well as a~$(R,\Z[\pi_1(X_2)] \times R_2)$-bimodule~$N$, there are twisted cross products
\begin{align*}
& \times \colon H_k(X_1;M) \otimes_\Z \overline{H_l(X_2;N)} \to H_{k+l}(X_1 \times X_2;M \boxtimes N),  \\
&\times \colon H^k(X_1;M) \otimes_\Z \overline{H^l(X_2;N)} \to H^{k+l}(X_1 \times X_2 ;M \boxtimes N).
\end{align*}
Here we identify~$\Z[\pi_1(X_1) \times \pi_1(X_2)]$ with~$\Z[\pi_1(X_1)] \otimes_\Z \Z[\pi_1(X_2)]$ and endow the~$(R_1,R_2)$-bimodule~$M \otimes_\Z N$ with the right~$\Z[\pi_1(X_1) \times \pi_1(X_2)]$-diagonal action given by~$(m_1,m_2) \cdot (\gamma_1 \otimes \gamma_2)=(m_1\gamma_1,m_2\gamma_2)$ so that it becomes a~$(R_1,\Z[\pi_1(X_1) \times \pi_1(X_2)] \times R_2)$-bimodule.

The definition and properties of these pairings will be discussed in greater detail in the appendix.
Presently,  we note a helpful characterisation of the $0$-th twisted homology with~$M \boxtimes N$-coefficients.

\begin{lemma}\label{lem:Ccoinvariants}
Let~$X$ be a 
space, let~$\varphi \colon \pi_1(X) \to \pi$ be a surjection, let~$\widehat{X} \to X$ be the~$\pi$-cover corresponding to~$\ker(\varphi)$,  and let~$M$ and~$N$ be~$(R,\Z[\pi])$--bimodules. 
The map~$(M\boxtimes N)\otimes_{\Z[\pi]}C_0(\widehat{X})\to M\otimes_{\Z[\pi]}\overline{N}$ given by~$(m\boxtimes n) \otimes [\operatorname{pt}] \mapsto m\otimes n$ induces an isomorphism of~$(R, R)$--modules
\[ \varpi_0 \colon H_0\big(X;M\boxtimes N \big) \to M \otimes_{\Z[\pi]} \overline{N},\]
where~$\overline{N}$ denotes~$N$ turned into a~$(\Z[\pi],R)$--bimodule using the involutions on~$R$ and~$\Z[\pi]$.

Additionally, the map~$m \otimes n \mapsto (m \boxtimes n) \otimes 1$ induces an isomorphism of~$(R, R)$--modules
$$M \otimes_{\Z[\pi]} \overline{N} \to (M \boxtimes N)  \otimes_{\Z[\pi]} \Z .$$
\end{lemma}
\begin{proof}
The first assertion follows because,  as a left~$R$-bimodule,~$H_0\big(X;M \boxtimes N \big)$ is isomorphic (via the augmentation) to the~$R$-bimodule of coinvariants
$ M\boxtimes N \Big/ \Z \langle 1 -g \mid g \in \pi \rangle$~\cite[Twisted~$H_0$-Proposition]{FriedlLectureNotes}.
That this is an isomorphism of~$(R,R)$-bimodules now follows from the definitions of the actions.
The second assertion follows from a direct verification.
\end{proof}

\subsection{Calculations for $Y \times S^1$}
\label{sub:YxS1}

This section introduces a coefficient system on $Y \times S^1$ that will be used frequently in the subsequent sections.
We then record a calculation involving this coefficient system and an application of the K\"unneth theorem.

\begin{construction}[A coefficient system on~$Y \times S^1$]
The composition~$\pi_1(Y \times S^1) \xrightarrow{p_Y} \pi_1(Y) \to  \pi$ of the projection and inclusion induced maps endows the group ring~$\Z[\pi]$ with the structure of a right~$\Z[\pi_1(Y\times S^1)]$-module, and thus of a~$(\Z[\pi],\Z[\pi_1(Y\times S^1)])$-bimodule.

In order to apply the Künneth theorem; we describe an equivalent description of this module.
Endow~$S^1$ with the trivial coefficient system~$\Z$.
Consider the projection~$p_{S^1} \colon Y \times S^1 \to S^1$.
Recall that~$p_Y^*(\Z[\pi]) \boxtimes_\Z p_{S^1}^*(\Z)$ denotes~$p_Y^*(\Z[\pi]) \otimes_\Z p_{S^1}^*(\Z)$ with the diagonal~$\Z[\pi_1(Y\times S^1)]$-action.
We assert that multiplication gives rise to the following isomorphism of~$(\Z[\pi],\Z[\pi_1(Y\times S^1)])$-bimodules:
\begin{align*}
\mu \colon p_Y^*(\Z[\pi]) \boxtimes_\Z p_{S^1}^*(\Z) &\to p_Y^*(\Z[\pi]) \\
a \otimes b &\mapsto ab.
\end{align*}
The fact that this is an isomorphism of left~$\Z[\pi]$-modules is clear.
For~$\gamma \in \Z[\pi_1(Y\times S^1)], a \in \Z[\pi]$ and~$b\in \Z$,  the following equalities establish the right $\Z[\pi_1(Y\times S^1)]$-linearity:
$$\mu((a \otimes b)\gamma)
=\mu(a (p_Y)_*(\gamma) \otimes b (p_{S^1})_*(\gamma))
=\mu(a (p_Y)_*(\gamma) \otimes b)
=a (p_Y)_*(\gamma) b
=\mu(a \otimes b)\gamma .
$$
\color{black}
\end{construction}

\begin{proposition}
\label{prop:KunnethYS1}
Using the notation above,  the cross product induces isomorphissm
\begin{align*}
\Kun \colon  H_2(Y;\Z[\pi]) \oplus H_1(Y;\Z[\pi]) &\to H_2(Y \times S^1;\Z[\pi] \boxtimes \Z) \cong H_2(Y\times S^1;\Z[\pi]), \\
(a, b) &\mapsto  i_Y(a)+(b \times 1_{H_1(S^1)}). \\
\Kun \colon
H^2(Y;\Z[\pi]) \oplus H^1(Y;\Z[\pi]) & \to H^2(Y \times S^1;\Z[\pi] \btimes_\Z \Z) \cong  H^2(Y \times S^1;\Z[\pi]) \\
(a, b) &\mapsto  (a  \times 1_{H^0(S^1)})+ (b \times 1_{H^1(S^1)})=p_Y^*(a)+(p_Y^*(b) \cup p_{S^1}^*(1_{H^1(S^1)})).
\end{align*}
\end{proposition}
\begin{proof}
Apply the K\"unneth theorem from Proposition~\ref{prop:KunnethAppendix} with~$M=R_1=\Z[\pi]$ and~$N=R_2=\Z$.
Note that both $M_1$ and $M_2$ are free abelian groups, so the theorem does indeed apply. 
We identify the~$(\Z[\pi],\Z)$-bimodule~$M_1 \otimes_\Z M_2=\Z[\pi] \otimes_\Z \Z$ with~$\Z[\pi]$.
Since the abelian group~$H_i(S^1)$ is either free or zero,  the~$\operatorname{Tor}_1^\Z$ terms and~$\operatorname{Ext}_\Z^1$ terms (in cohomology) vanish.
For~$i=1,2$, we then use the identifications~$H_i(Y;\Z[\pi]) \otimes_\Z H_{2-i}(S^1) \cong H_i(Y;\Z[\pi])$ and similarly in cohomology.
\end{proof}

\section{The secondary obstruction and even forms}
\label{sec:ProofGoal}

The goal of this section is to analyse the secondary obstruction and to prove Theorem~\ref{thm:VanishEvenIntro}.
Continuing with Notation~\ref{notation:SectionPrimarySecondary}, the main technical result of this section (namely Theorem~\ref{thm:GoalActionpdf}) shows that for any sesquilinear form~$q$ on~$H^2(Y_1;\Z[\pi])$, there exists a homotopy~$\mathbf{H}=(H_t)_{t \in [0,1]}$ from~$c_1|_{Y_1}$ to itself, such that the homotopy~$\Xi(\mathbf{H}) \colon B \to B$ obtained by applying homotopy extension belongs to~$\mathcal{G}$ and satisfies~$\Xi(\mathbf{H}) \cdot b(c_0,c_1)=b(c_0,c_1) + q+q^*$.
Theorem~\ref{thm:VanishEvenIntro} will follow fairly promptly: roughly speaking, when~$b(c_0,c_1)$ is even, it will then be possible to modify~$c_0$,  resulting in a new~$3$-connected map for which the resulting secondary invariant vanishes. 

\medbreak

In order to formulate the aforementioned technical theorem, we first introduce some notation and describe the ``homotopy extension map" $\Xi$ more precisely.

\begin{notation}
Set~$Y:=Y_1$
and assume that~$c_1|_Y \colon Y \to B$ is the inclusion.
We also fix once and for all basepoints~$y_0 \in Y$ and~$b_0:=c_1(y_0)\in B$.
We use~$(y_0,1) \in Y \times S^1$ as a basepoint for~$Y \times S^1$.
In what follows, we consider homotopies~$\mathbf{H}=(H_t \colon Y  \to B)_{t \in [0,1]}$ from~$c_1|_Y$ to~$c_1|_Y$.
Occasionally we will think of~$\mathbf{H}$ as a map~$ \mathbf{H} \colon (Y \times [0,1],Y \times \{0,1\}) \to (B,Y).$
More frequently, since~$\mathbf{H}$ is a map~$\mathbf{H} \colon Y \times [0,1] \to B$ with~$H_0=c_1|_Y=H_1$, we will think of~$\mathbf{H}$ as a map~$\mathbf{H} \colon Y \times S^1 \to B$ with~$\mathbf{H}|_{Y \times \{ 1\}}=c_1|_Y$ the inclusion.
We say that~$\mathbf{H}$ is \emph{trivial on the~$S^1$-factor} if the map~$S^1 \to B,t \mapsto H_t(y_0)$ is nullhomotopic.
Consider 
$$\Homotopy(c_1|_Y):=\lbrace \mathbf{H}:=(H_t \colon Y  \to B)_{t \in [0,1]} \mid H_0=c_1|_Y=H_1 \text{ and~$\mathbf{H}$ is trivial on~$S^1$-factor}  \rbrace.$$
Observe that the condition~$H_0=c_1|_Y=H_1$ ensures that~$\mathbf{H}$ is basepoint preserving.
\end{notation}

\begin{construction}[The homotopy extension map~$\Xi$]
\label{cons:Xi}
We construct a map
$$\Xi \colon \Homotopy(c_1|_Y) \to \hAut_Y(B),$$
where, recall, $\hAut_Y(B)$ denotes the group of rel.~$Y$ homotopy classes of maps~$B \to B$.
Apply the homotopy extension property to a homotopy~$\mathbf{H}=(H_t)_{t \in [0,1]}$ with~$H_0=c_1|_Y=H_1$.
The outcome is a homotopy~$\widetilde{H}_t  \colon B\to B$ that extends~$\mathbf{H}$ and satisfies~$\widetilde{H}_0=\id_B$.
Consider the map 
\begin{align*}
\Xi \colon \Homotopy(c_1|_Y) &\to \hAut_Y(B) \\
\mathbf{H} &\mapsto \widetilde{H}_1.
\end{align*}
Here note that~$\widetilde{H}_1 \colon B \to B$ is a homotopy equivalence because~$\wt{H}_1\simeq \wt{H}_0=\id_B$ by construction.
Also, the homotopy equivalence~$\widetilde{H}_1$ is rel.~$Y$ because~$\widetilde{H}_1|_Y=H_1=c_1|_Y$ is the inclusion.
\end{construction}

In general,  the homotopy equivalence $\Xi(\mathbf{H})$ does not belong to the group 
$$ \mathcal{G}=  \{ \varphi \in \hAut_{Y_1}(B) \mid  (\varphi \circ c_1)^*b_{X_1}^\partial = c_1^*b_{X_1}^\partial \}.$$
and so it is not possible to use~$\Xi(\mathbf{H})$ to act on $b(c_0,c_1)$.
We therefore consider the subset of~$\Homotopy(c_1|_Y)$ of homotopies whose image under $\Xi$ do belong to $\mathcal{G}$.

\begin{construction}[The subset~$\Homotopy(c_1|_Y)(c_1^*b_{X_1}^\partial)$.]
\label{cons:SimpleDef}
We already noted in Construction~\ref{cons:Xi} that for any homotopy~$\mathbf{H}=(H_t \colon Y \to B)_{t \in [0,1]}$ with~$H_0=c_1|_Y=H_1$, the homotopy equivalence~$\Xi(\mathbf{H})=\widetilde{H}_1$ satisfies~$\widetilde{H}_1 \circ c_1|_Y=H_1=c_1|_Y$.
We then consider the set of those~$\mathbf{H}$ such that~$\Xi(\mathbf{H})$ satisfies the second  defining condition of~$\mathcal{G}$:
$$\Homotopy(c_1|_Y)(c_1^*b_{X_1}^\partial):=\{  \mathbf{H} \in \Homotopy(c_1|_Y) \mid (\Xi(\mathbf{H}) \circ c_1)^*(b_{X_1}^\partial) =c_1^*b_{X_1}^\partial \}.$$
This way,  the inclusion~$\Xi(\Homotopy(c_1|_Y)(c_1^*b_{X_1}^\partial) \subset \mathcal{G}$ follows from the definition of~$\mathcal{G}$.
\end{construction}

The main technical result of this section is now the following.
\begin{theorem}
\label{thm:GoalActionpdf} 
For any homotopy~$\mathbf{H} \in \operatorname{Homotopy}(c_1)(c_1^*b_{X_1}^\partial)$, there is a corresponding sesquilinear form~$q_{\mathbf{H}} \in \Sesq(H^2(Y;\Z[\pi]))$ such that the homotopy equivalence $\Xi(\mathbf{H}) \colon B \to B$ acts on the set $\Herm(H^2(Y;\Z[\pi])$ of  hermitian forms as
$$\Xi(\mathbf{H})\cdot \lambda=\lambda+q_{\mathbf{H}}+q_{\mathbf{H}}^* \quad \quad \text{ for any } \lambda \in \Herm(H^2(Y;\Z[\pi]).$$
Conversely, 
if the map~$\ev^* \colon H^2(B;\Z[\pi])^{**} \to H_2(B;\Z[\pi])^*$ is an isomorphism and
\begin{itemize}
\item either the module~$H_2(B;\Z[\pi])$ is projective,
\item or $\pi$ is finite and $H_2(B;\Z[\pi]) \cong I \oplus \Z[\pi]^{\oplus k}$ for some $k \geq 0$,
\end{itemize}
then for any sesquilinear form~$q \in \Sesq(H^2(Y;\Z[\pi]))$, there is a homotopy~$\mathbf{H} \in \operatorname{Homotopy}(c_1)(c_1^*b_{X_1}^\partial)$ with~$q=q_{\mathbf{H}}.$
\end{theorem}

The proof of this theorem occupies the remainder of this section.
Given a homotopy $\mathbf{H}$, Section~\ref{sub:ActionOnhermitian} constructs the sequilinear form $q_{\mathbf{H}}$ and determines the action of~$\Xi(\operatorname{Homotopy}(c_1)(c_1^*b_{X_1}^\partial))$ on~$\Herm(H^2(Y;\Z[\pi])).$
Section~\ref{sub:Realisation} proves that for any~$q \in \Sesq(H^2(Y;\Z[\pi]))$, there exists a homotopy~$\mathbf{H} \in \operatorname{Homotopy}(c_1)(c_1^*b_{X_1}^\partial)$ with~$q=~q_{\mathbf{H}}$.

Before embarking on this proof,  we attempt to share some intuition.
\begin{remark}
\label{rem:Intuition}
We outline the intuition underlying Theorem~\ref{thm:GoalActionpdf} and its proof.
We emphasise that all of what follows is informal. 
First, recall the intuition underlying the definition of the hermitian form~$b(c_0,c_1)$.
Given $x \in H^2(Y;\Z[\pi])$, since~$\widetilde{X}_0$ and~$\widetilde{X}_1$ are simply-connected,  loops representing~$\PD_y(x)$ are nullhomotopic.
This leads to immersed discs~$D_0(x) \subset \wt{X}_0$ and~$D_1(x) \subset \wt{X}_1$.
Taking the union of these discs yields an immersed sphere~$S(x)=D_0(x) \cup D_1(x)$ representing the Poincaré dual of~$c^*(x') \in H^2(X_0 \cup_h X_1;\Z[\pi])$.
We then think of~$b(c_0,c_1)(x,y)$ as~$\lambda_{X_0 \cup_h X_1}(S(x),S(y))$.

The main idea of this section is that the discs can be modified in a way that changes~$b(c_0,c_1)$ by an even form. 
Here are slightly more (yet still informal) details underlying this idea.
\begin{itemize}
\item  A map~$\phi \colon H_1(Y;\Z[\pi]) \to H_2(Y;\Z[\pi])$ gives rise to a homotopy~$\mathbf{H}\colon Y \times I \to Y \to B$ from~$c_1$ to~$c_1$. 
Here is the intuition for building~$\mathbf{H}$: 
Modify the projection~$\proj_1 \colon Y \times I \to Y$ by mapping~$2$-cells of the form~$I \times \alpha$ to~$\proj_1(I \times \alpha)+\phi(\alpha)$ and compose the resulting map~$Y \times I \to Y$ with~$c_1 \colon Y \to B.$
\item Implanting this homotopy in a collar of~$\partial \widetilde{X}_i$ changes~$b(c_0,c_1)$, thought of as a hermitian form on~$H_1(Y;\Z[\pi])$, by~$q_{\mathbf{H}}+q_{\mathbf{H}}^*$ where~$q_\mathbf{H}$ is the sesquilinear form defined by $\phi$ via~$H_1(Y;\Z[\pi]) \xrightarrow{\phi}H_2(Y;\Z[\pi]) \cong H^1(Y;\Z[\pi])\to H_1(Y;\Z[\pi])^*$. 
Here, the idea is that changing~$c_1$ in a collar changes the sphere~$S(x)=D_0(x) \cup D_1(x)$ by adding~$\phi(x)$ in the collar of the discs~$D_0(x)$ and~$D_1(x)$. 
This changes~$\lambda_{X_0 \cup_h X_1}(S(x),S(y))$ by adding the expression~$\lambda_Y(\phi(x),y)+\lambda_Y(x,\phi(y))$ which corresponds to the expression~$q_{\mathbf{H}}+q_{\mathbf{H}}^*$.
\end{itemize}
Unfortunately, the presence of the maps~$c_0,c_1$ and~$c$ as well as numerous algebraic constraints make this geometric idea quite challenging and technical to carry out.
\end{remark}

\subsection{The action on hermitian forms}
\label{sub:ActionOnhermitian}

This section defines the sesquilinear form $q_{\mathbf{H}}$ mentioned in Theorem~\ref{thm:GoalActionpdf} and studies the action of $\Xi(\mathbf{H})$ on~$\Herm(H^2(Y;\Z[\pi]))$.
The definition of~$\Homotopy(c_1|_Y)(c_1^*b_{X_1}^\partial)$ from Construction~\ref{cons:SimpleDef} is simple to state but inconvenient to understand the action on~$\Herm(H^2(Y;\Z[\pi])).$
The objective is thus to reformulate the definition of~$\Homotopy(c_1|_Y)(c_1^*b_{X_1}^\partial)$ so as to better understand this action.
The main challenge is to find a condition on a homotopy~$\mathbf{H}$ that is equivalent to it satisfying~$(\Xi(\mathbf{H}) \circ c_1)^*(b_{X_1}^\partial) =c_1^*b_{X_1}^\partial$.

\medbreak
Recalling that~$j \colon B \to (B,Y_1)$ denotes the inclusion, we begin by recording a result concerning the fundamental class of~$Y \times S^1$. 

\begin{lemma}
\label{lem:FundamentalClass}
The following equation holds:
\begin{equation}
\label{eq:FundClassEquality}
(\widetilde{H}_1 \circ c_1)_*([X_1,Y])=(c_1)_*([X_1,Y])+j_*(\mathbf{H}_*([Y \times S^1])) \in H_4(B,Y).
\end{equation}
\end{lemma}
\begin{proof}
Let $M(i)$ be the mapping cylinder of the inclusion $i \colon Y\to X_1$. Since the inclusions of $Y$ into~$M(i)$ given by $Y\times\{0\}$ and $Y\times \{1\}$ are homotopic, there is a map $f\colon (X_1,Y)\to (M(i),Y\times\{1\})$ homotopic to the inclusion. In particular, $f$ is a homotopy equivalence and thus degree one.

Consider the following diagram.
\[\begin{tikzcd}
	(X_1,Y)\ar[r,"f"]\ar[d,"\times\{1\}"]&(M(i),Y\times\{1\})\ar[d,"c_1\cup\mathbf{H}"]\ar[dl,hook']\\
	(X_1\times I,Y\times \{1\})\ar[r,"{\widetilde{\mathbf{H}} \circ (c_1\times\id)}"]&(B,Y).
\end{tikzcd}\]
The lower triangle commutes by definition of $\wt{\mathbf{H}}$. 
The upper triangle commutes up to homotopy.
Since the diagram commutes up to homotopy and $f$ has degree one, 
\begin{equation}
	\label{eq:FundClassEquality-1}
	(\widetilde{H}_1 \circ c_1)_*([X_1,Y])=(c_1\cup\mathbf{H})_*([M(i),Y\times\{1\}]).
\end{equation}
In order to evaluate $(c_1\cup\mathbf{H})_*([M(i),Y\times\{1\}])$, consider the following commutative diagram:
$$
\xymatrix{
H_4(M(i),Y\times\{1\}) \ar[r]\ar[d]_{(c_1\cup\mathbf{H})_*}&H_4(M(i),Y\times\{0,1\}) \ar[r]^-\cong \ar[d]^{(c_1 \cup \mathbf{H})_*}& H_4(X_1,Y)\oplus H_4(Y\times I,Y\times\{0,1\}) \ar[d]^{\bsm (c_1)_*&(\mathbf{H})_* \esm}\\
H_4(B,Y)\ar[r]^-= &H_4(B,Y)\ar[r]^-= & H_4(B,Y).   \\
}
$$
Since the top row maps $[(M(i),Y\times \{1\}]$ to $([X_1,Y],[Y\times I,Y \times \{0,1\}]$, we deduce that
\begin{equation}
	\label{eq:FundClassEquality-2}
(c_1\cup\mathbf{H})_*[M(i),Y\times\{1\}]=(c_1)_*([X_1,Y])+\mathbf{H}_*([Y \times [0,1],Y \times \{ 0,1\}]).
\end{equation}
Next,  note that the quotient map $q\colon (Y \times I,Y \times \{0,1\})\to (Y\times S^1,Y)$ has degree one 
and consider the following commutative diagram:
\[
\begin{tikzcd}
	H_4(Y \times I,Y \times \{0,1\}) \ar[r,"q_*","\cong"']\ar[d,"{\mathbf{H}_*}"]&H_4(Y \times S^1,Y)  \ar[d,"\mathbf{H}_*"]& H_4(Y \times S^1)  \ar[d,"\mathbf{H}_*"] \ar[l,"j_*"',"\cong"]\\
	H_4(B,Y)\ar[r,"="] &H_4(B,Y) & H_4(B). \ar[l,"j_*"']  
\end{tikzcd}
\]
Since the top isomorphisms in the top row map fundamental class to fundamental class, it follows that~$\mathbf{H}_*([Y \times I,Y \times \{ 0,1\}])=j_*(\mathbf{H}_*([Y \times S^1]))$. 
Combining this with \eqref{eq:FundClassEquality-1} and \eqref{eq:FundClassEquality-2} establishes~\eqref{eq:FundClassEquality} and thus concludes the proof of the lemma.
\end{proof}

In order to understand the equation~$(\Xi(\mathbf{H}) \circ c_1)^*(b_{X_1}^\partial)=c_1^*b_{X_1}^\partial$, we need to unwind the definition of~$(\Xi(\mathbf{H}) \circ c_1)^*(b_{X_1}^\partial)$.
This requires some notation.

\begin{notation}
\label{not:mult0}
Consider the multiplication map~$\mult \colon \Z[\pi] \otimes_{\Z[\pi]} \overline{\Z[\pi]},(x,y) \mapsto x\overline{y}$ 
and, given a group $\pi$, a 
space $X$ and a surjection $\pi_1(X) \twoheadrightarrow \pi$, the composition
$$ \mult_0 \colon H_0(X;\Z[\pi] \btimes_\Z \Z[\pi]) \xrightarrow{\varpi_0,\cong} \Z[\pi] \otimes_{\Z[\pi]} \overline{\Z[\pi]} \xrightarrow{\mult} \Z[\pi].$$
Here the isomorphism $\varpi_0$ was introduced in Lemma~\ref{lem:Ccoinvariants}.

Given a pair $(X,A)$,  a calculation (carried out in Lemma~\ref{lem:capev}) shows that for~$\varphi \in H^k(X,A;\Z[\pi])$ and $x \in H_k(X,A;\Z[\pi])$,  the following equality holds:
\begin{equation}
\label{eq:lem:capev}
 \mult_0  (\varphi \cap x)
=\overline{\langle \varphi,x\rangle} \in \Z[\pi].
\end{equation}
Here recall that~$\langle \varphi,x\rangle$ denotes the evaluation of $\varphi$ at~$x$.

Furthermore, another calculation (carried out in Lemma~\ref{lem:cup}) shows that for~$\sigma \in H_4(X)$,  the pairing~$b_\sigma \colon H^2(X;\Z[\pi]) \times H^2(X;\Z[\pi]) \to \Z[\pi],(x,y)\mapsto \langle y, x \cap \sigma \rangle
$ agrees with the hermitian pairing
\begin{align}
\label{eq:lem:cup}
b_\sigma^\cup \colon H^2(X;\Z[\pi]) \times H^2(X;\Z[\pi]) &\to \Z[\pi] \nonumber \\
(x,y)& \mapsto \mult_0((x \cup y) \cap \sigma). 
 \end{align}
\end{notation}

The next lemma takes a first pass at unpacking the definition of~$(\Xi(\mathbf{H}) \circ c_1)^*(b_{X_1}^\partial)$.

\begin{lemma}
\label{lem:UnderstandPullback}
For~$a \in H^2(B;\Z[\pi])$ and~$r \in H^2(B,Y;\Z[\pi])$, the following equalities hold:
\begin{align*}
(\Xi(\mathbf{H}) \circ c_1)^*(b_{X_1}^\partial)(r,a)
=c_1^*b_{X_1}^\partial(r,a)+\mult_0((\mathbf{H}^*(j^*(r)) \cup \mathbf{H}^*(a)) \cap [Y \times S^1]) \in \Z[\pi].
\end{align*}
\end{lemma}
\begin{proof}
Recall that, by definition,~$b_{X_1}^\partial(x,y)=\langle y,x \cap [X_1,Y]\rangle.$
The definition of~$\mathbf{H}$ and the naturality of the evaluation map, then imply that
$$(\Xi(\mathbf{H}) \circ c_1)^*(b_{X_1}^\partial)(r,a)
=\langle (\widetilde{H}_1 \circ c_1)^*(a) ,(\widetilde{H}_1 \circ c_1)^*(r) \cap [X_1,Y] \rangle
=\langle a, r \cap (\widetilde{H}_1 \circ c_1)_*([X_1,Y]) \rangle.$$
The lemma therefore hinges on understanding~$(\widetilde{H}_1 \circ c_1)_*([X_1,Y])$.
Lemma~\ref{lem:FundamentalClass} and~\eqref{eq:lem:cup} yield
\begin{align*}
(\Xi(\mathbf{H}) \circ c_1)^*(b_{X_1}^\partial)(r,a)
&=\langle a, r \cap (c_1)_*([X_1,Y]) \rangle
+\langle a, r \cap j_*(\mathbf{H}_*([Y \times S^1])) \rangle \\
&=c_1^*b_{X_1}^\partial(r,a) 
+\langle  \mathbf{H}^*(a), \mathbf{H}^*(j^*(r))\cap [Y \times S^1] \rangle \\
&=c_1^*b_{X_1}^\partial(r,a) +b_{Y \times S^1}(\mathbf{H}^*(j^*(r)),\mathbf{H}^*(a)) \\
&=c_1^*b_{X_1}^\partial(r,a) +\mult_0((\mathbf{H}^*(j^*(r)) \cup \mathbf{H}^*(a)) \cap [Y \times S^1]).
\end{align*}
This concludes the proof of the lemma.
\end{proof}

Lemma~\ref{lem:UnderstandPullback} implies that,  in order to understand the pairing~$(\Xi(\mathbf{H}) \circ c_1)^*(b_{X_1}^\partial)$, we need to study the pairing~$(\alpha,\beta) \mapsto \mult_0((\mathbf{H}^*(\alpha) \cup \mathbf{H}^*(\beta)) \cap [Y \times S^1])$ for~$\alpha,\beta \in H^2(B;\Z[\pi])$.
This requires the following construction.

\begin{construction}
Associated to the homotopy~$\mathbf{H}$ is the map~$\phi_{\mathbf{H}}' \colon H^2(B;\Z[\pi]) \to 
H_1(Y;\Z[\pi])^*.$
obtained by composing~$\mathbf{H}^*$ with the inverse of the Künneth isomorphism and the projection onto the second component:
\begin{align*}\phi_{\mathbf{H}}' \colon H^2(B;\Z[\pi]) \xrightarrow{\mathbf{H}^*} H^2(Y \times S^1;\Z[\pi]) \xrightarrow{\Kun^{-1}} H^2(Y;\Z[\pi]) \oplus H^1(Y;\Z[\pi]) &\xrightarrow{\operatorname{proj}_2} H^1(Y;\Z[\pi]) \\
&\xrightarrow{\ev} H_1(Y;\Z[\pi])^*
\end{align*}
For brevity we set~$\psi:=\proj_2 \circ \Kun^{-1}$.
\end{construction}

The next couple of lemmas aim to prove the following technical result.
The reader wishing to skip ahead to the proof of this result can instead consult the proof following Lemma~\ref{lem:SecondTerm}.

\begin{proposition}
\label{prop:PairingbH}
For~$\alpha,\beta\in H^2(B;\Z[\pi])$,  we have the following equation in~$\Z[\pi]$:
$$
\mult_0((\mathbf{H}^*(\alpha) \cup \mathbf{H}^*(\beta)) \cap [Y \times S^1] )
=
\phi_\mathbf{H}'(\beta)(c_1|_Y^*(\alpha) \cap [Y]) + 
\overline{\phi_\mathbf{H}'(\alpha)(c_1|_Y^*(\beta) \cap [Y])}.
$$
\end{proposition}

The strategy is as follows.
Lemma~\ref{lem:HaHb} describes~$(\mathbf{H}^*(\alpha) \cup \mathbf{H}^*(\beta)) \cap [Y \times S^1]$ as a sum of two expressions.
Lemma~\ref{lem:FirstTerm} investigates the first and shows that after applying~$\mult_0$, it reduces to~$\phi_\mathbf{H}'(\beta)(c_1|_Y^*(\alpha) \cap [Y])$.
Lemma~\ref{lem:SecondTerm} applies a similar reasoning to the second term and the conclusion then follows promptly.

We begin by studying~$(\mathbf{H}^*(\alpha) \cup \mathbf{H}^*(\beta)) \cap [Y \times S^1]$.

\begin{lemma}
\label{lem:HaHb}
For~$\alpha,\beta\in H^2(B;\Z[\pi])$,  we have the following equation in~$H_0(Y \times S^1;\Z[\pi] \boxtimes \Z[\pi])$:
\begin{align}
\label{eq:TwoSummandsofHaHb}
(\mathbf{H}^*(\alpha) &\cup \mathbf{H}^*(\beta)) \cap [Y \times S^1]= \nonumber \\
&\left( \overbrace{p_Y^*(c_1|_Y^*(\alpha))}^{\in H^2(Y\times S^1;\Z[\pi])} \cup
\left( \overbrace{p_Y^*(\psi(\mathbf{H}^*(\beta)))}^{\in H^1(Y\times S^1;\Z[\pi])}   \cup \overbrace{p_{S^1}^*(1_{H^1(S^1)})}^{\in H^1(Y\times S^1;\Z)}\right)\right)
 \cap [Y \times S^1]  \nonumber \\
+&\left( \left( \underbrace{p_Y^*(\psi(\mathbf{H}^*(\alpha)))}_{\in H^1(Y \times S^1;\Z[\pi])} \cup \underbrace{p_{S^1}^*(1_{H^1(S^1)})}_{\in H^1(Y \times S^1;\Z)} \right) \cup \underbrace{p_Y^*(c_1|^*_Y(\beta))}_{\in H^2(Y \times S^1;\Z[\pi])}\right) \cap [Y \times S^1].
\end{align}
\end{lemma}
\color{black}
\begin{proof}
Given~$\alpha,\beta \in H^2(B;\Z[\pi])$, we start by understanding the cup product~$\mathbf{H}^*(\alpha) \cup \mathbf{H}^*(\beta).$
To this effect, we study the images of~$\mathbf{H}^*(\alpha)$ and~$\mathbf{H}^*(\beta)$ under the inverse Künneth isomorphism
\begin{align*}
H^2(Y \times S^1;\Z[\pi] \boxtimes_\Z \Z) \xrightarrow{\cong}H^2(Y;\Z[\pi]) \oplus H^1(Y;\Z[\pi]) \\
z  \mapsto (\iota_Y^*(z), \psi(z)).
\end{align*}
Here~$\iota_Y \colon Y \to Y \times S^1$ denotes the inclusion, whereas the formula for~$\psi$ is irrelevant.
Applying the formula for the twisted K\"unneth isomorphism from Proposition~\ref{prop:KunnethYS1}, we then obtain the following equality in~$H^2(Y\times S^1;\Z[\pi] \btimes \Z):$
\begin{align*}
\mathbf{H}^*(\alpha)
&= \operatorname{Kunneth}\circ  \operatorname{Kunneth}^{-1}(\mathbf{H}^*(\alpha)) \\
&=p_Y^*(\iota_Y^*(\mathbf{H}^*(\alpha)))+p_Y^*(\psi(\mathbf{H}^*(\alpha))) \cup p_{S^1}^*(1_{H^1(S^1)}) \\
&=p_Y^*(c_1|_Y^*(\alpha))+p_Y^*(\psi(\mathbf{H}^*(\alpha))) \cup p_{S^1}^*(1_{H^1(S^1)}).
\end{align*}
For~$\alpha,\beta \in H^2(B;\Z[\pi])$, we use this to calculate~$\mathbf{H}^*(\alpha) \cup \mathbf{H}^*(\beta)$.
To make sense of this cup product, it is helpful to return to thinking of~$H^2(Y \times S^1;\Z[\pi] \boxtimes \Z)$ as~$H^2(Y \times S^1;\Z[\pi])$ so that this cup product belongs to~$H^4(Y\times S^1;\Z[\pi]\btimes \Z[\pi])$.

First,~$H^4(Y;\Z[\pi] \boxtimes \Z[\pi])=0$ implies
$$\overbrace{p_Y^*(c_1|_Y^*(\alpha))}^{\in H^2(Y\times S^1;\Z[\pi])} \cup \overbrace{p_Y^*(c_1|_Y^*(\beta))}^{\in H^2(Y\times S^1;\Z[\pi])}
=p_Y^*(\overbrace{c_1|_Y^*(\alpha) \cup c_1|_Y^*(\beta)}^{\in H^4(Y;\Z[\pi]\btimes \Z[\pi])})=0.$$
Consider the isomorphism~$\operatorname{flip}_{23} \colon \Z[\pi]  \boxtimes \Z[\pi] \boxtimes \Z \boxtimes \Z \to \Z[\pi] \boxtimes \Z \boxtimes \Z[\pi] \boxtimes \Z$
given by~$\operatorname{flip}_{23}(a \otimes b \otimes c \otimes d)=a \otimes c \otimes b \otimes d.$
Write~$(\operatorname{flip}_{23})_*$ for the induced map on twisted cohomology.
The graded commutativity of the cup product and~$H^2(S^1)=0$ imply that
\begin{align*} 
&\left(p_Y^*(\psi(\mathbf{H}^*(\alpha))) 
\cup p_{S^1}^*(\overbrace{1_{H^1(S^1)})}^{\in H^1(S^1)}\right) 
\cup
\left(p_Y^*(\psi(\mathbf{H}^*(\beta)))
\cup p_{S^1}^*(\overbrace{1_{H^1(S^1)})}^{\in H^1(S^1)}\right) \\
&=-(\operatorname{flip}_{23})_*\left(
p_Y^*(\psi(\mathbf{H}^*(\alpha)))
\cup p_Y^*(\psi(\mathbf{H}^*(\beta)))
\cup 
\underbrace{p_{S^1}^*(1_{H^1(S^1)})
\cup p_{S^1}^*(1_{H^1(S^1)})}_{=0}
\right)  
=0.
\end{align*}
The announced calculation of~$(\mathbf{H}^*(\alpha) \cup \mathbf{H}^*(\beta)) \cap [Y \times S^1]$ now follows readily.
\end{proof}

We successively analyze the two summands that appear in Lemma~\ref{lem:HaHb}.
\begin{lemma}
\label{lem:FirstTerm}
The following equality holds in~$H_0(Y\times S^1;\Z[\pi] \btimes \Z[\pi])$:
\begin{align}
\label{eq:CupCalculationLemma}
&\left( \overbrace{p_Y^*(c_1|_Y^*(\alpha))}^{\in H^2(Y\times S^1;\Z[\pi])} \cup
\left( \overbrace{p_Y^*(\psi(\mathbf{H}^*(\beta)))}^{\in H^1(Y\times S^1;\Z[\pi])}   \cup \overbrace{p_{S^1}^*(1_{H^1(S^1)})}^{\in H^1(Y\times S^1;\Z)}\right)\right)
 \cap [Y \times S^1] \nonumber  \\
&=\left(\left(\overbrace{c_1|_Y^*(\alpha)}^{\in H^2(Y;\Z[\pi])} \cup \overbrace{\psi(\mathbf{H}^*(\beta))}^{\in H^1(Y;\Z[\pi])} \right) \cap [Y] \right) \times \overbrace{1_{H_0(S^1;\Z)}}^{\in H_0(S^1;\Z)}.
\end{align}
Furthermore, after applying~$\mult_0$, the following equality holds in~$\Z[\pi]$:
$$
\mult_0\left(\left(p_Y^*(c_1|_Y^*(\alpha)) \cup
\left(p_Y^*(\psi(\mathbf{H}^*(\beta)))   \cup p_{S^1}^*(1_{H^1(S^1)})\right)\right)
 \cap [Y \times S^1]\right)
=
\mult_0\left(
(c_1|_Y^*(\alpha) \cup \psi(\mathbf{H}^*(\beta)))
\cap [Y]\right).
$$
\end{lemma}
\begin{proof}
First, using the associativity
and naturality of the cup product,  
the left hand side of~\eqref{eq:CupCalculationLemma} can be rewritten as
$$
 \left(\left(\overbrace{p_Y^*(c_1|_Y^*(\alpha) \cup \psi(\mathbf{H}^*(\beta)))}^{\in H^3(Y\times S^1;\Z[\pi] \boxtimes \Z[\pi])} \right) \cup \overbrace{p_{S^1}^*(1_{H^1(S^1)})}^{\in H^1(Y\times S^1)}\right) \cap [Y \times S^1].
~$$
Dimensional considerations ensure that $[Y \times S^1]=[Y] \times [S^1] \in H_4(Y \times S^1).$
Using the definition of the cross product we can therefore rewrite the previous equality as
$$ 
\left(
\left( \overbrace{c_1|_Y^*(\alpha)}^{\in H^2(Y;\Z[\pi])} \cup \overbrace{\psi(\mathbf{H}^*(\beta))}^{\in H^1(Y;\Z[\pi])}
\right)
 \times \overbrace{1_{H^1(S^1)}}^{\in H^1(S^1)}
 \right) \cap ([Y] \times [S^1]).$$
Using properties of the cross product~see e.g. \cite[Lemma 5.2 item (2)]{ConwayNagel}), this equals
\begin{align*}
\overbrace{(c_1|_Y^*(\alpha) \cup \psi(\mathbf{H}^*(\beta))) \cap [Y])}^{\in H_0(Y;\Z[\pi])} \times \overbrace{(1_{H^1(S^1)} \cap [S^1])}^{\in H_0(S^1;\Z)}=\overbrace{(c_1|_Y^*(\alpha) \cup \psi(\mathbf{H}^*(\beta))) \cap [Y])}^{\in H_0(Y;\Z[\pi])} \times \overbrace{1_{H_0(S^1)}}^{\in H_0(S^1;\Z)}.
\end{align*}
To obtain the last sentence of the lemma,  apply Remark~\ref{rem:CrossLemma} from the appendix. 
\end{proof}

We focus on the second summand of Lemma~\ref{lem:HaHb}, namely
$$ \left( \left( \overbrace{p_Y^*(\psi(\mathbf{H}^*(\alpha)))}^{\in H^1(Y \times S^1;\Z[\pi])} \cup \overbrace{p_{S^1}^*(1_{H^1(S^1)})}^{\in H^1(Y \times S^1;\Z)} \right) \cup \overbrace{p_Y^*(c_1|^*_Y(\beta))}^{\in H^2(Y \times S^1;\Z[\pi])}\right) \cap [Y \times S^1].$$
Since twisted cup products only satisfy graded commutativity ``up to isomorphism", we cannot immediately exchange the terms in this expression and apply the same reasoning as in Lemma~\ref{lem:FirstTerm}.
Fortunately, this becomes possible after applying~$\mult_0$, as the next lemma demonstrates.

\begin{lemma}
\label{lem:SecondTerm}
The following equality holds in~$\Z[\pi]$:
\begin{align*}
&\mult_0\left(\left(\left(p_Y^*(\psi(\mathbf{H}^*(\alpha))) \cup p_{S^1}^*(1_{H^1(S^1)})\right) \cup p_Y^*(c_1|_Y^*(\beta)) \right) \cap [Y \times S^1]\right) \\
&=
\overline{\mult_0\left(\left(c_1|_Y^*(\beta) \cup \psi(\mathbf{H}^*(\alpha)\right) \cap [Y]\right)}.
\end{align*}
\end{lemma}
\begin{proof}
The first of the following equalities follows from the symmetry of twisted cup products (see \cref{lem:CupProductSymmetric}), whereas the second follows as in the proof of \cref{lem:FirstTerm}:
\begin{align*}
&\mult_0\left(\left(\left(p_Y^*(\psi(\mathbf{H}^*(\alpha))) \cup p_{S^1}^*(1_{H^1(S^1)})\right) \cup p_Y^*(c_1|_Y^*(\beta)) \right) \cap [Y \times S^1]\right) \\
&=\overline{\mult_0\left(\left(p_Y^*(c_1|_Y^*(\beta))\cup \left(p_Y^*(\psi(\mathbf{H}^*(\alpha))) \cup p_{S^1}^*(1_{H^1(S^1)})\right) \right) \cap [Y \times S^1]\right)}\\
&=
\overline{\mult_0\left(\left(c_1|_Y^*(\beta) \cup \psi(\mathbf{H}^*(\alpha)\right) \cap [Y]\right)}.
\end{align*}
This concludes the proof of the lemma.
\end{proof}

We can conclude the proof of Proposition~\ref{prop:PairingbH} which, recall, states that for~$\alpha,\beta\in H^2(B;\Z[\pi])$,  we have the following equation in~$\Z[\pi]$:
$$
\mult_0((\mathbf{H}^*(\alpha) \cup \mathbf{H}^*(\beta)) \cap [Y \times S^1] )
=
\phi_\mathbf{H}'(\beta)(c_1|_Y^*(\alpha) \cap [Y]) + 
\overline{\phi_\mathbf{H}'(\alpha)(c_1|_Y^*(\beta) \cap [Y])}.
$$
\begin{proof}[Proof of Proposition~\ref{prop:PairingbH}]
The combination of Lemmas~\ref{lem:HaHb},~\ref{lem:FirstTerm} and~\ref{lem:SecondTerm} gives
\begin{align*}
&\mult_0((\mathbf{H}^*(\alpha) \cup \mathbf{H}^*(\beta)) \cap [Y \times S^1] ) \\
&=
\mult_0(
(c_1|_Y^*(\alpha) \cup \psi(\mathbf{H}^*(\beta)))
\cap [Y])
+\overline{\mult_0\left(\left(c_1|_Y^*(\beta) \cup \psi(\mathbf{H}^*(\alpha)\right) \cap [Y]\right)}.
\end{align*}
Use the symmetry of twisted cup products (see \cref{lem:CupProductSymmetric}),
to rewrite this as
$$
\overline{\mult_0\left(
 \left(\psi(\mathbf{H}^*(\beta))
 \cup 
 c_1|_Y^*(\alpha)\right)
\cap [Y]\right)}
+\mult_0\left(
\left(
\psi(\mathbf{H}^*(\alpha))
\cup 
c_1|_Y^*(\beta)
\right) \cap [Y]\right).
$$
Applying the equality~$(x \cup y )\cap z=x \cap (y \cap z)$ and the 
equality~$\mult_0(\varphi \cap x)= \overline{\langle \varphi,x\rangle} \in \Z[\pi]$ from~\eqref{eq:lem:capev} now yields: 
\begin{align*}
   &\overline{\mult_0(
\psi(\mathbf{H}^*(\beta)) \cap (c_1|_Y^*(\alpha) \cap [Y] )}
+\mult_0(
\psi(\mathbf{H}^*(\alpha)) \cap (c_1|_Y^*(\beta) \cap [Y] ) \\
=
   &
\langle \psi(\mathbf{H}^*(\beta)), c_1|_Y^*(\alpha) \cap [Y] \rangle
+ \overline{\langle
\psi(\mathbf{H}^*(\alpha)),c_1|_Y^*(\beta) \cap [Y] \rangle }\\
=&\phi_\mathbf{H}'(\beta)(c_1|_Y^*(\alpha) \cap [Y])
+ 
\overline{\phi_\mathbf{H}'(\alpha)(c_1|_Y^*(\beta) \cap [Y])}.
\end{align*}
\color{black}
This concludes the proof of the proposition.
\end{proof}
\color{black}

The next proposition leads to the announced reformulation of the set~$\Homotopy(c_1|_Y)(c_1^*b_{X_1}^\partial).$

\begin{proposition}
\label{prop:H2H1}
Given~$\mathbf{H} \in  \Homotopy(c_1|_Y)$, the following assertions are equivalent:
\begin{enumerate}
\item~$(\Xi(\mathbf{H}) \circ c_1)^*(b_{X_1}^\partial)=c_1^*b_{X_1}^\partial$.
\item The map~$\phi_{\mathbf{H}}' \colon H^2(B;\Z[\pi]) \to H^1(Y;\Z[\pi])$ factors through~$H^2(Y;\Z[\pi])$, meaning that there exists a homomorphism~$\phi \colon H^2(Y;\Z[\pi]) \to H^1(Y;\Z[\pi])$ such that ~$\phi_{\mathbf{H}}'=\phi \circ c_1|_Y^*:$
$$
\xymatrix@C0.8cm{
 H^2(B;\Z[\pi])  \ar[r]^-{\mathbf{H}^*}  \ar[d]^{c_1|_Y^*} \ar@/^3pc/[rrr]^{\phi_{\mathbf{H}}'}
&H^2(Y \times S^1;\Z[\pi]) \ar[r]^-{\psi}
&H^1(Y;\Z[\pi]) \ar[r]^-{\ev}
 &H_1(Y;\Z[\pi])^*. \\
 H^2(Y;\Z[\pi]) \ar@{-->}[rrru]_-{\phi}
 }
$$
\end{enumerate}
\end{proposition}
\begin{proof}
Given~$\alpha,\beta\in H^2(B;\Z[\pi])$, set 
	\[b_{\mathbf{H}}(\alpha,\beta):=
 \phi_\mathbf{H}'(\beta)(c_1|_Y^*(\alpha) \cap [Y]) + 
\overline{\phi_\mathbf{H}'(\alpha)(c_1|_Y^*(\beta) \cap [Y])}. \]
Lemma~\ref{lem:UnderstandPullback} and Proposition~\ref{prop:PairingbH} imply that for every~$a \in H^2(B;\Z[\pi])$ and~$r \in H^2(B,Y;\Z[\pi])$,
$$(\Xi(\mathbf{H}) \circ c_1)^*(b_{X_1}^\partial)(r,a)=c_1^*b_{X_1}^\partial(r,a)+b_{\mathbf{H}}(j^*(r),a).$$
It therefore suffices to show that~$b_{\mathbf{H}}(j^*(r),a)=0$ for every ~$a \in H^2(B;\Z[\pi])$ and~$r \in H^2(B,Y;\Z[\pi])$ if and only if there exists a~$\phi \in \Hom(H^2(Y;\Z[\pi]),H_1(Y;\Z[\pi])^*)$ such that~$\phi_{\mathbf{H}}'=\phi \circ c_1|_Y^*.$

We first assume that~$\phi_{\mathbf{H}}'=\phi
 \circ c_1|_Y^*$ exists and prove that~$b_{\mathbf{H}}(j^*(r),a)=0$ for every ~$a \in H^2(B;\Z[\pi])$ and~$r \in H^2(B,Y;\Z[\pi])$. 
This follows because,  by exactness of the sequence of the pair~$(B,Y)$,  we have~$c_1|_Y^* \circ j^*=0$ and therefore 
\begin{align*}
b_{\mathbf{H}}(j^*(r),a)
&=\phi_\mathbf{H}'(a)(c_1|_Y^*(j^*(r)) \cap [Y]) + 
\overline{\phi_\mathbf{H}'(j^*(r))(c_1|_Y^*(a) \cap [Y])} \\
&=\phi_\mathbf{H}'(a)(\underbrace{c_1|_Y^*(j^*(r))}_{=0} \cap [Y]) + 
\overline{\phi(\underbrace{c_1|_Y^*(j^*(r))}_{=0})(c_1|_Y^*(a) \cap [Y])} 
=0.
\end{align*}
We now prove the converse: we assume that~$b_{\mathbf{H}}(j^*(r),a)=0$ for every~$a \in H^2(B;\Z[\pi])$ and every~$r \in H^2(B,Y;\Z[\pi])$ and prove that~$\phi$ exists.
Consider the exact sequence
$$
\xymatrix{
H^2(B,Y;\Z[\pi]) \ar[r]^-{j^*}& H^2(B;\Z[\pi]) \ar[r]^{(c_1)|^*_Y}\ar[d]_{\phi_{\mathbf{H}}'}& H^2(Y;\Z[\pi]) \ar[r]\ar@{-->}[ld]^{\phi}& 0. \\
&H_1(Y;\Z[\pi])^*&&
}
$$ 
The rightmost zero follows from the fact,  proved in Proposition~\ref{prop:PullbackProperties},  that~$(c_1)|^*_Y$ is surjective. 
This diagram shows that proving~$\phi_{\mathbf{H}}' \circ j^*=0$ suffices to establish the existence of~$\phi$.
To do so, we first note that by exactness we have for every~$a \in H^2(B;\Z[\pi])$ and~$r \in H^2(B,Y;\Z[\pi])$:
\begin{align*}
0&=b_{\mathbf{H}}(j^*(r),a) 
=\phi_\mathbf{H}'(a)(\overbrace{c_1|_Y^*(j^*(r))}^{=0} \cap [Y]) + 
\overline{\phi_\mathbf{H}'(j^*(r))(c_1|_Y^*(a) \cap [Y])} 
=\overline{\phi_\mathbf{H}'(j^*(r))(c_1|_Y^*(a) \cap [Y])}.
\end{align*}

Since the map~$c_1|_Y^* \colon H^2(Y;\Z[\pi]) \to H^2(B;\Z[\pi])$ is surjective, this implies
$\phi_\mathbf{H}'(j^*(r))(x)=0$ for every~$x \in  H_1(Y;\Z[\pi])$.
It follows that~$\phi_{\mathbf{H}}'(j^*(r))=0 \in H_1(Y;\Z[\pi])^*$ for every~$r \in H^2(B,Y;\Z[\pi])$.
As mentioned above,  this implies that~$\phi_{\mathbf{H}}'=\phi \circ c_1|_Y^*$ for some map~$\phi \colon H^2(Y;\Z[\pi]) \to H_1(Y;\Z[\pi])^*$, as required.
	\end{proof}
	
In order to describe the action of~$\Xi(\Homotopy(c_1|_Y)(c_1^*b_{X_1}^\partial))$ on~$\Herm(H^2(Y;\Z[\pi]))$, we construct the sesqulinear form $q_{\mathbf{H}}$ that appears in the statement of Theorem~\ref{thm:GoalActionpdf}.
\begin{construction}[{The sesquilinear form~$q_{\mathbf{H}}  \in \Sesq(H^2(Y;\Z[\pi]))$}]
\label{cons:qH}
Given a homotopy~$\mathbf{H} \in \Homotopy(c_1|_Y)(c_1^*b_{X_1}^\partial)$ and a map~$\phi \colon H^2(Y;\Z[\pi]) \to H_1(Y;\Z[\pi])^*$ with~$\phi'_{\mathbf{H}}=\phi \circ c_1|_Y^*$ (which exists by Proposition~\ref{prop:H2H1}),  consider the pairing 
\begin{align*}
q_{\mathbf{H}} \colon H^2(Y;\Z[\pi]) \times H^2(Y;\Z[\pi]) &\to \Z[\pi] \\
(a,b) &\mapsto 
\phi(b)(a \cap [Y]).
\end{align*}
\end{construction}

The next proposition describes the action of~$\Xi(\Homotopy(c_1|_Y)(c_1^*b_{X_1}^\partial))$ on~$\Herm(H^2(Y;\Z[\pi]))$ thereby proving the first half of Theorem~\ref{thm:GoalActionpdf}.
	
\begin{proposition}
\label{prop:HowItActs}
Given a homotopy~$\mathbf{H} \in \Homotopy(c_1|_Y)(c_1^*b_{X_1}^\partial)$,  the action of~$\Xi(\mathbf{H}) \in \mathcal{G}$ on~$\lambda \in \operatorname{Herm}(H^2(Y;\Z[\pi]))$ is by
$$\Xi(\mathbf{H})\cdot \lambda=\lambda+q_{\mathbf{H}}+q_{\mathbf{H}}^*$$
where~$q_{\mathbf{H}} \in \Hom(H^2(Y;\Z[\pi]),H^2(Y;\Z[\pi])^*)$ is the pairing from Construction~\ref{cons:qH} associated to~$\mathbf{H}$ and to a~$\phi \colon H^2(Y;\Z[\pi]) \to H_1(Y;\Z[\pi])^*$ with~$\phi'_{\mathbf{H}}=\phi \circ c_1|_Y^*$.
	\end{proposition}
	\begin{proof}
Recall that the action of~$\varphi \in \mathcal{G}$ on~$\lambda \in \operatorname{Herm}(H^2(Y;\Z[\pi]))$ is by~$\varphi \cdot \lambda=\lambda +b(\varphi)$ where~$b(\varphi)(a,b)=\langle \beta,\alpha \cap \xi\rangle$, where~$\alpha,\beta \in H^2(B;\Z[\pi])$ are such that~$c_1|_Y^*(\alpha)=a$ and~$c_1|_Y^*(\beta)=b.$
By~\eqref{eq:lem:cup}, this can be rewritten as
$$b(\varphi)(a,b)=\langle \beta,\alpha \cap \xi\rangle
=\mult_0((\alpha \cup \beta) \cap \xi)
.$$
Recall additionally that~$\xi \in H_4(B)$ is the unique class such that
$$j_*(\xi)=(\varphi \circ c_1)_*([X_1])-(c_1)_*([X_1]) \in H_4(B,Y).$$
We apply this definition to~$\varphi=\Xi(\mathbf{H}).$
First, note that
$$\Xi(\mathbf{H}) \cdot \lambda
=\lambda +b(\Xi(\mathbf{H}))
=\lambda +b(\widetilde{H}_1).$$
To determine the class~$\xi$, recall that we proved in Lemma~\ref{lem:UnderstandPullback} that 
$$(\widetilde{H}_1 \circ c_1)_*([X_1,Y])=(c_1)_*([X_1,Y])+j_*(\mathbf{H}_*([Y \times S^1])) \in H_4(B,Y).$$
We can therefore take~$\xi=\mathbf{H}_*([Y \times S^1])$ in the definition of~$b(\widetilde{H}_1)$.
Thus, for~$a,b \in H^2(Y;\Z[\pi])$ with~$c_1|_Y^*(\alpha)=a,c_1|_Y^*(\beta)=b$ for~$\alpha,\beta \in H^2(B;\Z[\pi])$,  the definition of~$b(\widetilde{H}_1)$ and Proposition~\ref{prop:PairingbH} imply that
\begin{align*} 
b(\widetilde{H}_1)(a,b)
&=\mult_0((\alpha \cup \beta) \cap \mathbf{H}_*([Y \times S^1]))  
=\phi_\mathbf{H}'(\beta)(c_1|_Y^*(\alpha) \cap [Y]) + 
\overline{\phi_\mathbf{H}'(\alpha)(c_1|_Y^*(\beta) \cap [Y])} \\
&=\phi \circ c_1|_Y^*(\beta)(c_1|_Y^*(\alpha) \cap [Y]) + 
\overline{\phi \circ c_1|_Y^*(\alpha)(c_1|_Y^*(\beta) \cap [Y])} 
=\phi(b)(a \cap [Y]) + 
\overline{\phi(a)(b \cap [Y])} \\
&=q_{\mathbf{H}}(a,b)+\overline{q_{\mathbf{H}}(b,a)}=q_{\mathbf{H}}(a,b)+q_{\mathbf{H}}^*(a,b).
\end{align*}
This concludes the proof of the proposition.
	\end{proof}

	\subsection{Realisation}
	\label{sub:Realisation}

In order to prove Theorem~\ref{thm:GoalActionpdf}, it remains to show that for any sesquilinear form~$q \in \Sesq(H^2(Y;\Z[\pi]))$, there is a~$\mathbf{H} \in \operatorname{Homotopy}(c_1)(c_1^*b_{X_1}^\partial)$ with~$q=q_{\mathbf{H}}$.
Here, given~$\phi \colon H^2(Y;\Z[\pi]) \to H_1(Y;\Z[\pi])^*$ with~$\phi_{\mathbf{H}}'=\phi \circ c_1|_Y^*$, recall from Construction~\ref{cons:qH} that 
\begin{equation}
\label{eq:qH}
q_{\mathbf{H}}(a,b)
:=\phi(b)(a \cap [Y]).
\tag{$\star$}
\end{equation}

\begin{plan}
\label{plan:ProgrammeFinalStep}
Fix~$q \in \Sesq(H^2(Y;\Z[\pi]))$.
In order to prove that there exists $\mathbf{H} \in \operatorname{Homotopy}(c_1)(c_1^*b_{X_1}^\partial)$ with~$q=q_{\mathbf{H}}$, we first note that there is a~$\phi \colon H^2(Y;\Z[\pi]) \to H_1(Y;\Z[\pi])^*$ with~$q(a,b)=\phi(b)(a \cap [Y])$
for every~$a,b \in H^2(Y;\Z[\pi]).$
Indeed the composition
$$\phi \colon H^2(Y;\Z[\pi]) \xrightarrow{q} H^2(Y;\Z[\pi])^* \xrightarrow{(\PD_Y^{-1})^*} 
H_1(Y;\Z[\pi])^* 
$$
is readily seen to satisfy
\begin{equation}
\label{eq:phiq}
 \phi(b)(a \cap [Y])=q(b)(a).
 \end{equation}
So we have to prove that for any~$\phi$ there is a~$\mathbf{H} \in \operatorname{Homotopy}(c_1)(c_1^*b_{X_1}^\partial)$ with~$\phi'_{\mathbf{H}}=\phi \circ c_1|_Y^*.$

Indeed, combining~\eqref{eq:phiq} with~\eqref{eq:qH},  it would then follow that
$$ q_{\mathbf{H}}(a,b)=\phi(b)(a \cap [Y])=q(a,b).$$
\end{plan}
Summarising, we need to prove that for every homomorphism~$\phi \colon H^2(Y;\Z[\pi]) \to H^1(Y;\Z[\pi])$ there is a~$\mathbf{H}\in \operatorname{Homotopy}(c_1)(c_1^*b_{X_1}^\partial)$ with~$\phi'_{\mathbf{H}}=\phi \circ c_1|_Y^*$.

We begin with two algebraic lemmas. 
In the first,  given a $\Z[\pi]$-module $V$, we write~$\ev_V \colon V \to~V^{**}$ for~$\ev_V(x)(f)=f(x).$
A quick verification shows that~$V^*\xrightarrow{\ev_{V^*}}V^{***}\xrightarrow{\ev_{V}^*}V^*$ is the identity on~$V^*$.
\begin{lemma}
\label{lem:DualExist}
Let $P$ be a~$\Z[\pi]$-module. for which $\ev_P$ is an isomorphism.
For every~$\Z[\pi]$-module~$L$ and every~$\Z[\pi]$-linear map~$f\colon P^*\to L^*$ there is a unique~$g\colon L\to P$ with~$g^*=f$.

The map~$\ev_P$ is an isomorphism if
\begin{enumerate}
	\item either $P$ is a projective~$\Z[\pi]$-module,
	\item or $\pi$ is finite and $P$ is stably isomorphic to~$\Z$ or to the augmentation ideal $I$.
\end{enumerate}
\end{lemma}
\begin{proof}
	The map $\ev_P\colon P\to P^{**}$ is an isomorphism by assumption and we define~$g:=\ev_P^{-1}\circ f^*\circ \ev_{L}$. 
	Since~$\ev_P$ is an isomorphism, we have~$(\ev_P^*)^{-1}=\ev_{P^*}$. 
	Thus taking the dual of $g$,  we obtain
	\[g^{*}=\ev_L^*\circ f^{**}\circ (\ev_P^*)^{-1}=\ev_L^*\circ f^{**}\circ \ev_{P^*}=\ev_L^*\circ \ev_{L^*}\circ f=f.\]
	
	It remains to show that~$g$ is uniquely determined by~$g^*=f$. 
	Assume that~$h^*=f$ for some map~$h\colon L\to P$. Then
$\ev_P\circ h=h^{**}\circ \ev_L=f^*\circ \ev_L=g^{**}\circ\ev_L=\ev_P\circ g.$
	Since~$\ev_P$ is an isomorphism,~$h=g$ and thus~$g$ is uniquely determined by~$g^*=f$ as claimed.

It is clear that $\ev_P$ is an isomorphism if $P$ is projective and we now show that the same holds if~$\pi$ is finite and $P$ is stably isomorphic to $\Z$ or to the augmentation ideal $I$. 
Since $\ev_{A\oplus B}=\ev_A\oplus\ev_B$ is suffices to show this for $P=I$ and $P=\Z$. 

We begin with $P=\Z$. 
The dual of $\Z$ is isomorphic to $\Z$ and is generated by the map that sends $1$ to the norm element $N\in\Z[\pi]$. 
Hence $\Z^{**}\cong \Z$ is generated by the map that sends $1\mapsto N$ to $N$. This is the image of $1$ under $\ev_\Z$ and hence $\ev_\Z$ is an isomorphism as claimed. 
	
	We now consider $P=I$. 
	Dualising the exact sequence $0\to I\to \Z[\pi]\to \Z\to0$, we obtain the exact sequence $0\to \Z^*\to \Z[\pi]^*\to I^*\to 0$ because $\Ext^1_{\Z[\pi]}(\Z;\Z[\pi])=H^1(\pi;\Z[\pi])=0$ for every finite group $\pi$. Hence also $0\to I^{**}\to \Z[\pi]^{**}\to \Z^{**}\to0$. Since $\ev_\Z$ and $\ev_{\Z[\pi]}$ are isomorphisms, so is~$\ev_I$ by naturality of~$\ev$ and the five lemma.
\end{proof}

\begin{lemma}
\label{lem:nice-lemma}
Let $M$ and $N$ be left $\Z[\pi]$-modules and let~$f\colon M\to N$ be a left~$\Z[\pi]$-linear map. 
If $f^* \colon N^* \to M^*$ is an isomorphism, then
$$\Hom_{\Z[\pi]}(f,A^*) \colon  \Hom_{\operatorname{left-}\Z[\pi]}(N,A^*) \to \Hom_{\operatorname{left-}\Z[\pi]}(M,A^*)$$
is an isomorphism of abelian groups for every right $\Z[\pi]$-module $A$.
 
If~$A$ is a $\Z[\pi]$-bimodule,  then $\Hom_{\Z[\pi]}(f,A^*)$ is an isomorphism of right~$\Z[\pi]$-modules.
\end{lemma}
\begin{proof}
    Using the tensor-hom adjunction, we obtain the commutative diagram of abelian groups
    \[\begin{tikzcd}
\Hom_{\operatorname{left-}\Z[\pi]}(N,A^*)\ar[r,"\cong"]\ar[d,"{\Hom_{\Z[\pi]}(f,A^*)}"]&\Hom_{\operatorname{bi-}\Z[\pi]}(N\otimes_\Z A,\Z[\pi])\ar[r,"\cong"]\ar[d,"{(f\otimes \id_A)^*}"]&\Hom_{\operatorname{right-}\Z[\pi]}(A,N^*)\ar[d,"{\Hom_{\Z[\pi]}(A,f^*)}","\cong"']\\
\Hom_{\operatorname{left-}\Z[\pi]}(M,A^*)\ar[r,"\cong"]&\Hom_{\operatorname{bi-}\Z[\pi]}(M\otimes_\Z A,\Z[\pi])\ar[r,"\cong"]&\Hom_{\operatorname{right-}\Z[\pi]}(A,M^*).
\end{tikzcd}\]
If follows that $\Hom_{\Z[\pi]}(f,A^*)$ is an isomorphism as claimed.
\end{proof}

We apply these lemmas to the situation at hand.

\begin{construction}[{The  map~$f \colon H_1(Y;\Z[\pi])\to H_2(B;\Z[\pi])$}]
\label{cons:Mapf}
Fix~$\phi \colon H^2(Y;\Z[\pi]) \to H_1(Y;\Z[\pi])^*$ and consider the evaluation~$\ev_B \colon H^2(B;\Z[\pi]) \to H_2(B;\Z[\pi])^*$.
Assume that 
$$\ev_B^* \colon H^2(B;\Z[\pi])^{**} \to H_2(B;\Z[\pi])^*$$ is an isomorphism.
Lemma~\ref{lem:nice-lemma} applied with $f=\ev_B$ and $A=H_1(Y;\Z[\pi])$ ensures that
$$\Hom(\ev_B,H_1(Y;\Z[\pi])^*) \colon \Hom(H_2(B;\Z[\pi])^*,H_1(Y;\Z[\pi])^*) \to \Hom(H^2(B;\Z[\pi]),H_1(Y;\Z[\pi])^*)~$$
is an isomorphism.
Since $\phi \circ c_1|^*_Y$ belongs to the target of this map, we deduce that there exists a map~$g \colon H_2(B;\Z[\pi]) \to H_1(Y;\Z[\pi])^*$ making the following diagram commute:
\[\begin{tikzcd}
	H^2(B;\Z[\pi])\ar[d,"\ev_B"]\ar[r,"c_1|_Y^*"]&H^2(Y;\Z[\pi])\ar[d,"\phi"]\\
	H_2(B;\Z[\pi])^*\ar[r,"g"]&H_1(Y;\Z[\pi])^*.
\end{tikzcd}\]
If we additionally assume that~$H_2(B;\Z[\pi])$ is $\Z[\pi]$-projective or that $\pi$ is finite and that~$H_2(B;\Z[\pi])$ is stably the augmentation ideal,  then Lemma~\ref{lem:DualExist} gives a unique 
$$f \colon H_1(Y;\Z[\pi])\to H_2(B;\Z[\pi])$$
whose dual satisfies~$f^*=g.$
\end{construction}

We describe a condition that is sufficient for~$\phi$ to be realised by a homotopy~$\mathbf{H}$.

\begin{proposition}
\label{prop:CanRealise}
Assume the map~$\ev^* \colon H^2(B;\Z[\pi])^{**} \to H_2(B;\Z[\pi])^*$ is an isomorphism and
\begin{itemize}
\item either the module~$H_2(B;\Z[\pi])$ is projective,
\item or $\pi$ is finite and $H_2(B;\Z[\pi]) \cong I \oplus \Z[\pi]^{\oplus k}$ for some $k \geq 0$.
\end{itemize}
If~$\mathbf{H} \colon Y \times S^1 \to B$ extends~$c_1|_Y \colon Y \to B$ and satisfies
$$ 
\xymatrix@C1.5cm{
H_2(Y;\Z[\pi]) \oplus H_1(Y;\Z[\pi]) \ar[r]^-{\bsm i_* & \left(- \times [S^1]\right) \esm} \ar[rd]_{((c_1|_Y)_*,f)}& H_2(Y \times S^1;\Z[\pi])  \ar[d]^{\mathbf{H}_*}  \\
& H_2(B;\Z[\pi]),
}
$$
where $f$ is the map from Construction~\ref{cons:Mapf}, then~$\mathbf{H}$ realises~$\phi$, meaning that
$$ 
\xymatrix{
H^2(B;\Z[\pi]) \ar[r]^-{c_1|_Y^*} \ar[rd]_{\phi_{\mathbf{H}}'}& H^2(Y;\Z[\pi]) \ar[d]^{\phi}  \\
& H_1(Y;\Z[\pi])^*.
}
$$
\end{proposition}
\begin{proof}
As explained in Construction~\ref{cons:Mapf}, the assumptions of the proposition ensure that~$f$ is defined.
During this proof, given two $\Z[\pi]$-modules~$A$ and~$B$,  we identify~$\Hom(A \oplus B,\Z[\pi])$ with the direct sum~$\Hom(A,\Z[\pi])\oplus \Hom(B,\Z[\pi]).$
Consider the following diagram in which the top and bottom squares are readily seen to commute:
\begin{equation}
\label{eq:KunnethCommute}
\xymatrix{
H^1(Y;\Z[\pi]) \ar[r]^-{\ev_Y}& \overline{\Hom(H_1(Y;\Z[\pi]),\Z[\pi])} \\
H^2(Y;\Z[\pi]) \oplus H^1(Y;\Z[\pi]) \ar[r]^-{\ev}\ar[u]^{\proj_2}& \overline{\Hom(H_2(Y;\Z[\pi]) \oplus H_1(Y;\Z[\pi]),\Z[\pi])}\ar[u]^{\proj_2^*} \\
H^2(Y \times S^1;\Z[\pi]) \ar[r]^-{\ev_{Y \times S^1}}\ar[u]_{\operatorname{Kun}^{-1},\cong}& \overline{\Hom(H_2(Y \times S^1;\Z[\pi]),\Z[\pi])}  \ar[u]^{\bsm i_* & \left(- \times [S^1] \right)\esm^*}\\
H^2(B;\Z[\pi]) \ar[r]^-{\ev_B,\cong}\ar[u]^{\mathbf{H}^*}& \overline{\Hom(H_2(B;\Z[\pi]),\Z[\pi])}. \ar[u]^{(\mathbf{H}_*)^*} \ar@/_8pc/[uuu]_{f^*}
}
\end{equation}
We claim that the middle square commutes, i.e. that~$\ev=(i_*,(- \times [S^1] ))^* \circ \ev_{Y \times S^1} \circ \operatorname{Kun}.$
Given~$\varphi \in H^2(Y;\Z[\pi]),\psi \in H^1(Y;\Z[\pi])$ and~$a \in H_2(Y;\Z[\pi]),b \in H_1(Y;\Z[\pi])$, we apply successively, the definition of~$\operatorname{Kun}$, the linearity (and then naturality) of the evaluation,  and the definition of the cross product to obtain
\begin{align*}
&\bsm i_* &(- \times [S^1] )\esm^* \circ \ev_{Y \times S^1} \circ \operatorname{Kun}(\varphi,\psi)(a,b) \\
&=\bsm i_* &(- \times [S^1] ) \esm^* \circ  \ev_{Y \times S^1}  \circ (p_Y^*(\varphi)+(p_Y^*(\psi) \cup p_{S^1}^*(1_{H^1(S^1)})))(a,b) \\
&=\langle p_Y^*(\varphi),i_*(a) \rangle +\langle p_Y^*(\psi) \cup p_{S^1}^*(1_{H^1(S^1)})), b\times [S^1]  \rangle 
=\langle \varphi,a \rangle +\langle \psi \times 1_{H^1(S^1)}, b\times [S^1]  \rangle  \\
&=\langle \varphi,a \rangle+\langle \psi,b \rangle
=\ev (\varphi,\psi)(a,b).
\end{align*}
For the penultimate equality we use Lemma~\ref{lem:capev} and Remark~\ref{rem:CrossLemma} to obtain
\begin{align*}
\langle \psi \times 1_{H^1(S^1)}, b\times [S^1] \rangle 
&=
\overline{\mult_0((\psi \times 1_{H^1(S^1)}) \cap (b \times [S^1]))}
=\overline{\mult_0((\psi \cap b) \times (1_{H^1(S^1)} \cap [S^1]))}
 \\
&=\overline{\mult_0(\psi \cap b)}=\langle \psi, b\rangle.
\end{align*}
We have now established the diagram above commutes.
Thus, as claimed,
\[
\phi_{\mathbf{H}}'
:=\ev_Y \circ \proj_2 \circ \Kun^{-1} \circ \mathbf{H}^*
=f^*\circ \ev_B
=g\circ \ev_B
=\phi \circ c_1|_Y^*,\]
where the third and fourth equality follow from the definitions of~$f$ and~$g$ in \cref{cons:Mapf}.
\end{proof}

Recall that we write the homology K\"unneth isomorphism as
\begin{align*}
\Kun \colon H_2(Y;\Z[\pi]) \oplus H_1(Y;\Z[\pi])  &\xrightarrow{\cong} H_2(Y \times S^1;\Z[\pi]) \\
(a,b) &\mapsto i_*(a)+b\times 1_{H_1(S^1)}.
\end{align*}
We can now carry out Plan~\ref{plan:ProgrammeFinalStep}.

\begin{proposition}
\label{prop:FinalStep?}
Assume the map~$\ev^* \colon H^2(B;\Z[\pi])^{**} \to H_2(B;\Z[\pi])^*$ is an isomorphism and
\begin{itemize}
\item either the module~$H_2(B;\Z[\pi])$ is projective,
\item or $\pi$ is finite and $H_2(B;\Z[\pi]) \cong I \oplus \Z[\pi]^{\oplus k}$ for some $k \geq 0$.
\end{itemize}
For every homomorphism~$\phi \colon H^2(Y;\Z[\pi]) \to H^1(Y;\Z[\pi])$,  there is a homotopy~$\mathbf{H} \colon Y \times [0,1] \to B$ extending~$c_1|_Y \sqcup c_1|_Y \colon Y \times \{ 0,1 \} \to B$, that is trivial on the~$S^1$-factor,   and such that~$\phi_{\mathbf{H}}'=\phi \circ c_1|_Y^*.$
In particular,
$$\mathbf{H} \in \operatorname{Homotopy}(c_1)(c_1^*b_{X_1}^\partial).$$
\end{proposition}
\begin{proof}
The first two assumptions ensure that there is a map~$f \colon H_1(Y;\Z[\pi]) \to H_2(B;\Z[\pi])$ associated to~$\phi$ as in Construction~\ref{cons:Mapf}.
Consider the following diagram
$$
\xymatrix{
H_2(Y;\Z[\pi]) \ar@{^{(}->}[d]^-{i} \ar[rd]^{\bsm \id \\ 0\esm} \\
H_2(Y \times S^1;\Z[\pi]) \ar[r]^-{\Kun^{-1}}  \ar@{->>}[d]^{j_*}& H_2(Y;\Z[\pi]) \oplus H_1(Y;\Z[\pi]) \ar[r]^-{\bsm 0 & f \esm} &H_2(B;\Z[\pi]). \\
H_2(Y \times S^1,Y;\Z[\pi])  \ar@{-->}[rru]_-{\psi}
}
$$
The fact that $i$ is injective and $j$ is surjective follows because the inclusion $i \colon Y \to Y \times S^1$ is split by the projection $\pr \colon Y \times S^1 \to Y$.
Thus, since $(0 \ f) \circ \Kun^{-1}$ vanishes on $H_2(Y;\Z[\pi])$, there is a map~$\psi \colon H_2(Y \times S^1,Y;\Z[\pi]) \to H_2(B;\Z[\pi])$ satisfying $\psi \circ j_*=(0 \ f) \circ \Kun^{-1}$.
Apply Theorem~\ref{thm:ObstructionTheoryTopTriangle} to $f_0:=c \circ \pr \colon Y \times S^1 \to B $ in order to obtain a map $\mathbf{H} \colon Y \times S^1 \to B$ that satisfies~$H|_Y=c\circ \pr|_Y=c|_Y$ and $\mathbf{H}_*=(c \circ \pr)_* \colon \pi_1(Y \times S^1) \to \pi_1(B)$ as well as $\psi \circ j_*=(c \circ \pr)_*-\mathbf{H}_*$.
It follows that 
\[(0,-f)  \circ \Kun^{-1}=\psi \circ j_*=(c \circ \pr)_* - \mathbf{H}_*=((c_1)_*,0)  \circ \Kun^{-1}- \mathbf{H}_*.\]
We thus obtain~$\mathbf{H}_*=((c_1)_*,f)  \circ \Kun^{-1}$.
Proposition~\ref{prop:CanRealise} then ensures that~$\phi_{\mathbf{H}}'=\phi \circ c_1|_Y^*$,  as required.
Since $\mathbf{H}_*=(c \circ \pr)_* \colon \pi_1(Y \times S^1) \to \pi_1(B)$,  we see that~$\mathbf{H}$ is trivial on the~$S^1$-factor.
Thus $\mathbf{H} \in \operatorname{Homotopy}(c_1)$ and, since~$\phi_{\mathbf{H}}'=\phi \circ c_1|_Y^*$,  Proposition~\ref{prop:H2H1} gives~$\mathbf{H} \in \operatorname{Homotopy}(c_1)(c_1^*b_{X_1}^\partial).$
\end{proof}

We can now prove the main technical result of this section.

\begin{customthm}{\ref{thm:GoalActionpdf}}
For any homotopy~$\mathbf{H} \in \operatorname{Homotopy}(c_1)(c_1^*b_{X_1}^\partial)$, there is a corresponding sesquilinear form~$q_{\mathbf{H}} \in \Sesq(H^2(Y;\Z[\pi]))$ such that the homotopy equivalence $\Xi(\mathbf{H}) \colon B \to B$ acts on the set $\Herm(H^2(Y;\Z[\pi])$ of  hermitian forms as
$$\Xi(\mathbf{H})\cdot \lambda=\lambda+q_{\mathbf{H}}+q_{\mathbf{H}}^* \quad \quad \text{ for any } \lambda \in \Herm(H^2(Y;\Z[\pi])).$$
Conversely, 
if the map~$\ev^* \colon H^2(B;\Z[\pi])^{**} \to H_2(B;\Z[\pi])^*$ is an isomorphism and
\begin{itemize}
\item either the module~$H_2(B;\Z[\pi])$ is projective,
\item or $\pi$ is finite and $H_2(B;\Z[\pi]) \cong I \oplus \Z[\pi]^{\oplus k}$ for some $k \geq 0$,
\end{itemize}
then for any sesquilinear form~$q \in \Sesq(H^2(Y;\Z[\pi]))$, there is a homotopy~$\mathbf{H} \in \operatorname{Homotopy}(c_1)(c_1^*b_{X_1}^\partial)$ with~$q=q_{\mathbf{H}}.$
\end{customthm}
\begin{proof}
The sesquilinear form $q_{\mathbf{H}}$ is constructed in Construction~\ref{cons:qH}, whereas the description of the action of~$\Xi(\mathbf{H})$ can be found in Proposition~\ref{prop:HowItActs}.
The realisation of a sesquilinear form~$q$ by a homotopy~$\mathbf{H}$ follows by combining Plan~\ref{plan:ProgrammeFinalStep} with Proposition~\ref{prop:FinalStep?}.
\end{proof}

We now prove Theorem~\ref{thm:VanishEvenIntro} whose statement we recall for the reader's convenience.

\begin{customthm}{\ref{thm:VanishEvenIntro}}
\label{thm:VanishEven}
Assume that $\cd(\pi) \leq 3$.
Let~$(X_0,Y_0)$ and~$(X_1,Y_1)$ be~$4$-dimensional Poincar\'e pairs with fundamental group~$\pi_1(X_i) \cong \pi$,  Postnikov $2$-type $B:=P_2(X_1)$ and~$\pi_1(Y_i) \to \pi_1(X_i)$ surjective for $i=0,1$,  and let~$h \colon Y_0 \to Y_1$ be a degree one homotopy equivalence.
Assume that~$\ev^*$ is an isomorphism as in Definition~\ref{def:HypothesesName}.

If there are~$3$-connected maps~$c_0 \colon X_0 \to B$ and $c_1 \colon X_1 \to B$ such that~$c_1|_{Y_1} \circ h = c_0|_{Y_0}$ as well as~$c_0^*b_{X_0}^\partial=c_1^*b_{X_1}^\partial$ and $b(c_0,c_1)$ is even hermitian,  then
$$b(c_0,c_1)=0 \in \Herm(H^2(Y_1;\Z[\pi]))/\mathcal{G}.$$
In fact, there is a $c_0' \simeq c_0$ with $c_1|_{Y_1} \circ h=c_0'|_{Y_0}$ such that $b(c_0',c_1)=0  \in \Herm(H^2(Y_1;\Z[\pi]))$.
\end{customthm}
\begin{proof}
Since the $c_i$ are $3$-connected, $\ev^*$ is an isomorphism if and only if $\ev_B^*$ is an isomorphism.
Recall also from Proposition~\ref{prop:StablyFreePD<4} that $\cd(\pi) \leq 3$ if and only if $\pi_2(X_i)$ is projective for $i=0,1$, which is in turn equivalent to~$H_2(B;\Z[\pi])$ being projective.
The assumptions of Theorem~\ref{thm:GoalActionpdf} are therefore satisfied.

Assume that~$b(c_0,c_1)=q+q^*$ for some sesquilinear form~$q$.
Theorem~\ref{thm:GoalActionpdf} produces a homotopy~$\mathbf{H}=(H_t \colon Y_1 \to B)_{t \in [0,1]}$ from $c_1|_{Y_1}$ to itself with $q_{\mathbf{H}}=-q$.
Consider also the homotopy equivalence~$\Xi(\mathbf{H}) \colon B \to B$ rel.~$Y_1$ obtained by applying the homotopy extension theorem (recall Construction~\ref{cons:Xi}).
Write~$g \colon B \to B$ for the homotopy inverse of $\Xi(\mathbf{H})$ and consider~$c_0':=g \circ c_0$.
Since~$\Xi(\mathbf{H}) \in \mathcal{G}$,  Propositions~\ref{prop:Indepc0} and~\ref{prop:HowItActs} show that
$$b(c_0',c_1)
=b(c_0,c_1)-b(g)
=b(c_0,c_1)+b(\Xi(\mathbf{H}))
=\Xi(\mathbf{H}) \cdot b(c_0,c_1)
=b(c_0,c_1)-(q+q^*)=0.$$
Since~$b(c_0,c_1)=b(c_0',c_1)+b(g)=g \cdot b(c_0',c_1)$, we deduce that $b(c_0,c_1)=0 \in \Herm(H^2(Y_1;\Z[\pi]))/\mathcal{G}$.
For the final sentence of the theorem, recall from Construction~\ref{cons:Xi} that $\Xi(\mathbf{H}) = \widetilde{H}_1$ is rel. $Y_1$ and is homotopic (though not rel. $Y_1$) to $\id_B$, whence $c_1|_{Y_1} \circ h=c_0'|_{Y_0}$ and~$c_0'\simeq c_0$.
This concludes the proof of the theorem.
\end{proof}

\section{The secondary obstruction for weakly even forms}
\label{sec:WeaklyEven}

We prove Theorems~\ref{thm:VanishSpinIntro} and~\ref{thm:obstruction0ImpliesEvenIntro}.
The first result follows promptly from Theorem~\ref{thm:VanishEven}, whereas the second requires some preliminary lemmas.

\begin{notation}
\label{not:IntroAgain}
Let~$(X_0,Y_0)$ and~$(X_1,Y_1)$ be~$4$-dimensional Poincar\'e pairs with fundamental group~$\pi_1(X_j) \cong \pi$,  Postnikov $2$-type $B:=P_2(X_1)$ and~$\iota_j \colon \pi_1(Y_j) \to \pi_1(X_j)$ surjective for~$j=0,1$.
Fix a degree one homotopy equivalence~$h \colon Y_0 \to Y_1$.
\end{notation}

The first main result of this section follows promptly from Theorem~\ref{thm:VanishEven}.

\begin{customthm}{\ref{thm:VanishSpinIntro}}
Fix $\pi$ with $\cd(\pi)\leq 3$.
Assume~$\ev^*$ is an isomorphism as in Definition~\ref{def:HypothesesName},
the equivariant intersection form~$\lambda_{X_0 \cup_h X_1}$ is weakly even,  and~$H_1(Y;\Z[\pi])\cong L \oplus T$ with~$L$ free and $T^*=0$.
If there are~$3$-connected maps~$c_0 \colon X_0 \to B,c_1 \colon X_1 \to B$ with~$c_1|_{Y_1} \circ h = c_0|_{Y_0}$ and~$c_0^*b_{X_0}^\partial=c_1^*b_{X_1}^\partial$,  then
$$b(c_0,c_1)=0 \in \Herm(H^2(Y_1;\Z[\pi]))/\mathcal{G}.$$
In fact, there is a $c_0' \simeq c_0$ with $c_1|_{Y_1} \circ h=c_0'|_{Y_0}$ such that $b(c_0',c_1)=0  \in \Herm(H^2(Y_1;\Z[\pi]))$.
\end{customthm}
\begin{proof}
The obstruction $b(c_0,c_1)$ arises by pulling back the equivariant intersection form of the~$4$-dimensional Poincar\'e complex~$X_0 \cup_h X_1$, and modding out by a subset of the resulting radical.
Thus if $\lambda_{X_0 \cup_h X_1}$ is weakly even hermitian,  then so is~$b(c_0,c_1)$.
Since~$H_1(Y;\Z[\pi])\cong L \oplus T$ with~$L$ free and~$T^*=0$, Lemma~\ref{lem:WeaklyEvenEven} ensures that~$b(c_0,c_1)$ is even.
Since~$b(c_0,c_1)$ is even,
Theorem~\ref{thm:VanishEven} ensures both that~$b(c_0,c_1)=0 \in \Herm(H^2(Y_1;\Z[\pi]))/\mathcal{G}$ and that there is a $c_0' \simeq c_0$ with~$c_1|_{Y_1} \circ h=c_0'|_{Y_0}$ such that~$b(c_0',c_1)=0  \in \Herm(H^2(Y_1;\Z[\pi]))$.
\end{proof}

The remainder of this section is devoted to the proof of Theorem~\ref{thm:obstruction0ImpliesEvenIntro}


\begin{notation}
Fix~$3$-connected maps~$c_0 \colon X_0 \to B,c_1 \colon X_1 \to B$ such that~$c_0|_{Y_0} =  c_1|_{Y_1} \circ h$ with~$c_0^*b_{X_0}^\partial =c_1^*b_{X_1}^\partial$ so that $b(c_0,c_1)$ is defined.
Elsewhere in this article (e.g.  in Section~\ref{sec:PrimarySecondary}) we have been writing~$c:=c_0 \cup c_1 \colon X_0 \cup_h X_1 \to B$.
We continue writing~$c$ for this map but, in this section alone, we write~$c_0 \cup c_1 \colon X_0 \cup_h X_1 \to  B \cup_{\id_{Y_1}} B$
so that~$c \colon X_0 \cup_h X_1 \to B$ is obtained from~$c_0 \cup c_1$ as $c=r \circ (c_0 \cup c_1)$ where $r \colon B \cup_{\id_{Y_1}} B \to B$ denotes the fold map.
Finally, we set~$B \cup_{Y_1} B:=B \cup_{\id_{Y_1}} B$ and $b_U:=b_{X_0 \cup_h X_1}$. 
\end{notation}
	
\begin{lemma}
\label{lem:EvenIff}
The following assertions hold:
\begin{itemize}
\item The form~$b_U$ is weakly even if and only if the form~$(c_0 \cup c_1)^*b_U$ is weakly even.
\item The form~$b(c_0,c_1)$ is weakly even if and only if the form~$r^*(c_0 \cup c_1)^*b_U$ is weakly even.
\end{itemize}
\end{lemma}
\begin{proof}
The first assertion is clear because~$c_0 \cup c_1$ induces an isomorphism on second (co)homology.
We therefore focus on the second.
The form~$b(c_0,c_1)$ is isometric (via~$(c_1|_{Y_1}^*)^{-1}$) to a form obtained from~$c^*b_U$ by modding out by a subset of its radical.
It follows that~$b(c_0,c_1)$ is weakly even if and only if~$c^*b_U=r^*(c_0 \cup c_1)^*b_U$ is weakly even. 
\end{proof}

\begin{lemma}
\label{lem:EvenIffTwoThings}
Consider the inclusion~$i_1 \colon B\to B \cup_{Y_1} B$ of the first factor and the composition
$$ \eta \colon H^2(B,Y_1;\Z[\pi]) \xleftarrow{\cong,\exc} H^2(B \cup_{Y_1} B,B) \xrightarrow{j_B^*} H^2(B \cup_{Y_1} B;\Z[\pi]).$$
The form~$(c_0 \cup c_1)^*b_U$ is weakly even if and only if~$r^*(c_0 \cup c_1)^*b_U$ and~$\eta(c_0 \cup c_1)^*b_U$ are weakly~even.
\end{lemma}
\begin{proof}
It suffices to prove that every~$x \in H^2(B \cup_{Y_1} B;\Z[\pi])$ can be written as~$r^*(x_1)+\eta(x_2)$ for~$x_1 \in H^2(B;\Z[\pi])$ and~$x_2 \in H^2(B,Y_1;\Z[\pi])$.
First, consider the long exact sequence of the pair~$(B \cup_{Y_1} B,B)$.
The definition of~$\eta$ together with the fact that the fold map~$r \colon B \cup_{Y_1} B \to B$ splits the inclusion~$i_2 \colon B \to B \cup_{Y_1} B$ then leads to the following commutative diagram in which the bottom row (and therefore the top row) is exact:
$$
\xymatrix{
 0 \ar[r]& H^2(B,Y_1;\Z[\pi]) \ar[r]^{\eta}\ar[d]^{\exc}_\cong& H^2(B \cup_{Y_1} B;\Z[\pi]) \ar[r]^-{i_2^*}\ar[d]^=& H^2(B;\Z[\pi]) \ar[r]\ar[d]^=& 0 \\
  0 \ar[r]& H^2(B\cup_{Y_1} B,B;\Z[\pi]) \ar[r]^{j_B^*}& H^2(B \cup_{Y_1} B;\Z[\pi]) \ar[r]^-{i_2^*}& H^2(B;\Z[\pi]) \ar[r]& 0.
 }
 $$
Since~$i_2^*r^*=\id_{H^2(B;\Z[\pi])}$, we have~$x-r^*i_2^*(x)\in \ker(i_2^*)=\im(\eta)$.
Setting~$x_1:=i_2^*(x) \in H^2(B;\Z[\pi])$, it follows that~$x-r^*(x_1)=\eta(x_2)$ for some~$x_2 \in H^2(B,Y_1;\Z[\pi])$,  therefore proving the lemma.
\end{proof}

To prove the next lemma,  it is convenient to consider the equivariant intersection form of $X_i$ on cohomology instead of homology.
Since the notation $b_{X_i}$ is already in use, we instead write 
\begin{align*}
\PD_{X_i}^*\lambda_{X_i} \colon H^2(X_i,Y_i) \times H^2(X_i,Y_i) &\to \Z[\pi], \\
(x,y)&\mapsto\lambda_{X_i}(\PD_{X_i}(x),\PD_{X_i}(y)).
\end{align*}
Using $j_i \colon H_2(X_i) \to H_2(X_i,Y_i)$ to denote the inclusion induced map, observe that
$$ \PD_{X_i}^*\lambda_{X_i}(x,y)
=\langle \PD_{X_i}^{-1} \circ j_i \circ  \PD_{X_i}(y), \PD_{X_i}(x)\rangle
=\langle j_i^*(y), \PD_{X_i}(x)\rangle.
$$

\begin{lemma}
\label{lem:FormEvenImpliesThing}
The following equality holds:
$$\eta (c_0 \cup c_1)^*b_U=c_0^* \PD_{X_0}^*\lambda_{X_0}.$$
In particular, if~$\lambda_{X_0}$ is weakly even, then~$\eta (c_0 \cup c_1)^*b_U$ is weakly even.
\end{lemma}
\begin{proof}
For the last 
sentence, note that if~$\lambda_{X_0}$ is weakly even, then so is~$c_0^* \PD_{X_0}^*\lambda_{X_0}$.
We therefore focus on proving that~$\eta (c_0 \cup c_1)^*b_U=c_0^* \PD_{X_0}^*\lambda_{X_0}.$
Recall that for~$x,y \in H^2(X_0;\Z[\pi])$, we have~$ \PD_{X_0}^*\lambda_{X_0}(x,y)=\langle j_0^*(y),\PD_{X_0}(x)\rangle$.

For $a,b\in H^2(B,Y_1;\Z[\pi])$, it therefore follows that
\begin{equation}
\label{eq:c1*bX1}
c_0^*\PD_{X_0}^*\lambda_{X_0}=\langle j_0^*c_0^*(b),\PD_{X_0}(c_0^*(a))\rangle.
\end{equation}
The remainder of the proof consists of calculating $\eta (c_0 \cup c_1)^*b_U(a,b)$ and verifying that we obtain the same outcome.
First, we note that
\begin{align}
\label{eq:pullbackbUFirstStep}
 \eta(c_0 \cup c_1)^*b_U(a,b)
&=b_U((c_0 \cup c_1)^* \circ \eta(b),(c_0 \cup c_1)^* \circ \eta(a))  \nonumber \\
&=\langle (c_0 \cup c_1)^*\circ \eta(b),\PD_{X_0 \cup_h X_1}\circ(c_0 \cup c_1)^*\circ \eta(a)\rangle .
\end{align}
The main intermediate step consists of the following claim.
\begin{claim}
Writing~$i^{X_0 \to X_0 \cup_h X_1} \colon X_0 \to X_0 \cup_h X_1$ for the inclusion,  the following equality holds:
$$\PD_{X_0 \cup_h X_1}  \circ (c_0 \cup c_1)^* \circ \eta=i_*^{X_0 \to X_0 \cup_h X_1} \circ \PD_{X_0} \circ c_0^* \colon H^2(B,Y_1;\Z[\pi]) \to H_2(X_0 \cup_h X_1,\Z[\pi]). $$
\end{claim} 
\begin{proof}
The equality in the statement is equivalent to the outermost routes of the following diagram (in which $\Z[\pi]$-coefficient are understood) agreeing:
$$
\xymatrix@R0.6cm{
H^2(B,Y_1)\ar[r]^{c_0^*}_\cong \ar@/_5pc/[dd]_{\eta}&H^2(X_0,Y_0)\ar[r]^{\PD_{X_0}}&H_2(X_0)\ar[rd]^{i^{X_0 \to X_0 \cup_h X_1}_*}&\\
H^2(B\cup_{Y_1}B,B)\ar[r]^{(c_0 \cup c_1)^*}_\cong\ar[u]^{\exc}_\cong \ar[d]_{j_B^*}&H^2(X_0 \cup_h X_1,X_1)\ar[r]^{\incl^*}\ar[u]^{\exc}_\cong&H^2(X_0 \cup_h X_1)\ar[r]^{\PD_{X_0 \cup_h X_1}}&H_2(X_0 \cup_h X_1).\\
H^2(B\cup_{Y_1}B)\ar[urr]_{(c_0 \cup c_1)^*}
}
$$
We verify that this diagram does indeed commute.
The square and triangle clearly commute,  and~$\eta =j_B^* \circ \exc$ by definition.
To see that the pentagon commutes,  rewrite it as 
$$
\xymatrix@C2cm{
H^2(X_0,Y_0;\Z[\pi])\ar[r]^-{-\cap [X_0]}&H_2(X_0;\Z[\pi])\ar[d]^{i^{X_0 \to X_0 \cup_h X_1}_*}  \\
H^2(X_0 \cup_h X_1,X_1;\Z[\pi])\ar[u]_{\exc,\cong}\ar[r]^-{-\cap \incl_*([X_0])}\ar[d]^{\incl^*}& H_2(X_0 \cup_h X_1;\Z[\pi])  \\  
H^2(X_0 \cup_h X_1;\Z[\pi])\ar[r]^{-\cap[X_0 \cup_h X_1]} & H_2(X_0 \cup_h X_1;\Z[\pi])\ar[u]_=
}
$$
and use the naturality of relative cap product as e.g.  in~\cite[18.1.1 (1)]{TomDieck}.
More precisely,  in the notation from this reference,  for the top square we use naturality with respect to the map~$(X_0,Y_0,\emptyset) \to (X_0 \cup_h X_1,X_0,\emptyset)$ (with $u=[X_0] \in H_4(X_0,Y_0)$) whereas for the bottom square it is with respect to~$(X_0 \cup_h X_1,\emptyset,\emptyset) \to (X_0 \cup_h X_1,X_0,\emptyset)$ (with~$u=[X_0 \cup_h X_1] \in H_4(X_0 \cup_h X_1)$).
This concludes the proof of the claim.
\end{proof}
Using the claim,  we can rewrite~\eqref{eq:pullbackbUFirstStep} as
$$ \eta(c_0 \cup c_1)^*b_U(a,b)
=\langle (c_0 \cup c_1)^*\circ\eta(b),i_*^{X_0 \to X_0 \cup_h X_1} \circ \PD_{X_0} \circ c_0^*(a)\rangle.
$$
Using the naturality of evaluation and the equality~$(c_0 \cup c_1)\circ i^{X_0 \to X_0 \cup_h X_1}=i^{B \to B \cup_{Y_1} B}\circ c_0$,  this equality can in turn be recast as
$$ \eta(c_0 \cup c_1)^*b_U(a,b)
=\langle c_0^*\circ i_{B \to B \cup_{Y_1} B}^*\circ \eta(b), \PD_{X_0} \circ c_0^*(a)\rangle.
$$
This expression is seen to equal~\eqref{eq:c1*bX1} by verifying that~$c_0^*\circ i_{B \to B \cup_{Y_1} B}^* \circ\eta=j_0^* \circ c_0^*$; indeed all the maps besides $c_0^*$ are inclusion induced.
\end{proof}

\begin{proposition}
\label{prop:WeaklyEvenIff}
Assume that the hermitian form~$\lambda_{X_0} \cong \lambda_{X_1}$ is weakly even and that the first obstruction vanishes: $c_0^*b_{X_0}^\partial=c_1^*b_{X_1}^\partial$.
The following assertions are equivalent:
\begin{enumerate}
\item the hermitian form~$b(c_0,c_1)$ is weakly even.
\item the hermitian form~$b_U$ is weakly even.
\end{enumerate}
\end{proposition}
\begin{proof}
If $\lambda_{X_0} \cong \lambda_{X_1}$ is weakly even, then Lemma~\ref{lem:FormEvenImpliesThing} ensures that~$\eta (c_0 \cup c_1)^*b_U$ is weakly even.
The combination of the first item of Lemma~\ref{lem:EvenIff} and Lemma~\ref{lem:EvenIffTwoThings} then shows that~$b_U$ is weakly even if and only if~$r^*(c_0 \cup c_1)^*b_U$ is weakly even.
The second item of Lemma~\ref{lem:EvenIff} shows that this is equivalent to~$b(c_0,c_1)$ being weakly even.
\end{proof}


Finally we prove the main result of this section, namely Theorem~\ref{thm:obstruction0ImpliesEvenIntro}.

\begin{customthm}{\ref{thm:obstruction0ImpliesEvenIntro}}
\label{thm:obstruction0ImpliesEven}
Assume that the hermitian form~$\lambda_{X_0} \cong \lambda_{X_1}$ is weakly even.
If the secondary obstruction~$b(c_0,c_1)$ vanishes in~$\Herm(H^2(Y_1;\Z[\pi]))/\mathcal{G}$, then~$b_U$ is weakly even.
\end{customthm}
\begin{proof}
We claim that if~$\lambda_{X_1}$ is weakly even, then so is~$b(\varphi)$ for every $\varphi \in \mathcal{G}$.
Recall that~$b(\varphi)$ is obtained by running the same definition as for~$b(c_0,c_1)$ but with~$(X_1,\varphi \circ c_1)$ in place of~$(X_0,c_0).~$
Repeat the proofs of Lemmas~\ref{lem:EvenIff},~\ref{lem:EvenIffTwoThings} and~\ref{lem:FormEvenImpliesThing} but replacing all instances of~$U:=X_0 \cup_h X_1$ by~$X_1 \cup_{Y_1} X_1:=X_1 \cup_{\id_{Y_1}} X_1$.
Since we assumed that~$\lambda_{X_1}$ is weakly even, the outcome is an analogue of Proposition~\ref{prop:WeaklyEvenIff}:~$b(\varphi)$ is weakly even if and only if~$b_{X_1 \cup_{Y_1} X_1}$ is weakly even.
In particular,~$b(\varphi)$ being weakly even is independent of~$\varphi$ and so,  the claim reduces to proving  that~$b(\id_B)$ is weakly even. 
In this case,  we can take~$\xi=0$ in Construction~\ref{cons:bphi},  from which it follows that~$b(\id_B)=0$ which is indeed weakly even.
This concludes the proof of the claim.

We conclude the proof of the theorem.
If the pairing~$b(c_0,c_1)$ vanishes in~$\Herm(H^2(Y_1;\Z[\pi]))/\mathcal{G}$, then~$b(c_0,c_1)=b(\varphi)$ for some~$\varphi \in \mathcal{G}$,  and the claim ensures then ensures that this pairing is weakly even.
Proposition~\ref{prop:WeaklyEvenIff} implies that~$b_U$ is weakly even.
\end{proof}

\section{The secondary obstruction for odd forms}
\label{sec:Odd}

The goal of this section is to prove Theorem~\ref{thm:VanishOddIntro}.
We recall the relevant notation, state the main technical result needed in the proof,  explain how Theorem~\ref{thm:VanishOddIntro} follows and then prove the technical step.

\begin{notation}
\label{not:Hat}
Recall from the end of Section~\ref{sub:MainTechnicalIntro} that given a 
$3$-dimensional Poincar\'e complex~$Y$ and an epimorphism $\pi_1(Y) \twoheadrightarrow \pi$, we set
$$I^1(Y;\Z[\pi]):=\im(H^1(Y;\Z[\pi])
 \xrightarrow{\ev}
   \Hom(H_1(Y;\Z[\pi]),\Z[\pi])
\to
  \Hom(H_1(Y;\Z[\pi]),\Z_2)),$$
where the second map is induced by augmentation modulo $2$.
Given $x \in \Z[\pi]$, write $(x)_1 \in \Z$ for the coefficient of the neutral element and consider the group homomorphism
\begin{align*}
\widehat{} \ \colon \Herm(H^2(Y;\Z[\pi])) &\to I^1(Y;\Z[\pi]) \\
b &\mapsto (x \mapsto b(\PD_{Y}^{-1}(x),\PD_{Y}^{-1}(x))_1 \mod 2).
\end{align*}
When $\pi$ has no nontrivial elements of order $2$, the kernel of this map consists of weakly even forms on $H^2(Y;\Z[\pi])$ (Lemma~\ref{lem:WeaklyEvenAlgebra}) or equivalently (by Lemma~\ref{lem:WeaklyEvenEven}),  if~$H_1(Y;\Z[\pi])\cong L \oplus T$ with $L$ free and $T^*=0$, of even forms.

Given a $4$-dimensional Poincar\'e pair~$(X,Y)$ with $\pi_1(Y) \twoheadrightarrow \pi_1(X)=:\pi$ surjective,  we associate to a homomorphism~ $\phi \colon H_1(Y;\Z[\pi]) \to H_2(X;\Z[\pi])$ the hermitian form
\begin{align*}
q_\phi  \colon H^2(Y;\Z[\pi]) \times H^2(Y;\Z[\pi]) &\to \Z[\pi] \\
(x,y) &\mapsto \lambda_{X}(\phi(\PD_Y(x)),\phi(\PD_Y(y))).
\end{align*}
\end{notation}

The main technical result of this section is the following.

\begin{theorem}
\label{thm:OddChange}
Let~$(X_0,Y_0)$ and~$(X_1,Y_1)$ be~$4$-dimensional Poincar\'e pairs with fundamental group~$\pi_1(X_i) \cong \pi$,  Postnikov $2$-type $B:=P_2(X_1)$ and~$\pi_1(Y_i) \to \pi_1(X_i)$ surjective for $i=0,1$,  and let~$h \colon Y_0 \to Y_1$ be a degree one homotopy equivalence.
Let~$c_0 \colon X_0 \to B,c_1 \colon X_1 \to B$ be
$3$-connected maps such that~$c_0|_{Y_0} =  c_1|_{Y_1} \circ h$ with~$c_0^*b_{X_0}^\partial =c_1^*b_{X_1}^\partial$.

Assume that one of the two following conditions hold:
\begin{itemize}
\item either the map~$H^2(\pi;\Z[\pi]) \to H^2(Y_1;\Z[\pi])$ is injective, and $H_1(Y_0;\Z[\pi]) \cong L \oplus T$, with $L$ a free $\Z[\pi]$-module and $T^*=0$.
\item or $\pi$ is finite and~$H_1(Y_0;\Z[\pi]) \cong L \oplus T \oplus \Z^k$, with~$L$ a free~$\Z[\pi]$-module and~$T^*=0$.
Here,~$\Z$ is endowed with the $\Z[\pi]$-module structure induced by the trivial $\pi$-action.
\end{itemize}
For every map $\phi \colon H_1(Y_0;\Z[\pi]) \to H_2(X_0;\Z[\pi])$, there is a $3$-connected map $c_0' \colon X \to B$ satisfying~$c_0'|_{Y_0} =  c_1|_{Y_1} \circ h$ as well as~$(c_0')^*b_{X_0}^\partial =c_1^*b_{X_1}^\partial$, and 
$$ \wh b(c_0',c_1)=\wh b(c_0,c_1)-\widehat{h^*q_\phi}.$$
\end{theorem}

We explain how this result implies Theorem~\ref{thm:VanishOddIntro}, whose  statement we now recall.

\begin{customthm}{\ref{thm:VanishOddIntro}}
\label{thm:VanishOdd}
Assume~$\cd(\pi) \leq 3$.
Let~$(X_1,Y_1)$ be a~$4$-dimensional Poincar\'e pair with~$\pi_1(X_1) \cong~\pi$,  Postnikov $2$-type $B:=P_2(X_1)$ and~$\pi_1(Y_1) \to \pi_1(X_1)$ surjective,  and let~$c_1 \colon X_1 \to B$ be a~$3$-connected map.
Assume that~$\ev^*$ is an isomorphism and that~$H^2(\pi;\Z[\pi]) \to H^2(Y_1;\Z[\pi])$ is injective.
If~$\lambda_{X_1}$ is odd and~$H_1(Y_1;\Z[\pi]) \cong L \oplus T$ with $L$ free and $T^*=0$,  then 
$$  \Herm(H^2(Y_1;\Z[\pi]))/\mathcal{G}=0.$$
In particular, given another~$4$-dimensional Poincar\'e pair with~$\pi_1(X_0) \cong \pi$,  Postnikov $2$-type $B$, and~$\pi_1(Y_0) \to \pi_1(X_0)$ surjective, and a $3$-connected map $c_0 \colon X_0 \to B$ with~$c_0|_{Y_0} =  c_1|_{Y_1} \circ h$ with~$c_0^*b_{X_0}^\partial =c_1^*b_{X_1}^\partial$, the secondary obstruction vanishes:
$$b(c_0,c_1)=0 \in \Herm(H^2(Y_1;\Z[\pi]))/\mathcal{G}.$$
\end{customthm}
\begin{proof}
By Proposition~\ref{prop:Herm/G}, it is equivalent to show that~$I^1(Y;\Z[\pi])/\mathcal{G}=0$, i.e. that every element~$f \in I^1(Y;\Z[\pi])$ can be written as $f=\widehat{b}(\varphi)$ for some $\varphi \in \mathcal{G}$.
Given $f \in I^1(Y;\Z[\pi])$,  define 
$$\phi \colon H_1(Y_0;\Z[\pi]) \to H_2(X_0;\Z[\pi])$$
 by picking a basis $x_1,\ldots,x_n$ for $L$,  setting $\phi(h_*^{-1}(x_i)):=0$ if $f(x_i)=0$
 and $\phi(h_*^{-1}(x_i)):=y$ if~$f(x_i)=1$.
This determines $\phi$: since $T^*=0$ and $H_2(X_0;\Z[\pi])$ is projective, $\phi$ necessarily vanishes on $T$.
By construction $\widehat{h^*q_\phi}=f.$
Theorem~\ref{thm:OddChange} (applied with $X_0=X_1$ and $c_0=c_1$)
gives a $3$-connected map $c_1' \colon X_1 \to B$ with~$c_1'|_{Y_1} =  c_1|_{Y_1} \circ h,(c_1')^*b_{X_1}^\partial =c_1^*b_{X_1}^\partial$, and 
$$ \wh b(c_1',c_1)=\wh b(c_1,c_1)-\widehat{h^*q_\phi}=\widehat{h^*q_\phi}=f.$$
Since $c_1' \in \mathcal{S}_0(c_1)$, Proposition~\ref{prop:Transitive}  (again with $X_0=X_1$) implies that there is a homotopy equivalence~$\varphi \in \mathcal{G}$ with~$c_1' \simeq \varphi \circ c_1$.
Proposition~\ref{prop:Indepc0} then implies that $b(c_1',c_1)=b(c_1,c_1)-b(\varphi)=-b(\varphi)$.
Thus $f=\widehat{b}(c_1',c_1)=\widehat{b}(\varphi)$ with $\varphi \in \mathcal{G}$, as required.
\end{proof}

The remainder of this section is devoted to the proof of Theorem~\ref{thm:OddChange}.

\subsection{Changing compatible triples}

The first step in the proof of Theorem~\ref{thm:OddChange} consists of showing how a given compatible triple can be modified using a homomorphism~$\phi\colon H_1(Y;\Z[\pi])\to H_2(X;\Z[\pi])$.
This is the content of Proposition~\ref{prop:ChangeTriple} below.

	\begin{notation}
	\label{not:psiphi}
Let~$(X,Y)$ be a~$4$-dimensional Poincar\'e pair with~$\pi_1(Y) \to \pi_1(X)=:\pi$ surjective.
	Given~$\phi\colon H_1(Y;\Z[\pi])\to H_2(X;\Z[\pi])$, define~$\psi\colon H_2(X,Y;\Z[\pi])\to H_2(X;\Z[\pi])$ as the composition
\begin{align*}
 \psi \colon H_2(X,Y;\Z[\pi])
 &\xleftarrow{\PD_X,\cong}H^2(X;\Z[\pi])\xrightarrow{\ev_X} H_2(X;\Z[\pi])^* \xrightarrow{\phi^*} H_1(Y;\Z[\pi])^*  \\ &\xleftarrow{\ev_Y, \cong} H^1(Y;\Z[\pi])/H^1(\pi;\Z[\pi]) \xrightarrow{\delta} H^2(X,Y;\Z[\pi])\\
 & \xrightarrow{\PD_X,\cong}H_2(X;\Z[\pi]).
	\end{align*}
	Here,~$\delta$ is the connecting homomorphism for the pair~$(X,Y)$,  and to obtain the penultimate map, we use that~$H^1(X;\Z[\pi])\cong H^1(\pi;\Z[\pi])$.
Finally we write~$i\colon Y\to X$ and~$j\colon (X,\emptyset)\to (X,Y)$ for the inclusions.
	\end{notation}
	
	The next lemma describes how, starting from the compatible triple $(\id_{H_2(X;\Z[\pi])},\id_{H_2(X,Y;\Z[\pi])},\id_{Y})$,  the maps~$\phi$ and $\psi$ lead to a new compatible triple. 
	\begin{lemma}
	\label{lem:Changeidid}
	Let~$(X,Y)$ be a~$4$-dimensional Poincar\'e pair with~$\pi_1(Y) \to \pi_1(X)=:\pi$ surjective.
The following is a compatible triple:
$$(F,G,\id_{Y_0}):=(\id_{H_2(X_0;\Z[\pi])}-\psi\circ j_*,\id_{H_2(X_0,Y_0;\Z[\pi])}+j_*\circ  \phi \circ \partial,\id_{Y_0}).$$
	\end{lemma}
	\begin{proof}
	Using $\partial \colon H_2(X,Y;\Z[\pi]) \to H_1(Y;\Z[\pi])$ to denote the connecting homomorphism,  we need to show that $i_* \circ F=i_*,\partial \circ G=\partial$ and $G \circ j_*=j_* \circ F$.
By definition of $F$ and $G$, this amounts to proving that $\psi \circ j_* \circ i_*=0,	\partial \circ j_* \circ \phi \circ \partial =0$ and $j_*\circ \phi\circ \partial  \circ j_*=j_* \circ \psi \circ \circ j_*$.
The first two equalities are clear because exactness guarantees that $j_* \circ i_*=0$ and $\partial \circ j_*=0$.
For the third, observe that both expressions are zero, one because of $\partial  \circ j_*=0$ and second because~$j \circ \psi$ contains the expression $j \circ \PD_X \circ \delta=\PD_X \circ j^* \circ \delta=0. $

For~$(F,G,\id_{Y})$ to be a compatible triple, it remains to verify that~$\lambda_X^\partial(F(x),G(y))=\lambda_X^\partial(x,y)$ for every~$x \in H_2(X;\Z[\pi])$ and every~$y \in H_2(X,Y;\Z[\pi]).$
	In other words, we have to show that 
$$\lambda_X^\partial(x,y)=\lambda_X^\partial((\id_{H_2(X_0;\Z[\pi])}-\psi\circ j_*(x)),(\id_{H_2(X_0;\Z[\pi])}+j_*\circ \phi\circ \partial)(y)).$$
We first establish~$\lambda_X(\psi \circ j_*(x),j_* \circ \phi \circ \partial(y))=0$ and then~$\lambda_X^\partial(\psi \circ j_*(x),y)-\lambda_X^\partial(x,j_* \circ \phi \circ \partial(y))=0$.

We assert that $\lambda_X^\partial(\psi \circ j_*(x),j_* \circ \phi \circ \partial(y))=0$.	
Observe that the image of $\psi$ is contained in~$\im(i \colon H_2(Y;\Z[\pi]) \to H_2(X;\Z[\pi]))$ (to see this, use $\PD_X \circ \delta=i \circ \PD_Y$).
This observation ensures that $\psi \circ j_*(x)$ pairs trivially under $\lambda_X^\partial$ with every element of $H_2(X,Y;\Z[\pi])$.
The assertion follows immediately.

Next we claim that $\lambda_X^\partial(\psi \circ j_*(x),y)=\lambda_X^\partial(x,j_* \circ \phi \circ \partial(y))$ for all~$x \in H_2(X;\Z[\pi]),y \in H_2(X,Y;\Z[\pi])$.
	Using Lemma~\ref{lem:Changeidid}, we hence have to show that~$\langle \PD_X^{-1}(\psi\circ j_*(x)
	),y\rangle=\langle \PD_X^{-1}(x),j_* \circ \phi \circ \partial(y)\rangle $.
This is a direct calculation which, besides the third equality, follows from the definitions:
\begin{align*}
\lambda_X^\partial(\psi \circ j_*(x),y)
&=\langle \PD_X^{-1}(\psi\circ j_*(x)),y\rangle
=\ev_X(\PD_X^{-1}(\psi\circ j_*(x)))(y)
=\phi^*(\PD_X^{-1}(j_*(x)))(\partial(y)) \\
&=\langle \PD_X^{-1}(j_*(x)),\phi \circ \partial(y)\rangle 
=\langle j^* \circ \PD_X^{-1}(x),\phi \circ \partial(y)\rangle
=\langle \PD_X^{-1}(x),j_* \circ \phi \circ \partial(y)\rangle \\
&=\lambda_X^\partial(x,j_* \circ \phi \circ \partial(y)).
	\end{align*}
To justify the third equality,  note that~$\psi$ involves~$\PD_X$ in the last step of its definition and that the following square commutes:
	\[\begin{tikzcd}
		H^1(Y;\Z[\pi])/H^1(\pi;\Z[\pi])\ar[d,"\ev_Y","\cong"']\ar[r,"\delta"]&H^2(X,Y;\Z[\pi])\ar[d,"\ev_X"]\\
		H_1(Y;\Z[\pi])^*\ar[r,"{\partial^*}"]&H_2(X,Y;\Z[\pi])^*.
	\end{tikzcd}\]
This concludes the proof of the claim and thus the proof that~$(F,G,\id_{Y})$ is a compatible triple. 
	\end{proof}
	
	The next proposition shows how $\phi$ and $\psi$ can be used to change a given compatible triple.

	\begin{proposition}
	\label{prop:ChangeTriple}
Let~$(X_0,Y_0)$ and~$(X_1,Y_1)$ be~$4$-dimensional Poincar\'e pairs with fundamental group~$\pi_1(X_i) \cong \pi$,  and~$\pi_1(Y_i) \to \pi_1(X_i)$ surjective for $i=0,1$,  and let~$h \colon Y_0 \to Y_1$ be a degree one homotopy equivalence.	
Given a compatible triple~$(F,G,h)$ and~$\phi\colon H_1(Y_0;\Z[\pi])\to H_2(X_0;\Z[\pi])$,  the following is again a compatible triple
$$(F-F\circ \psi\circ j_*,G+G\circ j_*\circ \phi\circ \partial,h).$$
	\end{proposition}
	\begin{proof}
	Given two compatible triples $(F_1,G_1,h_1)$ and~$(F_2,G_2,h_2)$, the componentwise composition $(F_1 \circ F_2,G_1 \circ G_2,h_1 \circ h_2)$ is again a compatible triple.
	Since Lemma~\ref{lem:Changeidid} ensures that the triple~$(\id_{H_2(X_0;\Z[\pi])}-\psi\circ j_*,\id_{H_2(X_0,Y_0;\Z[\pi])}+j_* \circ \phi \circ \partial,\id_{Y_0})$ is compatible,  it follows that 
$$(F-F\circ \psi\circ j_*,G+G\circ j_*\circ \phi\circ \partial,h)=(F,G,h) \circ (\id_{H_2(X_0;\Z[\pi])}-\psi\circ j_*,\id_{H_2(X_0,Y_0;\Z[\pi])}+j_*\circ \phi \circ \partial,\id_{Y_0})$$
is again a compatible triple.
	\end{proof}

\subsection{Proof of Theorem~\ref{thm:OddChange}}
After some preliminary lemmas, we prove Theorem~\ref{thm:OddChange}.

\begin{notation}
Fix a $4$-dimensional Poincar\'e pair $(X,Y)$ with $\pi_1(Y) \to \pi_1(X):=\pi$ surjective and Postnikov $2$-type $B:=P_2(X)$.
Additionally, fix~$\phi \colon H_1(Y;\Z[\pi]) \to H_2(X;\Z[\pi])$ and consider the associated~$\psi\colon H_2(X,Y;\Z[\pi])\to H_2(X;\Z[\pi])$ (recall Notation~\ref{not:psiphi}).
Let~$c_0',c_0 \colon X \to B$ be~$3$-connected maps such that $c|_Y:=c_0|_Y=c_0'|_Y$ and assume that $c|_Y$ is an inclusion. Furthermore, assume that $(c_0,c_0')$ induces the compatible triple
$$(F,G,\id_Y):=(\id_{H_2(X;\Z[\pi])}-\psi\circ j_*,\id_{H_2(X,Y;\Z[\pi])}+ j_*\circ \phi\circ \partial,\id_Y).$$
Proposition~\ref{prop:RelativeFormsCohomologyAreTheSame} ensures that~$(c_0')^*b_X^\partial=c_0^*b_X^\partial$
so that the secondary obstruction $b(c_0,c_0')$ is defined.
Here, recall that we write~$j \colon (X,\emptyset) \to (X,Y)$ for the inclusion induced map.
\end{notation}

\begin{lemma}
\label{lem:Technical1}
For $x \in H^2(Y;\Z[\pi])$ and $y\in H^2(X,Y;\Z[\pi])$, we have
	\[b_X^\partial(y,(c_0')^*(\wh x)-c_0^*(\wh x))=\lambda_X^\partial(\PD_X(y),j_* \circ \phi(\PD_Y(x)))\] 
		 where~$\wh x \in H^2(B;\Z[\pi])$ is a preimage of~$x\in H^2(Y;\Z[\pi])$ under~$H^2(B;\Z[\pi])\xrightarrow{c|_Y^*} H^2(Y;\Z[\pi])$. 
\end{lemma}
\begin{proof}
The equality~$(c_0')_*^{-1}(c_0)_*=\id_{H_2(X,Y;\Z[\pi])}+ j_*\circ \phi\circ \partial\colon H_2(X,Y;\Z[\pi])\to H_2(X,Y;\Z[\pi])$ implies that 
	\begin{equation}
		\label{eq:cc_0}
		(c_0)_*=(c_0')_*+(c_0')_*\circ j_*\circ \phi\circ \partial.
	\end{equation}
	We begin with the following intermediate calculation, where the third equality uses \eqref{eq:cc_0} and the fifth equality uses that $(c_0')^*b_X^\partial=c_0^*b_X^\partial$ as well as $\partial(\PD_X \circ c_0^*(\wh x))=
	\PD_Y(c|_Y^*\wh x)
	=\PD_Y(x)$:
	\begin{align*}
\langle y,c_0^*(\wh x)\cap[X]\rangle
&=\langle (c_0')^*((c_0')^*)^{-1}(y),c_0^*(\wh x)\cap[X]\rangle
=\langle ((c_0')^*)^{-1}(y),(c_0')_*(c_0^*(\wh x)\cap[X])\rangle\\
&=\langle ((c_0')^*)^{-1}(y),(c_0)_*(c_0^*(\wh x)\cap[X])\rangle-\langle ((c_0')^*)^{-1}(y),(c_0')_* \circ j_*\circ \phi\circ \partial(c_0^*(\wh x)\cap[X])\rangle\\
		&=\langle ((c_0')^*)^{-1}(y),\wh x\cap (c_0)_*[X]\rangle-\langle ((c_0')^*)^{-1}(y),(c_0')_*\circ j_*\circ \phi\circ \partial(c_0^*(\wh x)\cap [X])\rangle\\
		&=\langle ((c_0')^*)^{-1}(y),\wh x\cap (c_0')_*[X]\rangle-\langle ((c_0')^*)^{-1}(y),(c_0')_* \circ j_*\circ \phi(\PD_Y(x))\rangle\\
		&=\langle y,(c_0')^*(\wh x)\cap[X]\rangle-\langle y,j_*\circ \phi(\PD_Y(x))\rangle.
	\end{align*}
Applying Lemma~\ref{lem:otheradjoints} and the conclusion of this calculation yields:
 \begin{align*}
 	b_X^\partial(y,(c_0')^*(\wh x)-c_0^*(\wh x))
 	&=\overline{\langle y,((c_0')^*(\wh x)-c_0^*(\wh x))\cap [X]\rangle}
=\overline{\langle y,j_*\circ \phi(\PD_Y(x))\rangle}\\
&=b_X^\partial(y,\PD^{-1}_X\circ j_*\circ \phi(\PD_Y(x))) 
=\lambda_X^\partial(\PD_X(y),j_* \circ \phi(\PD_Y(x))).
 \end{align*}
 This concludes the proof of the lemma.
\end{proof}

Continuing with the notation from Lemma~\ref{lem:Technical1}, since $(c_0')^*(\wh x),c_0^*(\wh x) \in H^2(X;\Z[\pi])$ both map to~$x\in H^2(Y;\Z[\pi])$,  the difference $(c_0')^*(\wh x)-c_0^*(\wh x)$ lies in $\im(j^* \colon H^2(X,Y;\Z[\pi]) \to H^2(X;\Z[\pi]))$.
This observation is used implicitly in the statement of the next lemma. 

\begin{lemma}
\label{lem:Technical2}
Let~$x\in H^2(Y;\Z[\pi])$, and assume~$\wh x \in H^2(B;\Z[\pi])$ satisfies $c|_Y^*(\wh x)=x$.
For an element~$y\in H^2(X,Y;\Z[\pi])$ with~$j^*(y)=(c_0')^*(\wh x)-c_0^*(\wh x) \in H^2(X;\Z[\pi])$, the following holds:
	\[b_X^\partial(y,(c_0')^*(\wh x)-c_0^*(\wh x))=\lambda_X^\partial(\phi(\PD_Y(x)),j_* \circ \phi(\PD_Y(x))).\]
\end{lemma}
\begin{proof}
The equality $\lambda_X^{\partial}(a,j_*(b))=\lambda_X(a,b)$ and $\lambda_X$ being hermitian imply $\lambda_X^{\partial}(a,j_*(b))=\overline{\lambda_X^\partial(b,j_*(a))}$.
Using Lemma~\ref{lem:Technical1}, the equality $\lambda_X^{\partial}(a,j_*(b))=\overline{\lambda_X^\partial(b,j_*(a))}$ and $j^*(y)=(c_0')^*(\wh x)-c_0^*(\wh x)$,
we obtain
	\begin{align*}
b_X^\partial(y,(c_0')^*(\wh x)-c_0^*(\wh x))
&=\lambda_X^\partial(\PD_X(y),j_* \circ \phi(\PD_Y(x)))
=\overline{\lambda_X^\partial(\phi(\PD_Y(x)),j_*\PD_X(y))} \\
&=\overline{\lambda_X^\partial(\phi(\PD_Y(x)),\PD_X(j^*(y)))}
=\overline{b_X^\partial(\PD_X^{-1} \circ \phi(\PD_Y(x)),j^*(y))}\\
&=\overline{b_X^\partial(\PD_X^{-1} \circ \phi(\PD_Y(x)),(c_0')^*(\wh x)-c_0^*(\wh x))}
=\overline{\lambda_X^\partial(\phi(\PD_Y(x)),j_* \circ \phi(\PD_Y(x)))} \\
&=\lambda_X^\partial(\phi(\PD_Y(x)),j_* \circ \phi(\PD_Y(x))).
	\end{align*}
	This concludes the proof of the lemma.
\end{proof}

In order to make sense of the next result,  it is helpful to recall from Notation~\ref{not:Hat} that the form~$q_\phi  \colon H^2(Y;\Z[\pi]) \times H^2(Y;\Z[\pi]) \to \Z[\pi]$ is defined as~$
q_\phi(x,y)= \lambda_{X}(\phi(\PD_Y(x)),\phi(\PD_Y(y))).$

\begin{proposition}
	\label{prop:changing-b-FG}
The following equalities hold in $I^1(Y;\Z[\pi])$:
$$\wh b(c_0,c_0')=\wh b(c_0',c_0')+\widehat{q}_\phi=\widehat{q}_\phi.$$
\end{proposition}
\begin{proof}
During this proof, we write $X \cup X:=X \cup_{\id_Y} \unaryminus X$.
For~$x\in H^2(Y;\Z[\pi])$,  the definition of the secondary obstruction
yields  
\begin{align}
\label{eq:bc0c0}
b(c_0',c_0')(x,x)&=
b_{X\cup X}((c_0' \cup c_0')^*(\wh x),(c_0' \cup c_0')^*(\wh x)), \\
b(c_0,c_0')(x,x) \nonumber
&=b_{X\cup X}((c_0\cup c_0')^*(\wh x),(c_0\cup c_0')^*(\wh x)),
\end{align}
 where~$\wh x \in H^2(B;\Z[\pi])$ is a preimage of~$x$ under~$H^2(B;\Z[\pi])\xrightarrow{c|_Y^*} H^2(Y;\Z[\pi])$. 
 
The remainder of the proof is devoted to calculating $q_{\phi}(x,x).$
First,  some notation.
Write
\begin{align*}
&j_{X \cup X} \colon (X \cup X,\emptyset)\to (X \cup X,Y) \\
&i_k \colon (X,Y)\to (X \cup X,Y)  \quad \text{for } k=0,1.
\end{align*} 
The subscript in $j_{X \cup X}$ is used to distinguish this map from $j \colon (X,\emptyset) \to (X,Y)$.

Thinking of $j_{X \cup X}$ as a map $(X \cup X ,\emptyset,\emptyset)\to (X \cup X,\emptyset,Y)$, note that $j_{X \cup X}^*=\id_{H^2(X\cup X;\Z[\pi])}$.
Thus for $\alpha \in H^2(X\cup X;\Z[\pi])$ and $\sigma \in H_4(X\cup X)$ the relation between cap and cups (as e.g. in~\cite[18.1.1 (1)]{TomDieck}) ensures that
\begin{equation}
\label{eq:CapTripleTrick}
(j_{X\cup X})_*(\alpha \cap \sigma)=\alpha \cap (j_{X\cup X})_*(\sigma). 
\end{equation}
Choose~$z\in H^2(X\cup X;Y;\Z[\pi])$ so that~$j_{X \cup X}^*(z)=(c_0' \cup c_0')^*(\wh x)-(c_0\cup c_0')^* (\wh x) \in H^2(X\cup X;\Z[\pi])$.
  Such a preimage exists because the summands map to $x \in H^2(Y;\Z[\pi]).$
\begin{claim}
The following equality holds:
$$ q_\phi(x,x)=b_{X\cup X}(j_{X\cup X}^*(z),j_{X\cup X}^*(z)).$$
\end{claim}
\begin{proof}
 Applying the equation in~\eqref{eq:CapTripleTrick} to~$\alpha=j_{X \cup X}^*(z) \in H^2(X\cup X;\Z[\pi]),$ yields
\begin{equation*}
(j_{X\cup X})_*(j_{X \cup X}^*(z) \cap [X \cup X])=j_{X \cup X}^*(z) \cap (j_{X \cup X})_*([X \cup X]).
\end{equation*}
Using this equality as well as 
$$(j_{X \cup X})_*([X \cup X])=(i_0)_*([X])-(i_1)_*([X])\in H_4(X \cup X,Y) \cong  H_4(X,Y) \oplus H_4(X,Y),$$
we then obtain
 \begin{align*}
 	b_{X\cup X}(j_{X\cup X}^*(z),j_{X\cup X}^*(z))
 	&=\langle z,(j_{X\cup X})_*(j_{X\cup X}^*(z)\cap [X\cup X])\rangle\\ 	
 	&=\langle z,j_{X\cup X}^*(z)\cap (j_{X\cup X})_*([X\cup X])\rangle\\
 	&=\langle z,j_{X\cup X}^*(z)\cap (i_0)_*([X])\rangle-\langle z,j_{X\cup X}^*(z)\cap (i_1)_*([X])\rangle\\
 	 	 	&=\langle z,(i_0)_*(i_0^*\circ j_{X\cup X}^*(z)\cap [X])\rangle-\langle z,(i_1)_*(i_1^*\circ j_{X\cup X}^*(z)\cap [X])\rangle\\
 	 	&=\langle z,(i_0)_*(j^* \circ i_0^*(z)\cap [X])\rangle-\langle z,(i_1)_*(j^* \circ i_1^*(z)\cap [X])\rangle\\
 	&=\langle i_0^*(z),j^*\circ i_0^*(z)\cap [X]\rangle-\langle i_1^*(z),j^*\circ i_1^*(z)\cap [X]\rangle.
 \end{align*}
In the fifth equality, the relation $i_0^* \circ j_{X\cup X}^*=j^* \circ i_0^*$ (where $i_0^*$ is relative on one side and absolute on the other) holds because  $i_0\circ j, j_{X\cup X} \circ i_0 \colon (X,\emptyset) \to (X\cup X,Y)$ are both given by inclusion of the first factor.
The same reasoning is also used with~$i_1$ in place of $i_0$.
 
The definition of $z\in H^2(X\cup X;Y;\Z[\pi])$ ensures that~$i_1^*(z)=0 \in H^2(X,Y;\Z[\pi])$ and that~$i_0^*(z) \in H^2(X,Y;\Z[\pi])$ satisfies~$j^*(i_0^*(z))=(c_0')^*(\wh x)-c_0^*(\wh x) \in H^2(X;\Z[\pi])$. 
These observations together with Lemma~\ref{lem:otheradjoints} yield
\begin{align*}
b_{X\cup X}(j_{X\cup X}^*(z),j_{X\cup X}^*(z))
&=\langle i_0^*(z),j^*i_0^*(z)\cap [X]\rangle 
=\langle i_0^*(z),((c_0')^*(\wh x)-c_0^*(\wh x))\cap [X]\rangle \\
&=\overline{b_X^\partial(i_0^*(z),(c_0')^*(\wh x)-c_0^*(\wh x))}.
\end{align*}
Combining this equality with Lemma~\ref{lem:Technical2} (applied with~$y=i_0^*(z)$) and the definition of $q_\phi$, we get
\begin{align*}
b_{X\cup X}(j_{X\cup X}^*(z),j_{X\cup X}^*(z))
&=\overline{b_X^\partial(i_0^*(z),(c_0')^*(\wh x)-c_0^*(\wh x))} 
=\overline{\lambda^\partial_X(\phi(\PD_Y(x)),j_*\circ \phi(\PD_Y(x)))} \\
&=\lambda^\partial_X(\phi(\PD_Y(x)),j_*\circ \phi(\PD_Y(x))) 
=\lambda_X(\phi(\PD_Y(x)),\phi(\PD_Y(x))) \\
&=q_\phi(x,x).
\end{align*}
This concludes the proof of the claim.
\end{proof}
We can now conclude the proof of the proposition.
The claim, together with the equality~$j_{X \cup X}^*(z)=(c_0' \cup c_0')^*(\wh x)-(c_0\cup c_0')^*(\wh x)$ (by definition of $z$) and~\eqref{eq:bc0c0},  yields
\color{black}
\begin{align*}
q_\phi(x,x)
&=b_{X\cup X}(j_{X\cup X}^*(z),j_{X\cup X}^*(z)) \\
&=b_{X\cup X}((c_0' \cup c_0')^*(\wh x)-(c_0\cup c_0')^*(\wh x),(c_0' \cup c_0')^*(\wh x)-(c_0\cup c_0')^*(\wh x)) \\
&=b(c_0',c_0')(x,x)+b(c_0,c_0')(x,x)
-b_{X\cup X}((c_0' \cup c_0')^*(\wh x),(c_0\cup c_0')^*(\wh x))
-b_{X\cup X}((c_0\cup c_0')^*(\wh x),(c_0' \cup c_0')^*(\wh x))\\
&=b(c_0',c_0')(x,x)+b(c_0,c_0')(x,x)
-b_{X\cup X}((c_0' \cup c_0')^*(\wh x),(c_0\cup c_0')^*(\wh x))
-\overline{b_{X\cup X}((c_0' \cup c_0')^*(\wh x),(c_0\cup c_0')^*(\wh x))}.
\end{align*}
The equality $\wh b(c_0,c_0')=\wh b(c_0',c_0')+\widehat{q}_\phi$ follows from the definition of~$\ \wh{} \ $.
Finally, $b(c_0',c_0') \equiv 0$, by definition of the secondary obstruction, whence the equalities~$\wh b(c_0,c_0')=\wh b(c_0',c_0')+\widehat{q}_\phi=\widehat{q}_\phi.$ 
This concludes the proof of the proposition.
\end{proof}

We can now prove Theorem~\ref{thm:OddChange}.
Recall that under a certain number of assumptions, this result states that for every map $\phi \colon H_1(Y_0;\Z[\pi]) \to H_2(X_0;\Z[\pi])$, there is a $3$-connected map $c_0' \colon X \to B$ satisfying~$c_0'|_{Y_0} =  c_1|_{Y_1} \circ h$ as well as~$(c_0')^*b_{X_0}^\partial =c_1^*b_{X_1}^\partial$, and~$ \wh b(c_0',c_1)=\wh b(c_0,c_1)-\widehat{h^*q_\phi}.$

\begin{proof}[Proof of Theorem~\ref{thm:OddChange}]
Our first task is to construct the map~$c_0'$.
Consider the mapping cylinder $M(c_0|_{Y_0})$ of~$c_0|_{Y_0} \colon Y_0 \to B$ and the homotopy equivalence $B \to M(c_0|_{Y_0})$.
The composition~$Y_0 \xrightarrow{c_0|_{Y_0}}B\xrightarrow{\simeq}M(c_0|_{Y_0})$ is homotopic to the inclusion $Y_0\subset M(c_0|_{Y_0})$.
Applying the homotopy extension property,  we deduce that
the composition $X_0\xrightarrow{c_0}B\xrightarrow{\simeq}M(c_0|_Y)$ is homotopic to a map $\wh c_0 \colon X_0 \to M(c_0|_Y)$ such that~$\wh c_0|_{Y_0}$ is the inclusion~$Y_0\subset M(c_0|_{Y_0})$.
Composing this homotopy with the homotopy equivalence~$\proj  \colon M(c_0|_Y) \xrightarrow{\simeq}~B$ shows that $\proj \circ  \wh c_0  \simeq c_0 \colon X_0 \xrightarrow{} B$ rel. $Y_0$.

Apply Proposition~\ref{prop:Simplify} with $X_0=X_1$ to get a $3$-connected $\wh c_0' \colon X_0\to M(c_0|_Y)$ with~$(\wh c_0')^*b_{X_0}^\partial =\wh c_0^*b_{X_0}^\partial$ and $\wh c_0'|_{Y_0}=\wh c_0|_{Y_0}$ and such that $(\wh c_0',\wh c_0)$ realises the compatible triple
$$(\id_{H_2(X_0;\Z[\pi])}-\psi\circ j_*,\id_{H_2(X_0,Y_0;\Z[\pi])}+ j_*\circ \phi\circ \partial,\id_{Y_0}).$$
Here, it is legitimate to apply Proposition~\ref{prop:Simplify} because~\cite[Lemma 6.1]{ConwayKasprowskiKinvariant} shows that $ j_*$ maps the~$k$-invariant to zero, thus ensuring that this compatible triple is $k$-invariant preserving.
Define 
$$c_0':=\proj\circ \wh c_0' \colon X_0 \to B.$$
We have $(\wh c_0')^*b_{X_0}^\partial =\wh c_0^*b_{X_0}^\partial$ and~$\wh c_0|_{Y_0}=\wh c_0'|_{Y_0}$. 
Applying $\proj$ and the rel. $Y_0$ homotopy $\proj\circ \wh c_0 \simeq c_0$, it follows that $(c_0')^*b_{X_0}^\partial = c_0^*b_{X_0}^\partial$ and~$c_0|_{Y_0}=c_0'|_{Y_0}$.
From this and $c_0 \in \mathcal{S}_0(c_1)$,
it follows that~$c_0' \in \mathcal{S}_0(c_1)$.
In particular, the pairing $b(c_0',c_1)$ is defined.

It remains to show that~$ \wh b(c_0',c_1)=\wh b(c_0,c_1)-\widehat{h^*q_\phi}.$
Apply the transitivity of Proposition~\ref{prop:Transitive} to get a homotopy equivalence~$\mathcal{\varphi} \in \mathcal{G}$ with $\varphi \circ c_0=c_0'$ rel. $Y_0$. 
Proposition~\ref{prop:Action} then implies that 
$$b(c_0',c_1)=b(\varphi \circ c_0,c_1)=b(c_0,c_1)-b(\varphi).$$
We claim that $b(\varphi)=h^* b(\wh c_0',\wh c_0),$
as the conclusion that $\wh b(c_0',c_1)=\wh b(c_0,c_1)-h^*\widehat{q}_\phi$ will then follow promptly.
By definition of $b(\varphi)$, for $x,y \in H^2(Y_1;\Z[\pi])$ and $\widehat{x},\widehat{y} \in~H^2(B;\Z[\pi])$
with $c_1|_{Y_1}^*(\wh x)=x$ and $c_1|_{Y_1}^*(\wh y)=y$,
 we have~$b(\varphi)(x,y)=\langle \widehat{y},\widehat{x} \cap \xi\rangle$, with~$\xi =(c_0' \cup c_0)_*([X_0 \cup_{\id} X_0])\in H_4(B)$.
We now calculate
\begin{align*}
	b(\varphi)(x,y)
	&=\langle \widehat{y},\widehat{x} \cap \xi\rangle 
	=\langle \widehat{y},\widehat{x} \cap (c_0' \cup c_0)_*([X_0 \cup_{\id} X_0])\rangle \\
	&=\langle (c_0' \cup c_0)^*(\widehat{y}),(c_0' \cup c_0)^*(\widehat{x}) \cap [X_0 \cup_{\id} X_0]\rangle \\
&= \langle (\wh c_0' \cup \wh c_0)^* \proj^* (\widehat{y}),(\wh c_0' \cup \wh c_0)^*\proj^*(\widehat{x}) \cap [X_0 \cup_{\id} X_0]\rangle  \\
&= \langle  \proj^* (\widehat{y}),\proj^*(\widehat{x}) \cap (\wh c_0' \cup \wh c_0)_*([X_0 \cup_{\id} X_0])\rangle  \\
&= b(\wh c_0',\wh c_0)(h^*(x),h^*(y)) \\
&= h^*b(\wh c_0',\wh c_0)(x,y).
\end{align*}
For the penultimate equality, we used the definition of $b(\wh c_0',\wh c_0)$ together with the fact that $\proj^*(\wh x)$ is a preimage of $h^*(x)$ by $\wh c_0|_{Y_0}^*$ (and similarly, $\proj^*(\wh y)$ is a preimage of $h^*(y)$):
\[\wh c_0|_{Y_0}^*\proj^*(\wh x)=c_0|_{Y_0}^*(\wh x)=h^*c_1|_{Y_1}^*(\wh x)=h^*(x).\]
Thus $b(\varphi)=h^* b(\wh c_0',\wh c_0),$ establishing the claim.

We can now conclude.
Indeed, using the claim, it follows that \[\wh b(c_0',c_1)
=\wh b(c_0,c_1)-\wh b(\varphi)
=\wh b(c_0,c_1)-\widehat{h^*b}(\wh c_0',\wh c_0)
=\wh b(c_0,c_1)-\widehat{h^*q_\phi},\] where in the last equality, we used Proposition~\ref{prop:changing-b-FG}.
\end{proof}

\section{Groups to which the theorems apply}
\label{sec:cd3}

This section describes situations in which our results apply.
In what follows, we let~$(X_0,Y_0)$ and~$(X_1,Y_1)$ be $4$-dimensional Poincar\'e pairs with $\pi_1(X_0) \cong \pi_1(X_1):=\pi$ that satisfy the conditions listed in Notation~\ref{not:Intro}.

Section~\ref{sub:cd3} discusses examples of groups with~$\cd(\pi) \leq 3$ as this condition appeared in Theorem~\ref{thm:MainTheoremBoundaryIntro} as well as in Proposition~\ref{prop:StablyFreePD<4} (which characterised when $\pi_2(X)$ is projective).
Section~\ref{sub:WhiteheadGamma} recalls the definition of the Whitehead functor~$\Gamma(-)$ and lists examples of pairs~$(\pi_1,\pi_2)$ for which~$\Gamma(\pi_2) \otimes_{ \Z[\pi_1]} \Z$ is torsion-free: this hypothesis was present in Theorem~\ref{thm:MainTheoremBoundaryIntro}.
Section~\ref{sub:H2piZpiInj} gives situations in which the map~$H^2(\pi;\Z[\pi]) \to H^2(Y_i;\Z[\pi])$ is injective, as this condition appeared in Theorems~\ref{thm:PrimaryObstructionRecastIntro} and~\ref{thm:VanishOddIntro}.
Finally, Section~\ref{sub:ev*Iso} considers examples when the dualised evaluation map~$\ev^* \colon H^2(X_i;\Z[\pi])^{**} \to H_2(X_i;\Z[\pi])^*$ is an isomorphism,  as this assumption also played an important role in Theorems~\ref{thm:VanishOddIntro},~\ref{thm:VanishSpinIntro} and~\ref{thm:VanishEvenIntro}.

\subsection{Examples of groups of cohomological dimension $\leq 3$.}
\label{sub:cd3}

Examples of groups of cohomological dimension at most $3$ include the trivial group ($\cd=0$),  free groups ($\cd=1$),  surfaces groups and (nonfree) Baumslag--Solitar groups $BS(m,n)$ (both families have~$\cd=2$),  as well as fundamental groups of closed oriented aspherical $3$-manifolds ($\cd=3$).
Free products of the above yield new groups with~$\cd \leq 3$.
It is also helpful to recall that finite groups have infinite cohomological dimension.

Note that finite groups are good, whereas among the aforementioned groups with $\cd \leq 3$, 
examples of good groups include
$\Z$, the solvable Baumslag--Solitar groups~$BS(1,n)$ (this includes~$BS(0,1)=\Z$ and~$BS(1,1)=\Z^2$) and~$\Z^3$.
We refer to~\cite[Chapter 19]{DET} for a discussion of good groups but note that elementary amenable groups are good.

\subsection{The $\Gamma$ group condition}
\label{sub:WhiteheadGamma}

The \emph{Whitehead group} of an abelian group~$A$ consists of the group~$\Gamma(A)$ generated by the elements of~$A$ (we write~$v(a) \in \Gamma(A)$ for the element corresponding to~$a \in A$) and subject to the relations
\begin{align*}
&v(-a)=v(a) \quad \text{ for all }  a \in A, \\
&v(a+b+c)=v(b+c)+v(a+c)+v(a+b)-v(a)-v(b)-v(c),   \quad \text{ for all } a,b,c \in A.
\end{align*}
If~$A$ is free abelian with basis~$\mathcal{B}$ then, by~\cite[page 62]{WhiteheadCertainExact},~$\Gamma(A)$ is free abelian with basis
$$\{v(b),v(b+b')-v(b)-v(b')\}_{b\neq b' \in \mathcal{B}}.$$
In this case, it is customary to consider $\Gamma(A)$ as the subgroup of symmetric elements of $A \otimes A$ given by sending $v(a)$ to $a \otimes a$.
We refer to~\cite{WhiteheadCertainExact} for more details on this construction but simply note that for $X$ a $4$-dimensional Poincar\'e complex~$\Gamma(\pi_2(X)) \otimes_{\Z[\pi_1(X)]} \Z$ is known to be torsion free for finite fundamental groups $\pi_1(X)$ with 4-periodic cohomology~\cite{HambletonKreck} (e.g. finite cyclic groups) and more generally for finite groups where only the~$2$-Sylow subgroup of the fundamental group has~$4$-periodic cohomology~\cite{Bauer}.
This is also known to hold for finite groups whose~$2$-Sylow subgroup is abelian with at most two generators~\cite{KasprowskiPowellRuppik} 
or dihedral~\cite{KasprowskiNicholsonRuppik}.

We note that the same is true in our setting.
\begin{proposition}
\label{prop:GammaGroup}
Assume that $\pi$ is a finite group whose $2$-Sylow subgroup is dihedral, abelian with 2 generators or has 4-periodic cohomology.
If~$(X,Y)$ is a 4-dimensional Poincar\'e pair with~$\pi_1(Y)\to \pi_1(X)=:\pi$ surjective, then $\Z \otimes_{\Z[\pi_1(X)]} \Gamma(\pi_2(X))$ is torsion-free.
\end{proposition}
\begin{proof}
Lemma~\ref{lem:pi2Projective-new} gives finitely generated free~$\Z[\pi]$-modules~$C_0$,~$C_1$ and~$C_2$ and an exact sequence of~$\Z[\pi]$-modules~$0\to \pi_2(X)\to C_2\to C_1\to C_0\to \Z\to 0$.
Write~$\ker(d_2):=\pi_2(X)$.
The fact that~$\Z \otimes \Gamma(\ker(d_2))$ is torsion-free now follows from (a simplification of) the argument in~\cite[Proof of Theorem 1.1 in Section 4]{KasprowskiPowellRuppik}, \cite[Theorem 4.1]{KasprowskiNicholsonRuppik} and~\cite[Theorem~2.1 and Remark~2.5]{HambletonKreck}.
We emphasise that in these articles 
the passage from the~$2$-Sylow subgroups relies on the work of Bauer~\cite{Bauer}.
We also note that for closed $4$-manifolds, the situation is actually more delicate since $\Z \otimes \Gamma(\coker(d_2))$ must also be shown to be torsion-free.
\end{proof}

\subsection{Situations in which $H^2(\pi;\Z[\pi]) \to H^2(Y_i;\Z[\pi])$ is injective.}
\label{sub:H2piZpiInj}

Let $Y$ be a closed oriented~$3$-manifold and let $\pi_1(Y) \twoheadrightarrow \pi$ be a surjection.

\begin{proposition}
\label{prop:TorsionImpliesH2mapInj}
If~$H_1(Y;\Z[\pi])^*=0$, then~$H^2(\pi;\Z[\pi]) \to H^2(Y_1;\Z[\pi])$ is injective.
\end{proposition}
\begin{proof}
The UCSS gives rise to an exact sequence~$H_1(Y;\Z[\pi])^*\to H^2(\pi;\Z[\pi])\to H^2(Y;\Z[\pi])$. 
	Since~$H_1(Y;\Z[\pi])^*=0$,  it follows that~$H^2(\pi;\Z[\pi]) \to H^2(Y_1;\Z[\pi])$ is indeed injective.
\end{proof}

\begin{proposition}
\label{prop:H2piZpi=0}
The induced homomorphism~$H^2(\pi;\Z[\pi]) \to H^2(Y;\Z[\pi])$ is injective in the following circumstances:
\begin{enumerate}
\item $\pi$ is free, finite or a $\PD_3$ group.
\item $\pi$ is a free product of a free group and $\PD_3$ groups.
\item $\pi=\pi_1(\Sigma)$ is the fundamental group of an orientable surface, $Y=\Sigma\times S^1$, and $\pi_1(\Sigma \times S^1) \to \pi$ is the projection onto the first coordinate.
\end{enumerate}
\end{proposition}
\begin{proof}
The cases where $\pi$ is free or a~$\PD_3$ group follow from the second case, but we treat them separately as a warm up.
For the groups in the first statement, we argue that~$H^2(\pi;\Z[\pi])=0$ from which the conclusion is immediate.
For free and finite groups this is clear,  whereas for a $\PD_3$ group $G$,  one notes that 
$$ H^2(G;\Z[G]) \cong H_1(G;\Z[G])=0.$$
Next,  we consider the case where $\pi=F * *_{i=1}^rG_i$ is a free product of a free group $F$ and $\PD_3$-groups~$G_1,\ldots,G_r$.
Since $BH\vee BH'$ is a model for $B(H*H')$, repeated applications of the Mayer-Vietoris sequence show that
    \[H^2(\pi;\Z[\pi])\cong H^2(F;\Res_F^\pi \Z[\pi])\oplus\bigoplus_{i=1}^rH^2(G_i;\Res_{G_i}^\pi\Z[\pi])\cong 0\oplus \bigoplus_{i=1}^rH_1(G_i;\Res_{G_i}^\pi\Z[\pi]),\]
    where the last isomorphism follows from Poincar\'e duality.
    Since $\Res_{G_i}^\pi\Z[\pi]$ is free as a $\Z[G_i]$-module (see e.g.~\cite[pages 13-14]{Brown}), $H_1(G_i;\Res_{G_i}^\pi\Z[\pi])=0$ and the result follows.

The third statement follows by noting that $Y:=\Sigma\times S^1\to \Sigma\simeq B\pi$ admits a section and, in particular, the induced map~$H^2(\pi;\Z[\pi])\to H^2(Y;\Z[\pi])$ is injective. 
\end{proof}

Note that the third case of Proposition~\ref{prop:H2piZpi=0} is relevant if one is attempting to decide whether a $4$-dimensional Poincar\'e pair is homotopy equivalent to~$(\Sigma_g \times D^2,\Sigma_g \times S^1)$.

\begin{remark}
We conclude with an example where $H^2(\pi;\Z[\pi]) \to H^2(Y;\Z[\pi])$ is not injective.
Consider the setting where~$F_2\cong \pi_1(S^1\times S^2\# S^1\times S^2)\to \pi:= \Z^2$ is abelianisation. 
The induced map~$\Z\cong H^2(\pi;\Z[\pi])\to H^2(Y;\Z[\pi])$ factors through $H^2(BF_2;\Z[\pi])=H^2(S^1\vee S^1;\Z[\pi])=0$,  is therefore trivial, and thus non-injective. 
This shows that Theorem~\ref{thm:1d3dIntro} cannot be applied when trying to decide whether a given Poincar\'e pair is homotopy equivalent to a $4$-dimensional Poincar\'e pair of the form~$(X,S^1\times S^2\# S^1\times S^2)$ with $\pi_1(X) \cong \Z^2$.
\end{remark}

\subsection{Examples for which~$\ev^* \colon H^2(X;\Z[\pi])^{**} \to H_2(X;\Z[\pi])^*$ is an isomorphism.}
\label{sub:ev*Iso}

Let $(X,Y)$ be a $4$-dimensional Poincar\'e pair with $\pi_1(Y) \twoheadrightarrow \pi_1(X):=\pi$ a surjection.

\begin{proposition}
\label{prop:ev*Iso}
The dualised evaluation map
 \[\ev_X^*\colon H_2(X;\Z[\pi])^{**}\to H^2(X;\Z[\pi])^*\] is an isomorphism in the following circumstances:
\begin{enumerate}
\item $\pi$ is free or finite.
\item $\pi=\pi_1(\Sigma)$ is an orientable surface group or a solvable Baumslag--Solitar group $BS(1,n)$.
\item $\pi$ is a $\PD_3$ group.
\item $\pi$ is a free product of free and $PD_3$ groups.
\end{enumerate}  
\end{proposition}
\begin{proof}
Most of the first and third cases follow from the fourth, but we treat them separately as a warm up.
We begin with some general considerations.
Looking at the second diagonal of the UCSS yields the exact sequence 
$$ 0\to H^2(\pi;\Z[\pi])\to H^2(X;\Z[\pi])\xrightarrow{\ev_X}H_2(X;\Z[\pi])^*\xrightarrow{d_3} H^3(\pi;\Z[\pi]).$$
Since $H_1(X;\Z[\pi])=0$ and $H_0(Y;\Z[\pi])\to H_0(X,\Z[\pi])$ is an isomorphism (because $Y \to X$ is~$\pi_1$-surjective), it follows that $H^3(X;\Z[\pi]) \cong H_1(X,\partial X;\Z[\pi])=0$.
Looking at the third diagonal of the UCSS,  this implies that $\coker(d_3)=0$ and so we obtain the short exact sequence
\begin{equation}\label{eq:ucss-zeros-Bs-again}
0\to H^2(\pi;\Z[\pi])\to H^2(X;\Z[\pi])\xrightarrow{\ev_X}H_2(X;\Z[\pi])^*\to H^3(\pi;\Z[\pi])\to 0.
\end{equation}
It follows that for $\pi$ a free or a finite group,  $\ev_X$ is an isomorphism and dualising therefore gives the first assertion.
We move on to the second assertion.
If~$\pi$ is the fundamental group of an orientable surface $\Sigma$, then $H^3(\pi;\Z[\pi]) \cong H^3(\Sigma;\Z[\pi])=0$ and $H^2(\pi;\Z[\pi])^*\cong H^2(\Sigma;\Z[\pi])^* \cong \Z^*=0.$
By~\cite[Lemmas 6.2 and~6.4]{HambletonKreckTeichner}, the same is true if~$\pi$ is a solvable Baumslag--Solitar group. Hence,  dualising \eqref{eq:ucss-zeros-Bs-again}, we see that  $\ev_X^*$ is an isomorphism as claimed.
Next we consider the third assertion, i.e. the case where $\pi$ is a $\PD_3$ group.
In this case~$H^2(\pi;\Z[\pi])\cong H_1(\pi;\Z[\pi])=0$ and~$H^3(\pi;\Z[\pi]) \cong \Z$.
Again,  we deduce that~$H^3(\pi;\Z[\pi])^*=0$ and the assertion follows by dualising~\eqref{eq:ucss-zeros-Bs-again} and using that $\Ext^1_{\Z[\pi]}(H^3(\pi;\Z[\pi]),\Z[\pi]) \cong H^1(\pi;\Z[\pi]) \cong H_2(\pi;\Z[\pi)=0$:
$$
0 \to \overbrace{H^3(\pi;\Z[\pi])^*}^{=0 } \to H_2(X;\Z[\pi])^{**}\xrightarrow{\ev_X^*} H^2(X;\Z[\pi])^* \to \overbrace{\Ext^1_{\Z[\pi]}(H^3(\pi;\Z[\pi]),\Z[\pi])}^{=0}.
$$
It remains to prove the fourth assertion, i.e. the case where~$\pi\cong F**_{i=1}^rG_i$ with $F$ a free group and~$G_i$ a $PD_3$-group.
The argument is essentially the same as \cite[Lemmas~5.6 and~5.8]{HKPR}, but we repeat it for the reader's convenience.
Since $H^2(\pi;\Z[\pi])=0$ (by the proof of Proposition~\ref{prop:H2piZpi=0}),  dualising~\eqref{eq:ucss-zeros-Bs-again} yields the exact sequence
\begin{equation}\label{eq:ucss-zeros-Bs-hommed}
\begin{tikzcd}
[column sep=small]
0\ar[r]& H^3(\pi;\Z[\pi])^*\arrow[r]&H_2(X;\Z[\pi])^{**}\arrow[r,"{\ev_X^*}"]&H^2(X;\Z[\pi])^*\ar[r]&\Ext^1_{\Z[\pi]}(H^3(\pi;\Z[\pi]),\Z[\pi]).
\end{tikzcd}
\end{equation}
We assert that~$H^3(\pi;\Z[\pi])^*=0$.
For each $i$, we have $H^3(G_i;\Z[G_i]) \cong H_0(G_i;\Z[G_i])\cong \Z$. 
A Mayer--Vietoris argument as in the proof of  Proposition~\ref{prop:H2piZpi=0} 
By~\cite[Lemma 4.9]{HKPR}, we have~$H^3(\pi;\Z[\pi])\cong\oplus_{i=1}^r E_i$, where $E_i=\Z[\pi]\otimes_{\Z[G_i]}\Z = \Ind_{G_i}^{\pi} \big(\Z\big)$. 
Taking duals gives
\[E_i^*=\Hom_{\Z[\pi]}(\Ind_{G_i}^{\pi} \big(\Z\big),\Z[\pi]) \cong \Hom_{\Z[G_i]}(\Z,\Res_{G_i}^\pi \Z[\pi]).\]
Since $G_i$ is infinite, the only $G_i$-fixed point in $\Z[G_i]$ is the trivial element,
and thererefore $ \Hom_{\Z[G_i]}(\Z, \Z[G_i])=0$.
Since $\Res^{\pi}_{G_j} \Z[\pi]$ is free, $\Hom_{\Z[G_j]}(\Z,\Res_{G_j}^{\pi} \Z[\pi])=0$.
Thus $E_i^*=0$.
We conclude that~$H^3(\pi;\Z[\pi])^*\cong \bigoplus_{i=1}^rE_i^*=0$ as asserted.

Next,  we assert that~$\Ext^1_{\Z[\pi]}(H^3(\pi;\Z[\pi]),\Z[\pi])=0$. 
Apply the adjunction from Lemma~\ref{lem:ExtAdjunction}, we obtain
\begin{align*}
    \Ext^1_{\Z[\pi]}(H^3(\pi;\Z[\pi]),\Z[\pi]) &\cong\Ext^1_{\Z[\pi]}(\oplus_{i=1}^rE_i,\Z[\pi])\cong
\prod_{i=1}^r\Ext^1_{\Z[\pi]}(E_j,\Z[\pi]) \\
&\cong \prod_{i=1}^r\Ext^1_{\Z[\pi]}\big(\Ind_{G_i}^\pi \big(\Z\big),\Z[\pi]\big)\cong \prod_{i=1}^r\Ext^1_{\Z[G_j]}\big(\Z,\Res_{G_i}^\pi \Z[\pi]\big).
\end{align*}
We consider each factor separately. By Poincar\'e  duality, we have
\[\Ext^1_{\Z[G_i]}\big(\Z,\Res_{G_i}^\pi \Z[\pi]\big)\cong H^1\big(G_i;\Res_{G_i}^\pi \Z[\pi]\big)\cong H_2(G_i; \Res_{G_i}^\pi \Z[\pi]).\]
Since $\Res_{G_j}^\pi\Z[\pi]$ is free, this is trivial.
Thus~$\Ext^1_{\Z[\pi]}(H^3(\pi;\Z[\pi]),\Z[\pi])$, as asserted.

Since the assertions show that $H^*(\pi;\Z[\pi])=0$ and $\Ext^1_{\Z[\pi]}(H^3(\pi;\Z[\pi]),\Z[\pi])=0$, it follows from~\eqref{eq:ucss-zeros-Bs-hommed} that~$\ev_X^*$ is an isomorphism, as required.
\end{proof}

It remains to establish a result that was used during the proof of Proposition~\ref{prop:ev*Iso}.

\begin{lemma}
\label{lem:ExtAdjunction}
Let $\pi$ be a group.
If~$G\leq \pi$ is a subgroup,  $M$ is a $\Z[G]$-module and~$N$ is a~$\Z[\pi]$-module then, for every $k \geq 1$, there is an isomorphism
    \[\Ext_{\Z[G]}^k(M,\Res^\pi_G N)\cong \Ext_{\Z[\pi]}^k(\Ind_G^\pi M,N).\]
\end{lemma}
\begin{proof}
 If~$F_*$ is a free $\Z[G]$-resolution of~$M$, then~$\Ind_G^\pi F_*\cong \Z[\pi]\otimes_{\Z[G]} F_*$ is a sequence of free $\Z[\pi]$-modules.
    Since~$\Z[\pi]$ is free as a $\Z[G]$-module
it is flat as a $\Z[G]$-module,  so $\Ind_G^\pi(-)=\Z[\pi]\otimes_{\Z[G]}-$ preserves exactness and therefore~$\Ind_G^\pi F_*$ is a free~$\Z[\pi]$-resolution of~$\Ind_G^\pi M=\Z[\pi]\otimes_{\Z[G]} F_*$. 
    Thus
    \begin{align*}
    \Ext_{\Z[\pi]}^k(\Ind_G^\pi M,N)
    &\cong H^k(\Hom_{\Z[\pi]}(\Ind_G^\pi F_*,N))\cong H^k(\Hom_{\Z[G]}(F_*,\Res_G^\pi N)) \\
    &\cong \Ext_{\Z[G]}^k(M,\Res^\pi_G N)
    \end{align*}
    as claimed.
\end{proof}


Our final result does not directly address $\ev_X^*$ being an isomorphism. 
Instead it shows that this condition ensures that~$\Herm(\ev_X) \colon \Herm(H_2(X;\Z[\pi])^*) \to\Herm(H^2(X;\Z[\pi]))$ is also an isomorphism.
This was used during the proof of Theorem~\ref{thm:MainTheoremBoundary}.

\begin{proposition}
\label{prop:Herm-iso}
If the map
\[\ev_X^*\colon H_2(X;\Z[\pi])^{**}\to H^2(X;\Z[\pi])^*\] is an isomorphism of abelian groups,  then so is
 \[\Herm(\ev_X) \colon \Herm(H_2(X;\Z[\pi])^*) \to\Herm(H^2(X;\Z[\pi])).\] 
 In particular $\Herm(\ev_X) $ is injective if $\pi$ is a Baumslag--Solitar group, an orientable surface group, a free group, a $\PD_3$ or a free product of a free group and $\PD_3$-groups.
\end{proposition}
\begin{proof}
The proof is essentially the same as \cite[Proposition~5.9]{HKPR} but we repeat it for the reader's convenience.
Applying the covariant functor $\Hom_{\Z[\pi]}(H_2(X;\Z[\pi])^*, -)$ to $\ev_X^*$, we see that
\[(\ev_X^*)_*\colon \Hom_{\Z[\pi]}(H_2(X;\Z[\pi])^*,H_2(X;\Z[\pi])^{**})\to \Hom_{\Z[\pi]}(H_2(X;\Z[\pi])^*,H^2(X;\Z[\pi])^{*})\]
is an isomorphism.
Next, apply \cref{lem:nice-lemma} with $f=\ev_X$ and $A=H^2(X;\Z[\pi])$ to see that
\[\Hom_{\Z[\pi]}(\ev_X,H^2(X;\Z[\pi])^*)\colon \Hom_{\Z[\pi]}(H_2(X;\Z[\pi])^*,H^2(X;\Z[\pi])^*)\to \Hom_{\Z[\pi]}(H^2(X;\Z[\pi]),H^2(X;\Z[\pi])^*)\] is an isomorphism.
Consider the following diagram in which the top horizontal is the composition~$\Hom_{\Z[\pi]}(\ev_X,H^2(X;\Z[\pi])^*)\circ (\ev_X^*)_*$:
\begin{equation}\label{eq:sesq}
\begin{tikzcd}
\Hom_{\Z[\pi]}(H_2(X;\Z[\pi])^*,H_2(X;\Z[\pi])^{**}) \ar[r]  \ar[d,phantom,sloped,"\cong"]&\Hom_{\Z[\pi]}(H^2(X;\Z[\pi]),H^2(X;\Z[\pi])^*)\ar[d,phantom,sloped,"\cong"]\\
\Sesq(H_2(X;\Z[\pi])^*)\arrow[r,"\Sesq(\ev_X)"]   &\Sesq(H^2(X;\Z[\pi])).
\end{tikzcd}
\end{equation}
Here $\Sesq(\ev_X)(b)(\alpha,\beta)=b(\ev_X(\alpha),\ev_X(\beta)).$
This diagram commutes: indeed, unravelling the definitions shows that for~$\alpha,\beta \in H^2(X;\Z[\pi])$ and~$\varphi\in \Hom_{\Z[\pi]}(H_2(X;\Z[\pi])^*,H_2(X;\Z[\pi])^{**})$, one has~$\Hom_{\Z[\pi]}(\ev_X,H^2(X;\Z[\pi])^*)\circ (\ev_X^*)_*(\varphi)(\alpha)=\varphi(\ev_X(\alpha))(\ev_X(\beta))$.
In particular, it follows that $\Sesq(\ev_X)$ is an isomorphism.

The set of hermitian forms on~$H^2(X;\Z[\pi])^*$ arises as the fixed point set of the involution on~$\Sesq(H^2(X;\Z[\pi])^*)$ given by~$b \mapsto ((x,y) \mapsto \overline{b(y,x)}$).
Thus~$\Herm(\ev_X)$ is obtained by restricting~$\Sesq(\ev_X)$ to hermitian forms.
It follows that~$\Herm(\ev_X)$ is an isomorphism, as required.
\end{proof}

\appendix

\section{Twisted homology}
\label{sec:TwistedHomology}

The aims of this section are to fix our conventions concerning twisted homology and cohomology as well as to collect several results that are used throughout this article.
Proofs that consist of direct verifications are omitted.

\subsection{The evaluation map}
\label{sub:Evaluation}

Fix a pair~$(X,Y)$ and a ring~$R$ that is endowed with an involution.
Let~$M$ and~$M'$ be~$(R,\Z[\pi_1(X)])$-bimodules and let~$S$ be an~$(R,R)$-bimodule.
Furthermore, let 
$$\Theta \colon M' \times M \to S$$ be a nonsingular 
$\pi$-invariant sesquilinear pairing, meaning that for every~$m \in M,m' \in M'$ and~$r,s \in R,\gamma \in \pi_1(X)$, it satisfies~$\langle m'\gamma,m\gamma\rangle
=\langle m',m\rangle$ and~$\langle rm'\gamma,sm\gamma\rangle
=r\langle m',m\rangle \overline{s}$.
Reformulating, if~$\overline{M}$ denotes the~$(\Z[\pi_1(X)],R)$-bimodule obtained by using the involutions on~$R$ and~$\Z[\pi_1(X)]$, then~$\Theta$ is an~$(R,R)$-linear map 
$$M' \otimes_{\Z[\pi_1(X)]} \overline{M} \to S.$$
We proceed to the main construction of this section.

\begin{construction}
\label{cons:Evaluation}(The evaluation~$\ev \colon H^i(X,Y;M) \to \overline{\Hom_{\text{left-}R}(H_{i}(X,Y;M'),S)}$.)

This evaluation map is defined in three steps. 
Consider the map
\begin{align*}
 \Hom_{\text{right-}\Z[\pi_1(X)]}(\overline{C_*(\widetilde{X},\widetilde{Y})},M) &\to \overline{\Hom_{\text{left-}R}(M' \otimes_{\Z[\pi_1(X)]} C_*(\widetilde{X},\widetilde{Y}),M' \otimes_{\Z[\pi_1(X)]} \overline{M})} \\
     f &\mapsto \left((m' \otimes \sigma) \mapsto (m' \otimes f(\sigma))  \right).
\end{align*}
Composing with the pairing~$\Theta$ gives rise to the chain map 
  \begin{align*}
\kappa \colon \Hom_{\text{right-}\Z[\pi_{1}(X)]}(\overline{C_*(\widetilde{X},\widetilde{Y})},M) &\to \overline{\Hom_{\text{left-}R}(M' \otimes_{\Z[\pi_1(X)]} C_*(\widetilde{X},\widetilde{Y}),S)}, \\
    f &\mapsto \left((m' \otimes \sigma) \mapsto \Theta(m',f(\sigma) ) \right).
  \end{align*}
The nonsingularity of $\Theta$ ensures this is an isomorphism of cochain complexes of left~$R$-modules.
  
Evaluating cocycles on homology classes yields a homomorphism of left~$R$-modules
  \[E \colon H^{i}(\overline{\Hom_{\text{left-}R}(M' \otimes_{\Z[\pi_1(X)]} C_{\ast}(\widetilde{X},\widetilde{Y}),S)}) \xrightarrow{} \overline{\Hom_{\text{left-}R}(H_{i}(X,Y;M'),S)}.\]
  Finally,  the evaluation map~$\ev$, which is also left~$R$-linear,  is obtained by composing~$\kappa$ with~$E$:
\begin{align*} \ev \colon H^{i}(X,Y;M) &\xrightarrow{\kappa,\cong} H^{i}(\overline{\Hom_{\text{left-}R}(M' \otimes_{\Z[\pi_1(X)]} C_{\ast}(\widetilde{X},\widetilde{Y}),S)}) \\ &\xrightarrow{E} \overline{\Hom_{\text{left-}R}(H_{i}(X,Y;M'),S)}.
\end{align*}
In what follows we often write $\langle \varphi,x\rangle$ instead of $\ev(\varphi)(x).$
\end{construction}

\begin{example}
\label{ex:EvaluationZpi}
Let $\varphi \colon \pi_1(X) \twoheadrightarrow \pi$ be an epimorphism.
Continuing with the notation above, the example that most relevant to this article is the case where~$M=M'=R=S=\Z[\pi]$ (with the module structure $\Z[\pi_1(X)]$-module structure induced by $\varphi$) and the pairing is 
\begin{align*}
\Z[\pi] \otimes_{\Z[\pi]} \overline{\Z[\pi]}& \xrightarrow{\mult} \Z[\pi] \\
(x,y) &\mapsto x\overline{y}.
\end{align*}
In this case, the evaluation map satisfies
$$\ev([f])([m' \otimes \sigma])=m'\overline{f(\sigma)}.$$
\end{example}

\subsection{Twisted cap and cup products}
\label{sub:CapCup}

Let~$X$ be a space with fundamental group $\pi$ and with universal cover~$\widetilde{X}$. Let~$R$ be a ring with involution, and let~$M,N$ be~$(R,\Z[\pi])$-bimodules. 
Given an~$(i+j)$-chain~$\sigma \in C_{i+j}(\widetilde{X})$, we use~$\lrcorner \sigma \in C_i(\widetilde{X})$ and~$\llcorner \sigma \in C_j(\widetilde{X})$ to denote the front-face~$i$-chain and back-face~$j$-chain of~$\sigma$. 
They are defined on a singular simplex~$s_{i+j} \colon \Delta^{i+j} \to \widetilde X$ as follows and then extended linearly to chains: in barycentric coordinates, define projections~$p_\lrcorner \colon \Delta^{i+j} \to \Delta^i$ by~$[v_0, \ldots, v_{i+j}] \mapsto [v_0, \ldots, v_i]$ and~$p_\llcorner \colon \Delta^{i+j} \to \Delta^j$ by~$[v_0, \ldots, v_{i+j}] \mapsto [v_i, \ldots, v_{i+j}]$. Now define~$\lrcorner s_{i+j} = s_{i+j} \circ p_\lrcorner$ and~$\llcorner s_{i+j} = s_{i+j} \circ p_\llcorner$.
Given twisted cochains~$f \in C^i(X;M)=\Hom_{\text{right-}\Z[\pi]}(\overline{C_i(\widetilde{X})},M)$ and~$g \in C^j(X;N)= \Hom_{\text{right-}\Z[\pi]}(\overline{C_j(\widetilde{X})},N)$, a chain~$\sigma \in C_{i+j}(\widetilde{X})$, and a twisted chain~$b \otimes \sigma' \in C_k(X;N)$, one sets 
\begin{align*} 
 \big( f\cup g \big) (\sigma)&:=f( \lrcorner \sigma) \boxtimes g(   \llcorner \sigma )
	\in M\boxtimes N ,  \\
 f\cap (b \otimes \sigma' )  &:= \big( f( \lrcorner \sigma') \boxtimes b \big) \otimes  \llcorner \sigma' 
	\in C_{k-i}(X;M\boxtimes N). 
	\end{align*}
A verification shows that~$\cup$ and~$\cap$ define~$(R,R)$-bilinear maps
\begin{align*}
\cup \colon C^i(X;M) \otimes_\Z \overline{C^j(X;N)} \to C^{i+j}(X;M \boxtimes N), \\
\cap \colon C^i(X;M) \otimes_\Z \overline{C_k(X;N)} \to C_{k-i}(X;M \boxtimes N).
\end{align*}
Indeed we have
\begin{align*}
&(rf \cup gs)(\sigma)
=(rf)( \lrcorner \sigma) \boxtimes (\overline{s}g)(   \llcorner \sigma )
=rf( \lrcorner \sigma) \boxtimes \overline{s}g(   \llcorner \sigma )
=r(f( \lrcorner \sigma) \boxtimes g(   \llcorner \sigma ))s
=(r(f \cup g)s)(\sigma), \\
&(rf) \cap (b \otimes \sigma')s
=( rf( \lrcorner \sigma') \boxtimes \overline{s}b ) \otimes  \llcorner \sigma' 
=r\left(\big( f( \lrcorner \sigma') \boxtimes b \big) \otimes  \llcorner \sigma' \right)s.
\end{align*}
\color{black}
It can be checked that~$\cup$ and~$\cap$ are chain maps and descend to (co)homology, leading to the main definition of this subsection. 

\begin{definition} \label{def:TwistedCup}
Let~$X$ be a space with fundamental group $\pi$ and, and let~$M,N$ be~$(R,\Z[\pi])$-bimodules. The chain maps~$\cup$ and~$\cap$ defined above respectively induce the~$(R,R)$-linear \emph{twisted cup product} and \emph{twisted cap product}  
\begin{align*}
\cup \colon H^i(X;M) \otimes_\Z \overline{H^j(X;N)} \to H^{i+j}(X;M \boxtimes N), \\
\cap \colon H^i(X;M) \otimes_\Z \overline{H_k(X;N)} \to H_{k-i}(X;M \boxtimes N).
\end{align*}
\end{definition}

We note that the relation~$\big(\alpha \cup \beta \big) \cap \sigma = \alpha \cap \big( \beta \cap \sigma \big)$ holds between the cap and the cup product, as in the untwisted case~\cite[Theorem~5.2~(3)]{BredonTopology}.

\begin{remark}
\label{rem:RelativeCup}
We recall briefly how cup and cap products extend to pairs, referring to~\cite[Part~XXIV]{FriedlLectureNotes} for more details.
Recall that a triple $(X,Y,Z)$ of spaces is called an \emph{excisive triad} if the inclusion $(Y, Y \cap Z) \to (Y \cup Z,Z)$ induces isomorphisms on homology.
Similarly, following~\cite[Part XXIV]{FriedlLectureNotes}, if $X$ admits a universal cover~$p \colon \widetilde{X}\to X$, we say a triple $(X,Y,Z)$ is \emph{an universally excisive triad} if $(\widetilde{X},p^{-1}(Y),p^{-1}(Z))$ is an excisive triad.
Given such an universally excisive triad, Definition~\ref{def:TwistedCup} extends to a relative cup and cap products
\begin{align*}
\cup \colon H^i(X,Y;M) \otimes_\Z \overline{H^j(X,Z;N)} \to H^{i+j}(X,Y \cup Z;M \boxtimes N), \\
\cap \colon H^i(X,Y;M) \otimes_\Z \overline{H_k(X,Y\cup Z;N)} \to H_{k-i}(X,Z;M \boxtimes N).
\end{align*}
In the sequel, we will mostly be concerned with the case where $Y=\emptyset$ or $Z=\emptyset.$
\end{remark}

Recalling the map isomorphism $ \varpi_0 \colon H_0\big(X;M\boxtimes N \big) \to M \otimes_{\Z[\pi]} \overline{N}$  from Lemma~\ref{lem:Ccoinvariants},  we also note the graded symmetry of the twisted cup product.
\begin{lemma}
\label{lem:CupProductSymmetric}
Let~$(X,Y)$ be a pair with fundamental group $\pi$ and, let~$M$ be an~$(R,\Z[\pi])$-bimodule, let~$\Psi \colon M \otimes_{\Z[\pi]} \overline{M} \to R$ be a hermitian form.
Then for any $a \in H^k(X;M)$, any~$ b \in H^l(X;N)$ and any $\sigma \in H_{k+l}(X,Y)$, the cup product satisfies the following graded symmetry:
$$ \Psi \circ \varpi_0\circ ((b \cup a) \cap \sigma)=(-1)^{kl}\overline{\Psi \circ \varpi_0((a \cup b) \cap \sigma)} \in R.$$
In particular,  for $\Psi=\mult \colon \Z[\pi] \otimes_{\Z[\pi]} \overline{\Z[\pi]} \to \Z[\pi],(x,y) \mapsto x\overline{y}$,  and~$\mult_0:=\mult \circ \varpi_0$ as in Notation~\ref{not:mult0},  this yields
$$\mult_0((b \cup a) \cap \sigma)=(-1)^{kl}\overline{\mult_0(((a \cup b) \cap \sigma)} \in \Z[\pi].$$
\end{lemma}
\begin{proof}
Consider the following diagram of abelian groups: 
$$
\xymatrix{
H^k(X,Y;M) \otimes H^l(X;M)\ar[r]^-{(-1)^{kl}\cup}\ar[d]^{\operatorname{swap}}&
H^{k+l}(X,Y;M \boxtimes M) \ar[r]^-{-\cap \sigma}\ar[d]^{\operatorname{flip}}&
H_0(X;M \boxtimes M)\ar[r]^-{\varpi_0}\ar[d]^{\operatorname{flip}}&
M \otimes \overline{M}\ar[r]^-\Psi\ar[d]^{\operatorname{flip}}&
R \ar[d]^{x \mapsto \overline{x}} \\
H^l(X;M) \otimes H^k(X,Y;M)\ar[r]^-{\cup}&
H^{k+l}(X,Y;M \boxtimes M) \ar[r]^-{-\cap \sigma}&
H_0(X;M \boxtimes M)\ar[r]^-{\varpi_0}&
M \otimes \overline{M}\ar[r]^-\Psi&
R. \\
}
$$
The statement amounts to proving that this diagram commutes.
The left square commutes by~\cite[Twisted Cup Product-Commutativity Proposition]{FriedlLectureNotes}.
The second square from the left commutes by naturality of the cap product.
The commutativity of the third square is readily verified.
The rightmost square commutes because $\Psi$ is hermitian.
\end{proof}
\color{black}

We conclude by noting the relation between cap products and evaluation.

\begin{lemma}
\label{lem:capev}
Let $\pi$ be a group, let $(X,Y)$ be a pair,
and let $\varphi \colon \pi_1(X) \twoheadrightarrow \pi$ be a surjection.
For~$\varphi \in H^k(X,Y;\Z[\pi])$ and $x \in H_k(X,Y;\Z[\pi])$, we have
$$ \mult \circ \varpi_0  (\varphi \cap x)
=\overline{\langle \varphi,x\rangle} \in \Z[\pi].$$
\end{lemma}
\begin{proof}
Without loss of generality we can assume that $x=[m' \otimes \sigma]$.
Writing $f$ for a cocycle representative of $\varphi$, the definition of the cap product and Example~\ref{ex:EvaluationZpi} give
\begin{align*}
\mult \circ \varpi_0(\varphi \cap x)
&=\mult \circ \varpi_0(f(\sigma) \otimes m' \otimes \pt)
=\mult(f(\sigma) \otimes m')
=f(\sigma)\overline{m'}.
=\overline{\ev([f])([m' \otimes \sigma])} \\
&=\overline{\langle \varphi,x\rangle}.
\end{align*}
This concludes the proof of the lemma.
\end{proof}
\color{black}

\subsection{Twisted intersection forms}
\label{sub:hermitianForms}

This section has two goals.
The first is to recall the definition of the equivariant intersection form.
The second, given a 
space~$X$, an epimorphism~$\varphi \colon \pi_1(X) \twoheadrightarrow \pi$ and~$\sigma \in H_4(X)$,
is to collect some properties of the pairing
\begin{align*}
b_\sigma \colon H^2(X;\Z[\pi]) \times H^2(X;\Z[\pi]) &\to \Z[\pi] \\
(x,y)& \mapsto \langle y, x \cap \sigma \rangle.
\end{align*}
Namely we will reformulate this pairing using cup products, show that it is hermitian and relate it to the equivariant intersection form in the case where~$X$ is a Poincar\'e complex and~$\sigma$ is the fundamental class of $X$.

\begin{definition}
\label{def:EquivariantIntersectionForm}
Let~$(X,Y)$ be a~$4$-dimensional Poincar\'e pair with $\pi:=\pi_1(X)$.
\begin{itemize}
\item The \emph{cohomological relative equivariant intersection form} is
\begin{align*}
b_X^\partial  \colon H^2(X,Y;\Z[\pi]) \times H^2(X;\Z[\pi]) & \to  \Z[\pi] \\
(\alpha,\beta) &\mapsto \langle \beta,\PD_X(\alpha) \rangle
=\langle \beta,\alpha \cap [X] \rangle.
\end{align*}
\item The \emph{relative equivariant intersection form} is
\begin{align*}
\lambda_X^\partial  \colon H_2(X;\Z[\pi]) \times H_2(X,Y;\Z[\pi]) & \to  \Z[\pi] \\
(x,y) &\mapsto \langle \PD_X^{-1}(y),x\rangle.
\end{align*}
\end{itemize}
\end{definition}

\begin{remark}
\label{rem:PairingCohomHom}
Observe that Poincar\'e duality induces an isometry between these two pairings in the sense that for~$\alpha \in H^2(X,Y;\Z[\pi])$ and~$\beta \in H^2(X;\Z[\pi])$ one has
$$ \lambda_X^\partial(\PD_X(\alpha),\PD_X(\beta))
=\langle \PD^{-1}(\PD_X(\beta)),\PD_X(\alpha)  \rangle
=\langle \beta,\PD_X(\alpha)  \rangle
=b_X^\partial(\alpha,\beta).$$
\end{remark}

In what follows, it is helpful to recall from Notation~\ref{not:mult0} that we write
$$\mult_0 \colon H_0(X;\Z[\pi] \boxtimes \Z[\pi]) \xrightarrow{\varpi_0}
\Z[\pi] \otimes_{\Z[\pi]} \overline{\Z[\pi]}
\xrightarrow{\mult}
\Z[\pi].
$$
The next lemma expresses the pairing $b_\sigma$ in terms of the cup product.

\begin{lemma}
\label{lem:cup}
Let $\pi$ be a group, let $X$ be a 
space, and let $\varphi \colon \pi_1(X) \twoheadrightarrow \pi$ be a surjection.
The pairing~$b_\sigma$ agrees with the hermitian pairing
\begin{align*}
b_\sigma^\cup \colon H^2(X;\Z[\pi]) \times H^2(X;\Z[\pi]) &\to \Z[\pi] \\
(x,y)& \mapsto \mult_0((x \cup y) \cap \sigma).
 \end{align*}
In particular the pairing~$b_\sigma$ is hermitian.
\end{lemma}
\begin{proof}
The graded symmetry of the cup product (recall Lemma~\ref{lem:CupProductSymmetric}) ensures that $b_\sigma^\cup$ is hermitian.
The equality $b_\sigma=b_\sigma^\cup$ now follows directly from the cup-cap-formula
and Lemma~\ref{lem:capev}:
\begin{align*}
b_\sigma^\cup(x,y)
=\overline{b_\sigma^\cup(y,x)}
=\overline{\mult_0((y \cup x) \cap \sigma)} 
=\overline{\mult_0((y \cap (x \cap \sigma)}
=\langle y,x  \cap \sigma\rangle 
=b_\sigma(x,y).
\end{align*}
The fact that $b_\sigma$ is hermitian follows from the fact that $b_\sigma^\cup$ is hermitian.
\color{black}
\end{proof}


The same proof yields the following result for Poincar\'e pairs.
\begin{lemma}
\label{lem:cupRelative}
Let $(X,Y)$ be a $4$-dimensional Poincar\'e pair with $\pi:=\pi_1(X)$.
The pairing~$b_X^{\partial}$ agrees with the pairing
\begin{align*}
b_X^{\partial,\cup} \colon H^2(X,Y;\Z[\pi]) \times H^2(X;\Z[\pi]) &\to \Z[\pi] \\
(x,y)& \mapsto \overline{\mult(\varpi_0((y \cup x) \cap [X]))}.
 \end{align*}
\end{lemma}
\begin{proof}
This follows directly from the relation between the cap and cup products and Lemma~\ref{lem:capev}
\begin{align*}
b_X^{\partial,\cup}(x,y)
&=\mult_0((y \cup x) \cap [X]) 
=\mult_0((y \cap (x \cap [X]))
=\overline{\langle y, x \cap [X]\rangle} \\
&=\overline{b_M^{\partial}(x,y)}.
\end{align*}
 This concludes the proof of the lemma.
 \color{black}
\end{proof}

\subsection{Twisted cross products}
\label{sub:Cross}

Let~$X_1,X_2$ be spaces, let $R_1,R_2,$ and $R$ be rings.
Let~$M$ be a~$(R_1,\Z[\pi_1(X_1)] \times R)$-bimodule, and let~$N$ be a~$(R,\Z[\pi_1(X_2)] \times R_2)$-bimodule.
Identify the ring~$\Z[\pi_1(X_1) \times \pi_1(X_2)]$ with~$\Z[\pi_1(X_1)] \otimes_\Z \Z[\pi_1(X_2)]$ and endow the~$(R_1,R_2)$-bimodule~$M \otimes_\Z N$ with the right~$\Z[\pi_1(X_1) \times \pi_1(X_2)]$-diagonal action given by~$(m_1,m_2) \cdot (\gamma_1 \otimes \gamma_2)=(m_1\gamma_1,m_2\gamma_2)$ so that it becomes a~$(R_1,\Z[\pi_1(X_1) \times \pi_1(X_2)] \times R_2)$-bimodule.
Recall that the Eilenberg-Zilber theorem ensures there are chain maps
\begin{align*}
\times' \colon  C_*(\widetilde{X}_1) \otimes C_*(\widetilde{X}_2) \to C_*(\widetilde{X}_1 \times \widetilde{X}_2) \\   
\theta \colon  C_*(\widetilde{X}_1 \times \widetilde{X}_2) \to C_*(\widetilde{X}_1) \otimes C_*(\widetilde{X}_2) 
\end{align*}
which are unique up to chain homotopy, are natural in $X$ and $Y$, and that are chain homotopy inverses of one another.
We now begin the construction of the cross products.

\begin{construction}[Twisted cross products on chain complexes and homology]
\label{cons:CrossHomology}
A \emph{cross product} on the twisted chain complexes is the $(R,R)$-linear chain homotopy equivalence with
\begin{align*}
	\times \colon C_k(X_1;M) \otimes_\Z \overline{C_l(X_2;N)} &\to C_{k+l}(X_1\times X_2;M \boxtimes N)\\
	\Big( a \otimes_{\Z[\pi_1(X_1)]} \wt \sigma_{X_1}, b \otimes_{\Z[\pi_1(Y)]} \wt \sigma_{X_2} \Big) &\mapsto \big(a \boxtimes b \big) \otimes_{\Z[\pi_1(X_1\times X_2)]} \widetilde{\sigma}_{X_1} \times' \widetilde{\sigma}_{X_2}.
\end{align*}
To see that $\times$ is indeed a chain homotopy equivalence, express it as the composition
\begin{align*}
\times \colon C_k(X_1;M) \otimes_\Z C_l(X_2;N) 
&=\left( M \otimes_\Z C_*(\widetilde{X}_1)\right) \otimes_\Z \left( N \otimes_\Z C_*(\widetilde{X}_2)\right)  \\
&\cong 
(M \otimes_\Z N) \otimes_{\Z[\pi_1(X_1) \times \pi_1(X_2)]}  
C_k(\widetilde{X}_1) \otimes C_l(\widetilde{X}_2) \\
&\xrightarrow{\id \otimes \times' } 
(M \boxtimes_R N) \otimes_{\Z[\pi_1(X_1 \times X_2)]} 
C_{k+l}(\widetilde{X}_1 \times \widetilde{X}_2) \\
&=C_{k+l}(X_1 \times X_2;M \boxtimes_R B).
\end{align*}
Since $\times'$ is unique up to chain homotopy, it then makes sense to define the \emph{twisted cross product} to be the~$(R,R)$-linear map induced by $\times$ on homology:
$$ \times \colon H_k(X_1;M) \otimes_\Z \overline{H_l(X_2;N)} \to H_{k+l}(X_1 \times X_2;M \boxtimes N).$$
\end{construction}
\begin{construction}[Twisted cross products on cochain complexes and cohomology.]
\label{cons:Crosscohomology}
A \emph{cross product} on twisted cochain complexes is the $(R,R)$-linear chain homotopy equivalence
\begin{align*}
C^k(X_1;M) \otimes_\Z C^l(X_2;N) 
&=\Hom_{\Z[\pi_1(X_1)]}\left(\overline{C_k(\widetilde{X}_1)},M\right) \otimes_\Z\Hom_{\Z[\pi_1(X_2)]}\left(\overline{C_l(\widetilde{X}_2)},N\right) \\
&\xrightarrow{h,\simeq}
\Hom_{\Z[\pi_1(X_1) \times \pi_1(X_2)]}\left(\overline{C_k(\widetilde{X})\otimes C_l(\widetilde{X}_2)},M \boxtimes N\right)  \\
&\xrightarrow{\theta^*,\simeq}
\Hom_{\Z[\pi_1(X_1 \times X_2)]}\left(\overline{C_{k+l}(\widetilde{X}_1\times \widetilde{X}_2)},M \boxtimes N\right) \\
&=C^{k+l}(X_1 \times X_2;M \boxtimes N),
\end{align*}
where the map $h$ is induced by mapping $f \otimes g$ to~$(\sigma_X \otimes \sigma_{X_2}) \mapsto f(\sigma_X) \otimes g(\sigma_{X_2}).$
To see that $h$ is indeed a chain homotopy equivalence,
consider the following commutative diagram in which the horizontal maps arise from the fact that $X_1,X_2$ have the homotopy type of finite CW complexes, and the map $h^{\CW}$ is defined similarly to $h$ but using the cellular chain complex:
$$
\xymatrix@C0.4cm{
\Hom \left(\overline{C_k(\widetilde{X}_1)},M\right) \otimes_\Z\Hom \left(\overline{C_l(\widetilde{X}_2)},N\right)
\ar[r]^-{\simeq} \ar[d]^{h}
&
\Hom \left(\overline{C_k^{\CW}(\widetilde{X}_1)},M\right) \otimes_\Z\overline{\Hom(C_l^{\CW}(\widetilde{X}_2),N)} \ar[d]_\cong^{h^{\CW}}
\\
\Hom \left(\overline{C_k(\widetilde{X}_1)\otimes C_l(\widetilde{X}_2)},M \boxtimes N\right)
\ar[r]^-{\simeq}
&
\Hom \left(\overline{C_k^{\CW}(\widetilde{X}_1)\otimes C_l^{\CW}(\widetilde{X}_2)},M \boxtimes N\right).
}
$$
Since $\theta$ is unique up to chain homotopy, it then makes sense to define the \emph{twisted cross product} to be $(R,R)$-linear map induced by $\times$ on cohomology:
$$\times \colon H^k(X_1;M) \otimes_\Z \overline{H^l(X_2;N)} \to H^{k+l}(X_1 \times X_2 ;M \boxtimes N).$$
\color{black}
\end{construction}

Next we recall the statement of the K\"unneth theorem, assuming for simplicity that~$R=\Z$.

\begin{proposition}
\label{prop:KunnethAppendix}
Assume that $X_1,X_2$ are finite CW complexes and that $M$ and $N$ are free as $\Z$-modules.
The cross product gives rise to the following split short exact sequences of $(R_1,R_2)$-bimodules:
\begin{align*}
0 \to \bigoplus_{i=1}^n H_i(X_1;M) \otimes_\Z H_{n-i}(X_2;N) \xrightarrow{\times} H_n(X_1 \times X_2;M \boxtimes N) \to  \bigoplus_{i=1}^n \operatorname{Tor}_1^R(H_i(X_1;M), H_{n-i-1}(X_2;N)) \to 0 \\
0 \to \bigoplus_{i=1}^n H^i(X_1;M) \otimes_\Z H^{n-i}(X_2;N) \xrightarrow{\times} H^n(X_1 \times X_2;M \boxtimes N) \to  \bigoplus_{i=1}^n \operatorname{Ext}^1_R(H^i(X_1;M),H^{n-i-1}(X_2;N)) \to 0.
\end{align*}
\end{proposition}
\begin{proof}
In homology this follows by combining the Künneth theorem for chain complexes (which applies since both $M$ and $N$ are free as $\Z$-modules) and the definition of the cross product from Construction~\ref{cons:CrossHomology}:
$$
\xymatrix@C0.3cm{
0\ar[r]
&\bigoplus\limits_{i=1}^n H_i(X_1;M) \otimes_\Z H_{n-i}(X_2;N)
\ar[r] \ar[rd]^-{\times}
&H_n\left(
C_*(X;M) \otimes_\Z \overline{C_*(X_2;N)}\right)\ar[r]\ar[d]^-{\times}_\cong
&  \bigoplus\limits_{i=1}^n \operatorname{Tor}_1^R(H_i(X;M) \otimes_\Z H_{n-i-1}(X_2;N)\ar[r] 
&0\\
 &&{\underbrace{H_n\left(C_*(X_1 \times X_2;M \boxtimes N)\right).}_{=H_*(X_1\times X_2;M \boxtimes N)}}&&
}
$$
\color{black}
In cohomology this follows by combining the Künneth theorem for cochain complexes (which applies since both $M$ and $N$ are free as $R$-modules) and the definition of the cross product from Construction~\ref{cons:CrossHomology}:
$$
\xymatrix@C0.3cm{
0\ar[r]
&\bigoplus\limits_{i=1}^n H^i(X_1;M) \otimes_\Z H^{n-i}(X_2;N)
\ar[r] \ar[rd]^-{\times}
&H_n\left(
C^*(X_1;M) \otimes_\Z \overline{C^*(X_2;N)}\right)\ar[r]\ar[d]^-{\times}_\cong
&  \bigoplus\limits_{i=1}^n \operatorname{Ext}^1_R(H_i(X_1;M) \otimes_\Z H_{n-i-1}(X_2;N)\ar[r] 
&0\\
 &&{\underbrace{H_n\left(C^*(X_1 \times X_2;M \boxtimes N)\right).}_{=H^*(X_1\times X_2;M \boxtimes N)}}&&
}
$$
This concludes the proof of the proposition.
\color{black}
\end{proof}

We conclude with a remark that was used during the proofs of Lemma~\ref{lem:FirstTerm} and Proposition~\ref{prop:CanRealise}.
\begin{remark}
\label{rem:CrossLemma}
Let $Y$ be a space, let $\pi_1(Y) \to \pi$ be an epimorphism, and endow $Y \times S^1$ with coefficient system $\pi_1(Y \times S^1) \to \pi_1(Y) \to \pi.$
Under the identification
$$ H_0(Y;\Z[\pi] \boxtimes \Z[\pi]) \times H_0(S^1) \xrightarrow{\times} H_0(Y \times S^1;(\Z[\pi] \boxtimes \Z[\pi]) \boxtimes \Z) \cong H_0(Y \times S^1,\Z[\pi] \boxtimes \Z[\pi]),$$
the following equality holds:
$$\mult_0([(\gamma_1 \otimes \gamma_2) \otimes [\operatorname{pt}]] \times n1_{H_0(S^1)})
=n\gamma_1\overline{\gamma}_2
$$
for any~$\gamma_1,\gamma_2\in \Z[\pi]$ and~$n\in \Z$.
\end{remark}

\section{$3$-connected maps and isometries of intersection forms}

The goal of this section is to prove Proposition~\ref{prop:RelativeFormsCohomologyAreTheSame} that was used during the proofs of Propositions~\ref{prop:Simplify} and~\ref{prop:3ConnImpliesCompatible}.
This requires the following lemma which describes the effect of applying duality to the absolute class instead of the relative class.

\begin{lemma}
\label{lem:otheradjoints}
    Let $(X,Y)$ be a $4$-dimensional Poincar\'e pair with $\pi:=\pi_1(X)$.
The pairing~$b_X^{\partial}$ agrees with the pairing
\begin{align*}
    H^2(X,Y;\Z[\pi]) \times H^2(X;\Z[\pi]) &\to \Z[\pi] \\
(\alpha,\beta)& \mapsto \overline{\langle \alpha,\PD_X(\beta)\rangle}.
\end{align*}
while the pairing $\lambda_X^\partial$ agrees with the pairing
\begin{align*}
    H_2(X;\Z[\pi]) \times H_2(X,Y;\Z[\pi]) &\to \Z[\pi] \\
(x,y)& \mapsto \overline{\langle \PD_X^{-1}(x),y\rangle}.
\end{align*}
\end{lemma}
\begin{proof}
The proof follows promptly from Lemmas~\ref{lem:cupRelative} and \ref{lem:CupProductSymmetric} and~\ref{lem:capev}
Indeed, by \cref{lem:cupRelative} and \cref{lem:CupProductSymmetric},
    \[b_X^{\partial}(\alpha,\beta)=\overline{\mult_0((\beta \cup \alpha) \cap [X])}=\mult_0((\alpha \cup \beta) \cap [X])=\mult_0(\alpha \cap (\beta \cap [X])).\]
    By \cref{lem:capev},
    \[\mult_0(\alpha \cap (\beta \cap [X]))=\overline{\langle \alpha,\beta\cap [X]\rangle}=\overline{\langle \alpha,\PD_X(\beta)\rangle}.\]
    This completes the proof for $b_X^\partial$. Since $\PD_X$ is an isomorphism, there exist $\alpha,\beta$ with $x=\PD_X(\alpha)$ and $y=\PD_X(\beta)$. We then compute using \cref{rem:PairingCohomHom},
    \[\lambda_X^\partial(x,y)=\lambda_X^\partial(\PD_X(\alpha),\PD_X(\beta))=b_X^\partial(\alpha,\beta).\]
    The first part of the lemma now gives
    \[b_X^\partial(\alpha,\beta)=\overline{\langle \alpha,\PD_X(\beta)\rangle}=\overline{\langle \PD_X^{-1}(x),y\rangle}.\]
    This completes the proof of the lemma.
    \color{black}
\end{proof}

\begin{notation}
Let~$(X_0,Y_0)$ and~$(X_1,Y_1)$ be~$4$-dimensional Poincar\'e pairs with~$\pi_1(X_j) \cong \pi$,
Postnikov $2$-type $B:=P_2(X_1)$ and~$\iota_j \colon \pi_1(Y_j) \to \pi_1(X_j)$ surjective for~$j=0,1$.
Fix a homotopy equivalence~$h \colon Y_0 \to Y_1$ and assume there is a group isomorphism~$u \colon \pi_1(X_0) \to \pi_1(X_1)$ that satisfies~$u \circ \iota_0=\iota_1 \circ h_*$.
Let~$c_0 \colon X_0 \to B,c_1 \colon X_1 \to B$ be
$3$-connected maps with $c_1$ the inclusion.
For elements~$x\in H_2(B;\Z[\pi])$ and~$y\in H_2(B, Y_1;\Z[\pi])$, we write
$$(c_i)_*^{-1}\lambda_i(x,y):=\lambda_{X_i}^\partial((c_i)_*^{-1}(x),(c_i)_*^{-1}(y)).$$
\end{notation}

\begin{proposition}
\label{prop:RelativeFormsCohomologyAreTheSame}
The following assertions are equivalent:
\begin{enumerate}
\item The pairings~$(c_0)_*^{-1}\lambda_{X_0}^\partial,(c_1)_*^{-1}\lambda_{X_1}^\partial \colon  H_2(B;\Z[\pi]) \times  H_2(B,Y_1;\Z[\pi]) \to \Z[\pi]$ are equal,
\item The pairings~$c_0^*b_{X_0}^\partial,c_1^*b_{X_1}^\partial \colon H^2(B,Y_1;\Z[\pi]) \times  H^2(B;\Z[\pi]) \to \Z[\pi] $ are equal.
\end{enumerate}
\end{proposition}
\begin{proof}
Consider the maps~$H^2(B,Y_1;\Z\pi)\to H_2(B;\Z\pi)$ and $H^2(B;\Z\pi)\to H_2(B,Y_1;\Z\pi)$ obtained by capping with~$(c_i)_*([X_i])$ for $i=0,1$.
Denote both of these maps by~$\PD_{cX_i}$.
Observe that~$\PD_{cX_i}=(c_i)_* \circ \PD_{X_i} \circ c_i^*$ implies both that~$\PD_{cX_i}$ is an isomorphism for $i=0,1$ and that the equality~$(c_0)_*^{-1}\lambda_{X_0}^\partial=(c_1)_*^{-1}\lambda_{X_1}^\partial$ is equivalent to the following equality holding for all~$x\in H_2(B,Y_1;\Z\pi)$ and~$y\in H_2(B;\Z\pi)$:
	\[\langle \PD_{cX_0}^{-1}(x),y\rangle=\langle \PD_{cX_1}^{-1}(x),y\rangle.\]
	Since~$\PD_{cX_i}$ is an isomorphism, this is equivalent to the following holding for all~$x'\in H^2(B;\Z\pi)$ and~$y'\in H^2(B,Y_1;\Z\pi)$:
	\[
\overbrace{\langle \PD_{cX_0}^{-1} \circ \PD_{cX_0}(x'),\PD_{cX_1}(y')\rangle}^{=\langle x',\PD_{cX_1}(y')\rangle}
	=\langle \PD_{cX_1}^{-1}\circ \PD_{cX_0}(x'),\PD_{cX_1}(y')\rangle.\]
	Using Lemma~\ref{lem:otheradjoints} twice, first for~$c_1^*b_{X_1}^\partial$, then for~$c_0^*b_{X_0}^\partial$,  the right hand side can be rewritten as 
	\[\langle \PD_{cX_1}^{-1} \circ \PD_{cX_0}(x'),\PD_{cX_1}(y')\rangle=\overline{\langle y',\PD_{cX_0}(x')\rangle}=\langle x',\PD_{cX_0}(y')\rangle.\]
	Thus~$(c_0)_*^{-1}\lambda_{X_0}^\partial=(c_1)_*^{-1}\lambda_{X_1}^\partial$ is equivalent to the following equality holding for all~$x'\in H^2(B;\Z\pi)$ and~$y'\in H^2(B,Y_1;\Z\pi)$:
	\[\langle x',\PD_{cX_1}(y')\rangle=\langle x',\PD_{cX_0}(y')\rangle.\]
By definition, this is equivalent to~$c_0^*b^\partial_{X_0}=c_1^*b^\partial_{X_1}$.
\end{proof}

\section{Obstruction theoretic results}
\label{sec:ObstructionTheory}

The goal of this section is to prove results in obstruction theory that were used in Sections~\ref{sec:NoPostnikov} and~\ref{sec:ProofGoal}.
Whereas these these results can probably be obtained using~\cite{Baues-obstruction-theory}, we prefer to give hands on proofs.
In what follows, given a CW pair~$(X,Y)$, we write~$j \colon Y \to X$ for the inclusion.
Also, we write $f_0^*\pi_2(B)$ for the abelian group~$\pi_2(B)$ considered as a $\Z[\pi_1(X_0)]$-module via~$\gamma \cdot x:=f_*(\gamma)x.$

\begin{proposition}
\label{prop:ObstructionTheory}
Let~$(X,Y)$ and $(B,B')$ be pairs of spaces that are homotopy equivalent to CW pairs,  and let~$f_0 \colon (X,Y) \to (B,B')$ be a map.
Assume that $\pi_i(B)=0$ for $i>2$ and set~$\pi:=\pi_1(B)$.
Suppose furthermore that $Y \to X$ is a cofibration.
For every class~$[d] \in H^2(X,Y;f_0^*\pi_2(B))$, there exists a map~$f_1 \colon X \to B$ satisfying~$(f_0)_*=(f_1)_*$ on~$\pi_1$ and~$f_0|_Y=f_1|_Y$,
and making the following diagram commute:
\[\begin{tikzcd}
    H_2(X;\Z[\pi])\ar[r]\ar[d,"{(f_0)_*-(f_1)_*}"']&H_2(X,Y;\Z[\pi])\ar[d,"{(f_0)_*-(f_1)_*}"]\ar[dl,"{\ev([d])}"']\\
    H_2(B;\Z[\pi])\ar[r]&H_2(B,B';\Z[\pi]).
\end{tikzcd}\]
\end{proposition}
\begin{proof}
During this proof, all chain complexes and homology groups of~$B$ and~$(B,B')$ (resp.~$X$ and~$(X,Y)$) are taken with~$\Z[\pi]$ coefficients (resp. ~$f_0^*\Z[\pi]$ coefficients).
Also, a short verification (which uses that $Y \to X$ is a cofibration) shows that the general case follows from the case where~$(X, Y)$ and~$(B,B')$ are CW pairs and $f_0$ is cellular.
We now assume this is the case and work with cellular chain complexes and cellular homology.
%

Let~$C_2(X,Y)\xrightarrow{d} \ker d_2^B/\im d_3^B \cong \pi_2(B)$ be a cocycle representing~$[d]$.
Since~$C_2(X,Y)$ is free, we can pick a lift~$\wh d\colon C_2(X,Y)\to \ker d_2^B$ of~$d$. 
Since~$d$ is a cocycle,~$d\circ d_3^{X,Y}=0$. Hence since~$C_3(X,Y)$ is free, there exists a map~$e\colon C_3(X,Y)\to C_3(B)$ such that~$d_3^B\circ e=\wh d\circ d_3^{X,Y}$. 

Let~$\proj_i\colon C_i(X)\to C_i(X)/C_i(Y)=C_i(X,Y)$ be the projection and consider the chain map
\[\begin{tikzcd}
	C_3(X)\ar[r,"d_3^X"]\ar[d,"(f_0)_3+e\circ\proj_3"']&C_2(X)\ar[r,"d_2^X"]\ar[d,"(f_0)_2+\wh d\circ\proj_2"]&C_1(X)\ar[r,"d_1^X"]\ar[d,"(f_0)_1"]&C_0(X)\ar[d,"(f_0)_0"]\\
	C_3(B)\ar[r,"d_3^B"']&C_2(B)\ar[r,"d_2^B"']&C_1(B)\ar[r,"d_1^B"']&C_0(B).
\end{tikzcd}\]
Here,  commutativity of the left square follows from the definition of~$e$, while commutativity of the middle square follows since~$\wh d$ takes values in~$\ker d_2^B$.
Note for later use that this map induces a chain map~$C_*(X,Y) \to C_*(B,B')$.

The chain map agrees with~$(f_0)_*$ on~$C_*(Y)$,~$C_1(X)$ and~$C_0(X)$.
By \cite[Theorem~16]{WhiteheadCombinatorial2}, the partial realisation~$f_0|_{Y^{(3)}\cup X^{(1)}} \colon Y^{(3)}\cup X^{(1)} \to B$ extends to a realisation~$f_1'\colon X^{(3)}\to B$ of the chain map.
Since~$\pi_i(B)=0$ for~$i>2$, we can extend~$f_1'$ to a map~$f_1\colon X\to B$ with~$f_1|_{Y\cup X^{(1)}}=f_0|_{Y\cup X^{(1)}}$.
By construction~$(f_1)_*\colon C_2(X)\to C_2(B)$ is given by~$(f_0)_2+\wh d\circ \proj_2$ so that the map induced on absolute homology is given by~$(f_0)_2+\ev([d])\circ j_*$ 
while the map on relative cohomology is given by~$(f_0)_2+j_*\circ \ev([d])$ as needed. 
The condition~$f_1'|_{Y\cup X^{(1)}}=f_0|_{Y\cup X^{(1)}}$ ensures that~$(f_1)_*=(f_0)_*$ on fundamental groups. 
\end{proof}

The following result generalises the commutativity of the upper triangle in Proposition~\ref{prop:ObstructionTheory}.
\begin{theorem}
\label{thm:ObstructionTheoryTopTriangle}
Let~$(X,Y)$ be a pair of spaces that is homotopy equivalent to a CW pair, and let~$f_0 \colon X \to B$ be a~$\pi_1$-isomorphism where $\pi_i(B)=0$ for $i>2$.
Set~$\pi:=\pi_1(B)$.
Given 
$$\psi\colon H_2(X,Y;\Z[\pi])\to H_2(B;\Z[\pi]),~$$ 
there exists a map~$f_1 \colon X \to B$ satisfying~$(f_1)_*=(f_0)_*$ on~$\pi_1$,~$f_1|_Y=f_0|_Y$,  and such that 
$$  (f_0)_*-(f_1)_*=\psi \circ j\colon  H_2(X;\Z[\pi]) \to H_2(B;\Z[\pi]).$$ 
\end{theorem}
\begin{proof}
Consider Theorem~\ref{thm:RealiseAlgebraic3Type} with~$(X_0,Y_0)=(X,Y), (X_1,Y_1)=(B,B)$ as well as~$u=(f_0)_*,h=f_0|_Y$ and~$F=(f_0)_*$.
Since~$B$ is~$3$-coconnected,  applying this result with~$c_1=\id$ then implies that~$(f_0)_*(k_{X_0,Y_0}^{\nu_0})=f_0|_Y^*(k_{B,B}^{\nu_B})$ for any choice of maps~$\nu_0 \colon X_0 \to B\pi$ and~$\nu_B \colon B \to B\pi$.
As~$j_*(k_{X_0,Y_0}^{\nu_0})=0$ (see e.g.~\cite[Lemma 6.1]{ConwayKasprowskiKinvariant}), we deduce~$((f_0)_*-\psi\circ j_*)(k_{X_0,Y_0}^{\nu_0})=f_0|_Y^*(k_{B,B}^{\nu_B})$.
A second application of 
Theorem~\ref{thm:RealiseAlgebraic3Type} (but this time with~$F=(f_0)_*- \psi \circ j_*$) therefore ensures that there exists a map~$f_1 \colon X \to B$ (which is denoted by~$c_0$ in that theorem) with~$(f_1)_*=(f_0)_*- \psi \circ j_*$ and~$f_1|_Y \simeq f_0|_Y$.
Applying the homotopy extension theorem, we can assume~$f_1|_Y= f_0|_Y$.
\end{proof}
\section*{Conflict of Interest and Data Availability Statements}
On behalf of all authors, the corresponding author states that there is no conflict of interest. 

The manuscript has no associated data.
\bibliographystyle{alpha}
\bibliography{BiblioHomotopyBoundary}
\end{document}